\documentclass[11pt,a4paper]{article}
\usepackage[a4paper]{geometry}%for changing page geometry
\usepackage[british]{babel}
\usepackage[utf8]{inputenc}
\usepackage{a4, graphicx,sidecap,splitidx,multicol
}

\usepackage{verbatim} %für \begin{comment}

\usepackage{amsmath,amssymb,amsthm}
\usepackage{algpseudocode,algorithm}
\usepackage{extpfeil}

\usepackage{tikz}
\usetikzlibrary{calc}
\usetikzlibrary{math}
\usetikzlibrary{shapes.geometric}
\usetikzlibrary{patterns}
\usetikzlibrary{arrows.meta}
\makeatletter
\long\def\@makecaption#1#2{%
  \vskip\abovecaptionskip
  \sbox\@tempboxa{
					\colorbox{gray!10}{#1: #2}
					}%
  \ifdim \wd\@tempboxa >\hsize
    \colorbox{gray!10}{\parbox{\dimexpr\linewidth-2\fboxsep-2\fboxrule}{#1: #2}}\par
  \else
    \global \@minipagefalse
    \hb@xt@\hsize{\hfil\box\@tempboxa\hfil}%
  \fi
  \vskip\belowcaptionskip}
\makeatother

\newcommand{\changes}[1]{}

\providecommand{\printonly}[1]{}
\providecommand{\changes}[1]{#1}

\makeindex

\newcommand{\ra}{\rightarrow}

\newcommand{\Ra}{\Rightarrow}

\newcommand{\LR}{\Leftrightarrow}

\newcommand{\ar}[1]{\xrightarrow{#1}}
\newcommand{\arsquare}[1]{\xrightarrow[]{#1\square}}

\newcommand{\Rq}{\Rightarrow\quad}
\newcommand{\LRq}{\Leftrightarrow\quad}

\newcommand{\fa}{\forall}
\newcommand{\ex}{\exists}
\newcommand{\leer}{\emptyset}
\newcommand{\eps}{\varepsilon}
\newcommand{\vi}{\varphi}

\newcommand{\N}{\mathbb{N}}

\newcommand{\Z}{\mathbb{Z}}

\newcommand{\R}{\mathbb{R}}

\newcommand{\inv}{^{-1}}
\newcommand{\id}{\operatorname{id}}

\newcommand{\Int}{\operatorname{int}}

\newcommand{\diam}{\operatorname{diam}}

\newcommand{\aff}{\operatorname{aff}}

\newcommand{\conv}{\operatorname{conv}}

\newcommand{\rel}{\operatorname{rel}}
\newcommand{\relint}{\operatorname{rel\,int}}

\newcommand{\ba}{\begin{align*}}

\newcommand{\m}{Menge }

\newcommand{\T}{\mathcal T}

\newcommand{\subsetmarking}{\tilde}

\newcommand{\Tu}{\operatorname{Tw}}
\newcommand{\tu}{\operatorname{tw}}
\newcommand{\Wald}{\operatorname{fo}}

\newcommand{\parent}{\operatorname{pa}}
\newcommand{\pa}{\parent}

\newcommand{\children}{\operatorname{children}}
\newcommand{\leaves}{\operatorname{leaves}}

\newcommand{\est}{\operatorname{est}}
\newcommand{\Simplexe}{\mathbb T}
\newcommand{\Sub}{\operatorname{Sub}}

\newcommand{\refine}{\operatorname{refine}}
\newcommand{\Vnew}{V_{\operatorname{new}}}
\newcommand{\Eref}{E_{\operatorname{ref}}}
\newcommand{\Vertices}{\mathcal V}
\newcommand{\Edges}{\mathcal E}
\newcommand{\rstr}{\operatorname{rstr}}
\newcommand{\Rstr}{\operatorname{Rstr}}
\newcommand{\Descext}{\operatorname{Desc_{ext}}}
\newcommand{\Cubes}{\mathcal C}
\newcommand{\CoN}{\Cubes_{\frac1N}}
\newcommand{\Help}{\mathcal H}
\newcommand{\Hull}{V}

\renewcommand{\mid}{\operatorname{mid}}

\renewcommand{\S}{\mathcal S}

\newcommand{\emphindex}[1]{\index{#1}\emph{#1}}
\newcommand{\normalindex}[1]{\index{#1}{#1}}

\newcommand{\quasiuniform}[1]{%
for every $j\in\N_{\geq 1}$, %there is an admissible \emph{quasi-uniform refinement} 
the quasi-uniform refinement $\Simplexe^{n,j}%
$ is regular,
}\index{quasi-uniform refinement!IsoCoChange}\index{T^{n,j}@$\Simplexe^{n,j}$ -- quasi-uniform refinement}
\newcommand{\towerstep}[1]{%
for all $S\in\Simplexe$ with $hS > h_0:=2+\lfloor\log_2(n)\rfloor$ and for all $j\in \{h_0{+}1,\allowbreak\dots,\allowbreak hS\}$, there is a vertex $p_j$ of a hyperlevel $(j{-}h_0)$ ancestor of $S$, such that the tower layer $\Tu^j S
$ is included in 
\ifnum#1=1
the union of the simplices of the patch $
\else
$\bigcup
\fi
\Simplexe^{n,j-h_0}p_j$.
\ifnum#1=1

\fi
}
\newcommand{\milestone}[1]{%
Suppose 
ReTaHyCo
for $\T_0$. Then
\begin{enumerate}
\item
\ifnum#1=1{\label{it:quasi-uniform}}\fi
\quasiuniform{#1}
\item
\ifnum#1=1{\label{it:towerstep}\index{h_0@$h_0$}}\fi
\towerstep#1
\end{enumerate}
}

\usepackage{hyperref}

\newcommand{\lemma}{Lemma}

\theoremstyle{plain}
\newtheorem{thm}{Theorem}[section]
\newtheorem{pps}[thm]{Proposition}
\newtheorem{lem}[thm]{\lemma}
\newtheorem{kor}[thm]{Corollary}
\newtheorem{folg}[thm]{Corollary}

\newtheorem{innerlemmanonum}{\lemma}
\newenvironment{lemmanonum}[1]
{\renewcommand\theinnerlemmanonum{#1}\innerlemmanonum}
{\endinnerlemmanonum}

\newtheorem{innercornonum}{Corollary}
\newenvironment{cornonum}[1]
{\renewcommand\theinnercornonum{#1}\innercornonum}
{\endinnercornonum}

\theoremstyle{definition}
\newtheorem{defn}[thm]{Definition}

\newtheorem*{bem}{Remark}
\newtheorem*{bemn}{Remarks}

\theoremstyle{remark}

\newenvironment{dedication}
  {\clearpage           % we want a new page
   \thispagestyle{empty}% no header and footer
   \vspace*{\stretch{1}}% some space at the top 
   \itshape             % the text is in italics
   \raggedleft          % flush to the right margin
  }
  {\par % end the paragraph
   \vspace{\stretch{3}} % space at bottom is three times that at the top
   \clearpage           % finish off the page
  }

\begin{document}
\newgeometry{margin=2.5cm}
\titlepage
\noindent
\begin{minipage}{9cm}
\large{
Humboldt-Universität zu Berlin\\
%Faculty of Mathematics and Natural Sciences\\
Mathematisch-Naturwissenschaftliche Fakultät\\
%Department of Mathematics}
Institut für Mathematik}
\end{minipage}
\hspace{\fill}
	\begin{minipage}{3cm}
		\flushright
		{\includegraphics[width=\linewidth]{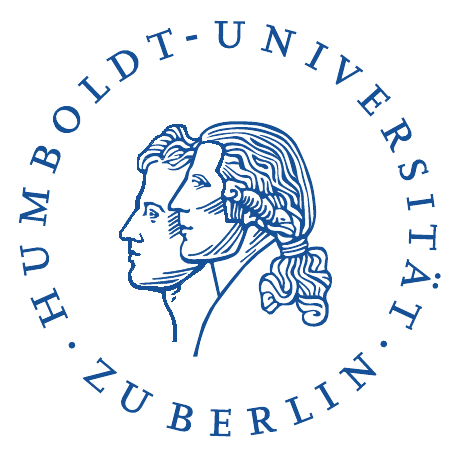}} % Hochschullogo
	\end{minipage}

%		\vspace{2 cm}
		\vspace{1 cm}
\begin{center} 
		\LARGE{The Constant in the Theorem of Binev--Dahmen--DeVore--Stevenson \\and a Generalisation of it}\\
%		\vspace{.6cm}
%		\Large{Die Konstante im Satz von Binev-Dahmen-DeVore-Stevenson \\und eine Verallgemeinerung desselben}
\vspace{.6cm}
\vspace*{\fill}
	
\begin{tikzpicture}[scale=.8]
\def\square{(-2,2) rectangle ++(2,2)	 -- ++(-1,-1) -- ++(-1,+1)};
\def\twosquares{\square[xshift=2cm] \square [xshift=-2cm] };
\foreach \x in {0,...,5} {
\foreach \y in {0,...,3}	{
		\draw[scale=2^(-\x),rotate=90*\y] \twosquares [cm={.5,.5,-.5,.5,(0,0)}] \twosquares;}	}
\end{tikzpicture}
%%anderer Quelltext für das Bild
%\begin{tikzpicture}
%\def\square{(-2,2) rectangle ++(2,2)	 -- ++(-1,-1) -- ++(-1,+1)};
%\def\twosquares{\square[xshift=2cm] \square [xshift=-2cm] [rotate=90]};
%\def\eightsquares{\twosquares\twosquares\twosquares\twosquares [cm={.5,.5,-.5,.5,(0,0)}]};
%\draw \eightsquares\eightsquares\eightsquares\eightsquares\eightsquares\eightsquares;
%\end{tikzpicture}
\end{center}
\vspace*{\fill}

\begin{flushleft}
%\noindent
\begin{minipage}{\textwidth}
\raggedright{
\noindent\makebox[\linewidth]{\rule{\textwidth}{0.4pt}}
\vspace{0cm}
\\
%		\vspace*{\fill}
\noindent\begin{tabular}{ll}
\large{\textbf{Masterarbeit}}&\normalsize zur Erlangung des akademischen Grades		\emph{Master of Science} %(M. sc.)
\vspace*{.25cm}
\\
\emph{eingereicht von} & \text{Lukas Gehring%
,}\\
\emph{geboren} & am 18.10.1992 in Zittau,\\
				& Institut für Mathematik, Friedrich-Schiller-Universität Jena\\
				& Ernst-Abbe-Platz 2, 07743 Jena, Germany,\\
				& lukas.gehring@uni-jena.de
%\emph{in} & \text{Zittau} \\
%\vspace{\baselineskip}
\vspace*{.25cm}\\
%\emph{Supervised by}
\emph{Gutachter/innen} & \text{Prof.~Dr.~Carsten \textsc{Carstensen}} \\ 
				 & \text{Dr.~\textsc{Ma} Rui}
\vspace*{.25cm}\\
\emph{Eingereicht am} & Institut für Mathematik der Humboldt-Universität zu Berlin\\
						& im August~2020.
\vspace*{.25cm}\\
\emph{2020 Mathematics} & Primary 65N50%Mesh generation, refinement, and adaptive methods for boundary value problems involving PDEs
, 05E45%Combinatorial aspects of simplicial complexes
, 52C22%Tilings in n dimensions (aspects of discrete geometry)
, 65D18%Numerical aspects of computer graphics, image analysis, and computational geometry
, 68R10%Graph theory (including graph drawing) in computer science [See also 05Cxx, 90B10, 90C35] Recent zbMATH articles in MSC 68R10 23922
\\
\emph{Subject Classification.}
\vspace*{.25cm}\\
 \emph{Key words} & Standard bisection, newest vertex bisection, mesh refinement, \\
  \emph{and phrases.} & adaptive finite element method, regular triangulations, \\
  						& Binev--Dahmen--DeVore theorem, closure estimate
\end{tabular}%\\
%\vspace*{0cm}
%\centering{\emph{A thesis submitted in partial fulfilment of the requirements for the degree of Master of Science at the Department of Mathematics at Humboldt-Universität zu Berlin.}}
}
\end{minipage}
\end{flushleft}	
\restoregeometry
%%%%%%%%%%%%%%%%%%%%%%%%%%%%%%%%
\printonly{
	\newpage
	\thispagestyle{empty}
	\mbox{}
	}
\newpage
\begin{dedication}
Glory to God alone!
\vspace*{1.5 Em}
\\

For we know in part, \\and we prophesy in part: \\but when that which is perfect is come, \\that which is in part shall be done away. \\(Paul)
\end{dedication}
\printonly{\newpage
	\thispagestyle{empty}
	\mbox{}
	}
\newpage
%\section{Notation}
%A Subset $C$ of an affine space is called convex, if $\fa a,b\in C, \lambda\in[0,1]. \lambda a+(1-\lambda)b\in C$.
%
%% [Wird kaum gebraucht. Kann vielleicht weg.]
%We will distinguish between
%\begin{align*}
%\conv\{p_0,\dots,p_k\}&:=\left\{\sum a_ip_i~\middle|~ a_i\in \left]0,1\right]\right\} \text{ and}\\
%\Conv\{p_0,\dots,p_k\}&:=\left\{\sum a_ip_i~\middle|~ a_i\in \left[0,1\right]\right\}.
%\end{align*}

%\section{Adaptive Mesh Refinement}

%%%%%%%%%%%%%%%%%%%%%%%%%%%%%%%%%%
%\justifying
\begin{abstract}
A triangulation of a polytope into simplices is refined recursively. In every refinement round, some simplices which have been marked by an external algorithm are bisected and some others around also must be bisected to retain regularity of the triangulation.
%This thesis estimates the ratio of the total number of marked simplices the total number of bisected simplices from above. 
The ratio of the total number of marked simplices and the total number of bisected simplices is bounded from above. 
%(2.) 
%More %to the point, 
%precisely, 
Binev, Dahmen and DeVore \cite{BDV} proved under 
%the assumption of 
a certain initial condition 
%that this ratio is bounded by a constant 
a bound that depends only on the initial triangulation. %Here, their theorem is proven in any dimension in a new way leading to a better bound for the ratio. 
This thesis proposes a new way to obtain a better 
%upper 
bound 
%for the ratio 
in any dimension.
%Furthermore it is proved without initial conditions but with the initialisation of some additional structure, the common NVB structure. 
Furthermore, the result is proven for a weaker initial condition, invented by Alkämper, Gaspoz and Klöfkorn \cite{Gaspoz}, who also found an algorithm to realise this condition for any regular initial triangulation. Supposably, it is the first proof for a Binev--Dahmen--DeVore theorem in any dimension with always practically realiseable initial conditions without an initial refinement. Additionally, the initialisation refinement proposed by Kossaczký and Stevenson \cite{Kossaczky,Stevenson} is generalised, and the number of recursive bisections of one single simplex in one refinement round is bounded from above by twice the dimension, sharpening a result of Gallistl, Schedensack and Stevenson \cite{GSS}.
\end{abstract}
\newpage
\tableofcontents
\newpage
\section{Introduction}
\subsection{Motivation}
The questions discussed in this thesis arise from a more famous one: What is the most efficient method to solve a partial differential equation (PDE) numerically? It is one of the fundamental questions of applied mathematics, how to get a most exact possible solution with least possible computation %expenditure.
cost.
%Since complexity theory is a field of mathematics with 
%%as many conjectures as theorems 
%many open conjectures and since there is a large variety of PDEs, a 
A
general answer to this question 
%is not within reach for a long time yet%. But 
likely will not be attainable anytime soon%
, but
within a certain class of finite element methods, the \emph{Adaptive Finite Element Methods (AFEM)}, there are partial answers for these questions, the \emph{theory of optimality} \cite{axioms}.

Finite Element Methods are based on a partition of the domain $\Omega$ of the PDE into pieces of simple shapes as simplices or cuboids. In this thesis, the pieces are simplices. Given such a partition, both the function space where a solution of the PDE is searched and the PDE are \emph{discretised}. The functions considered to be numerical solutions are simple on each simplex, e.g.\ piecewise affine or polynomial functions. Then instead of the PDE, a finite system of equations is solved, but this discretisation process will not be explained in detail here. Numerical solutions are evaluated by an \emph{error estimator }that quantifies how far the numerical solution is from the exact solution. This error estimation is calculated by summing up \emph{residuals} for each simplex. In areas where the exact solution or its derivatives change rapidly, especially near singularities, the simple functions on the simplices cannot approximate it well. To get a good numerical solution, the partition needs to be fine in these areas. Having a fine mesh everywhere would result in high computation cost. Prediction where these areas of oscillation of the exact solution occur may be difficult. A remedy is, to use the numerical solutions on a coarse mesh and special estimators to decide where these areas are and this coarse mesh should be refined. This leads to Adaptive Finite Element Methods which follow the loop: 
%Starting with a coarse triangulation, 
%%Adaptive Finite Element Methods
%AFEMs
%% work as follows: Starting with a coarse mesh a 
%follow the loop:

\begin{figure*}[!h]
\centering{
\includegraphics*[width=\textwidth]{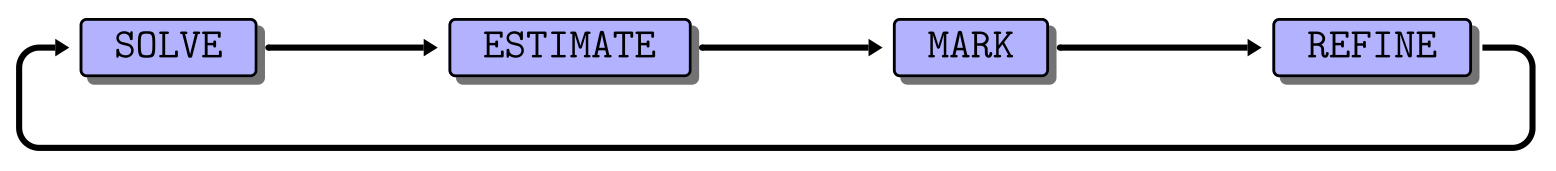}
\small{(Figure from \cite{axioms} by courtesy of C.~Carstensen)}
}
\end{figure*}
 SOLVE the discrete PDE on a coarse mesh numerically, ESTIMATE the contribution of each element to the error, MARK the elements with large estimates, REFINE the mesh dividing at least the marked elements, before starting the next loop, until the error becomes as small as desired.
In this thesis, the solution and estimation step are not discussed at all, so the marking algorithm is supposed 
%as a black box.
to be unpredictable.

%Certain assumptions on the mesh refinement are crucial for the theory.
A critical part of 
%that theory 
the theory of optimality
consists of 
%statements 
%assumptions
a theory
about mesh refinement.
%\subsection{Properties}
The refinement algorithm should satisfy three properties:
\begin{itemize}
\item
%\textbf{conformity} (in German \emph{Regularität, reguläre Triangulierungen}): 
\textbf{Regularity} %\index{regularity} 
(also 
%called 
known as
\emph{conformity }or interdiction of \emph{hanging nodes}): 
The intersection of two 
%triangles (
simplices
%) 
has to be a common subsimplex, i.e.\ the empty set, or exactly one vertex, or a whole edge, etc. 
%(reasons for this postulation: For regular $P_1$ methods: Adding hanging nodes may not enlarge the discretisation space. For Crouzeix-Raviart methods: The discrete functions shall be continuous in the edge mids. Where is the mid of an edge which is not shared by the two neighbouring triangles? Definitions had to be generalized. dPG methods? In general hard theory.)
\item
%%Problem: shape regularity kann auch bei beschränktem minimal angle schlecht werden. man müsste determinante oder raumwinkel ... also lieber gleich volumen
%\textbf{Minimal/maximal angle conditions}: There should be an $\eps>0$ such that all appearing interior angles (in all triangulations) are smaller than $\pi-\eps$ and greater than $\eps$ respectively. 
\textbf{Shape regularity}: There should be a lower bound for $|S|/\diam (S)^n$ for all appearing simplices $S$ of all appearing triangulations.
%(reasons: see CPDE1. Needed 
(This is needed 
for convergence/good convergence rates etc. 
%Very obtuse angles 
Small values for $|S|/\diam (S)^n$
may cause high gradients in interpolation.)
\item
\textbf{Locality}:
If a simplex is divided, usually some additional simplices around have to be divided, to retain conformity. However, this should not lead to a uniform refinement of the whole mesh, these additional divided simplices should lie \emph{around} the first one. 
\end{itemize}
This thesis exclusively treats the bisection method by Kossaczk\'y and Mau\-bach \cite{Kossaczky, Maubach} which ensures 
%the minimal/maximal angle condition.
shape regularity.
%, so dividing a simplex always means bisecting it.
%%Stimmt ja nicht mehr!!!

During the refinement process, 
%there are some simplices, we \emph{want} to subdivide and others we \emph{must} subdivide 
some simplices \emph{should} be subdivided to reduce the residual on them and others \emph{must} be subdivided
to retain regularity. For efficiency of the computation, their \emph{numbers} are important. That is why 
%our 
the 
principal aim 
here 
is to find a bound for the %quotient of the total number of the latter and the total number of the former.
number of the latter in dependence of the number of the former. 
In 2 dimensions, Binev, Dahmen and DeVore \cite{BDV} proved under the assumption of a certain initial condition that 
%this ratio
the ratio of the total number of bisected simplices and the number of marked ones
is bounded by a number dependent only on the initial triangulation (often called a \emph{closure estimate}). This number is called the \emph{BDV constant} here. 
In their original proof, 
%due to Binev, Dahmen and DeVore 
every marked simplex spends dollars to all simplices, which could be possibly created owing to the bisection of the marked one. 
%The 
%%second (and 
%key
%%) 
%step, explaining why %$\sum_{T\in M} \lambda(T',T)\gtrsim 1$
%every cell of the final triangulation gets at least $c>0$ dollars is %a little bit strange. 
%an ingenious proof of an existence theorem -- and so it does not pay any attention to the value coming for the BDV constant.
The ingenious key step, which explains why every cell of the final triangulation gets at least $c>0$ dollars, does not pay attention to tightness of the estimate.
However, 
%Stevenson 
%%found it short enough or just convenient to repeat 
%repeated 
%it unrevised in \cite{Stevenson} where he generalised the result for $n$-dimensional triangulations. While 
%this idea 
it
has been used and modified by several authors% \cite{Stevenson,Karkulik,Holst}, while 
: Stevenson \cite{Stevenson} generalised the BDV result for the bisection method described by Kossaczký and Maubach in $n$ dimensions; Karkulik, Pavlicek and Praetorius \cite{Karkulik,KarkulikE} proved it in 2D without an initial condition; and Holst, Licht and Lyu proved in \cite{Holst} a purely combinatorial version of the BDV theorem whose constant depends only on the mesh topology but not on geometric quantities. Meanwhile,
the author of this thesis has not found any new proof neither doing without this key step nor targeted on reduction of the constant.
 
%Here,
In this work, 
%their 
the BDV
theorem is proven in any dimension in a new way leading to a better bound for the ratio. Furthermore, it is proven without initial conditions 
but with the initialization invented by Alkämper, Gaspoz and Klöfkorn \cite{Gaspoz} instead.

Note the
%The 
paper \cite{Atalay} of Atalay and Mount% should be mentioned
, 
who solved a very similar but different problem. In their problem definition, in every round a simplex is bisected, but the refinement to regularity is done only at the end of the process.

\subsection{Schedule}
This thesis is structured as follows: Section 2 briefly introduces common notions and some notation, Section 3 the bisection method of Kossaczký and Maubach. One aim of Section 4 is to 
%derive
recall
%a 
the
binary forest structure (in terms of graph theory) on the set of simplices which can be generated by the refinement process introduced %by Binev, Dahmen and DeVore 
in \cite{BDV}. 
The second one is to derive that all appearing simplices lie in a finite number of similarity classes.
The 
%other 
third
one is 
%a characterisation of 
to characterise
the set of regular triangulations made of these simplices leading to a lattice structure (in terms of set theory) on these triangulations (i.e.\ a coarsest common refinement function, also known as overlay from \cite{Cascon}).
%a characterisation of the set of regular triangulations which can be generated by the refinement process.
With these preparations, Section~5 is devoted to prove the classical theorem of Binev--Dahmen--DeVore 
(BDV). At first, the idea of the proof is presented 
%using 
for
%a toy problem.
a 1-dimensional analogue, the \emph{pile game}.
%As an 
%first 
%interlude, 
Sections 6 and 7 
%are kind of an interlude and 
do not deal directly with the BDV theorem.
It is not possible to impose the initial condition on an arbitrary initial triangulation. Kossaczk\'y and Stevenson proposed a preparing refinement of an arbitrary mesh 
%providing 
which provides
a mesh satisfying the initial condition, but this preparation %divides the volume of the simplices by $\frac2 (n+1)!$, but
%worsens
deteriorates the shape regularity ${\lvert S\rvert}/{(\diam S)^n}$ by factor $2/(n{+}1)!$. A more general preparation refinement is presented in Section 6.
%Section 6 deals with the preparing refinement to realise the initial condition as a generalisation of the algorithm, Kossaczk\'y \cite{Kossaczky} and Stevenson \cite{Stevenson} gave and 

%%cite? articles are already mentioned in the intro. places in the documents later.

Section 7 generalises
%As an interlude 
a theorem of Gallistl, Schedensack and Stevenson \cite{GSS}. 
%is generalised in Section 6. It states that the maximal number of recursive refinements of one simplex during one refinement step is bounded. 
They proved under the usual initial conditions that in a single refinement step, the number of recursive bisections of one simplex is bounded.
In fact, 
even
weaker 
%assumptions 
initial conditions
already imply that for each positive integer $k$, there always exists a 
\emph{quasi-uniform refinement}, i.e.\ a 
triangulation obtained by the 
%above-mentioned refinement rounds 
regular bisections where all simplices are at least $k{\cdot} n$ 
%and at most 
but less than
$(k{+}1)n$ times bisected recursively, where $n$ is the dimension of the simplices.

%For that reason, the 
To avoid an adverse preparing refinement of the initial triangulation, Al\-käm\-per, Gaspoz and Klöfkorn \cite{Gaspoz} 
%invented a weaker initial condition, which is simply realised on an arbitrary triangulation. For a very similar
weakened the initial condition 
%showed how 
to initialise the adaptive mesh refinement without a preparing refinement. 
%For this modification, the BDV result is proven in Section 8.
Section 8 proves the BDV result for this modification.

Unfortunately, the text has become extensive compared to the model proofs, but a thorough setup of a theory requires some pages. Moreover, it provides some shortcuts.
%%%%%%%%%%%%%%%%%%%%%%%%%%%%%%%%%
\newpage
\section
%{Preliminaries}
{Notation specialities and common notions}\label{sec:preliminaries}
\paragraph{Syntax of quantified expressions.} To separate a logical formula from its quantification, a baseline dot is used here, e.g.\ $\fa n{\in}\N\,\ex m{\in} \N.m>n$.
\paragraph{Round brackets are sometimes left out.}
Functions like $l$, $h$, $\Tu$ are sometimes used without round brackets: $lS,hS,\Tu S$ means the same as $l(S)$, $h(S)$, $\Tu(S)$.
\paragraph{$\boldsymbol{n=\dim \Omega\geq m\geq 0}$.}\index{n@$n$ -- $\dim\Omega$.}\index{m@$m$ -- integer in $\{0,\dots,n\}$}
Throughout this thesis, an $n$-dimensional polytope $\Omega$ is fixed and $n$ 
%is 
mostly denotes
its dimension. As needed for 
%many 
lower-dimensional objects, $m$ will usually be an integer between 0 and $n$.
\paragraph{Cardinality $\boldsymbol{\#}$.}\index{$\#$ -- cardinality}
%Also the symbol $\#$ is used for the cardinality to avoid ambiguity.
The function $\#$ denotes the cardinality of a set.
\paragraph{Lebesgue measure $\boldsymbol{\lvert\bullet\rvert}$.}
\index{$\lvert\bullet\rvert$ -- Lebesgue measure}
The function $\lvert\bullet\rvert$ denotes %both the cardinality of a set and 
the Lebesgue measure of a set.
\paragraph{Affine space.}%\cite{Affiner_Raum}
\index{affine space}
Let be given a set $A$, an $m$-dimensional real vector space $V$ and a mapping $-:A\times A\ra V$, such that the following properties hold true:
\begin{itemize}
\item
$\fa p,q,r\in A.(r{-}q)+(q{-}p)=r{-}p$.
\item
$\fa p\in A, v\in V\ex q\in A.v=q{-}p$.
\end{itemize}
Then the triple $(A,V,-)$ is an affine space. If it is clear, what $V$ and $-$ are, the affine space is sometimes denoted by $A$. If $q-p=v$, one also writes $p+v=q$. Based on these definitions, every expression $\sum_{i=1}^M a_i p_i$ with $a_i\in\R$, $\sum_{i=1}^M a_i=0$ means a unique vector in $V$ and with $\sum_{i=1}^M a_i=1$ a unique point in $A$. The latter is called an \emph{affine combination} of $p_1,\dots,p_M$. The \emph{dimension of the affine space} $A$ is the dimension of its vector space $V$.

\paragraph{Affine hull $\boldsymbol\aff$.}
\index{affine hull}\index{aff@$\aff$ -- affine hull}
The \emph{affine hull} of a subset $S$ of an affine space is
\[
\aff(S):=\left\{\sum_{i=1}^M a_is_i ~\middle|~s_i\in S,\sum_{i=1}^M a_i=1\right\}.
\] 
\paragraph{Convex set, convex hull $\boldsymbol{\conv}$, line segment $\overline{xy}$, dimension of $\sim$.}
\index{convex set}\index{convex hull}\index{conv@$\conv$ -- convex hull}
A subset $C$ of an affine space is \emph{convex}, if for any two points $p,q\in C$, the line segment $\overline{pq}:=p+[0,1](q{-}p)$ is contained in $C$.
The \emph{convex hull} of a subset $S$ of an affine space, i.e.\ the smallest convex set including $S$, is denoted by $\conv(S)$. 
For two points $x,y$ of an affine space, 
$\conv\{x,y\}$ is the line segment $\overline{xy}$. The \emph{dimension of a convex set} is the dimension $\dim(\aff C)$ of its affine hull.
%the line segment $\conv\{x,y\}$ is also denoted by $\overline{xy}$\index{$\overline {xy}$ -- line segment between $x$ and $y$}.
\paragraph{Extreme point.}
\index{extreme point}
A point $x$ of a subset $U$ of an affine space is 
%one of its \emph{extreme points}, 
%%it vermeiden
an \emph{extreme point} of $U$,
%if the only representation of this point as convex combination $\sum a_j p_j$ (with \sum a_j=1) of points $p_j\in U$ is $1x$ itself.
if it cannot be obtained by a convex combination of other points of $U$.
%\begin{defn}[relative interior]
\paragraph{Relative interior.}
\index{relative interior}\index{rel@$\rel\Int$ -- relative interior}
The \emph{relative interior} of a convex set $C$ is its interior within the affine hull of $C$ and denoted by $\rel\Int C$.
%\end{defn}
\paragraph{Face of a convex set.}
\index{face}
%A \emph{face} $F$ of a convex set $C$ is a subset \emph{affinely closed in $C$}, i.e.\ the intersection $C\cap \aff F$ equals $F$, with the additional property that $C\setminus F$ is convex.
A \emph{face} $F$ of a convex set $C$ is a convex subset with the property
$$\fa x,y\in C. x\in F\text{ or }F\cap\relint\overline{xy}=\leer.$$

%Equivalent definitions 
An equivalent definition
for a face 
%and proofs for the following lemmas 
%are 
is 
provided in Definition \ref{def:face} in the appendix.
\paragraph{Simplex, subsimplex, vertex, edge, hyperface, $\boldsymbol{\Sub}$.}
\index{simplex}\index{subsimplex}\index{hyperface}\index{vertex}\index{vertices}\index{edge}\index{simplex@$m$-simplex}\index{Sub@$\Sub$ -- set of the subsimplices}\index{V@$\Vertices$ -- set of the vertices}\index{E@$\Edges$ -- set of the edges}
An $m$-\emph{simplex} 
is the convex hull of $m{+}1$ affinely independent points 
$p_0, \dots, p_m$
in $\R^n$ for some $m\in\{-1,\allowbreak\dots,\allowbreak n\}$, which are its \emph{vertices}.
(Thus, the $(-1)$-simplex is the empty set.) 
It is denoted by $\triangle p_0\dots p_m$\index{$\triangle$ -- simplex spanned by}, alternatively to $\conv\{p_0,\dots,p_m\}$.
A \emph{subsimplex} of it is the convex hull of a subset of these points, a 1-subsimplex is an \emph{edge}, an $(m{-}1)$-subsimplex a \emph{hyperface}. The set of subsimplices of a simplex $S$ is denoted by $\Sub(S)$, the set of its vertices by $\Vertices(S)$ and the set of its edges by $\Edges(S)$.

The faces of a simplex are its subsimplices. (This is Theorem \ref{SeitenUntersimplexe} in the appendix.)

%\begin{defn}[patch, interior of a patch]\label{def:int patch}
\paragraph{Patch.}
\index{patch}\index{T@$\T p$ -- patch of $p$}
The \emph{patch} of a point $p$ in a triangulation $\T$ is the set of all simplices containing this point:
\begin{align*}
\T p:=\{S\in\T~|~p\in S\}.
\end{align*}

\paragraph{Interior of a patch.}
%\index{interior of a patch}
\index{int@$\Int(\T p)$ -- interior of the patch $\T p$}
%\label{def:int patch} 
The \emph{interior of a patch} $\Int (\T p)$ is $\Omega\setminus\bigcup(\T\setminus\T p)$, the complement of the union of all simplices \emph{not} in $\T p$. (Note: While the patch is a set of simplices, its interior is a subset of $\Omega$.)

%\begin{defn}[some graph theory]
%\subsection{Binary Trees}
%In a directed graph, the \emph{outdegree} of a vertex is the number of edges pointing away from it and its \emph{indegree} is the number of edges pointing towards it.

\paragraph{Binary tree.}
\index{binary tree}
A \emph{binary tree} is a directed graph with the following properties:
\begin{itemize}
\item
It has no cycles.
\item
%All vertices have indegree 0 or 1.
For each node, either 0 or 1 edge points towards it.
\item
There is exactly one node, the \emph{root}, such that for every node, there is a (directed) path from the root to it.
\end{itemize}
%\item
%All vertices have outdegree 0 or 2.
Additionally to these axioms, a binary tree is \emph{full}\index{full binary tree}, if for each node, either 0 or 2 edges point away from it.

\begin{figure}[ht]
\centering{
\begin{tikzpicture}[information text/.style={fill=gray!10,inner sep=1ex},
edge from parent/.style={draw,->},
level distance=10mm,
level 1/.style={sibling distance=28mm},
level 2/.style={sibling distance=14mm},
]
\node [rectangle, draw]{root}
	child {node [rectangle,draw] {parent $\pa(S)$}	
		child {node [draw,circle] {}}			
		child {node (S) [draw,diamond] {{\color{white}$S$}}
			child {node [draw,diamond] {$S_1$}
				child {node (S3) [draw,diamond] {}}			
				child {node (S4) [draw,diamond] {}}
			}			
			child {node (S2) [draw,diamond] {\color{white}$S_2$}}
		}
	}
	child {node [coordinate] {}
		child {node [draw,circle] {}}			
		child {node [draw,circle] {}}
	}
	;
\draw (S) node [rectangle, draw]{$S$};
\draw (S2) node [circle, draw]{$S_2$};
\draw (S3) node [circle, draw]{};
\draw (S4) node [circle, draw]{};

\begin{scope}[xshift=3.5cm]
\filldraw[gray!10] (-.8,.5) rectangle (4,-2.3);
\draw node [circle, draw]{};
\draw (1,0) 	node [right] {leaf};
\draw (0,-.6) 	node (Menge) {$\{S_1,S_2\}$};
\draw (1,-.6) 	node [right] {$\operatorname{children}(S)$};
\draw (0,-1.2) 	node [rectangle, draw]{};
\draw (1,-1.2) 	node [right] {ancestor of $S$};
\draw (0,-1.8) 	node [diamond, draw]{};
%\draw (Menge) 	node [right]{$=$};
\draw (1,-1.8) 	node [right] {descendant of $S$};
\end{scope}
\end{tikzpicture}
}
\caption{Terminology of binary tree}
\end{figure}
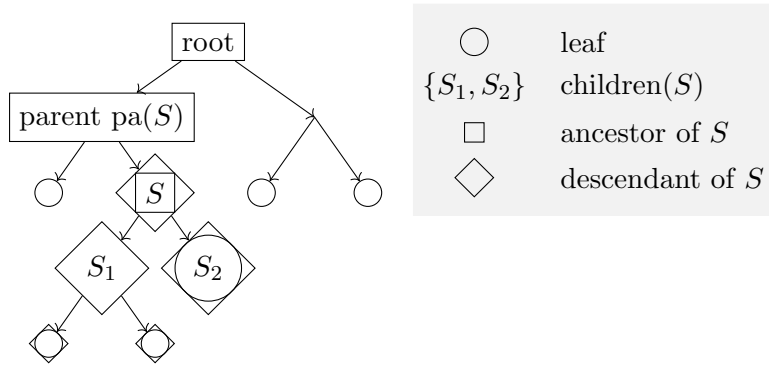

The 3 properties imply that the root is the only node with indegree 0. (If it had an ingoing edge, there would be a cycle. If there would be another such node, there could not be a path from the root to it.)

%A node with outdegree 0 is called \emph{leaf}.\index{leaf} If there is an edge from a node $A$ to a second $B$, we call the first \emph{parent}\index{parent}\index{pa@$\pa$ -- parent} of the second, writing $A=\pa(B)$, and the second \emph{child}\index{child} of the first, denoting the set of the children of $A$ by $\children(A)$.\index{children@$\children$} If there is a path from a node to another (and we rate a single node among the paths), %we call the first \emph{ancestor} of the second and the second \emph{descendant} of the first. As convenient, we add the attribute \emph{proper}, if they are not identical.
%the first is an \emphindex{ancestor} of the second and the second is a \emphindex{descendant} of the first. As convenient, a \emph{proper}\index{ancestor!proper}\index{descendant!proper} ancestor/descendant of a node is an ancestor/descendant of it different from the node itself. 
A node with outdegree 0 is called \emph{leaf}.\index{leaf} If there is an edge from a node $A$ to a second $B$, the first is the \emph{parent}\index{parent}\index{pa@$\pa$ -- parent} of the second, writing $A=\pa(B)$, and the second is a \emph{child}\index{child} of the first. The set of the children of $A$ is denoted by $\children(A)$.\index{children@$\children$} If $A$ has two (distinct) children $B$ and $C$, $B$ is the \emphindex{sibling} of $C$.
If there is a path from a node to another (and a single node counts among the paths), %we call the first \emph{ancestor} of the second and the second \emph{descendant} of the first. As convenient, we add the attribute \emph{proper}, if they are not identical.
the first is an \emphindex{ancestor} of the second and the second is a \emphindex{descendant} of the first. As convenient, a \emph{proper}\index{ancestor!proper}\index{descendant!proper} ancestor/descendant of a node is an ancestor/descendant of it which is different from the node itself. %The \emph{level} of a node is the length of the path from the root to it (the number of its edges).

%An \emph{infinite complete binary tree} is a full binary tree, where every node has outdegree 2. A \emph{binary forest} is a union of disjoint binary trees. 
An \emphindex{infinite perfect binary tree} is a full binary tree without leaves. A \emphindex{binary forest} is a union of disjoint binary trees. 

%Roughly speaking, 
Throughout this thesis,
a forest $F$ is always associated with its set of nodes, $x\in F$ means that $x$ is a node of it. 
%\end{defn}
%%%%%%%%%%%%%%%%%%%%%%%%%%%%%%%%%%%%%%%%%
\newpage
%\section{Maubach's Bisection Algorithm}
%\section{Maubach's bisection algorithm}
\section{The bisection algorithm of Kossaczký and Maubach}
%\begin{wrapfigure}
%%[15]
%{r}[2cm]{0.3\textwidth}
%%  \begin{center}
%%  \label{pic:refinement edge}
%%  \caption{}
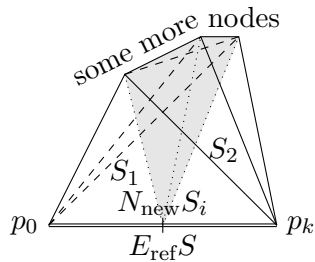
\begin{SCfigure}[1.1][ht]
%\begin{center}
\begin{tikzpicture}[information text/.style={fill=gray!10,inner sep=1ex}]
\coordinate (a) at (0,0);
\coordinate (b) at (3,0);
\coordinate (c) at (1,2);
\coordinate (d) at (2,2.5);
\coordinate (e) at (2.5,2.5);
\coordinate (f) at (1.5,0);

\filldraw[gray!20] (c) -- (d) -- (e) -- (f);
%\draw (0,0) node[anchor=east] {$p_0$}-- node[below]{%refinement edge 
%$E_{\text{ref}}S$} (3,0) node[anchor=west] {$p_k$} -- (1,2)  -- cycle;
\draw[double] (0,0) node[anchor=east] {$p_0$}-- node[below]{$E_{\text{ref}}S$} (3,0) node[anchor=west] {$p_k$};
\draw (3,0) -- (1,2)  -- (0,0);
\draw (3,0) -- (2,2.5) -- node[near end, sloped,above] {some more} (1,2);
\draw[dashed] (0,0) -- (2,2.5);
\draw (3,0) -- (2.5,2.5) --node[near start, sloped,above]{\quad nodes} (2,2.5);
\draw[dashed] (0,0) -- (2.5,2.5);
\draw[dashed] (1,2) -- (2.5,2.5);
%\draw (1,2) --  (2,2.5) -- (2.5,2.5);
\draw (1.5,-.1) -- (1.5,.1);
\draw[dotted] (1,2) -- (1.5,0) node[above] {%new node 
$N_{\text{new}}S_i$} -- (2,2.5);
\draw[dotted] (1.5,0) -- (2.5,2.5);
\draw (1,.7) node {$S_1$};
\draw (2.3,1) node {$S_2$};
%\draw (1.5,0) node[below=5mm,text width=3.5cm,
%information text]
%  {
%  Bisection of a simplex into two: $E_{\text{ref}}S$ is the refinement edge, $N_{\text{new}}S_i$ the new node of the children.
%  };
\end{tikzpicture}
\caption{Bisection of a simplex into two: $E_{\text{ref}}S$ is the refinement edge, $N_{\text{new}}S_i$ the new node of the children.}
%\end{center}
%\end{wrapfigure}
\end{SCfigure}

One can easily show that bisection of a simplex into two only works, if all but two vertices are shared by both parts, and the edge between those two is bisected{}% (see Figure 1)
.

%We want the two simplices to have equal volume, so we will cut this edge always in its centre. 
The question is, which edge should be chosen. %If we bisect the simplex successively, choosing these edges without consideration, we might loose shape regularity (boundedness of smallest and biggest angle from 0 and $\pi$ respectively). On that score, we demand for a simplex some additional structure (this will be the \emph{T-Array}) determining firstly the edge to bisect and secondly the structure of the two parts for recursion.
If a simplex is bisected successively and these edges are chosen without consideration, shape regularity 
%(boundedness of smallest and biggest angle from 0 and $\pi$ respectively) 
%(minimal/maximal angle conditions)
might get lost
as shown in Figure \ref{bad choice good choice}.

\begin{figure}[ht]
\centering{
%%bad
\begin{tikzpicture}[information text/.style={fill=gray!10,inner sep=1ex}]
\foreach \y in {0,1,1.5,1.75,2}
\draw (0,0) -- (2,\y);
\draw (2,0) -- (2,2);
\draw (1,0) node[below=0mm,text width=4cm,
information text]
  {
  Bad choice of bisection edges: long edges, very acute angles
  };
\end{tikzpicture}
%%good
\begin{tikzpicture}[information text/.style={fill=gray!10,inner sep=1ex}, scale=.5]
\draw (4,0) -- (0,0) -- (4,4) -- (4,0) --(2,2) -- (2,0) -- (1,1) -- (1,0);
\draw (2,0) node[below=0mm,text width=4cm,
information text]
  {
  Good choice: Edge length converges to 0, finite number of angles
  };
\end{tikzpicture}
}
\caption{}\label{bad choice good choice}
\end{figure}

To avoid this, some additional structure is demanded for a simplex (this will be the \emph{T-array}) determining firstly the edge to bisect and secondly the structure of the two parts for recursion. 
%Mitchell's 2-dimensional \emph{newest vertex bisection} \cite{Mitchell}, where the additional structure simply consists of tagging a refinement edge which is presented in Subsection \ref{sec:NVB}.
%%In 1992, %Joseph M. 
%Maubach found the subsequently discussed generalisation \cite{Maubach} of the 2-dimensional \emph{newest vertex bisection} \cite{Mitchell}, where the additional structure simply consists of tagging a 
%\emph{newest vertex}
%refinement edge. 
Characteristic of Maubach's algorithm is that this edge is chosen based only on some arrangement of its vertices, while lengths or any geometrical data of the simplex are not taken into account for its choice. (Contrastingly, the so called \emph{longest edge bisection} always chooses the longest edge of a simplex for its bisection. Unexpectedly, the theory of longest edge bisection is not simpler.)
%%%%%%%%%%%%%%%%%%%%%%%%%%%%%%%%%%%%%%%%%%
\begin{defn}[tagged simplex, T-array, type, horizontal and vertical part]
\index{tagged simplex}\index{T-array}\index{type}\index{horizontal part}\index{vertical part}%\label{Maubach}
\index{simplex!tagged}
A \emph{tagged simplex} $S$ is an $m$-simplex %with $m\geq 1$ 
with its vertices $p_0,\dots,p_m$ %of a $m$-dimensional simplex 
arranged into a \emph{T-array} depicted as
\[
\begin{pmatrix}
p_0&\dots& p_{k}\\
&p_{k+1}\\
&\vdots\\
&p_m
\end{pmatrix}
\]
for 
%any 
a variable 
$k\in\{0,\dots,m\}$, called the \index{t@$t$ -- type of a simplex} \emphindex{type} $t(S)$ of the simplex; with the \emph{horizontal part} $\begin{pmatrix}p_0 &\dots& p_k\end{pmatrix}$ in the first row and the \emph{vertical part} $\begin{pmatrix}p_{k+1} &\dots& p_m\end{pmatrix}^T$ below. (The vertical part is not linked to a certain horizontal vertex but should be imagined as centred below the horizontal part.) If the vertical part is empty, i.e.\ if the type equals $m$, the T-array has \emphindex{full type}.
%for any $k\in\{-1,0,\dots,n\}$, the \emph{type} $t(S)$ of the simplex, with the horizontal part $(p_0,\dots,p_k)$ in the first row and the vertical part $(p_{k+1},\dots,p_n)^T$ below. 
%In inexact language, 
Roughly speaking,
the T-array itself is considered as a tagged simplex, and the horizontal and the vertical parts as T-arrays and thereby also as simplices. Relations and operations between sets like $S\subset T$ and the volume $\lvert S\rvert$ always refer to the simplices. For a T-array $S$, the expression $\conv(S)$\index{conv@$\conv$ -- forget the T-array} or the figure of speech “to forget the T-array”\index{forget the T-array} refers explicitly to the simplex \emph{without} the T-array. %Forget the T-array (of a tagged simplex).
%The horizontal and the vertical part are considered as 
\end{defn}

\begin{defn}[Maubach's bisection, transposition, refinement edge $\Eref$, new vertex $\Vnew$]\label{Maubach}
\index{bisection}\index{transposition of a T-array}
%The tagged simplex above is bisected into the two \emph{children}
If the tagged simplex above has type $k\neq 0$, it can be bisected into the two \emph{children}
\[
\begin{pmatrix}
%\qquad&\boldsymbol{p_1}\qquad\dots& p_{n-k}\\
%&\boldsymbol{\frac12 (p_0+p_{n-k})}\\
%&p_{n-k+1}\\
\quad&p_1\quad\dots& p_{k}\\
&\boldsymbol{\frac12 (p_0+p_{k})}\\
&p_{k+1}\\
&\vdots\\
&p_m
\end{pmatrix}
\quad\text{and}\quad
\begin{pmatrix}
p_0&\dots\quad p_{k-1}&\quad\\
&\boldsymbol{\frac12 (p_0+p_{k})}\\
&p_{k+1}\\
%p_0&\dots\qquad\boldsymbol{p_{n-(k+1)}}&\qquad\\
%&\boldsymbol{\frac12 (p_0+p_{n-k})}\\
%&p_{n-k+1}\\
&\vdots\\
&p_m
\end{pmatrix}
\]
(the whitespace at the margin of the horizontal part and the boldface of $\frac12 (p_0+p_k)$ should only emphasize the difference from the parent here),
%so $k$ is decremented and if $k=0$, the child arrays, being column vectors then, are turned back into row vectors (i.e.\ $(p_1, \frac{p_0+p_1}2, p_2,\dots,p_n)$ and $(p_0, \frac{p_0+p_1}2, p_2,\dots,p_n)$) and $k$ jumps onto $n$.
so their type has been decremented. 

If the type $k$ is 0, the T-array, being a column vector, can be \emph{transposed} into a row vector (i.e.\ $\begin{pmatrix}p_0 &\dots& p_m\end{pmatrix}^T\mapsto \begin{pmatrix}p_0 &\dots& p_m\end{pmatrix}$) and $k$ jumps onto $m$. (The geometry remains unchanged during such a transposition.) After transposition, the transposed T-array can be bisected into two like the other T-arrays. A type 0 simplex and its transposed simplex are distinguished only for technical reasons. When T-arrays are not important, they are identified.
%If the type $k=0$, we can turn the one horizontal vertex into a vertical one (still at the top ) the T-array, being a column vector, can be \emph{transposed} into a row vectors (i.e.\ $(p_0, \dots,p_n)$ and $(p_0, \dots,p_n)$) and $k$ jumps onto $n$. (The geometry remains unchanged during such a transposition.) A type 0 simplex and its translated simplex are distinguished only for technical reasons. If T-arrays are not important, they are identified.

Figures, how this bisection “looks like”, are provided in Subsection \ref{RefCo}, together with some theory about the shape of these simplices.

If $k\neq 0$, the edge $\overline{p_0p_{k}}$ is called \emphindex{refinement edge} $\Eref(S)$\index{E_ref@$\Eref$ -- refinement edge} of the parent simplex $S$. A simplex of type 0 does not have a refinement edge. For 
%both 
each
of the child simplices $T$, $\frac12 (p_0+p_{k})$ is the only one of its vertices, which is not %also 
a vertex of the parent, hence the \emphindex{new vertex} $\Vnew T$\index{V_new@$\Vnew$ -- new vertex}.

The definition directly implies the following: If one bisects a simplex $S$ of type $k$ successively into a simplex $T$ of type $j$ without a transposition in between, the horizontal part of $T$ will be a centrepiece of the horizontal part of $S$, 
%and the vertical part of $S$ will be the tail of the vertical part of $T$, while the upper vertical vertices are centres between two horizontal vertices each.
and the vertical part of $T$ consists of centres between horizontal vertices of $S$, succeeded by the vertical part of $S$.
\end{defn}
\begin{defn}[reflected T-array]
For such a tagged simplex $T$, the T-array with the same vertical part, but with the \index{reflected T-array}\emph{reflected} horizontal part $\begin{pmatrix}p_{k} &\dots& p_0\end{pmatrix}$
%, i.e.\
is the \emph{reflected T-array} $T_R$. 
%Considering a tagged simplex $T$ and all its descendants and forgetting their tags, $T_R$ produces the same simplices{}%, so we can identify $T$ and $T_R$
The children of $T$ and $T_R$ coincide up to reflexion.
\end{defn}
%\subsection{Special Case: 2D (newest vertex bisection)}
\subsection{Special case: 2D (newest vertex bisection, NVB)}\label{sec:NVB}
For 2-simplices, Maubach's algorithm is equivalent to 
%the following:
Mitchell's newest vertex bisection \cite{Mitchell}: One edge of a triangle is marked as its refinement edge. This edge must be bisected first. The refinement edges of the children are defined as the other two edges of the parent triangle.

To see the equivalence to Maubach's algorithm,
look at Figure \ref{NVB}. There is no difference between types, if the refinement edges coincide. In higher dimensions, they do differ.

%%NVB
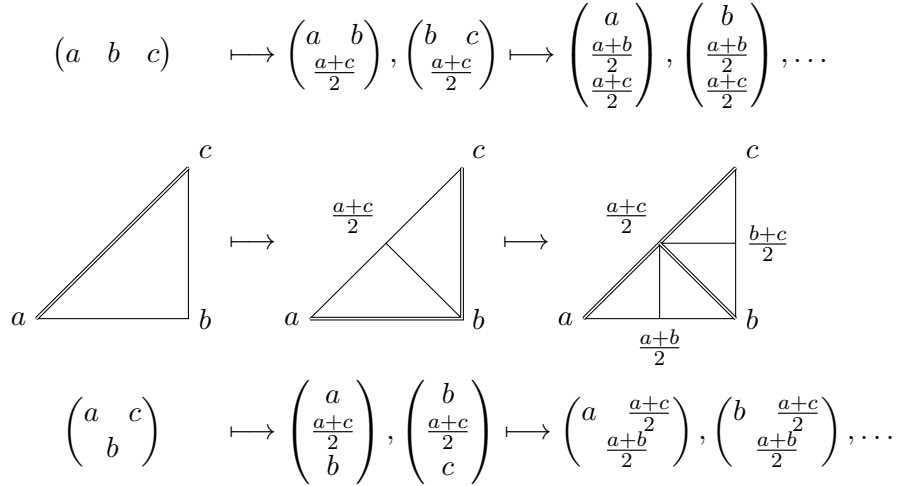
\begin{figure}[ht]
\centering{
\begin{tikzpicture}[information text/.style={fill=gray!10,inner sep=1ex}]
\tikzmath{\x = 3.6;}
\def\m{-1.2}
\begin{scope}[yshift=-.5cm]
\draw (1,4) node{$\begin{pmatrix} a&b&c \end{pmatrix}$};
\draw (\x+\m,4) node[right]{$\longmapsto\begin{pmatrix} a\quad b\\ \frac{a+c}{2} \end{pmatrix},\begin{pmatrix} b\quad c\\ \frac{a+c}{2} \end{pmatrix}
%$};
%\draw (2*\x+\m,4) node[right]{$
\longmapsto\begin{pmatrix} a\\ \frac{a+b}{2}\\ \frac{a+c}{2} \end{pmatrix},\begin{pmatrix} b\\ \frac{a+b}{2}\\ \frac{a+c}{2} \end{pmatrix}, \dots$};
\end{scope}

\def\mytriangle{(0,0) node[left] {$a$} -- (2,0) node[right] {$b$} -- (2,2)  node[above right] {$c$} -- cycle};
\draw \mytriangle;
\draw[double] (0,0) -- (2,2);

\def\mytriangles{
					\draw (\m,1) node[right]{$\longmapsto$};
					\draw \mytriangle;
					\draw (2,0)--(1,1) node[above left]{$\frac{a+c}2$};
					};
					
\begin{scope}[xshift=\x cm]
	\mytriangles;
	\draw[double] (0,0) -- (2,0) -- (2,2);
\end{scope}

\begin{scope}[xshift=2*\x cm]
	\mytriangles;
	\draw (1,0)node[below]{$\frac{a+b}2$} --(1,1) -- (2,1)node[right]{$\frac{b+c}2$};
	\draw[double] (0,0) -- (2,2);
	\draw[double] (2,0) -- (1,1);
\end{scope}

\begin{scope}[yshift=-1.5cm]
\draw (1,0) node{$\begin{pmatrix} a\quad c\\b \end{pmatrix}$};
\draw (\x+\m,0) node[right]{$\longmapsto\begin{pmatrix} a\\ \frac{a+c}{2} \\ b \end{pmatrix},\begin{pmatrix} b\\ \frac{a+c}{2}\\ c \end{pmatrix}
%$};
%\draw (2*\x+\m,4) node[right]{$
\longmapsto\begin{pmatrix} a\quad \frac{a+c}{2}\\ \frac{a+b}{2} \end{pmatrix},\begin{pmatrix} b\quad \frac{a+c}{2}\\ \frac{a+b}{2} \end{pmatrix}, \dots$};
%\draw (4.5,0) node[below=9mm,text width=9cm,
%information text]
%  {
%  newest vertex bisection: Type does not matter, if $E_{\text{ref}}$ (the double lines) are determined.
%  };
\end{scope}
\end{tikzpicture}
}
\caption{Newest vertex bisection: Both the above and the below T-arrays match for the recursive bisection of the depicted triangles. Type does not matter, if 
%$E_{\text{ref}}$ (the double lines) 
the (doubly drawn) refinement edges coincide.}\label{NVB}
\end{figure}
%to see that

\subsection{How it is interlinked with the \emph{tagged simplex} in literature}\label{connection to literature}
%Already in 1992,
%Kossaczk\' y \cite{Kossaczky} proposed bisection of a tetrahedron by embedding it into a parallelotope and partitioning it in the same manner as the cube in Section \ref{RefCo}.

%Also in 1992,
%In the same year,
Maubach \cite{Maubach} presented the following bisection algorithm:

``The bisection of an $n$-simplex presented below only involves the ordering of the vertices
of this simplex. Initially, all coarse grid $n$-simplices $T$ are said to be of level $l(T)= 0$. If a
simplex $T$ is bisected, the two created simplices are called its descendants, and the ordering
of their vertices is defined by the following bisection step:
\vspace{\baselineskip}\\
Bisect (simplex):
\\
BEGIN

Let $k := n - l(\text{simplex}) \mod n$;

Get simplex vertices: $x_0, x_1,\dots,x_{n-1}, x_n$;

Create the new vertex: $z = \frac12\{x_0 + x_k\}$;

Create descendant$_0$: $x_0, x_1,\dots,x_{k-1}, z, x_{k+1},x_n$;

Create descendant$_1$: $x_1,x_2,\dots,x_{k}, z, x_{k+1},x_n$;

Let $l(\text{descendant}_0) := l(\text{simplex}) + 1$;

Let $l(\text{descendant}_1) := l(\text{simplex}) + 1$;
\\
END.''
%\begin{algorithmic}
%Bisect (simplex):
%\BEGIN
%\LET $k := n - l(\text{simplex}) \mod n$;
%\end{algorithmic}

It is quite obvious that the ``simplex'' in that algorithm corresponds to the T-Array
$\begin{pmatrix}
x_0 & \dots & x_k\\
&\vdots\\
&x_n
\end{pmatrix}$
of type $k$.

Traxler \cite{Traxler}, whose notation was adopted by Stevenson \cite{Stevenson} and Alkämper, Gaspoz and Klöfkorn \cite{Gaspoz}, rearranged the vertices:

``%
%A simplex 
$S := (x_0,\dots,x_n)_g$ 
[\dots],
%is cut in only two pieces by the hyperplane $\{(x_0 + x_n)/2 , x_1,\dots,x_{n-1}\}$. This leaves the hyperfaces {x 0 ...Xn_ 1} and {x 1 . . . x n} intact and divides only one edge (the edge x0x,, which we call the axis of the simplex) right in the middle, producing a single 
new node $y := (x_0 + x_n)/2$.
Using the simplex type $\gamma := (g \mod n) \in \{0,\dots, n - 1\}$, the two children simplices
are defined as
\begin{align*}
C_0 := (x_0,y, 
\underbrace{x_1,\dots,x_\gamma}_{\rightarrow}, 
\underbrace{x_{\gamma+1},\dots,x_{n-1}}_{\rightarrow})_{g+1}\\
C_n := (x_0,y, 
\underbrace{x_1,\dots,x_\gamma}_{\rightarrow}, 
\underbrace{x_{n-1},\dots,x_{\gamma+1}}_{\leftarrow})_{g+1}.
\end{align*}
The arrows point in the direction of increasing indices.''

$S$ corresponds to the T-array
$
T=
\begin{pmatrix}
x_0 & x_{\gamma+1} \quad \dots & x_n\\
&x_1\\
&\vdots\\
&x_{\gamma}
\end{pmatrix}
$
of level $g$ (called \emph{generation} there).
This means that the type $\gamma$ in that language there and the type $t(T)$ in this language here are associated by $t(T)=n-\gamma$.

Consequently, the children $C_0$ and $C_n$ above correspond to
\begin{align*}
\begin{pmatrix}
x_0&x_{\gamma+1}\quad\dots & x_{n-1}\\
&y\\
&x_1\\
&\vdots\\
&x_{\gamma}
\end{pmatrix}
\text{ and }
\begin{pmatrix}
x_n&\dots & x_{\gamma+1}\\
&y\\
&x_1\\
&\vdots\\
&x_{\gamma}
\end{pmatrix}.
\end{align*}
While $C_0$ coincides with the second child in Definition \ref{Maubach}, $C_n$ differs from the first child 
%of the bisection in Definition \ref{Maubach} 
by reflexion, but Stevenson identifies a T-arrays and its reflected.
%\subsection{Remark on the name of the bisection algorithm}
\begin{bem}[on the designation “Newest Vertex Bisection” and on history]
%“Newest Vertex Bisection” -- normally this term would only fit two-dimensional cases. 
Normally the term ``Newest Vertex Bisection'' only fits the two-dimensional case.
In recent years, however, it has also been used for higher dimensional generalisations %\cite{30years},
\cite{Gaspoz}. Once a triangle has been bisected into two triangles, NVB bisects one of these  triangles by cutting straight through the newest vertex (result of the previous bisection) and through the midpoint of the opposite edge. In higher dimensions, a cut bisecting a simplex always runs through multiple (all but 2) vertices and thus cannot be defined by a \emph{newest} vertex whatsoever.

%The designation “Newest Vertex Bisection” is appropriate only for the 2-di\-men\-sio\-nal case and has not been used for higher dimensional generalisations until the last years %\cite{thirtyyears}, 
%\cite{Gaspoz}. If a triangle has been bisected into two, a newest vertex bisection bisects one of the parts by a straight cut through the newest vertex, i.e.\ this one which emerged after the preceding bisection, and the center of the opposite edge. In higher dimensions, a cut which bisects a simplex into two simplices always runs through several (all but 2) vertices. Thus, by no means it is defined by a newest one among them -- in whichever sense.

%The bisection method is old, maybe immemorial. Already in 1957, Whitney \cite{Whitney} describes a subdivision of a simplex into $2^{n}$ simplices, calling it \emph{standard subdivision}. It equals the $n$-fold recursive bisection of a type $n$ simplex by Maubach's algorithm.
The bisection method is classical. Already in 1957, Whitney \cite{Whitney} delineated a subdivision of a simplex into $2^{n}$ simplices, which he designated \emph{standard subdivision}. It represents an $n$-fold recursive bisection of a type $n$ simplex by Maubach's algorithm.

%In the 1990s, several bisection algorithms were invented for tetrahedra and eventually for $n$-simplices which have the following in common: All of them define for a simplex with a certain structure two children with a belonging structure for each. Apart from initial steps, this definition of the children depends only on combinatorial data (i.e.\ certain tags and orders of the vertices, edges, etc.). It does not depend on the geometrical data as it is the case for Longest Edge Bisection. Furthermore the definition of the children depends only on the structure of the simplex itself, but not on the surrounding simplices or any other influences.
During the 1990s several bisection algorithms for tetrahedrons and eventually for $n$-simplices were found. They all had one thing in common: for a simplex with a specific structure, they define two direct descendants and corresponding structures. The definition of these direct descendants depends -- disregarding initial refinement steps -- exclusively on combinatorial data including specific tags and orders of vertices and not, however, on geometrical data, edges etc.\ (as for Longest Edge Bisection). Furthermore, the only determining factor for this definition is the structure of the simplex itself, not surrounding simplices or any other factors.

%Up to initial steps, these algorithms are all equivalent: From the second or third generation, for every simplex with a structure $\Sigma_1$ in the language of Algorithm 1, there is a structure $\Sigma_2$ in the language of Algorithm 2, such that both algorithms generate the same binary tree of descendants.
Apart from initial refinement steps all of these algorithms are equivalent: From the second or third generation, for every simplex with a structure $\Sigma_1$ in the language of Algorithm 1 there is a structure $\Sigma_2$ in the language of Algorithm 2 such that both algorithms generate the same binary tree of descendants.

%Among these, Kossaczký were the first esteeming a tetrahedron as embedded in a parallelepiped such that in the 0th generation the refinement edges are the space diagonals, in the 1st generation the face diagonals etc.\ and that after 3 recursive bisection steps the triangulation resembles the initial triangulation. This is obviously generalisable to $n$ dimensions, what Maubach did in the same year. Although Maubach's structure is condensed into an array plus a level $l$, as cited above.
Already in 1992,
Kossaczk\' y \cite{Kossaczky} proposed bisection of a tetrahedron by embedding it into a parallelepiped and partitioning it in the same manner as the cube in Subsection \ref{RefCo}.
%Kossaczký
%%, in 1992, 
%was the first who delineated that a tetrahedron (as one of six components with equal volume) can be embedded into a parallelepiped. Therefore, the refinement edges are the space diagonals in the 0th generation and the face diagonals in the 1st generation etc. After 3 recursive bisection steps the triangulation resembles almost the initial triangulation. 
Apparently, this can be generalised for $n$ dimensions, and this is what Maubach did the same year 1992. Maubach's structure though is not an embedding into the parallelepiped. %Instead, it is very condensed into an array of vertices and level $l$ of simplex, %(modulo $n$), 
%similar to the T-array defined above. 
The parallelepiped and its diagonals are represented only indirectly (in the form of coordinates).

%The mentioned articles do not give their ``common'' algorithm a name. 
The cited authors do not give their ``universal'' algorithm a name. 
%Most of them bear titles including words 
The titles mostly include keywords 
like ``bisection'', ``refinement'', ``local'' and ``simplex'', which are suitable for Longest Edge Bisection as well.

In the literature of fixed-point and homotopy methods, the bisection algorithm is known as %``Whitney's policy''.
``Whitney's policy''.
Other sensible names would be ``cube-'' or ``parallelotope-induced bisection'', ``Kossaczký--Maubach bisection'' or, reviving Whitney's designation, ``standard bisection''.
\end{bem}
%%%%%%%%%%%%%%%%%%%%%%%%%%%%%%%%%%%%%%%%%
\newpage
\section{Basic structures}
Having introduced common notions in Section \ref{sec:preliminaries}, further general tools are developed now: Bisection generates a binary tree of tagged simplices from an initial tagged simplex. 
Level and hyperlevel of a T-array $S$ count how many times an initial simplex had to be bisected and transposed, respectively, to get $S$. Reference simplices, which are so to speak the standard case, are used to prove the shape regularity of all simplices generated by bisection in Subsection \ref{RefCo}, a result which was already shown by Kossaczký. 
The forest $\Simplexe$ consists of all simplices which can be generated from an initial partition by the bisection algorithm (the so called \emph{admissible} simplices) (Subsection \ref{sec:forest}), the forest of a partition is the set of all ancestors of the simplices of this partition, as they were introduced by 
%Cascon, Kreuzer, Nochetto and Siebert in \cite{Cascon}. 
Binev, Dahmen and DeVore in \cite{BDV}.
%In Section \ref{sec:regular triangulations}, the notion \emph{regular} is defined to generalise what is familiar as interdiction of \emph{hanging} vertices in the 2D case, a set of simplices is regular, if the intersection of any pair of simplices is a common subsimplex. Admissible triangulations are regular sets consisting of admissible simplices. 
Admissible triangulations are regular sets consisting of admissible simplices and generalise what is familiar as interdiction of \emph{hanging} vertices in the 2D case. The purpose of Subsection \ref{sec:regular triangulations} is to characterise the forests of admissible triangulations among all subsets of $\Simplexe$. This characterisation directly leads to a lattice structure (in terms of set theory) on the set of admissible triangulations and implies the result of Cascon, Kreuzer, Nochetto and Siebert \cite[Lemma 3.7]{Cascon} on the overlay of triangulations.
Lastly, under the assumption of so called \emph{compatibility conditions}, the forests of admissible triangulations are characterised alternatively in Subsection \ref{sec:IC}, which is essential for the proof of the BDV theorem afterwards.
\subsection{The binary tree of a T-array}
%By the definition of the children of a tagged simplex, 
%Let us identify the type 0 simplices and their transposed. Then 
%Maubach's bisection algorithm defines an infinite complete binary tree originating in an initial simplex as root. (It is actually a (non-full) binary tree, where all simplices have two children, except the type 0 simplices with their transposed as children. We distinguish between )

%For convenience both in the proofs and in the application, we give two definitions of a tree of an initial simplex.
%\begin{defn}[extended tree of an initial simplex]
%Bisection and transposition of a T-array define a binary tree with an initial simplex as root. In this tree, we distinguish between bisection edges and transposition edges.
%\end{defn}
%\begin{defn}[complete binary tree of an initial simplex]
%In the extended tree of an initial simplex, consolidate the type 0 simplices with their transposed to an equivalence class. The tree with these equivalence classes is the complete binary tree of the initial simplex.
%\end{defn}
\begin{defn}[binary tree and extended binary tree of a T-array, $\Descext$]\label{defn:binary tree}
%Der erweiterte Binärbaum eines T-Arrays $R$ sei der gerichtete Graph ohne Mehrfachkanten, aber mit zwei verschiedenen Kantentypen, nämlich Bisektions- und Transpositionskanten, der durch folgende Vorschrift entsteht:
Starting with an $n$-dimensional T-array $R$ (with $n\geq 1$) two directed graphs without multiple edges are defined, the extended binary tree of the T-array and its (normal) binary tree. 

The \emphindex{extended binary tree} has two different types of edges, namely bisection and transposition edges and is constructed as follows:
\begin{itemize}
\item
Start with the node $R$.
\item
Repeat the following steps ad infinitum:
\begin{itemize}
\item
For each (new) node $S$ of type $t(S)\geq 1$, add the children of the T-array $S$ (as defined in Definition \ref{Maubach}) to the set of nodes and (the edges from $S$ to its children) to the set of bisection edges.
%\item
%füge seine Kinder zur Knotenmenge hinzu.
%\item
%falls $t(S)\leq 1$, füge (Kanten von $S$ zu seinen Kindern) zur Menge der Bisektionskanten hinzu.
\item
%falls $t(S)=0$, füge (die Kante von $S$ zu $S^T$) zur Menge der Transpositionskanten hinzu.
For each (new) node $S$ of type $t(S)=0$, add $S^T$ to the set of nodes and the edge from $S$ to $S^T$ to the set of transposition edges.
\end{itemize}
\end{itemize}
Ancestors and descendants in the extended binary tree are briefly called \emphindex{extended ancestors} and \emphindex{extended descendants} respectively. The set of extended descendants of a T-array $T$ is denoted by $\Descext T$\index{Descext@$\Descext$ -- set of extended descendants}.

In the \emph{binary tree}\index{binary tree!of a T-array} of a T-array, the T-arrays of type $n$ are skipped 
%by transposing them straight away. 
by transposing and bisecting the type 0 simplices in \emph{one} step.
Thence the instructions read here:
%ist die Vorschrift etwas anders:
\begin{itemize}
\item
Start with the node $R$. If $R$ has the type
%0,
$n$, 
start with $R^T$.
\item
Repeat the following steps ad infinitum:
\begin{itemize}
\item
For each (new) node $S$ of type $t(S)\geq
%2
1$, add its children to the set of nodes and (the edges from $S$ to its children) to the edge set.
\item
For each (new) node $S$ of type $t(S)=
%1
0$, add 
%seine transponierten Kinder 
the children of $S^T$
to the node set and (the edges from $S$ to the children of $S^T$) to the set of edges.
\end{itemize}
\end{itemize}
%Im Binärbaum werden die T-Arrays vom Typ $0$ also übersprungen, indem sie gleich transponiert werden.
\end{defn}

\begin{lem}\label{binary tree well-defined}
%In diesem Graphen zeigt zu jedem Knoten nur eine Kante hin. Er ist also ein Binärbaum.
In the binary tree and in the extended binary tree, 
%towards each node points only one edge
for each node there is only one node pointing towards it; they are indeed binary trees.

In the binary tree, two distinct nodes represent two distinct simplices (without T-arrays).
\end{lem}
\begin{proof}[Sketch of proof]
%We prove the lemma for the binary tree first.
The lemma is proved for the binary tree first.
Within this proof, the notions of a parent and a child are used for arbitrary directed graphs. One after the other, one proves the successive statements:\\
%Ist $T$ ein Kind, von $S$, so gilt $T\subset S$.\\
For the children $S_1,S_2$ of a T-array $S$, %vom Typ $\leq 1$ 
it holds that
%\begin{itemize}
%\item
%S_1,S_2&\subsetneq S
%\item
%|S_1|=|S_2|&=\frac12\lvert S\rvert
%\item
%|S_1\cap S_2|&=0.
%\end{itemize}
%
%\begin{align*}
$S_1,S_2\subsetneq S$, %\\
$|S_1|=|S_2|=\frac12\lvert S\rvert$, %\\
 and $|S_1\cap S_2|=0$.
%\end{align*}

Per induction one concludes that:\\
If $T$ is a proper descendant of $S$, then 
%$T\subsetneq S$
$|T|\leq \frac12\lvert S\rvert$.\\
If $T$ is neither an ancestor nor a descendant of $S$, then $|S\cap T|=0$.\\
%Ein echter Nachfahre von $S$ kann nur dann gleiches Volumen wie $S$ haben, wenn $t(S)=0$ und der Nachfahre Kind von $S$ ist.
A proper descendant of $S$ never has the same volume as $S$.

After these preparations one argues as follows: Assume that $T$ had two distinct parents $S_1,S_2$ in the graph. Then $T\subset S_1\cap S_2$, so $|S_1\cap S_2|>0$. Hence $S_1$ is a proper descendant of $S_2$ or vice versa, %Wir nehmen o.\,B.\,d.\,A. den ersten Fall an. Nun unterscheiden wir 2 Fälle: Entweder, $t(T)=n$. Dann gilt $\conv(S_1)=\conv(T)=\conv(S_2)$, also $|S_1|=|S_2|$, oder $t(T)\leq n{-}1$, woraus 
contradicting $|S_1|=2|T|=|S_2|$. %folgt. In beiden Fällen ist $S_2$ ein Nachfahre von $S_1$ mit dem gleichen Volumen, also $t(S_1)=0$ und $S_2$ das einzige Kind von $S_1$. Dann kann $T$ nicht Kind von $S_1$ sein. Widerspruch.

The extended binary tree forms from the binary tree by duplicating each type 0 node into a type 0 node and a type $n$ node, drawing a transposition edge from the first to the last and turning the edge pointing towards the former node into an edge pointing towards the type 0 copy and the edges pointing away from it into edges pointing starting at the type $n$ copy.
%n
Hence, a node cannot have several parents, either.
\end{proof}
%\begin{defn}[Binärbaum eines T-Arrays]
%Der Binärbaum eines T-Arrays $R$ sei folgendes Abbild des erweiterten Binärbaumes: Die zueinander transponierten Paare, bestehend aus einem Typ-0- und einem Typ-$n$-Array, werden zu \emph{einem} Knoten zusammengezogen. Beide stehen für das gleiche Simplex. Alle anderen Knoten werden auf sich selbst abgebildet. Die Bisektionskanten von oder zu T-Arrays der Typen 0 oder $n$ werden nun mit diesem Paar verbunden. Die Transpositionskanten werden gelöscht. Die Knoten, die wir meistens Simplexe nennen werden, haben nun alle eine neue Ecke (bei den Typ-0-Typ-$n$-Paaren definiert durch den Typ-0-Partner) und eine Verfeinerungskante (bei den Paaren definiert durch den Typ-$n$-Partner).
%\end{defn}
The binary tree of a T-array is an infinite perfect binary tree. Unless otherwise specified, %we will 
the text always refer%
s to the binary tree. In the binary tree, all simplices have refinement edges. Although type 0 T-arrays have not got a refinement edge, in the binary tree, they get the refinement edge of their transposed, such that all simplices in the binary tree have a refinement edge.
%Das gleiche gilt für die Verwendung der Begriffe Kind, Vater, Nachkomme und Vorfahr.
%Maubachs oben vorgestellte Bisektion %eines (nummerierten) Simplex', die ihn in zwei nummerierte Simplizes zerlegt, 
%definiert einen vollständigen Binärbaum unendlicher Tiefe mit einem Startsimplex als Wurzel. %(Die beiden Teile der Bisektion seien die Kinder des Simplex.) 
%... (Wald)
%In this tree, 2 tagged simplices either intersect in a null set or one of both is the descendant of the other, and thus has at most half of the volume. Consequently, 2 distinct T-arrays of the forest have always two distinct simplices. This justifies, to call the nodes simply simplices.

The extended binary tree is rather used for some technical proofs. If some property has to be shown for all T-arrays in the extended binary tree, it is convenient to 
prove that property for the root and to show that both bisection of a type ${\geq} 1$ simplex and transposition conserves the property.
%do mathematical induction in the extended binary tree.
\begin{defn}[level and hyperlevel]
\index{level}\index{hyperlevel}
Fix a tagged simplex $R$ as root. For a descendant $S$ of %a %n initial 
%simplex 
$R$, its \emph{level} $l(S)$\index{l@$l$ -- level} is the number of edges in the path from the root to it in the %perfect 
binary tree. Consider also the path from the root to a T-array $S$ in the extended tree. The %number of bisection edges in this path is called \emph{level} $l(S)$, the 
number of transposition edges
in this path 
is called \emph{hyperlevel} \index{h@$h$ -- hyperlevel}$h(S)$. 
\end{defn}
Bisection decrements the type, while transposition increases it by $m$, so $h(T)\approx l(T)/m$.

\subsection{Reference coordinates and shape regularity}\label{RefCo}
This subsection illustrates Maubach's rule.

\begin{defn}[Kuhn simplex]
\index{Kuhn simplex}
%Ein \emph{Kuhn-Simplex} der linear unabhängigen Familie $(e_1,\dots,e_k)$ sei ein angeordnetes Simplex mit $p_j-p_{j-1}=e_j$.
%\begin{enumerate}
%\item
Let $(e_1,\dots,e_n)$ be the canonical basis of $\R^n$. %Let us call a tagged $m$-simplex of type $m$, where $p_j-p_{j-1}=\epsilon_je_{\sigma(j)}$ for some signs $\epsilon_j\in\{1,-1\}$ and some permutation $\sigma$ of $\{1,\dots,n\}$, \emph{Kuhn simplex}.
A \emph{Kuhn $m$-simplex} is a tagged $m$-simplex of type $m$, $\begin{pmatrix}p_0&\dots&p_m\end{pmatrix}$, such that $p_j-p_{j-1}=\epsilon_je_{\sigma(j)}$ for some signs $\epsilon_j\in\{1,-1\}$ and some permutation $\sigma$ of $\{1,\dots,n\}$.
\end{defn}
%\item
%Let a \emph{$(k,h)$ reference simplex} be a hyperlevel $h$ descendant of type $k$ of a Kuhn simplex.
\begin{defn}[reference simplex]\index{reference simplex}
A \emph{$(k,h)$ reference simplex} is a hyperlevel $h$ descendant of type $k$ of a Kuhn simplex 
in 
%its 
the
extended binary tree
of the Kuhn simplex.
\end{defn}
%\item
%Let $S$ be the tagged simplex $S$ of type $k$ and hyperlevel $h$ with the T-array
%\[
%\begin{pmatrix}
%p_0&\dots& p_{k}\\
%%&p_{k+1}\\
%&\vdots\\
%&p_n
%\end{pmatrix}.
%\]
%Then \emph{reference coordinates}
\begin{defn}[reference coordinates]\index{reference coordinates}
\emph{Reference coordinates}
of 
a T-array 
$S
=\begin{pmatrix}
p_0&\dots& p_{k}\\
%&p_{k+1}\\
&\vdots\\
&p_n
\end{pmatrix}
$ of type $k$ and hyperlevel $h$ 
are an affine mapping $\vi: \R^n\rightarrow \R^n$, such that the tagged simplex
\[
\vi(S):=
\begin{pmatrix}
\vi(p_0)&\dots& \vi(p_{k})\\
%&\vi(p_{k+1})\\
&\vdots\\
&\vi(p_n)
\end{pmatrix}
\]
is a $(k,h)$ reference simplex.
%\end{enumerate}
\end{defn}
\begin{bem}
Reference coordinates of a simplex are also reference coordinates for its children and (in case of type 0) for the transposed simplex.
\end{bem}

\begin{thm}[The cube-related geometry of reference simplices]\label{Cubetheorem}
Given a Kuhn simplex, 
let $T=\begin{pmatrix}
p_0&\dots& p_{k}\\
&p_{k+1}\\
&\vdots\\
&p_n
\end{pmatrix}
$ be 
a
T-array
%one of its descendants % of a Kuhn simplex 
of its extended tree
of type $k$ and hyperlevel $h$. %with the T-array
%\[
%\begin{pmatrix}
%p_0&\dots& p_{k}\\
%&p_{k+1}\\
%&\vdots\\
%&p_n
%\end{pmatrix}.
%\]
%(with $(p_a, \dots, p_{a+k})$ being a centrepiece of the horizontal part $(p_0,\dots,p_n)$ of its ancestor, of hyperlevel $h$ and type $n$). {refsim2}

%Then $q_j$ are the centres of $j$-dimensional subcubes $C_j$ of a $n$-dimensional unit cube $[0,1]^n$, such that each $C_j$ expect of $C_n$ is a subcube of $C_{j+1}$ and $p_a,\dots,p_{a+k}$ are vertices of a $k$-dimensional subcube $C_k$ of $C_{k+1}$. 

Then 
\begin{enumerate}
\item
the horizontal part is a $k$-dimensional Kuhn simplex scaled by factor $2^{-h}$,
\item
for $j=k,\dots,n{-}1$ there are $j$-dimensional subcubes $C_j$ of a 
$2^{-h}$-scaled
unit cube $C_n=a+2^{-h}[0,1]^n$ (for some $a\in\R^n$)%scaled by factor $2^{-h}$
, such that: 
\begin{itemize}
\item
each $C_j$ is a subcube (i.e.\ a hyperface) of $C_{j+1}$,
\item
for $j=k{+}1,\dots,n$, the vertical vertex $p_j$ is the centre of $C_j$,
%the vertical node $p_{k+1},\dots,p_n$ is the centre of $C_{k+1},\dots,C_n$ respectively.
\item
$p_0,\dots,p_{k}$ are vertices of $C_k$.
\end{itemize}
 %each $C_j$ is a subcube of $C_{j+1}$, for $j=k{+}1,\dots,n$ the $p_j$ are the centres of $C_j$ and $p_0,\dots,p_{k}$ are vertices of $C_k$.
\end{enumerate}

%Each hyperlevel $h$ descendant of the Kuhn simplex of type $k$ and hyperlevel $h$ is an image of $T$, %vom Level $n-k$
%scaled by factor $2^{-h}$, mapped with a coordinate permutation, reflected at hyperplanes orthogonal to the coordinate axes and translated. (These transformations without the scaling are: $\text{signed permutation matrix}\cdot\bullet+vector$.)
\end{thm}

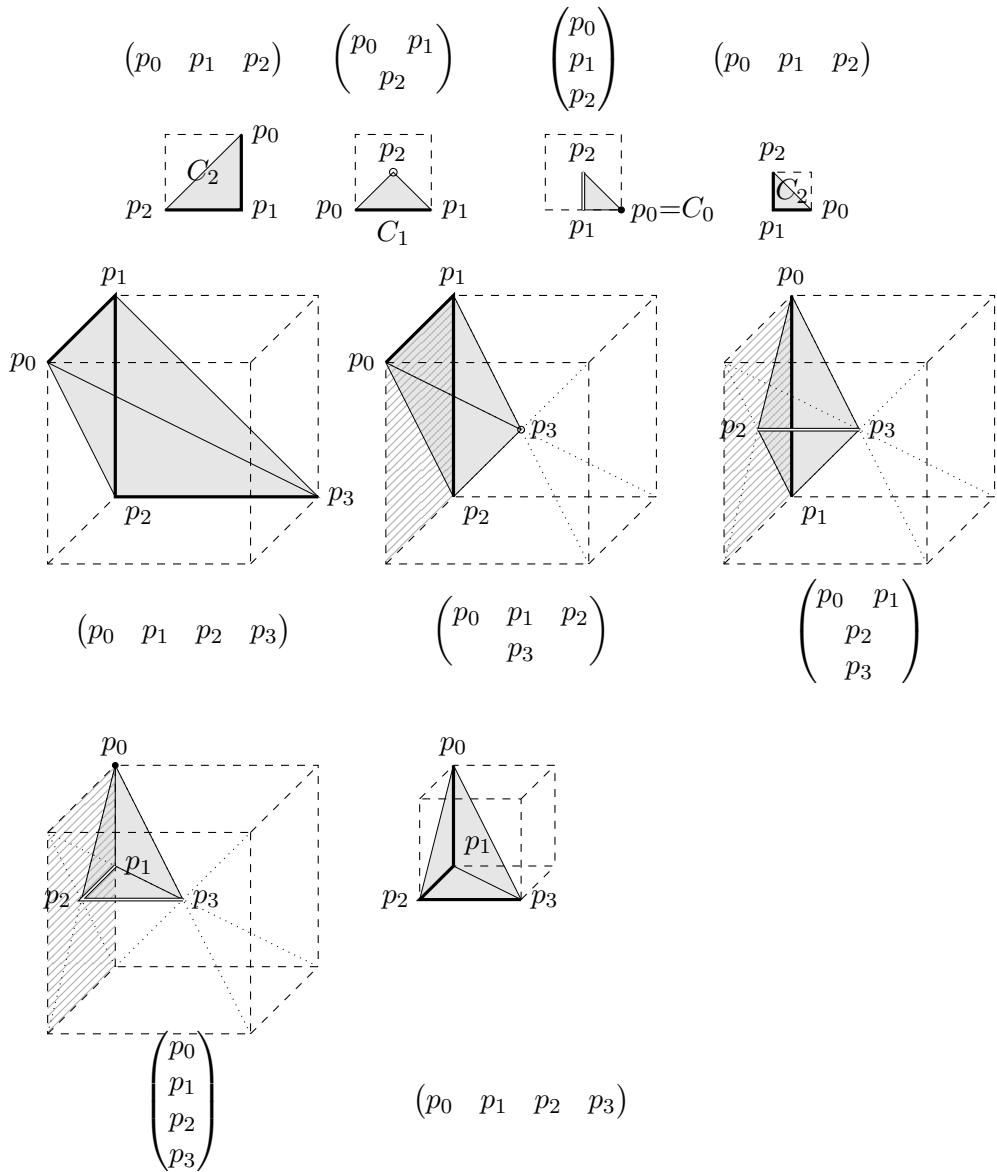
\begin{figure}
\centering{
%reference simplices of dim 2
\begin{tikzpicture}%[information text/.style={fill=gray!10,inner sep=1ex}]
					[information text/.style={fill=gray!10,inner sep=1ex}, 
					cube1/.style=dashed, 
%					cube2/.style={dashed,double},
					simplex/.style={very thin},
					diag/.style=dotted,
					canonical1/.style={very thick},
					canonical2/.style=double,
					]
\begin{scope}[yshift=0.2cm]
%Kuhn
\draw[cube1] (1,1) rectangle (0,0);
\filldraw[simplex, fill=gray!20] (0,0) -- (1,1) -- (1,0) -- cycle;
\draw[canonical1] (0,0) node[left]{$p_2$}-- (1,0) node[right]{$p_1$}--(1,1) node[right]{$p_0$};
\draw (.5,.5) node {$C_2$};
%%beschriftung
\draw (0.5,2) node{$\begin{pmatrix}p_0&p_1&p_2\end{pmatrix}$};

%type 1
\begin{scope}[xshift=2.5cm]
\draw[cube1] (1,1) rectangle (0,0);
\filldraw[simplex, fill=gray!20] (0,0) -- (.5,.5) -- (1,0) -- cycle;
\draw[canonical1] (0,0) node[left]{$p_0$} -- node[below]{$C_1$} 
					(1,0) node[right]{$p_1$}%--(1,1) node[right]{$p_0$}
;
\draw (.5,.5) node[above]{$p_2$} circle [radius=.05];
%\draw[canonical2] (.5,.5) node[above]{$p_2$} -- (.5,.5);
%%beschriftung
\draw (0.5,2) node{$\begin{pmatrix}p_0\quad p_1\\p_2\end{pmatrix}$};
\end{scope}

%type 0
\begin{scope}[xshift=5cm]
\filldraw[simplex, fill=gray!20] (1,0) 
%node[below]{$p_0=C_0$}
node[right]{$p_0{=}C_0$} 
-- (.5,.5) -- (.5,0) -- cycle;
\draw[cube1] (1,1) rectangle (0,0);
\draw[canonical2] (.5,0) node[below]{$p_1$}--(.5,.5) node[above]{$p_2$}%--(1,1) node[right]{$p_0$}
;
\filldraw (1,0) circle [radius=.04];
%%T-array
\draw (0.5,2) node{$\begin{pmatrix}p_0\\p_1\\p_2\end{pmatrix}$};
\end{scope}

%type 2
\begin{scope}[xshift=7.5cm]
\filldraw[simplex, fill=gray!20] (1,0) -- (.5,.5) -- (.5,0) -- cycle;
\draw[cube1] (1,0) rectangle (.5,0.5);
\draw[canonical1] (1,0)node[right]{$p_0$} -- (.5,0) node[below]{$p_1$}--(.5,.5) node[above]{$p_2$}%--(1,1) node[right]{$p_0$}
;
%\filldraw (1,0) circle [radius=.04];
\draw (.75,.25) node {$C_2$};
%%beschriftung
\draw (0.75,2) node{$\begin{pmatrix}p_0&p_1 & p_2\end{pmatrix}$};
\end{scope}
\end{scope}
\end{tikzpicture}

%\hline{10cm}

%%3D Kuhn

%\begin{scope}[yshift=-7cm]

\begin{tikzpicture}[scale=.89,
					information text/.style={fill=gray!10,inner sep=1ex}, 
					cube1/.style=dashed, 
%					cube2/.style={dashed,double},
					simplex/.style={very thin},
					diag/.style=dotted,
					canonical1/.style={very thick},
					canonical2/.style=double,
					]
\def\ys{-1};
\def\cube{
\draw[cube1] (3,3) rectangle (0,0);
\draw[cube1] (1,1) rectangle (4,4);
\draw[cube1] (0,0) -- (1,1); 
\draw[cube1] (3,3) -- (4,4); 
\draw[cube1] (3,0) -- (4,1); 
\draw[cube1] (0,3) -- (1,4);
}

%\begin{scope}[yshift=-4.5cm]

%%type3
\begin{scope}[xshift=0cm]
	\coordinate (p0) at (0,3);
	\coordinate (p1) at (1,4);
	\coordinate (p2) at (1,1);
	\coordinate (p3) at (4,1);

	\draw (2,\ys) node {$\begin{pmatrix} p_0&p_1&p_2&p_3 \end{pmatrix}$};
	\fill[gray!20] (p0)--(p2)--(p3)--(p1)--cycle;
	\cube;
	\draw[canonical1, %dashed
	]  
	(0,3) -- (1,4)--(1,1) -- (4,1); 
	
	\draw[style=simplex] 
	(1,4) node [above]{$p_1$}-- (4,1)node [right]{$p_3$} -- (0,3)node [left]{$p_0$} -- (1,1)node [below right]{$p_2$};
\end{scope}

%%type2
\begin{scope}[xshift=5cm]
	\coordinate (p0) at (0,3);
	\coordinate (p1) at (1,4);
	\coordinate (p2) at (1,1);
	\coordinate (p3) at (2,2);
	
	\fill[fill=gray!20] (p0)--(p1)--(p3)--(p2)--cycle;
	\fill[pattern=north east lines,pattern color=gray!60] (p0)--(p1)--(p2)--(0,0)--cycle;

	\draw (2,\ys) node {$\begin{pmatrix} p_0&p_1&p_2\\&p_3 \end{pmatrix}$};

	\cube;
	\draw[canonical1]  
	(p0) -- (p1) -- (p2); 
	
	\draw[style=simplex] 
	(p0) node [left]{$p_0$
	}-- (p2) node [below right]{$p_2$
	} %-- (p3)node [below right]{%$p_0$
	%} -- %(1,1)node [above right]{$p_2$};
	%cycle
	;
	\draw[simplex] (p0) -- (p3) node [right]{$p_3$};
	\draw[simplex] (p1) node [above]{$p_1$} -- (p3);
	\draw[simplex] (p2) -- (p3);
	
	\draw[diag] (0,3) -- (4,1);
	\draw[diag] (3,0) -- (1,4);
	\draw[diag] (1,1) -- (3,3);
	
	\draw (p3) circle [radius=.05];
\end{scope}

%%type1
\begin{scope}[xshift=10cm]
	\coordinate (p0) at (1,4);
	\coordinate (p1) at (1,1);
	\coordinate (p2) at (.5,2);
	\coordinate (p3) at (2,2);
	
	\fill[fill=gray!20] (p0)--(p2)--(p1)--(p3)--cycle;
	\fill[pattern=north east lines,pattern color=gray!60] (0,0)--(1,1)--(1,4)--(0,3)--cycle;
	\cube;
	\draw[canonical1]  
	(p0) -- (p1); 
	
	\draw (2,\ys) node {$\begin{pmatrix} p_0\quad p_1\\p_2\\p_3 \end{pmatrix}$};

	\draw[simplex] (p0) node [above]{$p_0$} -- (p3) node [right]{$p_3$};
	\draw[simplex] (p1) node [below right]{$p_1$} -- (p3);
	\draw[simplex] (p0) -- (p2) node [left]{$p_2$};
	\draw[simplex] (p1) -- (p2);
	
	\draw[diag] (0,3) -- (4,1);
	\draw[diag] (3,0) -- (1,4);
	\draw[diag] (1,1) -- (3,3);
	
	\draw[diag] (0,0) -- (1,4);
	\draw[diag] (1,1) -- (0,3);
	
	\draw[canonical2] (p2)--(p3);
\end{scope}

%%type0
\begin{scope}[yshift=-7cm]
\begin{scope}[xshift=0cm]
	\coordinate (p0) at (1,4);
	\coordinate (p1) at (1,2.5);
	\coordinate (p2) at (.5,2);
	\coordinate (p3) at (2,2);
	
	\fill[fill=gray!20] (p0)--(p2)--(p3)--cycle;
	\fill[pattern=north east lines,pattern color=gray!60] (0,0)--(1,1)--(1,4)--(0,3)--cycle;
	\cube;
%	\draw[canonical1]  	(p0) -- (p1); 
	\fill (p0) circle [radius=.05]; 

	\draw (2,\ys) node {$\begin{pmatrix} p_0\\ p_1\\p_2\\p_3 \end{pmatrix}$};
	
	\draw[simplex] (p0) node [above]{$p_0$} -- (p1) node [right]{$p_1$};
	\draw[simplex] (p0) -- (p2) node [left]{$p_2$};
	\draw[simplex] (p0) -- (p3) node [right]{$p_3$};
	\draw[simplex] (p1) -- (p3);
	
	\draw[diag] (0,3) -- (4,1);
	\draw[diag] (3,0) -- (1,4);
	\draw[diag] (1,1) -- (3,3);
	
	\draw[diag] (0,0) -- (1,4);
	\draw[diag] (1,1) -- (0,3);
	
	\draw[canonical2] (p1)--(p2)--(p3);
	%\draw (p3) circle [radius=.04];
\end{scope}
\begin{scope}[xshift=5cm]
	\coordinate (p0) at (1,4);
	\coordinate (p1) at (1,2.5);
	\coordinate (p2) at (.5,2);
	\coordinate (p3) at (2,2);
	
	\fill[fill=gray!20] (p0)--(p2)--(p3)--cycle;
%	\fill[pattern=north east lines,pattern color=gray!60] (0,0)--(1,1)--(1,4)--(0,3)--cycle;
	\draw[cube1] (2,2) rectangle (.5,3.5);
	\draw[cube1] (1,4) rectangle (2.5,2.5);
	\draw[cube1] (.5,2) -- (1,2.5); 
	\draw[cube1] (2,2) -- (2.5,2.5); 
	\draw[cube1] (.5,3.5) -- (1,4); 
	\draw[cube1] (2,3.5) -- (2.5,4);
%	\draw[canonical1]  	(p0) -- (p1); 
%	\fill (p0) circle [radius=.05]; 

	\draw (2,\ys) node {$\begin{pmatrix} p_0& p_1&p_2&p_3 \end{pmatrix}$};
	
	\draw[simplex] (p0) node [above]{$p_0$} -- (p2) node [left]{$p_2$};
	\draw[simplex] (p0) -- (p3) node [right]{$p_3$};
	\draw[simplex] (p1) -- (p3);
	
	\draw[canonical1] (p0)-- (p1) node [above right]{$p_1$}--(p2)--(p3);
	%\draw (p3) circle [radius=.04];
\end{scope}
\end{scope}

%\end{scope}%%3D-Bilder runterschieben.
\end{tikzpicture}
\caption{Reference simplices (descendants of Kuhn simplices) of different dimensions and types with their unit (sub-)\-cubes. 
%Above, the type is 
%The types of the reference triangles at the top row are (from the left to the right) 2, 1, 0, 2.
%, respectively. 
The thick edges are between successive horizontal vertices, the double ones between successive vertical ones. They are signed canonical unit vectors and bisected canonical unit vectors respectively. Note that the vertical vertices build up a scaled Kuhn simplex step by step.}
\label{refsim2}
}
\end{figure}
See Figure \ref{refsim2} for an illustration of Theorem \ref{Cubetheorem}.
\begin{proof}[Proof by mathematical induction% on $l(T)$
]%%Problem: We are in the extended binary tree.

For
$2n$ numbers $a_1,b_1,\dots,a_n,b_n\in \R$ with $|a_j-b_j|=c,$ 
$C_{a_1,b_1,\dots,a_n,b_n}=[a_1,b_1]\times\dots\times [a_n,b_n]$ is a cube with the vertices
%define the vertices of a $c$-scaled cube $C_{a_1,b_1,\dots,a_n,b_n}$ by
\begin{align*}
\Vertices C_{a_1,b_1,\dots,a_n,b_n}=\{a_1,b_1\}\times\dots\times\{a_n,b_n\}=\left\{x\in\R^n~\middle|~ \fa j. x_j\in\{a_j,b_j\}\right\}.
\end{align*}
\emph{Base case: $l(T)=0$.} 1. The root $R$ is a Kuhn simplex by definition.
2. 
%We only have 
It suffices
to find $C_n$ having $p_0,\dots,p_n$ 
%among 
as
its vertices. W.l.o.g.\ assume that $\sigma=\id$ for 
the above
%our 
Kuhn simplex. 
%Note that 
Since $p_{i}-p_{i-1}=\epsilon_ie_i$, 
$\fa j=1,\dots,n$ the $j$-th coordinate $p_{ij}$ of $p_i$ changes only between $p_{j-1}$ and $p_j$:
\begin{align*}
\underbrace{p_{0j}=\dots=p_{(j-1)j}}_{=:a_j},\underbrace{p_{jj}=\dots=p_{nj}}_{=:b_j},
\end{align*}
and $|p_{jj}-p_{(j{-}1)j}|=|\epsilon_j|=1$,
so $C_n=C_{a_1,b_1,\dots,a_n,b_n}$ fulfils 
%our 
the 
request.

\emph{Induction assumption.} For some reference simplex $T$, the statement is true for the parent $\pa (T)$ in the extended binary tree of a reference simplex.

\emph{Induction step.} 1st case: $t(\pa T)\geq 1$. 

1st statement: One gets the horizontal part of %the children of $%S
%\pa T$ 
$T$ by removing the first or the last vertex from the horizontal part of $\pa (T)$. By induction assumption, the horizontal part of $\pa (T)$ is a Kuhn $k$-simplex. The oddment $\begin{pmatrix}p_1 &\dots& p_k\end{pmatrix}$ and $\begin{pmatrix}p_0 &\dots& p_{k-1}\end{pmatrix}$ is a Kuhn $(k{-}1)$-simplex (with edge vectors $\epsilon_2e_{\sigma (2)},\dots,\epsilon_ke_{\sigma (k)}$ and $\epsilon_1e_{\sigma (1)},\dots,\epsilon_{k-1}e_{\sigma(k-1)}$), respectively.

2nd statement: %For the children of $T$, the cubes $C_k,\dots, C_n$ stay the same as they were for $T$. Also $p_{k+1}, \dots, p_n$ stay the same, so for their indices, nothing has to be proven. The just mentioned edge vectors of the horizontal part span a hyperface $C_{k-1}$ of $C_k$. Since $p_k-p_0=\sum_{j=1}^k\epsilon_je_{\sigma j}$ is the diagonal of $C_k$, the new vertex $\frac{p_0+p_k}2$ is its centre.
For %the children of $\pa T$
$T$, the centres $p_{k+1}, \dots, p_n$ stay the same, as they 
%were 
are
for $\pa (T)$, so the cubes $C_{k+1},\dots, C_n$ also stay the same, and for their indices nothing has to be proven. Since $p_k-p_0=\sum_{j=1}^k\epsilon_je_{\sigma j}$ 
%was 
is
the diagonal of $C_k$ in the parent, the new vertex $\frac{p_0+p_k}2$ is its centre, so $C_k$, too, stays the same. The aforementioned edge vectors of the horizontal part span a hyperface $C_{k-1}$ of $C_k$, which finally proves the statement for $k{-}1$.

2nd case: $t(\pa T)=0$, i.e.\ $T=(\pa T)^T$.

1st statement: The horizontal part of $%S
\pa (T)$ is $(p_0)$, its vertical part is $\begin{pmatrix}p_1 &\dots& \allowbreak p_n\end{pmatrix}^T$. $C_0$ is a 0-dimensional cube, i.e.\ a point. So $p_0$ is at the same time also centre of $C_0$, %so 
and all $p_j$ are centres of $C_j$. Let $2^{-h}e_{\sigma' j}$ be the scaled canonical basis vector being edge of $C_j$ but not of $C_{j-1}$ 
(see the 
%figure 
Figure \ref{fig:canonical basis vector}%below
).  Then
\begin{align*}
p_j-p_{j-1}=2^{-h-1}\epsilon'_je_{\sigma' j}
\end{align*}
for some $\epsilon'_j\in\{\pm 1\}$. Hence, the transposed simplex $T$ is a $2^{-h-1}$-scaled Kuhn simplex. The second statement has already been proven in the base case (for an unscaled Kuhn simplex).
%%canonical basis vector
%\begin{wrapfigure}{r}[2cm]{0.3\textwidth}
\begin{figure}[ht]
\centering{
\begin{tikzpicture}[information text/.style={fill=gray!10,inner sep=1ex}%, scale=.5
]
\draw[dashed] %(0,0) %node [left] {$a$}  
%-- 
(2,0) %node [right] {$d$} 
-- (2,2) %node [right] {$b$} 
-- (0,2) %-- cycle %node {$C_j$}
;
\draw (1,1) node {$C_j$} 
;
\draw[thick] (0,0) -- node [left] {$C_{j-1}$}(0,2);
\draw[thin,gray,->] (0,0) -- node [black,below] {$\pm2^{-h}e_{\sigma' j}$}(2,0);
\draw (2,1) node[right=
%7mm
2mm
,text width=3cm,
	information text]
	  {\caption{\newline Definition of the canonical basis vector $e_{\sigma' j}$.}\label{fig:canonical basis vector}

	  };
\end{tikzpicture}
}
\end{figure}
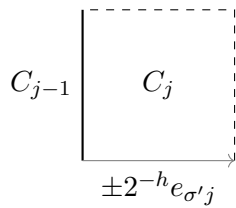
%\emph{Third proof}
%[base case]
%For every $j$ with $\epsilon_j=-1$ reflect $R$ at the hyperface orthogonal to $e_{\sigma j}$ (in arbitrary order). Such a reflexion preserves every of the edge vectors $e_{\sigma i}$ for $i\neq j$, but changes the sign of $e_{\sigma j}$ so composition of all them yields a Kuhn simplex, where the $\epsilon_j$ are all 1. Then we can give the vertices of the reflected Kuhn simplex $(q_0,\dots,q_n)$ by $q_j=q_0+\sum_{i=1}^j e_{\sigma i}$ and these are all vertices of $q_0+[0,1]^n$.
\end{proof}
\begin{folg}\label{similarity classes}\index{similarity classes}\index{minimal angle condition}
The edge vectors of descendants of a tagged simplex point into only finitely many directions altogether. Thus, the number of 
%angles between them as well as %is bounded. Moreover 
similarity classes and
the number of values for $\diam(T)^n/|T|$ 
%is 
are finite in the binary tree.
\end{folg}
Corollary \ref{similarity classes} was already known to Whitney \cite{Whitney}.
%\begin{folg}\label{edgelengths}
%Two distinct horizontal vertices of a $(k,h)$ reference simplex have the Chebyshev distance $2^{-h}$. All other edges have Chebyshev distance $2^{-h-1}$. The refinement edge has a Manhattan length of $2^{-h}\cdot k$.
%\end{folg}
The Kuhn simplices in Theorem \ref{Cubetheorem} give rise to the following corollary directly.
\begin{folg}\label{edgelengths}
%\begin{proof}
\begin{enumerate}
\item
%The 
Two distinct
horizontal vertices 
of a $(k,h)$ reference simplex
are vertices of 
a 
%the 
$2^{-h}$-scaled 
%unit cube $W_k$, 
Kuhn $k$-simplex,
and thus have pairwise Chebyshev distances of $2^{-h}$. 
\item For $k\neq 0$, the refinement edge of 
%the 
a $(k,h)$ reference simplex equals the refinement edge of that Kuhn $k$-simplex and has Manhattan length $2^{-h}k$. 
\item Every \emph{other} pair of vertices appears in one of the descendants after one transposition as a pair of horizontal vertices, 
%of this descendant, 
hence it has Chebyshev distance $2^{-h-1}$.
%\end{proof}
\end{enumerate}
\end{folg}
%\paragraph{Exercises. }
%\begin{enumerate}
%\item We call 2 tagged simplices \emph{equivalent}, if they generate the same tree of (untagged) simplices. What are the equivalence classes? (One part is hard to prove here.)
%
%\item Which are the directions, where the edges of descendants can point to? Which are the minimal, which are the maximal angles? What is the maximum of $\diam(T)^n/|T|$? How do these values change after linear transformations (estimations!)? What are these values for the regular simplex?
%
%\item Divide a cube into Kuhn simplices.
%\end{enumerate}
\subsection{The forest of admissible simplices 
%generated from an initial partition
}\label{sec:forest}
\begin{defn}[initial partition $\T_0$, forest of admissible simplices $\Simplexe$, forest of a partition]
\index{initial partition}\index{T_0@$\T_0$ -- initial partition}\index{forest of admissible simplices}\index{T@$\Simplexe$ -- forest of admissible simplices}\index{forest of a partition}\index{admissible simplices}
As a general assumption a fixed finite %tagged 
\emph{initial partition} $\T_0$ of an $n$-dimensional polytope $\Omega\subset \R^n$ into 
tagged
$n$-simplices is given,
%As a basic principle, we will subsequently assume to be given a fixed finite tagged initial partition $\T_0$ of a polytope $\Omega$ into simplices, 
i.e.\ their union is $\Omega$ and the intersection of two has Lebesgue measure 0. 
%We call the union of the infinite complete binary trees of the initial partition the binary forest of \emph{admissible simplices $\Simplexe$}.
The union of the infinite complete binary trees of the simplices of the initial partition will be called \emph{binary forest of admissible simplices $\Simplexe$}.

If %we 
now %restrict 
each of these trees is restricted to a finite full subtree still containing the initial simplex (the original root) (in other words if %we break 
the ongoing bisection is broken off sometime), the leaves of this subforest form a (not necessarily 
%conforming)
regular) partition of $\Omega$. Conversely, add to an arbitrary finite partition into admissible simplices all their ancestors to get such a full subforest with these simplices as leaves. %We will refer to 
These partitions %as 
are the \emph{admissible} ones and %to 
the corresponding subforest simply %as 
the \emph{forest (of the partition $P$)}, %and 
denoted %it 
by $\Wald(P)$\index{fo@$\Wald$ -- forest (of a partition)}.
\end{defn}
%\begin{defn}[level]
%Fixing this initial triangulation, we can revise our notion of a level, allocating every simplex $T$ two numbers $l(T)$ and $h(T)$, if they are level-$l(T)$-descendants and hyperlevel-$h(T)$-descendants of an initial simplex respectively.
%\end{defn}
\begin{defn}[finer, coarser, refinement, strictly/proper $\sim$]
\index{finer}\index{$\leq$ -- finer}\index{coarser}\index{$\geq$ -- coarser}\index{refinement}\index{strictly finer}\index{proper refinement}
%We call a 
A set of simplices $A$ 
is
\emph{finer} than a second $B$, %and write 
denoted by
$A\leq B$, if
\begin{align*}
\fa S\in A, T\in B.|S\cap T|>0\Ra S\subset T.
\end{align*}
Applying this relation to singletons $\{T\}$ of simplices, 
%we will 
also 
%say ``
$T$ is 
said to be
finer than\ \ldots%'' 
\ and %``
\ldots\ %is 
to be
finer than $T$%''
, respectively. $A$ is \emph{coarser} than $B$ (denoted by $A\geq B$), if $B\leq A$. %By a 
A 
\emph{refinement }of a triangulation of $\Omega$ %we refer to 
is 
a finer triangulation. (See Definition \ref{def:regular} below for the definition of a triangulation.) 
A partition (triangulation) $A$ is called \emph{strictly finer} (\emph{a proper refinement}) than (of) $B$, if it is finer than (a refinement of) $B$ and $\ex S\in A, T\in B.S\subsetneq T$. 
%A partition $A$ is called \emph{strictly finer} than $B$, if it is finer than and $\ex S\in A, T\in B.S\subsetneq T$.
%A triangulation $A$ is called \emph{a proper refinement} of $B$, if it is a refinement of $B$) and $\ex S\in A, T\in B.S\subsetneq T$.
\end{defn}
\begin{bem}
The relation finer is \emph{not} a partial order on the power set of admissible simplices, %whereas the relation “$A$ is a finer partition than $B$” is one.
but it is a partial order on the set of admissible partitions.
\end{bem}
\begin{lem}\label{feinerObermenge}
An admissible partition $P$ is finer than a second $Q$ if and only if its forest is a superset of the second.
\end{lem}
\begin{proof}
%(trivial) 
If the forest of $P$ is a superset of the forest of $Q$, then every leaf of $P$ is a subset of the leaf of $Q$, which is the ancestor of the leaf. Conversely, $T\subset S$ implies for admissible simplices that $T$ is a descendant of $S$. Hence a forest of an admissible partition containing $T$ must contain $S$ as well.
\end{proof}

\begin{thm}\label{count forest not triangulation}
For admissible partitions $\T$, it holds that:
\begin{align*}
\#\T-\#\T_0=\#\operatorname{\text{non-leaves}}(\Wald\T)=\frac12\#\left(\Wald(\T)\setminus \T_0\right).
\end{align*}
\end{thm}
\begin{proof}
For the initial triangulation and its forest, all three terms vanish. Bisecting a simplex is equivalent to adding two children of a leaf to the forest. This leaf becomes a non-leaf and there are two new leaves. This bisection increments the cardinality of the triangulation as well as the set of non-leaves in the forest, while the total number of nodes in the forest increases by 2.
\end{proof}

\subsection{%Conforming 
Regular 
triangulations and their characterisation}\label{sec:regular triangulations}
\subsubsection{Regular triangulations}
%We are especially interested in \emph{regular triangulations (simplicial decompositions)} from admissible simplices.
There is a particular interest in %\emph{conforming 
regular triangulations\index{regular (triangulation)}\index{conforming} (also known as \emph{conforming triangulations} or \emph{simplicial decompositions}\index{simplicial decomposition}) from admissible simplices. Here, both words \emph{regular} and \emph{triangulation} alike refer to this notion individually.
\begin{defn}\label{def:regular}
%[regular, 
[regular, triangulation, admissible triangulation, admissible forest]
%TODO shorten
%A pair of simplices $(S,T)$ is called 
%\emph{regular}, 
%if the intersection of $S$ and $T$ is a common subsimplex of $S$ and $T$. 
A finite set of simplices is %called 
\emph{regular}%
%\normalindex{regular}
%regular% and a \emphindex{triangulation}
, if 
%every of its pairs is 
%regular.
%conforming. %We call a 
the intersection $S\cap T$ of any two of these simplices is a common subsimplex of $S$ and $T$. A regular set of simplices is also called a \emphindex{triangulation}.
A 
triangulation $\T$\index{T@$\T$ -- admissible triangulation} of $\Omega$ into admissible simplices 
is
an \emphindex{admissible triangulation} and the forest of an admissible triangulation is an \emphindex{admissible forest}.
\end{defn}
\begin{bem}
The admissible triangulations are the meshes for the discretisation of the PDEs. Consideration of the corresponding forests will %us 
develop the structure of this partial ordered set.
\end{bem}

\subsubsection{Characterisation of admissible forests.}
%To characterise the admissible forests, some notions has to be defined at first.
Characterisation of the admissible forests requires some further 
%notions.
lemmas%
, which are proved in the appendix.
%\subsection{Seiten konvexer Mengen}
%Equivalent definitions for a face and proofs for the following lemmas are provided in the appendix.
%\begin{lem}\label{SeitenUntersimplexe}
%The faces of a simplex are its subsimplices.
%\end{lem}

%\begin{lem}\label{UntermengeTeilmenge}
\begin{lemmanonum}{\ref{UntermengeTeilmenge}}
Let $C$ be a convex set with a face $F$. If $T$ is a convex subset of $C$, then $F\cap T$ is a face of $T$.
%\end{lem}
\end{lemmanonum}
%\begin{proof}
%$F\cap T$ is affinely closed in $T$: 
%\begin{align*}
%\aff(F\cap T)\cap T&\subset \aff F\cap C=F\qquad | \cap T\\
%%\aff(F\cap T)\cap T&\subset T\\
%\aff(F\cap T)\cap T&\subset F \cap T.
%\end{align*}
%Additionally, $T\setminus (F\cap T)=T\cap (C\setminus F)$ is an intersection of convex sets, hence convex.
%\end{proof}
%%We will %almost 
%This 
%%Lemma \label{UntermengeTeilmenge} 
%lemma will be used in the form of the following corollary exclusively.
Instead of this lemma, the following corollary is almost exclusively used.
\begin{cornonum}{\ref{NachkommeUntersimplex}}
%If $S\cap T$ is a subsimplex of $T$ and $\hat T$ a descendant of $T$, so also $S\cap \hat T$ a subsimplex of $\hat T$.
Suppose $T$ and $S$ are simplices, $T'$ is a descendant of $T$ and $T\cap S$ is a subsimplex of $T$. Then $S\cap T'$ is a subsimplex of $T'$.
\end{cornonum}
%%Wird \label{eine} und \label{eine2} überhaupt gebraucht??
%%\begin{lem}\label{eine}
%\begin{lemmanonum}{\ref{eine}}
%%If a union of faces is convex, then it is a face itself.
%If a union of faces of a convex set $C$ is convex, then the union itself is a face of $C$.
%%\end{lem}
%\end{lemmanonum}
%Here, only the following conclusion is needed.
%\begin{cornonum}{\ref{eine2}}
%Let $C$ and $K$ be convex sets, and $C=\bigcup\left\{C_i~\middle|~i\in I\right\}$ for convex $C_i$. 
%%Then we have the following implication: 
%Then the following implication holds true.
%\begin{align*}
%\fa i.C_i\cap K \text{ is a face of $K$. }\quad\Rq C\cap K \text{ is a face of $K$.}
%\end{align*}
%\end{cornonum}
%\begin{proof}
%Apply Lemma \ref{eine} with $S_i=C_i\cap K$.
%\end{proof}

Hereafter, some binary relations on $\Simplexe$, denoted by arrows $T\ra S$ will occur, which are vividly worded by “$T$ demands\index{demand} $S$”.
\begin{defn}[closed with respect to a relation]
\index{closed with respect to a relation}
Let $\ra$ be a binary relation on $\Simplexe$. 
%We call a 
A
subset $W\subset \Simplexe$ 
is
\emph{closed} 
%closed
with respect to $\ra$ (or briefly $\ra$-closed\index{->-closed@$\ra$-closed}), if 
\begin{align*}
(T\ra S)\Ra (T\in W \Ra S\in W).
\end{align*} 
%Eine Menge ist genau dann abgeschlossen bezüglich $E$, wenn sie abgeschlossen bezüglich dem transitiven Abschluss von $E$ ist.
\end{defn}
\begin{defn}[$\ra$]
Let $\ra$ be the relation on %the admissible simplices 
$\Simplexe\times\left(\Simplexe{\setminus} \T_0\right)$
defined by
\begin{align*}
T\ra S \text{ (in words: ``$T$ \emph{demands }$S$.") }\quad:\LRq T\cap\parent S\notin\Sub (\parent S).
\end{align*}
(To say it in layman's terms: If and only if $T$ meets $\pa S$ in a \emph{fraction }of a subsimplex of $\pa S$, it demands bisection of $\pa S$, i.e.\ the children of $\pa S$.)
\end{defn}
\begin{thm}[First characterisation of admissible forests]\label{TpaS}
\index{characterisation of admissible forests!first}
%Let $\T_0$ be a partition of a polytope into tagged simplices. 
%Then a 
A
finite subgraph $W\subset\Simplexe$ containing $\T_0$ is an admissible forest if and only if it is closed w.r.t.\ $\ra$.
\end{thm}
\begin{proof}
\emph{1. The forest $W$ of an admissible triangulation is $\ra$-closed.} Let $T\in W$, $S\in \Simplexe$ with $T\ra S$. If $S$ were not in $W$, there would be a proper ancestor $\tilde S$ of $S$ in the triangulation. But %its 
the
intersection 
of $\tilde S$
with $T$ would be a subsimplex of $\tilde S$. Hence according to Corollary \ref{NachkommeUntersimplex}, $\parent(S)\cap T$ %=\parent(S)\cap (\tilde S \cap T)$ 
would be a subsimplex of $\parent(S)$, contradicting $T\ra S$.

\emph{2. A finite $\ra$-closed set $W$ containing $\T_0$ is a forest of a 
%conforming 
regular 
triangulation $\T$.} 
%Offenbar gilt $T\ra S$, wenn $S$ Bruder oder Vater von $T$ ist, also ist $W$ ein voller binärer Unterwald von $\Simplexe$, der alle Wurzeln enthält. 

\emph{2.1 $W$ is a full binary subforest of $\Simplexe$ with roots $\T_0$.}\\
$T\ra \operatorname{sibling} T$, because $T\cap\underbrace{\pa(\operatorname{sibling} T)}_{\pa T}=T\notin\Sub \pa T.$\\
$T\ra \pa T$, because $T\cap\pa\pa T=T\notin\Sub \pa\pa T$.
Consequently, $W$ is a full binary subforest of $\Simplexe$ and if $T$ is a root of $W$, i.e.\ does not have a parent in $W$, it does not have a parent in $\Simplexe$ either and is an initial simplex.

\emph{2.2 The corresponding set of leaves (partition) is 
%conforming
regular}, 
because whenever $S\cap T\notin\Sub S$, then $T\ra S'$ for each child $S'$ of $S$ and hence: If $T\in W$, $S$ cannot be a leaf of $W$.
%Sei $T\in \T$, $T\cap S\notin \Sub S$. 
\end{proof}
\begin{bemn}
\begin{itemize}
\item
The proof works also under more general preconditions: Simplices could be replaced by convex polytopes and the initial simplices (and polytopes respectively) even do not need to partition the domain.
\item
On the other hand, the theorem does not show that there are admissible triangulations at all (which are for example finer than a certain simplex). This requires further \emph{initial conditions} 
%we will discuss 
presented
in 
%a moment.
the next subsection.
\end{itemize}
\end{bemn}
\begin{folg}
[Overlay and underlay]
%[Finest common recoarsement and coarsest common refinement]
\label{overlay}
The set of admissible triangulations is a \emphindex{lattice} (in terms of set theory) respecting the finer relation. Concretely, for each two admissible triangulations $\T_1$ and $\T_2$, %ist 
%\begin{align*}
%\T_1\vee\T_2=\left\{T_1\cup T_2~\middle|~ T_1\in\T_1, T_2\in\T_2,|T_1\cap T_2|>0\right\}
%\end{align*}
\emph{the leaves of the intersection of the corresponding forests}
are the finest %zulässige Simplizialzerlegung, die gröber als beide ist, und 
common recoarsement (underlay)\index{finest common recoarsement}\index{underlay}
%\begin{align*}
%\T_1\wedge\T_2=\left\{T_1\cap T_2~\middle|~T_1\in\T_1, T_2\in\T_2,|T_1\cap T_2|>0\right\}
%\end{align*}
and \emph{the leaves of the union of the corresponding forests}
%die gröbste zulässige Simplizialzerlegung, die feiner als beide ist.
the coarsest common refinement\index{coarsest common refinement}\index{overlay} (overlay).
\end{folg}
\begin{proof}
According to Lemma \ref{feinerObermenge}, finer triangulations correspond to greater forests. So the forest of the finest common recoarsement can at most include the intersection $\Wald (\T_1)\cap \Wald (\T_2)$, for being a recoarsement. The intersection of $\ra$-closed finite sets including $\T_0$ is finite, $\ra$-closed and contains $\T_0$, thereby an admissible forest as well. Hence the leaves of this intersection form the finest common recoarsement. %Für 2 zulässige Simplexe ist $|T_1\cap T_2|$ nur positiv, falls $T_1$ ein Vorfahr von $T_2$ oder umgekehrt. Sagen wir beispielsweise $T_1$ ist ein Vorfahr von $T_2$. Dann ist $T_1\cup T_2=T_1$, was ein Blatt des Schnittes der Wälder ist.
The argumentation for the coarsest common refinement is analogue.
\end{proof}
\begin{bem}
Together with Theorem \ref{count forest not triangulation}, this implies \cite[Lemma 3.7]{Cascon}. The proof in \cite{Cascon} implicitly assumes that a non-regular partition would contain a non-regular vertex.
\end{bem}
\subsection{%Strong initial conditions, general refineability, pairwise compatibility 
General refineability, compatibility conditions
and 
revision of the characterisation theorem
% under preconditions
}\label{sec:IC}

The existence of arbitrarily fine finite triangulation is formulated as \emph{general refineability} in Definition \ref{GR}.
Initial conditions on $\T_0$ are supposed to ensure %that there exist arbitrary fine finite triangulations. 
it. 
The following initial conditions appear in this thesis: 
\begin{itemize}
\item
%TODO ReTaHyCo or IsoCoChange?
\emph{strong initial conditions} (\hyperlink{SIC}{SIC}, Definition \ref{SIC}), 
\item
the much weaker \emph{restricted T-array coincidence} (\hyperlink{ReTaCo}{ReTaCo}, Definition \ref{def:ReTaCo}) 
%in Section \ref{sec:ReTaCo}  
and 
\item
\emph{restricted T-array and hyperlevel coincidence} (\hyperlink{ReTaHyCo}{ReTaHyCo}, Definition \ref{def:ReTaHyCo}) which implies 
\item
\emph{isometric coordinate changes} (\hyperlink{IsoCoChange}{IsoCoChange}, Definition \ref{def:IsoCoChange}). %in Section \ref{sec:IsoCoChange}. 
\end{itemize}
Additionally, pairwise compatibility (\hyperlink{PC}{PC}, Definition \ref{def:PC}), a very weak condition on $\Simplexe$, is used to prove the second characterisation of admissible forests (Theorem \ref{Knotenschritt}).
\subsubsection{General refineability%, strong initial conditions and pairwise compatibility
}

\begin{defn}[tower, $\ra$-closure]
\index{tower}\index{->-closure@$\ra$-closure}
The \emph{\hypertarget{tower}{tower} }of an admissible simplex $S\in\Simplexe$, $\Tu(S)$\index{Tw@$\Tu$ -- tower} %oder $T{\ra}$ 
is the \emph{$\ra$-closure }of 
%$\T_0\cup\{S\}$, 
$S$, 
i.e.\ the smallest $\ra$-closed subgraph of $\Simplexe$ containing $S$.
%, united with $\T_0$.
\end{defn}
%Let $M\subset \Simplexe$ be finite. The lattice structure on the admissible triangulations directly implies equivalence of the following statements (even without supposing any initial condition):
%\begin{itemize}
%\item There is a (finite) admissible triangulation finer than $M$.
%\item For each $T\in M$, there is a (finite) admissible triangulation finer than $T$.
%\item For each $T\in M$, the tower $\Tu(T)$ is finite and therefore an admissible forest.
%\end{itemize}
\begin{defn}[generally refineable]\label{GR}\index{generally refineable}
If the tower $\Tu(T)$ is finite and therefore an admissible forest for each $T\in \Simplexe$, 
%If these three conditions are fulfilled for arbitrary $M\subset \Simplexe$ (or, sufficiently, all $T\in\Simplexe$), %we call 
%the set of admissible triangulations 
$\Simplexe$
is %called
\emph{generally refineable}.
%%%%%
%%Hermanns Vorschlag
%The set $\Simplexe$ is generally refineable, if the following equivalent conditions 
\end{defn}
\begin{bem}
The lattice structure on the admissible triangulations (Corollary \ref{overlay}) directly implies that for any finite subset $M$ of a generally refineable $\Simplexe$, there is a (finite) admissible triangulation finer than $M$, namely the triangulation with the forest $\bigcup_{S\in M} \Tu S$.
\end{bem}

\subsubsection{Strong initial conditions (SIC)}
\begin{defn}[strong initial conditions (\hypertarget{SIC}{SIC})]\label{SIC}
\index{initial conditions!strong}\index{strong initial conditions}\index{SIC -- strong initial conditions}
%\index{RefCoCo -- Reference coordinate coincidence}
%Eine Startzerlegung $\T_0$ erfülle die Startbedingung, falls zu jedem Simplex $S$ eine Abbildung $\phi_S: S\rightarrow \R^n$ auf das Referenzsimplex existiert, sodass für je zwei Simplexe $S$ und $T$ eine Abbildung $\psi\in G_n$ existiert mit %$\phi_S|_{S\cap T}\circ (\phi_T|{S\cap T})\inv$ eine 
%$(\psi\circ \phi_S)|_{S\cap T}=\phi_T|_{S\cap T}$. 
The \emph{strong initial conditions (%abbreviated by 
SIC)} are:
\begin{enumerate}
\item
$\T_0$ is 
%conforming. 
regular.
\item
All tagged simplices in $\T_0$ are of the same type.
\item
For each 2 tagged
%initial 
simplices $S$ and $T$ in $\T_0$, there are reference coordinates $\vi_S$ and $\vi_T$ coinciding on $S\cap T$.
\end{enumerate}
\end{defn}
Since 
%we can reuse 
reference coordinates 
can be reused
for children, the 3rd property is transferred for each \emph{admissible} simplex.
\begin{thm}[SIC allow uniform refinements]\label{thm:uniform refinements}\index{uniform refinements for SIC}
The strong initial conditions imply that also uniform refinements, i.e.\ the set of all level $l$ simplices, fulfil the strong initial conditions.
\end{thm}
\begin{proof}
It suffices to show one induction step, i.e.\ 
%we only need 
to prove it for the uniform refinement of level 1. The 2nd and the 3rd condition hold obviously. 
For a regular set of simplices $\{S,T\}$, the children of the simplices $S$ and $T$ form a 
%conforming 
regular 
triangulation if 
%and only if 
every common edge is either the refinement edge of both or of neither of them.
% (exercise). 
(This is explained in detail in the proof of Lemma \ref{PAF} in the case 2.2.) 
If the simultaneous type of the simplices is 0, transpose them all, to have well-defined refinement edges.

%But according 
According
to Corollary \ref{edgelengths}, the refinement edge of a $(k,h)$-simplex is unique\-ly determined as the preimage of the unique one in the reference simplex with Chebyshev length $2^{-h}$ and Manhattan length $2^{-h}k$. Since the reference coordinates coincide on $S\cap T$, this property is independent of the simplex, so every common edge is either the refinement edge of both or of neither of them.
%(exercise).
%Seien also $S$ und $T$ zwei Simplexe vom Level 1. Sind sie Geschwister, so ist $(S,T)$ regulär. Sind $S$ und $T$ die Kinder von $\tilde S\neq \tilde T$, 
\end{proof}
\begin{folg}
The SIC imply that the tower of a simplex $S$ cannot contain more simplices than all of level $\leq lS$. Thus, they imply general refineability.
\end{folg}
\begin{lem}[Levels of vertices]\label{eckenlevel}\index{level!of a vertex}\index{l@$l$ -- level!of a vertex}
If all uniform refinements are regular, the assignment $l(\Vnew T):=l(T)$ defines a function (which %we will 
is
also call%
ed 
\emph{level}) on $\Vertices\Simplexe\setminus\Vertices\T_0$.
\end{lem}
\begin{proof}
Any vertex in  $\Vertices\Simplexe\setminus\Vertices\T_0$ is the new vertex of some simplex, so is assigned to at least one number. It remains to prove uniqueness. Let $p=\Vnew(S)=\Vnew(T)$. It %is to show 
has to be shown
that $l(T)=l(S)$. Note that $p$ is not a vertex of a proper ancestor of $S$. If for example 
%$l(T)\leq l(S)-1$, 
$l(T)< l(S)$, 
then in the uniform refinement of level $l(T)$ $p$ would be a hanging node on the ancestor of $S$ of level $l(T)$.
\end{proof}

%SIC ensure that uniform refinements are admissible, so the tower of a simplex $S$ cannot contain more simplices than all of level $\leq l(S)$. Thus, they imply general refineability.

\begin{bemn}
%%Ausführlich gemacht
%The strong initial conditions 
%%are equivalent to 
%follow from
%those given by Stevenson \cite[at the beginning of Section 4]{Stevenson} (see below).

In Section \ref{sec:BDV}, SIC are assumed throughout. %and maybe the strongest sensible ones. We will not use others here. 
%But it is unknown, if for of an arbitrary given triangulation the vertices of each simplex can be arranged into T-arrays complying the strong initial conditions. 
An example for a 3D triangulation, which cannot be fit with T-arrays for each simplex
complying the strong initial conditions is given in \cite[Lemma 1.7.14, p.~26]{Schoen}.
Only for the 2-dimensional case, Biedl, Bose, Demaine and Lubiw showed that 
%in the 2-dimensional case 
SIC is realisable effectively \cite{biedl}.

%%Steht nun in IsoCoChange
%Karkulik, 
%%Pavli\v cek
%Pavlicek and Praetorius showed for the 2-dimensional case that 
%%conformity
%regularity
%of the initial mesh is sufficient for general refineability and even for %the linear regular closure (i.e.\ 
%the BDV theorem (see Theorem \ref{thm:BDV} on page \pageref{thm:BDV} below) \cite{Karkulik}. On the other hand, they got an additional factor for the constant in the BDV theorem. 
%%and Biedl, Bose, Demaine and Lubiw showed that in the 2-dimensional case SIC is realisable effectively \cite{biedl}. 
%For the $n$-dimensional case, recently Alkämper, Gaspoz and Klöfkorn gave weaker initial conditions easy to realise, which still guarantee general refineability \cite{Gaspoz}. 
For certain weaker initial conditions, Section \ref{sec:ReTaCo} shows general refineability and the existence of quasi-uniform refinements and Section \ref{sec:IsoCoChange} shows the BDV theorem.
%For these weaker initial conditions, the BDV theorem is harder to prove and the constant is worse.
%This is discussed in Section \ref{sec:IsoCoChange} below. 

%Nevertheless, for simplicity (especially in the proof of the BDV theorem) we will assume SIC throughout.
\end{bemn}
The remainder of this subsubsection is devoted to deduce 
%SIC from Stevenson's initial conditions 
the strong initial conditions from those given by Stevenson in
\cite[at the beginning of Section 4 on page 232]{Stevenson}.
\begin{lem}\label{reference parent}
\begin{enumerate}
\item\label{it:reference parent}
Let $S$ be a tagged simplex and $T$ be a reference simplex which is a child of $S$ in the extended binary forest. Then also $S$ is a reference simplex.
\item\label{it:reference coordinates parent}
Reference coordinates for a child of a T-array $S$ are also reference coordinates for $S$ itself.
\end{enumerate}
\end{lem}
\begin{proof}[Proof of \ref{it:reference parent}]
As a reference simplex, $T$ is the child of a reference simplex $S'$. 

\emph{1st case: $t(T)=n$.} There is only one simplex, whose child can be $T$, namely $S=T^T=S'.$

\emph{2nd case: $t(T)<n$.}
As in Lemma \ref{possibleancestorequalhyperlevel} below on page \pageref{possibleancestorequalhyperlevel}, there is an $\|\bullet\|_\infty$-isometry $\phi$ mapping signed canonical unit vectors to each other and the vertices of $S'$ to the vertices of $S$. $\phi$ maps the root $R'$ of $S'$, which is a Kuhn simplex, to a Kuhn simplex $R$, which has $S$ as a descendant. This shows that also $S$ is a reference simplex.
\end{proof}
\begin{proof}[Proof of \ref{it:reference coordinates parent}]
Let $\vi$ be reference coordinates for a child $T$ of $S$. Since $\vi(T)$ is a reference simplex and a child of $\vi(S)$, statement \ref{it:reference parent} shows that $\vi(S)$ is also a reference simplex. Consequently, $\vi$ are also reference coordinates for $S$.
\end{proof}

Stevenson assumes 3 initial conditions: that all initial simplices are of the same type $\gamma$, that the initial partition is conforming (i.e.\ regular) and the so called condition (b), which reads:
``Any two neighbouring\index{neighbour} [i.e.\ meeting in a hyperface] tagged simplices $T=(x_0,\dots,x_n)_\gamma$, $T'=(x'_0,\dots,x'_n)_\gamma$ from $P_0$ \emph{match} in the sense that if $\overline{x_0x_n}$ or $\overline{x'_0x'_n}$ is on $T\cap T'$, then $T$ and $T'$ are reflected neighbours\index{reflected neighbours} [i.e.\ the T-array of either $T$ or $T_R$ coincides with that of $T'$ in all but one position]. Otherwise, the pair of neighbouring children of $T$ and $T'$ are reflected neighbours.''

%In the language here, this means: Any two tagged simplices $T$ and $T'$ intersecting in a hyperface (this is meant by \emph{neighbouring}) match in the sense that if $\Eref T$ or $\Eref T'$ is on $T\cap T'$, then $T$ and $T'$ are \emph{reflected neighbours}, i.e.\ the T-array of either $T$ or $T_R$ coincides with that of $T'$ in all but one position. Otherwise, the pair of neighbouring children of $T$ and $T'$ are reflected neighbours.

\begin{lem}\label{reflected neighbours RefCoCo}
For reflected neighbours $T$, $T'$
and given reference coordinates $\vi_T$ for $T$%
, there exists 
%a pair of 
reference coordinates
$\vi_{T'}$ for $T'$
coinciding with $\vi_T$ on $T\cap T'$.
\end{lem}
\begin{proof}
Let 
$\begin{pmatrix}
p_0&\dots&p_k\\
	&\vdots\\
	&p_n
\end{pmatrix}\in\left\{T,T_R\right\}$ 
and
$T'=\begin{pmatrix}
p'_0&\dots&p'_k\\
	&\vdots\\
	&p'_n
\end{pmatrix}$ be reflected neighbours, i.e.\ $p_j=p'_j$ for all but one $j\in\{0,\dots,n\}$ which are the vertices $\Vertices T\cap \Vertices T'$.
Choose reference coordinates 
$%\vi_T,
\vi_{T'}$ 
%which maps the equal vertices of $T$ and $T'$ at equal (or reflected) positions to the same vertex in the reference simplex and also the two remaining vertices at the equal (or reflected) positions to the same vertex in the reference simplex each, i.e.\ 
with 
$\vi_{T'}(p'_j)=\vi_T(p_j)$ for all $j\in\{0,\dots,n\}$. These reference coordinates coincide on $T\cap T'=\conv\left(\Vertices T\cap \Vertices T'\right)$.
\end{proof}
Since Stevenson postulates an open domain $\Omega$ instead of the polytope $\Omega$ here and uses a slightly different definition of a regular triangulation, the following condition is needed to connect the theory in this thesis here with his one. 
\begin{defn}[hyperface-connected]\index{hyperface-connected}
%A triangulation $\T$ is hyperface-connected, if for every pair $S,T\in\T$ 
Two simplices $S,T\in\T$ are \emph{hyperface-connected}, if
there is a sequence $S=S_1,\dots,S_N=T$ such that any two successive simplices 
$S_j,S_{j+1}$ 
of the sequence meet in a hyperface
and the intersection $\bigcap_{j=1}^N S_j$
equals $S\cap T$.
%A triangulation is hyperface-connected, if any 
\end{defn} 
%Stevenson's triangulations are generally hyperface-connected (see the proof of \cite[Theorem 3.2]{Stevenson}).
In his definition of regularity, Stevenson asks for regularity only of hyperface-connected pairs of simplices.
\begin{lem}
For a hyperface-connected initial partition, Stevenson's initial conditions imply SIC
for all pairs of hyperface-connected simplices.
\end{lem}
\begin{proof}
Since the first and the second SIC are postulated also by Stevenson, only the third SIC, the existence of coinciding reference coordinates has to be concluded yet.
If two neighbouring simplices are reflected neighbours, then Lemma \ref{reflected neighbours RefCoCo} gives reference coordinates coinciding on their intersection. If otherwise their neighbouring children are reflected neighbours, then according to Lemma \ref{reference parent}.\ref{it:reference coordinates parent}, the reference coordinates for these children are also reference coordinates for the neighbouring simplices itself.

Now let $S$ and $T$ be two arbitrary simplices with a sequence $S=S_1$, $\dots$, $S_N=T$ as in the definition of hyperface-connectivity. Let $\vi_{S}$ be reference coordinates for $S$ and for each $j$ choose reference coordinates $\vi_{S_{j+1}}$ coinciding with $\vi_{S_j}$ on $S_j\cap S_{j+1}$. 
%by means of Lemma \ref{reflected neighbours RefCoCo}.
Then all $\vi_{S_j}$ coincide on $S_1\cap\dots\cap S_N=S\cap T$.
%For each pair $S_j,S_{j+1}$ there are reference coordinates $\vi_{S_j}$ and $\vi_{S_{j+1}}$ coinciding on $S_j\cap S_{j+1}$. For each pair, there is a free choice of $\vi_{S_j}$ 
\end{proof}
%\paragraph{Exercises.} 
%\begin{enumerate}
%\item
%The children of two simplices $S$ and $T$ form a regular triangulation if and only if every common edge is either refinement edge of both or of none of them.
%\item
%Show that the reference edge of the $(k,h)$ reference simplex is its unique one with Chebyshev length $2^{-h}$ and Manhattan length $2^{-h}k$.
%\item
%Given an initial triangulation of tagged simplices, how can you decide efficiently, if the strong initial conditions hold true? %(Hint: Identify the unit vectors.)
%\item
%(probably tedious) Show the equivalence with Stevenson's initial conditions.
%\item
%(probably tedious and difficult) What are the weakest possible initial conditions for the 
%%conformity 
%regularity of every uniform refinement? Stevenson did not give the answer, because he proved the necessity of his condition only assuming our 2nd one.
%\end{enumerate}
%%(1 is maybe boring, 4 and 5 are not suitable as obligatory homework.)

\subsubsection{Pairwise compatibility}
Another compatibility condition is formulated, which is even weaker than general refineability. The second characterisation of admissible forests (Theorem \ref{Knotenschritt} below) will be needed for further compatibility conditions later on, so it is proved for the absolutely necessary compatibility condition already here.

A reader who is not interested in this extended generality is recommended to read Lemma \ref{PAF}, its proof for SIC and Corollaries \ref{SubUSub} and \ref{hanging->new} and to continue with Section \ref{Revision of the characterisation theorem}.
\begin{defn}
%[(rooted) path in a forest]
[rooted path in a forest]
%Let $G$ be a directed graph. A \emphindex{path} in $G$ is a finite or infinite sequence of nodes $N_1,\dots$ such that there is an edge from $N_j$ to $N_{j+1}$ each. Here, the graph will always be a directed forest and the existence of this edge is equivalent to $\pa N_{j+1}=N_j$. The fact that in the forest of admissible simplices children are included in their parents gives rise to write $S_1\supset S_2\supset\dots$ instead of $S_1,S_2,\dots$. 
A 
finite or infinite
path $S_0\supset\dots$ 
in the forest of admissible simplices $\Simplexe$
is \emphindex{rooted}, if node $S_0$ is a root (in other words an initial simplex). Without further explanation, $S_1\supset\dots$ denotes an infinite, $S_1\supset\dots\supset S_k$ a finite path.
\end{defn}
\begin{defn}[pairwise compatible]\label{def:PC}
\index{pairwise compatible}\index{PC -- pairwise compatible}
The forest of admissible simplices $\Simplexe$ is \emph{pairwise compatible} (\hypertarget{PC}{PC}), if the following conditions are satisfied:
\begin{enumerate}
\item
The initial partition $\T_0$ is regular.
\item
If for an arbitrary regular pair $\{S,T\}\subset\Simplexe$ the refinement edges $\Eref S$ and $\Eref T$ both lie in $S\cap T$, they coincide.
\item
%Given a path $S_1\supset\dots$ in the forest $\Simplexe$ and an arbitrary edge $e$ of $S_1$, there exists a simplex $S_j$ in that path with $e=\Eref S_j$.
For any arbitrary edge $e$ in $\Edges(\Simplexe)$, there is no infinite path $S_1\supset\dots$ in the forest $\Simplexe$ with $e$ being an edge of each $S_j$.
%$\fa j\in\N_{\geq 1}.e\in\Edges S_j$.
\end{enumerate}
\end{defn}
The third condition is satisfied whenever Maubach bisection is used. The first two are implied by SIC as well as by ReTaCo (see Corollary \ref{ReTaCo=>PC} below on page \pageref{ReTaCo=>PC}).
\begin{bem}
Nevertheless, this notion is not restricted to Maubach bisection, but could be used for example also for Longest Edge Bisection\index{Longest Edge Bisection}.
\end{bem}

%\begin{lem}\label{PAF single}
%If $\Simplexe$ is pairwise compatible, then for each $S\in\Simplexe$ and each rooted path $T_0\supset\dots$ in the forest $\Simplexe$ there is a $k\in\N$ such that $\{S,T_k\}$ is regular. Moreover, given such a partner $T_k$ for $S$, for any descendant $S'$ of $S$ there is a descendant $T_l$ of $T_k$ such that $\{S',T_l\}$ is regular.
%\end{lem}
\begin{lem}
%[Pairwise arbitrarily fine]
%[Fundamental property of PC]
\label{PAF}
Suppose that the forest of admissible simplices $\Simplexe$ is pairwise compatible and let $S_0\supset\dots$, $T_0\supset\dots$ be two rooted paths in 
%it.
$\Simplexe$. 
Then there exists a sequence of pairs $(a_0,b_0)%=(0,0)
,(a_1,b_1),\dots\in\N^2$ such that:
\begin{itemize}
\item
$\{S_{a_j},T_{b_j}\}$ is regular.
\item
%$a_0=b_0=0$
%\item
%$(a_{j+1},b_{j+1})\in(a_j,b_j)+\{(1,0),(0,1),(1,1)\}$
%\item
%Neither $j\mapsto a_j$ nor $j\mapsto b_j$ is bounded, i.e.\ $\fa j\ex k>j.a_k=a_j+1$.
$j\mapsto a_j$ and $j\mapsto b_j$ are both monotonic increasing surjective maps from $\N$ to $\N$.
\end{itemize}
%In other words,
%$j\mapsto a_j$ and $j\mapsto b_j$ are both monotonic increasing surjective maps from $\N$ to $\N$.
\end{lem}
\begin{proof}[Proof assuming SIC additionally]
The SIC imply the regularity of uniform refinements (Theorem \ref{thm:uniform refinements}), thus $(a_j,b_j):=(j,j)$ fulfil the request.
\end{proof}
\begin{proof}[Proof by mathematical induction]

$\{S_0,T_0\}$ is regular
by definition.
Suppose that $\left\{S_{a_j},T_{b_j}\right\}$ is a regular pair.
%The first statement is that $\{S_{a+1},T_b\}$, $\{S_a,T_{b+1}\}$ or $\{S_{a+1},T_{b+1}\}$ is regular. %Consecutively, a few descendants $T_j\in \{T_i,T_{i+1},\dots\}$ are considered, beginning with $T_i$.
It suffices to give $a_k$ and $b_k$ for a few successors $k\in\{j{+}1,\dots,l\}$ with
\begin{align}
a_{j+1}=\dots=a_{l-1}=a_j,\quad
a_l=a_j{+}1\label{eq:asuccessors}
\end{align}
 and
\begin{align}
b_k\in b_{k-1}+\{0,1\} \text{ each}.\label{eq:bsuccessors}
\end{align}
Then it is clear that $(a_j,\dots,a_l)$ and $(b_j,\dots,b_l)$ are monotonic increasing without gaps (i.e.\ $\{a_j,\dots,a_l\}=[a_j,a_l]\cap\N$ and $\{b_j,\dots,b_l\}=[b_j,b_l]\cap\N$) and that the sequence $k\mapsto a_k$ exceeds $a_j$. (The unboundedness of $j\mapsto b_j$ is proved by interchanging $a$ and $b$ in the argumentation.)

%This is 
Equations \eqref{eq:asuccessors} and \eqref{eq:bsuccessors} are equivalently expressed using less indices as follows: 
\emph{If $\{S_a,T_b\}$ is regular, then there is a sequence of regular pairs $\{S_a,T_{b+1}\},\allowbreak\dots,\allowbreak\{S_a,T_{c-1}\},\allowbreak\{S_{a+1},T_c\}$ with $c\geq b$.}

This latter formulation is going to be proved now.
%Now values are assigned for the $a_k$ recursively, starting with $k=j$. 
Let $\Vertices\Eref S_a=
%\conv
\{p,q\}
%\overline {pq}
$
%Let $p$ be the node of $\Eref S_a$ being also a node of $S_{a+1}$, so 
and $\{q\}:=\Vertices\Eref S_a\setminus \Vertices S_{a+1}$, so
$$S_{a+1}=\conv\left(\left\{\frac{p+q}{2}\right\}\cup\Vertices S_a\setminus\{q\}\right).$$

Starting with $j=b$, let $\{S_a,T_j\}$ be regular. Several cases are distinguished:

1st case: $\Eref S_a$ is not an edge of $T_j$. It follows that $\Vnew S_{a+1}=\frac{p+q}2\notin T_j$. Therefore, 
$$S_{a+1}\cap T_j=\conv(\Vertices S_a\setminus \{q\})\cap T_j=\conv(\Vertices S_a\cap\Vertices T_j\setminus \{q\})\in\Sub S_{a+1}\cap \Sub T_j,
$$ 
%which concludes the induction step
hence, $c=j$.

2nd case: $\Eref S_a\in\Edges T_j$.

2.1 $\Eref T_j\notin\Edges S_a$. As above (with $S_a$ and $T_j$ interchanged) this implies regularity of $\{ S_a,T_{j+1}\}$. Then the distinction of cases is repeated with incremented $j$. But according to the third condition of pairwise compatibility, applied to the edge $\Eref S$ and the path $T_j\supset\dotsb$, this case cannot happen arbitrarily often consecutively, but sometime the first case or $\Eref T_j\in \Edges S_a$ will happen.

2.2 $\Eref T_j\in\Edges S_a$. Pairwise compatibility constrains
the refinement edges of $T_j$ and $ S_a$ to coincide. Thus, $\left\{S_{a+1},T_{j+1}\right\}$ is regular, because one of the two following cases arises: Either $T_{j+1}$ is the child of $T_j$ containing $p$, i.e.\ $\conv\left(\left\{\frac{p+q}2\right\}\cup\Vertices T_{j}\setminus\{q\}\right)$. Then the intersection 
$$S_{a+1}\cap T_{j+1}=\conv\left(\left\{\frac{p+q}2\right\}\cup\Vertices( S_a\cup T_j)\setminus \{q\}\right).
$$ 
Or $T_{j+1}$ is the child of $T_j$ containing $q$, i.e.\ $\conv\left(\left\{\frac{p+q}2\right\}\cup\Vertices T_j\setminus\{p\}\right)$ and then $S_{a+1}\cap T_{j+1}=\conv\left(\left\{\frac{p+q}2\right\}\cup\Vertices( S_a\cup T_j)\setminus \{p,q\}\right).$ (There is not even need for an exception, if $S_a=T_j$.)
%i.e.\ $\left\{S_{a+1},T_{j+1}\right\}$ is regular, 
Therefore, $c=j{+}1$.
\end{proof}

\begin{kor}\label{SubUSub}
For two arbitrary subsimplices $U,V\in \Sub\Simplexe$ of a pairwise compatible forest of admissible simplices $\Simplexe$, it holds that $U\cap V\in\Sub U\cup \Sub V$.
\end{kor}
\begin{proof}
Let $S,T\in\Simplexe$ with $U\in\Sub S,V\in\Sub T$. Let $S_0\supset\dots\supset S_i=S\supset\dots$ and $T_0\supset\dots\supset T_l=T\supset\dots$ be two rooted paths in $\Simplexe$. 
%According to 
Lemma \ref{PAF} provides 
%is a $T_k$ in that path such that $\{S_i,T_k\}$ is regular. W.l.o.g.\ $k\leq l$. Again according to Lemma \ref{PAF single},
%are 
pairs $(a_j=i,b_j)$ and $(a_k,b_k=l)$ corresponding to regular pairs $\{S,T_{b_j}\}$ and $\{S_{a_k},T\}$. W.l.o.g.\ assume that $j\leq k$ and deduce from the monotonicity of the sequences that %$i=a_j\leq a_k$ and 
$b_j\leq b_k=l$, so $T$ is a descendant and a subset of $T_{b_j}$.

$U\cap T_{b_j}=\conv\left(\Vertices(S\cap T_{b_j})\cap\Vertices U\right)$ is a subsimplex of $T_{b_j}$. Since $V$ is a convex subset of $T$ and $T_{b_j}$, Lemma \ref{UntermengeTeilmenge} implies that also $\left(U\cap T_{b_j}\right)\cap V=U\cap V$ is a face (i.e.\ a subsimplex) of $V$.
\end{proof}
\begin{kor}\label{hanging->new}
Suppose a pairwise compatible forest of admissible simplices $\Simplexe$% is pairwise compatible%
%ReSaCo
, let $p\in\Vertices\Simplexe$ and $p\in S\setminus \Vertices S$. Then there exists a proper descendant $S'$ of $S$ such that $p=\Vnew S'$.
\end{kor}
\begin{proof}
%Tei $p\in\Vertices T$, $T\in\Simplexe$. Tei $S_0\subset\dots\subset S_j=S\subset S_{j+1}\subset\dots$ ein gewurzelter Pfad im Wald wie in der Bedingung \ref{passt irgendwann} für PCC, so dass alle $S_k$ den Punkt $p$ (nicht notwendigerweise als Ecke) enthalten. (Von zwei Kindern kann man immer eins auswählen, das $p$ enthält.) Dann gibt es nach PCC ein $k$ mit $\{T,S_k\}$ konform, d.\,h. $p\in\Vertices S_k$. Für das minimale $k$ mit dieser Eigenschaft gilt sogar $p=\Vnew S_k$. Wäre $k\leq j$, so müsste $p$ auch Ecke von $S_j$ sein. (Ecken von Simplexen bleiben bei Bisektion Ecken.) Also ist $S_k$ Nachfahre von $S_j$.
Let $p\in\Vertices T$ for some $T\in\Simplexe$. Let $S_0\supset\dots\supset S_j=S\supset\dots$ be a rooted path in $\Simplexe$, such that all $S_k$ contain the point $p$ (not necessarily as vertex). (From 2 children one can always choose one containing $p$.) 

Then 
%Corollary \ref{PAF single} 
Lemma \ref{PAF} states that for some $k\in\N$, the pair $\{S_k,T\}$ is regular and since $p$ is a vertex of $T$, it must be a vertex also of $S_k\cap T$ and hence also of $S_k$. For the minimal $k$ with $p\in \Vertices(S_k)$, even $p=\Vnew S_k$. If $k\leq j$, $p$ would have to be a vertex also of $S_j$ (because vertices of simplices stay vertices during bisection). Therefore $S_k$ is a descendant of $S_j$.
\end{proof}
Another result of PC, being used in Section \ref{sec:IsoCoChange}, shall be sketched yet. In a regular triangulation $\T$ (even in a regular convex decomposition), the set 
%$\bigcup \{\relint\Sub S~|~S\in\T\}$ 
$\rel\Int\Sub\T$ of relative interiors of subsimplices of $\T$
is a partition of $\Omega$. These partitions of \emphindex{open subsimplices} form an \emphindex{open ternary forest}: When $S$ is bisected into $S_1$ and $S_2$, $\relint S$ is partitioned into $\relint S_1,\relint S_2$ and $\relint (S_1\cap S_2)$. PC ensures the unique bisection for the intersection of two simplices and thus the same partition into three for common subsimplices of two simplices. 

\begin{lem}
If the admissible forest $\Simplexe$ satisfies PC, then in the open ternary forest described above two admissible open subsimplices are either disjoint or one is an ancestor of the other and includes the latter.
\end{lem}
\begin{proof}[Sketch of a proof]
%The induction step, consisting in the investigation of one bisection is very simple: 
The proof resembles the proof of Lemma \ref{binary tree well-defined} with $\bullet\cap\bullet=\leer$ instead of $\lvert{\bullet\cap\bullet}\rvert=0$ and $\bullet\subsetneq\bullet$ instead of $\lvert\bullet\rvert\leq \frac12\lvert\bullet\rvert$: Roots are disjoint, children are included in parents, siblings are disjoint. Any two admissible simplices either have ancestors which are distinct roots or siblings or one is the ancestor of the other.
\end{proof}
\begin{kor}\label{intersectancestor}
If two admissible subsimplices intersect within their relative interior, the relative interior of one is an ancestor of the relative interior of the other in the open ternary forest and the former includes the latter.
\end{kor}
\begin{lem}\label{lem:midinjective}
The function $\mid: \Edges\Simplexe \ra \Omega$\index{mid@$\mid$} mapping an edge to its mid is injective.
\end{lem}
\begin{proof}%[Proof of Lemma \ref{lem:midinjective}]
Let $e,f$ be two admissible edges with common centre. Corollary \ref{intersectancestor} 
%in the Appendix 
implies that one is an ancestor of the other, say $e$ is an ancestor of $f$. But the only descendant of $e$ with the same centre as $e$ is $e$ itself, because $e$ is bisected at its centre.
\end{proof}

\subsubsection{Revision of the characterisation theorem}\label{Revision of the characterisation theorem}
Other demand relations are going to replace $\ra$ in the characterisation of admissible forests to provide insight in the geometric shape of towers.
\begin{defn}[$\ar{0}, \ar{1},\ar{01},\ar\Vertices,\ar\Edges$]\label{ar}
\index{->0@$\ar{0}$}\index{->1@$\ar{1}$}\index{->01@$\ar{01}$}\index{->V@$\ar\Vertices$}\index{->E@$\ar\Edges$}
$\ar{0}, \ar{1}$ and $\ar{01}$ are the relations
on $\Simplexe\setminus\T_0$ defined by
\begin{align*}
T\ar{0} S\quad&:\LRq \Vnew S =\Vnew T\\
T\ar{1} S\quad&:\LRq \Vnew S =\Vnew \parent T\\
T\ar{01} S\quad&:\LRq T\ar{0} S \text{ or } T\ar{1} S,
\end{align*}
%on $(\Simplexe\setminus\T_0)\times(\Simplexe\setminus\T_0)$ and $(\Simplexe\setminus\T_0\setminus\children \T_0)\times(\Simplexe\setminus\T_0)$ respectively.
%and
$\ar{\Vertices}$ is the relation on $\Vertices(\Simplexe)\setminus\Vertices(\T_0)$ with
\begin{align*}
q\ar\Vertices p \quad:\LRq \ex T\in\Simplexe \text{ with }q=\Vnew T,p=\Vnew\parent T
\end{align*}
%%\ar\Edges
and $\ar\Edges$ is the relation on $\Edges(\Simplexe)$ with 
\begin{align*}
f\ar\Edges e\quad:\LRq \ex S\in\Simplexe\text{ with }f=\Eref S, e=\Eref\pa S.
\end{align*}
Recall that in the binary tree, type 0 T-arrays, too, get a refinement edge. $\ar\Edges$ is only used in Section \ref{sec:IsoCoChange}.
\end{defn}

\begin{bemn}
\begin{enumerate}
\item
$\ar0$ is an equivalence relation.
\item\label{it:-> and levels}
According to Lemma \ref{eckenlevel}, \hyperlink{SIC}{SIC} implies the properties 
\begin{align*}
T\ar0 S&\quad\Rq l(S)=l(T)\\
T\ar1 S&\quad\Rq l(S)=l(T)-1\\
q\ar\Vertices p&\quad\Rq l(p)=l(q)-1.
\end{align*}
That is the reason to write $T\ar{1} S$ contrary to alphabetical order.
\end{enumerate}
\end{bemn}

%\begin{bem}
%$\ar0$ is an equivalence relation.
%\end{bem}
%\begin{folg}[of %Theorem \ref{Knotenschritt} and 
%Lemma \ref{eckenlevel}]
%Assuming SIC, the relations $\ar0,\ar1$ and $\ar\Vertices$ have the properties 
%\begin{align*}
%T\ar0 S&\quad\Rq l(S)=l(T)\\
%T\ar1 S&\quad\Rq l(S)=l(T)-1\\
%p\ar\Vertices q&\quad\Rq l(q)=l(p)-1.
%\end{align*}
%\end{folg}

\begin{lem}\label{ra-closed is ar01-closed}
Suppose 
%SIC
\hyperlink{PC}{PC} for $\Simplexe$
and $\dim\Simplexe=\{n\}$ with $n\geq 2$. 
Then the following statements hold true
\begin{enumerate}
\item\label{it:01_lem}
\begin{enumerate}
\item\label{it:Hauptaussage_lem}
A set $W\subset \Simplexe$ of admissible simplices
%including $\T_0$ 
is closed w.r.t.\ $\ra$ if and only if it is closed w.r.t.\ $\ar{01}$.

%(Zum Verständnis: Aufgefasst als Teilmengen von $\T^2$ kann man Relationen vereinigen. $T\left(\ar{0} \cup \ar{1}\right) S$ bedeutet also $T\ar0 S$ \emph{oder} $T\ar1 S$.)
\item\label{it:Kette_lem}
%Darüber hinaus gibt es für alle Paare $T\ar{\cl} S$ des transitiven Abschlusses von $\ar{0}\cup \ar{1}$ eine Kette $T\ar{0}S_1\ar{1}S_2\ar{1}\cdots\ar{1}S_n=S$.
Moreover, for each chain 
%$S_1({\ar0}\cup {\ar{1}})\cdots({\ar0}\cup {\ar{1}}) S_n$, 
$S_1\ar{01}\cdots\ar{01} S_m$, 
%wobei $\ar{01}:={\ar0}\cup {\ar{1}}$, 
there is a subchain $S_1\ar{0}S_{i_1}\ar{1}S_{i_2}\ar{1}\cdots\ar{1}S_{i_k}\ar{1}S_m$.
\end{enumerate}
\item Alternative 
%formulation: 
formulations: The sets closed w.r.t.\ $\ra$
are the sets 
%Let $\ar{\Vertices}$ be the relation on $\Vertices(\Simplexe)$ with
%\begin{align*}
%p\ar\Vertices q \quad:\LRq \ex T\in\Simplexe \text{ with }p=\Vnew T,q=\Vnew\parent T.
%\end{align*}
\begin{enumerate}
\item\label{it:Nnew_lem}
$%\T_0\cup 
\Vnew\inv V$, where %finite %subsets of $\Vnew (\Simplexe\setminus\T_0)$ 
sets of admissible vertices
which are closed w.r.t.\ $\ar\Vertices$
have to be inserted for $V$%
.
%\begin{align*}
%\left\{\T_0\cup \Vnew\inv V~\middle|~V\subset \Vnew (\Simplexe\setminus\T_0) \text{ is closed w.r.t.\ } \ar\Vertices\right\}.
%\end{align*}
%%\ar\Edges
\item\label{it:Eref_lem}
$%\T_0\cup 
\children(\Eref\inv(E))$, where finite sets of admissible edges which are closed w.r.t.\ $\ar\Edges$ have to be inserted for $E$.
\end{enumerate}
\end{enumerate}
\end{lem}
\begin{bem}
%\begin{enumerate}
%\item
%$\ar0$ is an equivalence relation.
%\item
Interpreting $T\ar{01}S$ as “$T$ demands $S$”,
a chain like in 
%Theorem 
\ref{it:Kette_lem} is something like an “argumentation”:
The tower $\Tu T$, the $\ar{01}$ closure of $T$, is the set
\begin{align*}
\left\{S~\middle|~\ex\,\text{chain } T\ar{01}\cdots\ar{01}S\right\}\stackrel{\ref{it:Kette_lem}}{=}\left\{S~\middle|~\ex\,\text{chain } T\ar{0}\cdots\ar{1}S\right\}.
\end{align*}
%\end{enumerate}
\end{bem}

\begin{proof}
Since $n\geq 2$, $\Vnew(\parent T)$ is always a vertex of $T$.

\emph{Proof of \ref{it:Hauptaussage_lem}.}
%Wieder wollen wir zunächst zeigen, dass zulässige Wälder unter $\ar0$ und $\ar1$ abgeschlossen sind. Ist $p=\Vnew S$, so ist es keine Ecke des Vaters von $S$. Also ist $\parent S\cap T$ eine Seite von $S$ mit extremalem Punkt $\Vnew T$, kann also kein Untersimplex von $\parent S$ sein. Nach Satz \ref{} folgt damit: Ein zulässiger Wald muss mit $T$ auch $S$ enthalten.
Let $W$ be closed w.r.t.\ $\ra$. To prove that $W$ is also closed w.r.t.\ $\ar{01} 
%{=} \ar0 {\cup} \ar1
$, let $T\in W$ and %$T\left(\ar0\cup\ar1\right) S$. 
$T\ar{01} S$. 
To prove $S\in W$, it suffices to show %$T\ar0 S$ oder $T\ar1 S$ 
$T\ra S$. 
If
%$T\left(\ar0\cup\ar1\right) S$, 
$T\ar{0} S$ or $T\ar1 S$, 
then the vertex $p=\Vnew S$ must be a vertex also of 
%$\Vertices(T)$ 
$T$, and is not a vertex of the parent of $S$. Consequently, $\parent( S)\cap T$ is a %Seite von $S$ 
set with extreme point $p\in\Vertices(T)$, so it cannot be a subsimplex of $\parent(S)$. That is $T\ra S$ by definition.

Conversely, let $W$ be an $\ar{01}$-closed set. 
%including $\T_0$. It suffices to 
To
show that $W$ is also $\ra$-closed, %und damit nach Satz ein zulässiger Wald.
let $T\in W$ and $T\ra S$, so $S\notin\T_0$. It suffices to show $S\in W$. 
%If $S\in\T_0$, it is already done. %%nicht nötig, weil ja $\pa S$ existieren muss.
%So let $S\notin \T_0$. 
A chain $T\ar{01}S_1\ar{01}\cdots\ar{01}S_n\ar{01} S$ is wanted. By definition, $T\ar1\parent(T)$, if $\parent(T)\notin\T_0$. Hence a good ansatz is a chain $T=\pa^0 T\ar1\cdots\ar1\pa^k T\ar0 \hat S\ar1\cdots\ar1\pa^j \hat S=S$ and %we only need 
an ancestor %$\tilde T\sup T$ 
of $T$ and a descendant %$\hat S$ 
of $S$ with common new vertex 
is needed
yet. %Für jede Ecke $p\in \Vertices(T)$, die in der Starttriangulierung nicht vorkommt, gibt einen Vorfahren $\pa^j(T)$ von $T$ mit $\Vnew(\pa^j T)$. Damit haben wir TODO $\pa^0 T\ar1\cdots\ar1\pa^j T$. $W$ enthält also mit $T$ auch $\pa^j(T)$ mit neuer Ecke $p$.
Recall that $T\ra S$ means that $U:=T\cap\parent(S)\notin \Sub(\parent S)$. %$U$ is a simplex. 
According to Corollary $\ref{SubUSub}$, $U$ is a subsimplex of $T$.
Not all vertices of $U$ can be vertices of $\parent(S)$, otherwise $U$ would be a subsimplex of $\parent(S)$, say $p\in \Vertices(U)\setminus\Vertices(\parent S)$. %$p\notin \Vertices(\T_0)$, denn sonst wäre $p$ kein hängender Knoten auf $\parent S$ und 
According to Corollary \ref{hanging->new}, there is a proper descendant of $\pa S$, i.e.\ a descendant of $S$ with new vertex $p$, and $p\notin\Vertices(\T_0)$. %According to Lemma $\ref{SubUSub}$, the simplex 
$U$ is a subsimplex of $T$, i.e.\ $p\in\Vertices U\subset \Vertices T$, so there exists an ancestor of $T$ with new vertex $p$. %Wegen $\pa^j(T)\ar0 \hat S$ muss $W$ auch $\hat S$ enthalten und damit schließlich auch $S$.
\qed

\emph{On \ref{it:Kette_lem}.} $S_1\ar0 S_2\ar0 S_3$ means $\Vnew S_1=\Vnew S_2=\Vnew S_3$ and implies $S_1\ar0 S_3$. $S_1\ar1 S_2\ar0 S_3$ means $\Vnew \parent S_1=\Vnew S_2=\Vnew S_3$ and implies the contraction $S_1\ar1 S_3$. Thus, all $\ar0$ not leading a chain $p_1\ar{01}\cdots\ar{01}p_m$ can be eliminated from it preserving a valid chain from $p_1$ to $p_m$.

\emph{\ref{it:Nnew_lem}}\ 
%is only a reformulation of 
follows from
%1.\ 
\ref{01}
solely by set operations. Note that preimages of $\Vnew$ are always $\ar0$-closed.\qed

\emph{On \ref{it:Eref_lem}.} 
%If $f\ar\Edges e$, then there is a simplex $S$ with $f=\Eref S$ and $e=\Eref \pa S$. And we know that always $S\ar1 \pa S$, so the .
Since $\mid\circ\Eref\circ \pa=\Vnew$ and according %According 
to Lemma \ref{lem:midinjective}, a $\ar\Edges$-closed set $E$ is in bijection with an $\ar\Vertices$-closed set $V$ of vertices by the function $\mid$. 
%Additionally, $\mid\circ\Eref\circ\pa=\Vnew$, so 
So
again 
mere set operations provide a robust solution for turning
%set operations are sufficient to transform
\ref{it:Nnew_lem} into \ref{it:Eref_lem}.
%Let $V$ be closed w.r.t.\ $\ar\Vertices$. Let $E$ be the preimage of $V$ under $\mid$. Let $S\in 
\end{proof}

\begin{thm}[Second characterisation of admissible forests]\label{Knotenschritt}
\index{characterisation of admissible forests!second}
%PCC
Suppose 
%SIC
\hyperlink{PC}{PC} for $\Simplexe$
and $\dim\Simplexe=\{n\}$ with $n\geq 2$. 
Then the following statements hold true.
\begin{enumerate}
\item\label{01}
%\begin{enumerate}
%\item\label{Hauptaussage}
A subset $W$ of $\Simplexe$ is an admissible forest if and only if %$\ar{01}:=\ar{0} \cup \ar{1}$ abgeschlossenen endlichen Teilmengen von $\Simplexe$, die $\T_0$ enthalten. 
\begin{itemize}
\item
$W$ is finite,
\item
$\T_0\subset W$,
\item 
$W$ is closed w.r.t.\ $\ar{01}$.
\end{itemize}

%(Zum Verständnis: Aufgefasst als Teilmengen von $\T^2$ kann man Relationen vereinigen. $T\left(\ar{0} \cup \ar{1}\right) S$ bedeutet also $T\ar0 S$ \emph{oder} $T\ar1 S$.)
%\item Alternative 
%%formulation: 
%formulations:
%Let $\ar{\Vertices}$ be the relation on $\Vertices(\Simplexe)$ with
%\begin{align*}
%p\ar\Vertices q \quad:\LRq \ex T\in\Simplexe \text{ with }p=\Vnew T,q=\Vnew\parent T.
%\end{align*}
%
%\begin{enumerate}
\item\label{it:Nnew}
The
%n the 
%set of
admissible forests 
%is
are the sets $\T_0\cup \Vnew\inv V$, where finite %subsets of $\Vnew (\Simplexe\setminus\T_0)$ 
sets of admissible vertices
which are closed w.r.t.\ $\ar\Vertices$
have to be inserted for $V$%
.
%\begin{align*}
%\left\{\T_0\cup \Vnew\inv V~\middle|~V\subset \Vnew (\Simplexe\setminus\T_0) \text{ is closed w.r.t.\ } \ar\Vertices\right\}.
%\end{align*}
%%\ar\Edges
\item\label{it:Eref}
The admissible forests are the sets ${\T_0\cup \children(\Eref\inv(E))}$, where finite sets of admissible edges which are closed w.r.t.\ $\ar\Edges$ have to be inserted for $E$.
%\end{enumerate}
\end{enumerate}
\end{thm}

\begin{proof}
Since $n\geq 2$, $\Vnew(\parent T)$ is always a vertex of $T$.

\emph{Proof of \ref{01}.} %According to 
%Theorem \ref{TpaS}, 
%the first theorem about the characterisation of admissible forests, 
It suffices to show that a subset $W\subset \Simplexe$ %is closed w.r.t.\ $\ar0\cup\ar1$ 
fulfils the 3 conditions of this theorem if and only if it %is closed w.r.t.\ the relation $\ra$ 
fulfils the 3 conditions from the first %theorem about the 
characterisation of admissible forests (Theorem \ref{TpaS} on page \pageref{TpaS}).
This is done by Lemma \ref{ra-closed is ar01-closed}.
\qed

%\emph{On \ref{Kette}.} $S_1\ar0 S_2\ar0 S_3$ means $\Vnew S_1=\Vnew S_2=\Vnew S_3$ and implies $S_1\ar0 S_3$. $S_1\ar1 S_2\ar0 S_3$ means $\Vnew \parent S_1=\Vnew S_2=\Vnew S_3$ and implies the contraction $S_1\ar1 S_3$. Thus, all $\ar0$ not leading a chain $p_1\ar{01}\cdots\ar{01}p_m$ can be eliminated from it preserving a valid chain from $p_1$ to $p_m$.
\emph{\ref{it:Nnew} and \ref{it:Eref}} follow directly from Lemma \ref{ra-closed is ar01-closed}.
\end{proof}

\subsection
{%$\refine(\T,S)$ 
The refine function
and the languages of partitions and forests}\label{sec:refine and languages}
%%Refine doch nicht verschoben, weil Remark \ar1 voraussetzt.
\begin{defn}[$\refine(\T,S)$]\index{refine(T,S)@$\refine(\T,S)$}
Let $\Simplexe$ be generally refineable. $\refine(\T,S)$ is the coarsest refinement of $\T$ strictly finer than $S\in\Simplexe$.
\end{defn}
The next lemma ensures that there exists a coarsest one among these refinements.
\begin{lem}\label{refineforest}
Let $\T$ be generally refineable. If the simplices $S_1,S_2$ are the children of $S$, then the towers $\Tu(S_1)$ and $\Tu(S_2)$ coincide and $\Wald(\refine(\T,S))=\Wald(\T)\cup\Tu(S_1)$.
\end{lem}
\begin{proof}
%\begin{enumerate}
%\item
%Recall that the tower of $S_1$ is defined as the $\ra$-closure of $S_1$ and according to Lemma \ref{ra-closed is ar01-closed}, it equals the $\ar{01}$-closure of $S_1$. 
%
%Then
$S_1\ra S_2%$ and $S_2
\ra S_1$ causes %that 
the towers of $S_1$ and $S_2$ to coincide.
%\item
%\begin{align*}
%\refine(\T,S)\leq \T\\
%\Wald\refine(\T,S)\supset \Wald \T.
%\end{align*}
%$\Wald
Thus, it is easy to recognise that $\Wald \T \cup \Tu S_1 \cup \Tu S_2=\Wald \T \cup \Tu S_1$ is the smallest $\ra$-closed superset of $\Wald\T\cup\{S_1,S_2\}$.
%\end{enumerate}
\end{proof}
\begin{bem}\label{refine algorithm}
Algorithms for $\refine(\T,S)$ are well known. Kossaczký \cite[Algorithm 4 on page 283]{Kossaczky} presented a suitable one for forests, where simplices $\ar1$-demand only lower level simplices (as it is the case for SIC, see Remark \ref{it:-> and levels} on Definition \ref{ar}). For general refineability, the algorithm of Arnold, Mukherjee and Pouly \cite[Algorithm on page 6]{Arnold} finds the coarsest refinement of $\T$ which is strictly finer than $S$ reliably. Some authors seem to believe that a simplicial partition without hanging nodes were necessarily regular. On the other hand, if the admissible simplices are pairwise compatible, the absence of hanging nodes is indeed equivalent to regularity and almost every serious bisection algorithm is pairwise compatible.
\end{bem}
%Let us summarise the setup:
\subsubsection{Summary of the setup: the languages of partitions and forests}
%We have
There is
\begin{itemize}
\item
an initial triangulation $\T_0$,
\item
a bisection rule, defining a forest of admissible triangulations $\Simplexe$,
\item
the relations $\ar{0}$ and $\ar1$, determining the admissible forests and towers.
\end{itemize}
%For the next section, we can forget all the %former 
%further structures as $T$-arrays, types, hyperlevels, reference coordinates, initial conditions, the definition of regular triangulations and others.
In the next section, most of the %former 
further structures as $T$-arrays, types, hyperlevels, reference coordinates, initial conditions, the definition of regular triangulations and others are not used anymore.

Up to here, all the concepts for partitions are converted into concepts for the corresponding forests, which do not refer to the partitions. %The following table 
Table \ref{tab:conversion table} 
sums up these conversions.%\vspace{\baselineskip}\\
\begin{table}[ht]
\bgroup
\def\arraystretch{1.5}
\begin{tabular}{p{4.5 cm}|c|p{4.2 cm}}
\emph{language of partitions}		&\emph{conversion}							&\emph{language of forests}\\
\hline
						&$\xrightarrow[\text{(add all ancestors)}]{\Wald }$&\\
%\hline
						&$\xleftarrow{\text{leaves}}$&\\
\hline
covers $\Omega$					&							&includes $\T_0$\\
\hline
finer							&							&	$\supset$\\
\hline
%\!\!\!\begin{tabular}{l}
overlay (coarsest common refinement)
%\end{tabular}\!\!\!
&							&	$\cup$\\
\hline
regular							&							&	$\ar{01}$-closed\\
\hline
$\#\T-\#\T_0$				&$\stackrel{\cdot 2}{\longrightarrow}	$ (Lemma \ref{count forest not triangulation})	&$\#\left(\Wald(\T)\setminus\T_0\right)$\\
\hline
%\!\!\!\begin{tabular}{l}
$\refine(\T,S)$ (coarsest re\-fine\-ment of $\T$
strictly finer than $S$)
%\end{tabular}
&	&	
%\!\!\!\begin{tabular}{l}
$\Wald(\T)\cup \Tu(\children(S))$ 
%($\ar{01}$-closure of 
%$\children(S)\cup \Wald(\T))$
%\end{tabular}\!\!\!
\end{tabular}
\egroup
\caption{Conversion of the partition language into the forest language}
\label{tab:conversion table}
\index{conversion table}
\end{table}
%
%\paragraph{Exercise.} 
%Write an algorithm computing $\refine(\T,S)$
%\begin{itemize}
%\item
%assuming SIC
%\item
%assuming only general refineability.
%\end{itemize}
%%%%%%%%%%%%%%%%%%%%%%%%%%%%%%%%%%%%%%
\newpage
\section{The (classical) theorem of Binev--Dahmen--DeVore %(BDV for SIC)
}\label{sec:BDV}
\subsection{Introduction}
The adaptive finite element method is the following algorithm: Start with an initial (coarse) triangulation $\T_0$. 
%Starting 
Beginning
with $j=0$, repeat \emph{rounds} consisting of:
\begin{itemize}
\item
Solve the PDE discretised on $\T_j$. An error estimator marks elements, whose bisection promises a ponderable improvement of the error.
\item
Let $\T_{j+1}$ be the coarsest refinement of $\T_j$ strictly finer than the marked triangles.
\end{itemize}

For this setup, in 2004 Binev, Dahmen and DeVore proved the following theorem  (Theorem 2.4 in \cite{BDV}). 
\begin{thm}
[BDV theorem for \hyperlink{SIC}{SIC}, triangulation formulation]\label{thm:BDV}
\index{BDV theorem!triangulation formulation} 
\index{BDV theorem!SIC} 
Suppose SIC. Then there is a constant $C$% \index{C@$C$ -- BDV constant!SIC}
\index{C@$C$|see {BDV constant}}%
, depending only on $\T_0$, such that whichever elements are marked, 
%by the estimator, 
%for refinement,
the total number of simplices in the triangulation 
%can be estimated 
is bounded
by $\#\T_0$ plus $C$ times the total number of marked ones (in all rounds together).
\end{thm}
It is going to be reformulated in the language of forests% in a moment
. Since 
%we are interested only in the triangulations 
only the triangulations interests
here, 
%we ignore the solution and the estimation and esteem 
solution and estimation are ignored and 
the marking 
is esteemed
as a black box marking any number of simplices. Additionally, 
%we do not concern% ourselves here
%
%it is of no interest%
%, 
how many times the loop is repeated, 
is not important, but only 
%with
the maximal possible total number of simplices in 
%our 
the 
triangulation in dependence on the total number of marked simplices. Thus, %we can 
assume %the number of marked simplices to be 1 every time
that only one simplex is marked each round -- without loss of generality: Instead of marking several ones in one round, the same simplices could be marked one after the other, only one per iteration, omitting %some if they
these which have already been refined meanwhile.

%[second formulation: MUss hier weg] Suppose SIC. Then there is a constant $C$, depending only on $\T_0$, such that for every refinement sequence of triangulations $\T_0,\dots,\T_N$, it holds $\#\T_N-\#\T_0\leq C\cdot N$.

This is formalised as follows.
\begin{defn}[refinement sequence]
\index{refinement sequence!of forests}\index{refinement sequence!of triangulations}
A \emph{refinement sequence of triangulations} is a sequence $\T_0,\dots,\T_N$
%\index{N@$N$ -- length of the refinement sequence} 
such that for any $j=0,\dots,N{-}1$, there is a simplex $S_j\in\T_j$ with $\T_{j+1}=\refine(\T_j,S_j)$. 
The sequence of corresponding forests $W_0=\Wald \T_0=\T_0,W_1=\Wald \T_1,\dots,W_N=\Wald \T_N$ is a \emph{\hypertarget{refinement sequence of admissible forests}{refinement sequence of admissible forests}}. 
\end{defn}
%\begin{bemn}
\begin{bem}
%\begin{enumerate}
%\item
For a refinement sequence of admissible forests $W_0%=\Wald \T_0=\T_0\subset W_1
\subset\dots\subset W_N%=\Wald \T_N
$,
$W_{j+1}$ is the smallest set of admissible simplices including $W_j$ and $\children(S_j)$. %This is $W_j\cup \Tu(\children(S_j))$. The children $T_1,T_2$ $\ar0$-demand each other, so $\Tu(T_1)=\Tu(T_2)$. Consequently, refinement sequences of admissible forests are characterised by
According to Lemma \ref{refineforest}, this is $W_j\cup \Tu(T_j)$ for $T_j$ being a child of $S_j$. Consequently, refinement sequences of admissible forests are characterised by
\begin{align}\label{refseqfor}
\fa j=0,\dots,N{-}1\,\ex T_j\in\children(\leaves W_{j}).W_{j+1}=W_j\cup \Tu T_j.
\end{align}
\end{bem}
%\item
\begin{lem}[Forest formulation of the BDV theorem]\label{forest formulation}
\index{forest formulation!of BDV theorem}\index{BDV theorem!forest formulation}
%In the forest language, the (equivalent) statement reads 
According to the conversion Table \ref{tab:conversion table}, the Binev--Dahmen--DeVore statement converts (for the same $\T_0$ and $C$) into: For every refinement sequence of forests $W_0=\T_0,W_1,\dots,W_N$, it holds 
\begin{align*}
\#\left(W_N\setminus \T_0\right)\leq 2C\cdot N.
\end{align*}
\end{lem}
%\end{enumerate}
%\end{bemn}

To introduce a new proof of the theorem, %we give a proof of 
the analogous theorem for a toy problem
is proven. 
But at first %let us describe 
the key problem
is presented%
: 
The additional number of simplices in one single refinement step is unbounded, as shown in  Figure \ref{Treppe}. This problem will be solved by a sequence $a_1,a_2,\dots,a_N$ with a bounded increase $a_k{-}a_{k-1}\leq C$ dominating the number of simplices in the refinement sequence, i.e.\ $a_k\geq \#(W_k\setminus \T_0)$. %and having a bound $2C$ for $a_k{-}a_{k-1}$. 
The elements $a_j$ of the sequence will have the form $\mu(\tu T_0\cup\dots\cup \tu T_{j-1})$ for a measure $\mu$ and an \emph{estimated tower $\tu$} defined below.

%%additional number in 1 step
\begin{figure*}[ht]
\centering{
\begin{tikzpicture}[information text/.style={fill=gray!10,inner sep=1ex}, scale=.5]
\fill[gray] (1,1) --(2,0) -- (2,2) -- cycle;
\draw (0,0) -- (8,8);
\draw (0,0) -- (1,1) --(2,0) -- (2,2) -- (4,0) -- (4,4) -- (8,0) -- (8,8) -- (16,0)-- cycle;
\draw[dotted] (1,1) -- (3,1);
\draw[dotted] (2,2) -- (6,2);
\draw[dotted] (4,4) -- (12,4);
\draw[dotted] (2,0) -- (4,2);
\draw[dotted] (4,0) -- (8,4);
\draw[dotted] (8,0) -- (12,4);
%\draw 
%(8,0) node[below=3mm
%%(16,2) node[right
%,text width=9cm,
%	information text]
%	  {
%	  	  In a single refinement step, the number of refined triangles is arbitrary large. The continuos edges are the edges before refinement, the dotted ones has to be added, when the grey triangle is bisected. (The figure can be continued to the right arbitrary often.)
%	  };
\end{tikzpicture}
}
%\caption
\caption{In a single refinement step, the number of refined triangles is arbitrary large. The continuous edges are the edges before refinement, the dotted ones have to be added, when the grey triangle is bisected. The figure can be continued on the right arbitrarily often.
}
\label{Treppe}
\end{figure*}
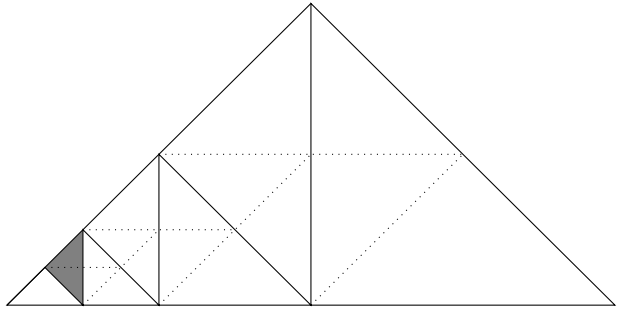

\subsection{A toy problem: the pile game}
\index{pile game}
Impatient readers are kindly recommended to skip this subsection.

Consider the following game (which is not 1D adaptive mesh refinement).
%We have the following set of bricks:
%\begin{align*}
%\{(2^{-k}[a,a+1])\times[k,k+1]\}
%\end{align*}
%where we esteem the last coordinate as height and make a pile of them. There is an upward rule and a downward rule (the pile rule). The upward-rule says that in every round we can add a brick on the top of a brick which is already there. The downward rule says: a brick demands all bricks of a lower level touching it. So as long as there are bricks, which demand bricks which are not already in the pile, these must be added.
%We build a 2-dimensional pile of bricks. 
The game consists in building a 2-dimensional pile of bricks.
At the beginning, there is prepared a horizontal row of equally sized bricks without gaps between them. 
%Our bricks must have certain admissible sizes and positions:
Generally all bricks are of the same height, but for every level, starting with the lowest level 0, there is a characteristic breadth, namely a brick one level higher is half as broad. It can be placed exactly on top of either the left or the right half of an underlying admissible position only 
and is called a \emph{child} of the underlying brick.

\begin{center}
%%admissible bricks
\begin{tikzpicture}[scale=2,yscale=.2,information text/.style={fill=gray!10,inner sep=1ex}]
\def\yend{5};\def\xend{5};
\draw[yshift=2cm, xshift=\xend cm-.7cm]
node[right=2mm,text width=5cm,information text]
{
The admissible bricks at their admissible positions.
%Die zulässigen Bausteine an ihren zulässigen Positionen.
};
\clip (.7,-.1) rectangle (\xend-.7,\yend-.7);
  \foreach \y in {0,...,\yend}
    {\tikzmath{\b = 2^(-\y);}
    \draw (-.5,\y) -- (\xend+.5,\y);
      \foreach \x in {0,\b,...,\xend}
	\draw (\x,\y) -- +(0,1);
    }
\end{tikzpicture}
\end{center}

%This is formalised by: The set of admissible bricks is 
This set of admissible bricks is formally described by 
\[
\left\{2^{-l}[k,k{+}1]\times\{l\}\text{ (depicted as } 2^{-l}[k,k{+}1] \times [l,l{+}1]) ~\middle|~k\in\Z, l\in\N\right\}.
\]
%In every round, choose one brick of the pile, choose the left or the right half of it and raise the pile there by one brick. %This means, put a brick on it, equally high and half as broad as the chosen brick either on the left or on the right half of it. 
%This is the \emph{upward rule}.
%After that you have to add further bricks as long as they are demanded by the following \emph{downward rule}: Every brick demands every subjacent brick which could adjoin it at least in a vertex. (Imagine, like on a pile of sand, it would fall down, if the pile were too steep.)

%We define a 
A
\emph{%steepness 
demand
relation}\index{demand} or steepness rule
determines,
%, which determines 
how the pile grows. That is to say, every brick of level $\geq 1$ in the pile \emph{demands} the two touching bricks underneath (one touching at the bottom side and one touching at a vertex) to be part of the pile as well%,
%%for stability. 
%expect of the bricks of level 0 which do not demand any brick, since there are no bricks of level $-1$%
.
Imagine, like on a pile of sand, it would fall down, if the pile were too steep.

\begin{center}
%%steepness relation
\begin{tikzpicture}[>=stealth,scale=2,information text/.style={fill=gray!10,inner sep=1ex}]
\begin{scope}[yscale=.2]
\draw (0,0) rectangle (1,1);
\draw (1,0) rectangle (2,1);
\filldraw[fill=gray!20] (.5,1) rectangle (1,2);
\draw[->] (.75,1.5) -- (.5,.5);
\draw[->] (.75,1.5) -- (1.5,.5);

\draw[shift={(2,1.5)}]
node[right=2mm,text width=7cm,information text]
{
The steepness relation: 
%Expect of the basement bricks, every brick demands the two touching subjacent bricks.
Every brick expect the basement bricks demands the two touching bricks underneath.
};
\end{scope}
\end{tikzpicture}
\end{center}

%At the beginning, there is already a 
In every round, add a new brick %which is either of level 0 or demands a brick which is already part of the pile (i.e.\ touches the already existing pile)
on some already existing brick. 
%We call it 
The new brick is called
the \emph{chosen} brick. 
After that, 
%you have to add 
further bricks 
have to be added
as long as they are demanded by bricks of the pile, see Figure \ref{fig:round}.

\begin{figure}[ht]
\centering{
%%choice of the next brick
\begin{tikzpicture}[>=stealth,scale=2,yscale=.2,information text/.style={fill=gray!10,inner sep=1ex}]
%\draw[shift={(.3,0)}]
%node[below=2mm,text width=9cm,information text,scale=1]
\draw
(2.1,1)
node[right,text width=6.2cm,information text,scale=1]
{
\caption{First step of a round: The grey pile is already there. Only one of the dashed 
%bricks can be chosen, since they are either from level 0 or demand an already existing brick.
children of existing bricks can be chosen.}
\label{fig:round}
};
\def\myrectangle#1#2[#3]{%% #1: Level, #2 x-Koordinate der linken unteren Ecke, #3 style
  \tikzmath{\b = 2^(-#1);} 
  \draw[#3] (#2,#1) rectangle +(\b,1);
  %\draw (0,#1) rectangle (\b,1);
 };
\def\arrowupwards#1#2[#3]{%% #1: Level, #2 x-Koordinate der linken unteren Ecke des unteren Bricks, #3 style
  \tikzmath{\b = 2^(-#1);%\c = (#2+\b,#1+.5);
	  }
  \foreach \s in {-1,1}
    {\begin{scope}[scale around={\s:(#2+.5*\b,0)},yscale=\s]
%%über-Eck-Pfeil
%      \draw[->,#3] (#2-.2*\b,#1+1.3) -- (#2+.2*\b,#1+.7);
%%Kinderpfeil
      \draw[->,#3] (#2+.3*\b,#1+1.3) -- (#2+.35*\b,#1+.7);
    \end{scope}
   }
    %\draw (0,#1) rectangle (\b,1);
 };
%\clip (-1.7,0) rectangle (2.4,3);
\clip (-1.2,-1.1) rectangle (2.1,3.1);

%%Version demanding brick or basement can be chosen
%\myrectangle{0}{-1}[fill=gray!20];
%\myrectangle{0}{0}[fill=gray!20];
%\myrectangle{1}{0}[fill=gray!20];
%
%\foreach \x in {-2,1,2}
%\myrectangle{0}{\x}[dashed];
%\foreach \x in {-1.5,-1,-.5,.5,1}
%\myrectangle{1}{\x}[dashed];
%\foreach \x in {-.25,0,.25,.5}
%\myrectangle{2}{\x}[dashed];
%
%\arrowupwards{0}{-1}[];
%\arrowupwards{0}{0}[];
%\arrowupwards{1}{0}[];

%%Version child can be chosen, basement already there
%%basement
\foreach \x in {-2,0}
\myrectangle{-1}{\x}[fill=gray!20];

\foreach \x in {-1,0}
\myrectangle{0}{\x}[fill=gray!20];

\myrectangle{1}{0}[fill=gray!20];

\foreach \x in {-2,1}
\myrectangle{0}{\x}[dashed];
\foreach \x in {-1,-.5,.5}
\myrectangle{1}{\x}[dashed];
\foreach \x in {0,.25}
\myrectangle{2}{\x}[dashed];

%%in der children version no arrows needed
%\arrowupwards{0}{-1}[];
%\arrowupwards{0}{0}[];
%\arrowupwards{1}{0}[];
\end{tikzpicture}

%%closure
\begin{tikzpicture}[>=stealth,scale=2,yscale=.2,information text/.style={fill=gray!10,inner sep=1ex}]
\draw[shift={(2,1)}]
node[right=2mm,text width=6.2cm,information text]
{
Second step of a round: After the dark grey brick has been chosen, the bricks on the right are added, since they are demanded successively.
};
\def\myrectangle#1#2[#3]{%% #1: Level, #2 x-Koordinate der linken unteren Ecke, #3 style, #4 Beschriftung
  \tikzmath{\b = 2^(-#1);} 
  \draw[#3] (#2,#1) rectangle +(\b,1);
  %\draw (0,#1) rectangle (\b,1);
 };

\clip (-1.2,-1.1) rectangle (2.1,3.1);

%%basement
\foreach \x in {-2,0}
\myrectangle{-1}{\x}[fill=gray!0];

\myrectangle{0}{-1}[fill=gray!0];
\myrectangle{0}{0}[fill=gray!0];
\myrectangle{1}{0}[fill=gray!0];

\foreach \x in {1}
\myrectangle{0}{\x}[fill=gray!20];
\foreach \x in {.5}
\myrectangle{1}{\x}[fill=gray!20];
%%
%\foreach \x in {.5}
%\myrectangle{2}{\x}[fill=gray!60];

%%Version child can be chosen, basement already there
\foreach \x in {.25}
\myrectangle{2}{\x}[fill=gray!60];

\draw[->] (.375,2.5) -- (.75,1.5);
\draw[->] (.75,1.5) -- (1.5,.5);
\end{tikzpicture}
%\caption{}
%%example
\begin{tikzpicture}[>=stealth,scale=2,yscale=.2,information text/.style={fill=gray!10,inner sep=1ex}]
\draw[shift={(1.2,1.5)}]
node[right=2mm,text width=7.6cm,information text]
{
Example for a pile game. The number stands for the round, in which the brick is added. The grey bricks are the chosen ones.
};
\def\myrectangle#1#2[#3]#4{%% #1: Level, #2 x-Koordinate der linken unteren Ecke, #3 style, #4 Beschriftung
  \tikzmath{\b = 2^(-#1);} 
  \draw[#3] (#2*\b,#1) rectangle +(\b,1);
  \draw (#2*\b+.5*\b,#1+.5) node{#4};
 };
 
\clip (-1.2,-1) rectangle (1.2,5);
\foreach \x in {-2,-1,0,1}
\myrectangle{0}{\x}[]{};

\myrectangle{1}{-2}[fill=gray!0]{3};
\myrectangle{1}{-1}[fill=gray!20]{1};
\myrectangle{1}{0}[fill=gray!0]{2};
\myrectangle{1}{1}[fill=gray!0]{4};

\myrectangle{2}{-2}[fill=gray!0]{3};
\myrectangle{2}{-1}[fill=gray!20]{2};
\myrectangle{2}{1}[fill=gray!20]{4};

\myrectangle{3}{-2}[fill=gray!20]{3};

\end{tikzpicture}
}
\end{figure}
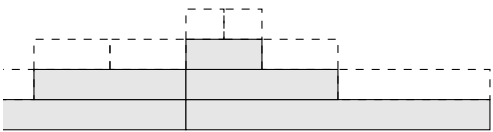

%The above listed elements of the structure are replaced as listed in table \ref{tab: replacement for pile game}. %(see also the pictures below):

%The pile game is related to the BDV theorem, because the 
The relation to the BDV theorem is not hard to spot: The
stages of the growing pile %at the beginnings of the rounds 
form a refinement sequence of forests, if the elements of the structure are replaced as listed in Table \ref{tab: replacement for pile game}.

\begin{table}[ht]
\bgroup
\def\arraystretch{1.5}
\begin{tabular}{p{5cm}|p{7cm}}
\emph{Replace} 					& \emph{by}\\
\hline
$\Omega$						& $\R$\\
\hline
$\T_0$ 							& Basement $\left\{[k,k{+}1]\times
													%[0,1]
													\{0\}
													~\middle|~k\in\N\right\}$, depicted as bricks $[k,k{+}1]\times[0,1]$
%%basement
\begin{tikzpicture}[scale=1.9,yscale=.2,information text/.style={fill=gray!10,inner sep=1ex}]
\def\yend{1};\def\xend{5};
%\draw[yshift=.5cm, xshift=\xend cm-.7cm]
%node[right=2mm,text width=5cm,information text]
%{
%The basement line.
%%Die zulässigen Bausteine an ihren zulässigen Positionen.
%};
\clip (.7,-.1) rectangle (\xend-.7,1);
  \foreach \y in {0,...,\yend}
    {\tikzmath{\b = 2^(-\y);}
    \draw (-.5,\y) -- (\xend+.5,\y);
      \foreach \x in {0,\b,...,\xend}
	\draw (\x,\y) -- +(0,1);
    }
\end{tikzpicture}
\\
\hline
children rule			 			
&%%children rule
\begin{tikzpicture}[%>=stealth,
scale=2,information text/.style={fill=gray!10,inner sep=1ex}]
\begin{scope}[yscale=.2]
\draw[white] (0,2) -- (0,2.5);
\draw (0,0) rectangle (1,1);
\draw (0,1) rectangle (.5,2);
\draw (.5,1) rectangle (1,2);
\draw[->] (.5,.25) -- (.25,1.5);
\draw[->] (.5,.25) -- (.75,1.5);
\end{scope}
\end{tikzpicture}
\\
\hline
$\Simplexe$ with level function $l$ 	& Set of admissible bricks $\left\{2^{-l}[k,k{+}1]\times
															%[l,l+1]
															\{l\}
															~\middle|~k\in\Z,l\in\N\right\}$, depicted as %bricks 
															$2^{-l}[k,k{+}1]\times[l,l{+}1]$\\

								&
%%admissible bricks
\begin{tikzpicture}[scale=1.9,yscale=.2,information text/.style={fill=gray!10,inner sep=1ex}]
\def\yend{5};\def\xend{5};
%\draw[yshift=2cm, xshift=\xend cm-.7cm]
%node[right=2mm,text width=5cm,information text]
%{
%The admissible bricks at their admissible positions.
%%Die zulässigen Bausteine an ihren zulässigen Positionen.
%};
\clip (.7,-.3) rectangle (\xend-.7,\yend-.7);
  \foreach \y in {0,...,\yend}
    {\tikzmath{\b = 2^(-\y);}
    \draw (-.5,\y) -- (\xend+.5,\y);
      \foreach \x in {0,\b,...,\xend}
	\draw (\x,\y) -- +(0,1);
    }
\end{tikzpicture}
%\\ 
%								&
								(Brick size of level $l$ is $2^{-l}$.)\\
\hline
$\ar0$							& $S\ar0 T \quad:\LRq S=T$\\
								&(trivial equivalence relation)\\
\hline
$\ar1$							&

%%steepness relation
\begin{tikzpicture}[%>=stealth,
scale=2,information text/.style={fill=gray!10,inner sep=1ex}]
\begin{scope}[yscale=.2]
\draw[white] (0,2) -- (0,2.5);
\draw (0,0) rectangle (1,1);
\draw (1,0) rectangle (2,1);
\filldraw[fill=gray!20] (.5,1) rectangle (1,2);
\draw[->] (.75,1.5) -- (.5,.5);
\draw[->] (.75,1.5) -- (1.5,.5);

%\draw[shift={(2,1.5)}]
%node[right=2mm,text width=7cm,information text]
%{
%The steepness relation: Every brick expect the basement bricks demands the two touching bricks underneath.
%};
\end{scope}
\end{tikzpicture} 
\\
%								&Brick demands all adjacent subjacent bricks.\\
%								&Expect of the basement bricks, every brick demands the two touching subjacent bricks.\\
								&Every brick expect the basement bricks demands the two touching bricks underneath.\\
\end{tabular}
\egroup
\caption{Replacements to turn adaptive mesh refinement into the pile game.}\label{tab: replacement for pile game}
\index{conversion table!adaptive mesh refinement vs pile game}
\end{table}

%We will use the 
The
same notions and variables 
will be used
for the pile game as %we 
were
used for the simplices. The rounds are the same as defined before: Choose a child of one (surface) brick of the existing pile/forest and unite the existing pile/forest with the tower of this child.

%\paragraph{}
%We pose the following questions: 
\paragraph{Questions:}
\begin{enumerate}
\item
How to get as many bricks as possible after $N$ rounds? (This question is meant, as a businessman would understand it: %We do 
Do
not ask for a \emph{proved} answer and do not concern about a handful of bricks.)
\item
What is the smallest constant $C$, such that for any possible pile game it holds that
\begin{align*}
\#\left(W_N\setminus \T_0\right)\leq 2C{\cdot} N?
\end{align*}
\end{enumerate}
\paragraph{Answer to the 1st question.}
%\begin{enumerate}
%\item
\emph{1st strategy “tower”:} 
At first, let us consider the ``tower'' strategy:
\emph{Always choose a child of a top brick.}

\begin{center}
%%tower
\begin{tikzpicture}[>=stealth,scale=3,yscale=.15,information text/.style={fill=gray!19,inner sep=1ex}]
% \draw[shift={(2,1.5)}]
% node[right=2mm,text width=6.6cm,information text]
% {
% };
\def\myrectangle#1#2[#3]{%% #1: Level, #2 x-Koordinate der linken unteren Ecke, #3 style
  \tikzmath{\b = 2^(-#1);} 
  \draw[#3] (#2*\b,#1) rectangle +(\b,1);
  %\draw (0,#1) rectangle (\b,1);
 };

\foreach \y in {1,2,3,4}
\foreach \x in {-1,0}
\myrectangle{\y}{\x}[fill=gray!0];

\foreach \y in {5}
 \foreach \x in {0}
 \myrectangle{\y}{\x}[fill=gray!0];
% \myrectangle{5}{1}[fill=gray!0];
% \myrectangle{5}{2}[fill=gray!0];
% 
% \foreach \y in {6}
% \foreach \x in {3}
% \myrectangle{\y}{\x}[fill=gray!0];
\begin{scope}[xshift=1.2cm]
\draw[shift={(1,1.5)}]
node[right=2mm,%text width=6.6cm,
information text]
  {
  towers
  };
  \def\myrectangle#1#2[#3]{%% #1: Level, #2 x-Koordinate der linken unteren Ecke, #3 style
    \tikzmath{\b = 2^(-#1);} 
    \draw[#3] (#2*\b,#1) rectangle +(\b,1);
    %\draw (0,#1) rectangle (\b,1);
  };

  \foreach \y in {1,2,3}
  \foreach \x in {-1,0,1}
  \myrectangle{\y}{\x}[fill=gray!0];

  \foreach \y in {4}
  \foreach \x in {0,1}
  \myrectangle{\y}{\x}[fill=gray!0];
  \myrectangle{5}{1}[fill=gray!0];
  \myrectangle{5}{2}[fill=gray!0];

  \foreach \y in {6}
  \foreach \x in {3}
  \myrectangle{\y}{\x}[fill=gray!0];
\end{scope}
\end{tikzpicture}
\end{center}

We observe that in the growing pile, there are at most 3 bricks per level, because 3 connected bricks in one layer cannot demand more than 3 subjacent ones and the height of the pile equals the number of rounds.
%Observe:
%\begin{itemize}
%\item
%$\leq 3$ bricks per level (3 connected bricks in one layer cannot demand more than 3 subjacent ones), 
%\item
%height = number of rounds. 
%\end{itemize}
Consequently, $\#\left(W_N\setminus \T_0\right)\leq 3N$.

%\emph{2nd strategy “quasitower”:}\index{quasitower} 
On the other hand, 
observe that 4 bricks can demand 4 subjacent bricks: 
%as in the tower
%(Largest “fix point” of $\ra$.) 
%%reason
\begin{tikzpicture}[>=stealth,xscale=4,scale=.34,information text/.style={fill=gray!19,inner sep=1ex}]
% \draw[shift={(2,1.5)}]
%node[below=2mm,text width=6.6cm,information text]
% {
% };
\def\myrectangle#1#2[#3]#4{%% #1: Level, #2 x-Koordinate der linken unteren Ecke, #3 style, #4 Beschriftung
  \tikzmath{\b = 2^(-#1);} 
  \draw[#3] (#2*\b,#1) rectangle +(\b,1);
  \draw (#2*\b+.5*\b,#1+.5) node{#4};
 };

\foreach \y in {1,2}
\foreach \x in {-2,-1,0,1}
\myrectangle{\y}{\x}[fill=gray!0]{};

%\draw[%shift={(1,1.5)}
%]
%node[below=0mm,%text width=6.6cm,
%information text]
%  {
%  4 bricks can demand 4 subjacent bricks.
%  };
  \def\myrectangle#1#2[#3]{%% #1: Level, #2 x-Koordinate der linken unteren Ecke, #3 style
    \tikzmath{\b = 2^(-#1);} 
    \draw[#3] (#2*\b,#1) rectangle +(\b,1);
    %\draw (0,#1) rectangle (\b,1);
  };
\end{tikzpicture}
(and these 4 bricks are the largest “fixed point” of $\ra$.)
This leads to the second, the ``quasitower'' strategy: Try to achieve 4 bricks per level 
in a pile that grows up by 1 layer per round approximately.
%Aim: height $\approx$ number of rounds.
%Realisation:
Figure \ref{fig:quasitower} shows how it can be realised.
\begin{figure}[ht]
\centering{
%%quasitower
\begin{tikzpicture}[>=stealth,xscale=14,yscale=.5,information text/.style={fill=gray!10,inner sep=1ex}]
% \draw[shift={(2,1.5)}]
%node[below=2mm,text width=6.6cm,information text]
% {
% };
\def\myrectangle#1#2[#3]#4{%% #1: Level, #2 x-Koordinate der linken unteren Ecke, #3 style, #4 Beschriftung
  \tikzmath{\b = 2^(-#1);} 
  \draw[#3] (#2*\b,#1) rectangle +(\b,1);
  \draw (#2*\b+.5*\b,#1+.5) node{#4};
 };

%\foreach \y in {1,2,3,4,5}
%\foreach \x in {-2,-1,0,1}
%\myrectangle{\y}{\x}[fill=gray!0]{\y};
\clip (-.51,-1) rectangle (.43,6);
\myrectangle{3}{-1}[draw=white,fill=gray!0]{$\dots$};
\myrectangle{1}{-1}[fill=gray!30]{1};
\myrectangle{2}{-1}[fill=gray!30]{2};

\myrectangle{1}{0}[fill=gray!0]{2};
\myrectangle{2}{0}[fill=gray!0]{3};
\myrectangle{1}{-2}[fill=gray!0]{$m{+}1$};
\myrectangle{1}{1}[fill=gray!0]{$m{+}3$};
\myrectangle{2}{-2}[fill=gray!0]{$m{+}1$};
\myrectangle{2}{1}[fill=gray!0]{$m{+}3$};

\begin{scope}[yshift=1cm]
\myrectangle{3}{-1}[fill=gray!30]{$m{-}1$};
\myrectangle{4}{-1}[fill=gray!30]{$m$};
\myrectangle{4}{-2}[fill=gray!30]{$m{+}1$};
\myrectangle{4}{0}[fill=gray!30]{$m{+}2$};
\myrectangle{4}{1}[fill=gray!30]{$m{+}3$};

\myrectangle{3}{-2}[fill=gray!0]{$m{+}1$};
\myrectangle{3}{0}[fill=gray!0]{$m$};
\myrectangle{3}{1}[fill=gray!0]{$m{+}3$};
\end{scope}

%\draw[%shift={(1,1.5)}
%]
%node[below=-3mm,%text width=6.6cm,
%information text]
%  {
%  quasitower and its realisation
%%\caption{quasitower and its realisation}\label{fig:quasitower}
%  };
  \def\myrectangle#1#2[#3]{%% #1: Level, #2 x-Koordinate der linken unteren Ecke, #3 style
    \tikzmath{\b = 2^(-#1);} 
    \draw[#3] (#2*\b,#1) rectangle +(\b,1);
    %\draw (0,#1) rectangle (\b,1);
  };
\end{tikzpicture}
\caption{Quasitower and its realisation. The number stands for the round, in which the brick is added. The grey bricks are the chosen ones. The first $m$ chosen bricks each demand one further brick. The $(m{+}1)$st and $(m{+}3)$rd brick demand $m{-}1$ bricks each.}\label{fig:quasitower}
}
\end{figure}
Therefore, the second strategy is more efficient.
%Conjecture: $2C=4$ is the smallest possible choice for the 2nd question.
%\item
%Claim: 
\paragraph{Answer to the 2nd question.}
\begin{thm}[BDV theorem for the pile game]
\index{BDV theorem!pile game}
For all $N\in\N$ and all refinement sequences of admissible forests for the pile game $\T_0=W_0\subset\dots\subset W_N$, it holds that $\#\left(W_N\setminus\T_0\right)\leq 4N$.
\end{thm}
%\begin{proof}
\noindent\emph{Plan of the proof:}

\begin{center}
\begin{tabular}{l|c|l}
\emph{estimate} & &\emph{by}\\
\hline
the counting measure $\#$ on $\Simplexe$ 	&$\leq$				&a measure $\mu$ on $\R\times\N$\\
\hline
the relation $\ra$ on $\Simplexe$ 	&$\Rightarrow$ 	&a relation $\ar{\text{est}}$ on $
%\Omega
\R
\times \N$\\
\hline
$\ra$-clos%ure $\Tu S$ of $S$ 
ed forest $W_j$
&$\subset$ 		&$\ar{\text{est}}$-clos%ure $\tu S$ of $S$
ed  $w_j\subset \R\times\N$,
\end{tabular}
\end{center}

such that 
\begin{align*}
\#(W_j\setminus \T_0)\leq \mu\left(\bigcup \left(W_j\setminus\T_0\right)\right)&\leq \mu(w_j),\\
\mu(w_0)=0
\quad\text{ and }\quad\mu(w_{j+1})&\leq \mu(w_{j})+4.
\end{align*}
This implies
\begin{align*}
%\Rightarrow 
\#(W_N\setminus \T_0)\leq \mu(w_N)\leq 4N.
\end{align*}

\begin{proof}
For a set of bricks $W$, let $W^k:=\{\text{level $k$ bricks of $W$}\}$.\index{W^k@$W^k$ -- $k$th layer of $W$!pile game}

\emph{Step 1: Estimate the counting measure ${\#\bullet - \#\T_0}$ by a measure $\mu$ on $\R\times\N$.} 

For arbitrary Lebesgue measurable $M=M^0{\times}\{0\}\,\cup\, M^1{\times}\{1\}\,\cup\, \dots \subset
%\Omega
\R
\times \N$, let
\index{m@$\mu$ -- measure estimating $\#$!pile game}
\begin{align*}
\mu(M):=
%\frac1d
\sum_{k=1}^\infty 2^k\left|M^k\right|.
\end{align*}
T%hen t
his measure $\mu$ is closely related to $\#$:
\begin{align}\label{mudef}
\text{For $j\neq 0$, }
\mu\left(\bigcup %\iota 
W^j\right)&=2^j\cdot\underbrace{\left|\bigcup W^j\right|}
_{
\tiny{\begin{matrix}\text{total brick length}\\ \text{of level $j$}\end{matrix}}
}
=\#W^j,\\
\mu\left(\bigcup %\iota 
W\right)&=\sum_{j=1}^\infty 2^j\left|\bigcup W^j\right|=\#W-\#W^0.
\end{align}
(Note that the sum starts at $j=1$ to exclude the initial bricks.) Note that $\mu$ measures also fractions of bricks.

%\begin{bem}
%If the sizes of initial bricks vary from 1, but are at least $d$, the size of a level $j$ brick is at least $2^{-j}d$ and 
%%we have:
%\eqref{mudef} remains true with a modified $\mu$:
%\begin{align*}
%\#W-\#W^0=\sum_{j=1}^\infty \frac{2^j}d\left|\bigcup W^j\right|&=:\mu\left(\bigcup %\iota 
%W\right)
%\end{align*}
%(which does not make much sense here, but will be needed in the simplex version of the proof).
%\end{bem}

\emph{Step 2: Estimate $\ra$.
%Try to 
First, find a relation $\ar{\iota}$ on $\R\times\N$ such that for a given brick $T$
\begin{align*}
\bigcup\left\{S~\middle|~T\ra S\right\}=\{s\in\R\times\N~|~\ex t\in T\subset\R\times\N.t\ar{\iota}s\}.
\end{align*}}
%(Then the $\iota$-image of a $\ra$-closure is the $\ar\iota$-closure of the image.)
Such a $\ar\iota$ is given by:
\index{->i@$\ar\iota$!pile game}
%%arIota
%\begin{wrapfigure}{l}{0.3\textwidth}
%\begin{center}
\begin{tikzpicture}[information text/.style={fill=gray!10,inner sep=1ex},yscale=.5,yshift=-1]
\draw (0,0) rectangle (1,1) (1,0) rectangle (2,1);
\draw (.5,1) rectangle (1,2);
\foreach \x in {0,.4,...,2}
	\draw[->] (1,1.5) -- (\x,.75);
%\draw (1,0) node[below=0mm,text width=2cm,
%	information text]
%	  {
%	  $\stackrel{\iota}{\rightarrow}$
%	  };
\end{tikzpicture}
%\end{center}
%\end{wrapfigure}
, which reads
\begin{align*}
(y,l{+}1)\ar\iota (x,l) \quad:\LRq 
%\begin{array}{l}
&\text{ $(y,l)$ lies at the boundary between}\\ &\text{ two level $l$ bricks
%}\\ 
%&\text{ 
and $x\in y+2^{-l}[-1,1]$.}
%\end{array}
\end{align*}
\emph{Now estimate $\ar\iota$ by $\ar\est$:} 
Define $\ar\est$ by
\index{->est@$\ar\est$ -- estimates $\ar\iota$!pile game}
\begin{center}
\begin{tikzpicture}[information text/.style={fill=gray!10,inner sep=1ex},yscale=.7]
\draw[dotted] (0,0) rectangle (1,1) (1,0) rectangle (2,1);
\draw[dotted] (1,1) --(1,2);
%\draw (.5,1) rectangle (1,2);
\draw[>=stealth,<->] (0,.5)-- node [below] {\tiny
{BS}
} (1,.5);
\draw[>=stealth,<->]  (1,.5) -- node [below] {\tiny
{BS}} (2,.5);
\draw (1,1.5) node [right] {\small
{(arbitrary point)}};
\foreach \x in {0,.4,...,2}
	\draw[->] (1,1.5) -- (\x,.75);
\draw 
%(1,0) node[below=0mm,
(2,.5) node[right=1mm,
%text width=2.7cm,
%text width=4.6cm,
%	information text
]
	  {
%	  $\stackrel{\text{est}}{\rightarrow}$\\
\small
{(let BS be the brick size of this level),}
	  };
%(BS = brick size of this level)%
\end{tikzpicture} 
\end{center}
in formulas
\begin{align}\label{estdef}
(y,l{+}1)\ar{\text{est}}(x,l) \quad:\LRq x\in y+2^{-l}[-1,1].
\end{align}
Since $\fa x,y\in\R\times\N.(y\ar\iota x\Ra y\ar{\text{est}} x)$, the $\ar\iota$-closure of a set $w\subset \R\times\N$ is included in the $\ar{\text{est}}$-closure of $w$. 
%\begin{bem}
%If the initial brick sizes vary from 1, but are at most 1, then this $\ar{\text{est}}$ still estimates 
%%$\ar\iota$.
%$\ra$.
%\end{bem}

Now repeat the game with $\ar{\text{est}}$ instead of $\ar\iota$, but the same choices for $T_j$ as follows: Again, let the tower $\Tu(S)$ of a brick $S$ be the $\ra$-closure of $S$. Let $W_0,\dots,W_N$ be a refinement sequence of forests, with a certain sequence of bricks $T_j\in \children(W_j)$ 
%in \eqref{refseqfor}. 
and $W_{j+1}=W_j\cup \Tu T_j$.
For a brick $S$, let \index{tw@$\tu$ -- estimates $\Tu$!pile game}$\tu(S)$ be the $\ar\est$-closure of $S$ and let \index{w@$w$ -- estimates $W$!pile game}$w_0:=\T_0$, $w_{j+1}:=w_j\cup\tu(T_j)$. %for the same $T_j$ as in \eqref{refseqfor}
Since $\tu(S)\supset \bigcup\Tu (S)$, it holds that $w_j\supset \bigcup W_j$ and also $\mu (w_j)\geq \mu\left(\bigcup W_j\right)\geq \#W_j$.

%\emph{Claim: $\mu (w_{j+1})- \mu (w_j)\leq 4$.}
Now it is claimed that $\mu (w_{j+1})- \mu (w_j)\leq 4$.

\emph{\hypertarget{Step 3 from the pile game}{Step 3}: $\mu (w_{j+1})- \mu (w_j)\leq \mu(\tu T_j)- \mu(\tu \pa T_j)$.}
This is obvious from 
%the following picture: 
Figure \ref{HaufendifferenzTurmdifferenz}. 

\begin{figure}[ht]
\centering{
%\begin{center}
%%HaufendifferenzTurmdifferenz
\begin{tikzpicture}[information text/.style={fill=gray!10,inner sep=1ex}]
\draw[thick] (0,0.1) -- (3,3.1) --(4,2.1) --(6,3.1)--(9,0.1);
\filldraw[fill=gray!25]		(.5,0)--(3.5,2.4)--(4,2) --(6,3)--(9,0);
\draw[<->,>=stealth]	(3,2)-- node[left]{1} (3,3);
\filldraw[dashed,fill=gray!50] (1,0) -- (3,2) -- (5,0);
\draw[dashed] (0,0) -- (3,3) -- (6,0);
\draw (7,2.1)--(7.5,2.1) node[right]{new pile};
\draw (7,2)--(6.5,2) node[left]{old pile};
\draw (4.5,1.5)--(5,1.5) node[right]{new tower 
%$\tu{T}$
};
\draw (4.5,.5)--(4,.5) node[left]{old tower
%$\tu{\pa T}$
};
%\draw (4.5,0) node[below=0mm,text width=11cm,
%	information text]
%	  {
%	  The difference between the piles is 
%%at most 
%included in
%the difference between the towers. (new tower $=\tu{T}$, old tower $=\tu{\pa T}$)
%	  };
\end{tikzpicture}
}
\caption{The difference between the piles is 
%at most 
included in
the difference between the towers. (new tower $=\tu{T}$, old tower $=\tu{\pa T}$)
}
\label{HaufendifferenzTurmdifferenz}
\end{figure}
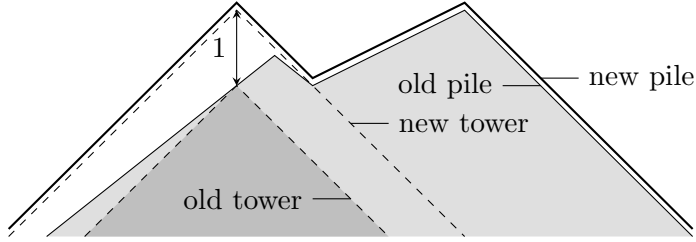

The following calculation proves it formally:
\begin{align*}
w_{j+1}\setminus w_j&
=\left(w_j\cup\tu(T_j)\right)\setminus w_j\\
&=\tu T_j\setminus w_j\\
&\subset \tu T_j\setminus\tu\parent T_j
\end{align*}
and
\begin{align}\label{setminustrick}
\mu(w_{j+1}\setminus w_j)\leq\mu(\tu T_j)-\mu(\tu\parent T_j),
\end{align}
because $T_j\ra\pa (T_j)$ implies $\pa (T_j)\in \Tu (T_j)$ and also $\tu(\pa T_j)\subset \tu T_j$, therefore $\mu\left(\tu T_j\setminus\tu\parent T_j\right)=\mu(\tu T_j)-\mu(\tu\parent T_j)$.

\emph{Step 4: $\mu(\tu S)-\mu(\tu\pa S)\leq 4$.
}
\emph{Idea in a figure: }

\begin{center}
\begin{tikzpicture}[information text/.style={fill=gray!10,inner sep=1ex}]
\fill[pattern=vertical lines, pattern color=gray!30] (1,0) -- (3,2) -- (5,0) -- (6,0) -- (3,3) -- (0,0) -- cycle;
\fill[pattern=horizontal lines, pattern color=gray!30] (0,0) -- (6,0) -- (5.1,.9) -- (.9,.9) -- cycle;
\draw%[dashed]
 (1,0) -- (3,2) -- (5,0);
\draw%[dashed]
 (0,0) -- (3,3) -- (6,0);
\draw[dashed] (1,.9) -- (3,2.9) -- (5,.9) -- cycle;
\draw[dashed,-{Stealth[scale=1.2]}]	(3,2)-- node[left]{1} (3,2.9);
\draw[dashed,-{Stealth[scale=1.2]},%>=stealth
]	(1,0)-- node[left]{1} (1,.9);
\draw[dashed,-{Stealth[scale=1.2]},%>=stealth
]	(5,0)-- node[right]{1} (5,.9);
\draw (4,2)--(4.3,2) node[right]{new tower};
\draw (4.5,.5)--(4.2,.5) node[left]{old tower};
\draw (.0,1.5) node[left=1mm,text width=6cm,
	information text]
	  {
	  Lift the old tower into the top of the new one and note: 
	  %The set difference between the towers has the same measure as the first layer of the new tower.
	  The difference of the measures of the new tower and the old tower equals the measure of the first layer of the new tower.
	  };
\end{tikzpicture}
\end{center}

%The $\ar\est$towers have a very regular shape: 
All $\ar\est$-towers have the same shape: 

\begin{center}
%%esttower
\begin{tikzpicture}[scale=.95,information text/.style={fill=gray!10,inner sep=1ex}]
\draw (-6.2,1) -- (-6,1) -- (-6,2) -- (-2,2) -- (-2,3) -- (0,3) -- (0,4) -- (1,4) -- (1,3) -- (3,3) -- (3,2) -- (7,2) -- (7,1) --(7.2,1);
\draw[dotted] (0,3)rectangle(1,4) 
	(-2,2) rectangle (0,3) 
	(1,2) rectangle (3,3) 
	(-6,1) rectangle (-2,2) 
	(3,1) rectangle (7,2) 
	(-6,.5)--(-6,1)
	(7,.5)--(7,1);
\draw (.5,3.5) node {$S$};
\draw[<->] (0,3.1)-- node [below] {\small{BS($lS$)}%$=\frac12$BS($lS{-}1$)
} (1,3.1);

\draw[<->] (0,2.3)-- node [above] {BS($lS{-}1$)} (-2,2.3);
\draw[<->] (3,2.3)-- node [above] {BS($lS{-}1$)} (1,2.3);
\draw[<->] (-2,2.1)-- node [below] {$2\frac12$BS($lS{-}1$)=$1\frac14$BS($lS{-}2$)} (3,2.1);

\draw[<->] (-6,1.3)-- node [above] {BS($lS{-}2$)} (-2,1.3);
\draw[<->] (3,1.3)-- node [above] {BS($lS{-}2$)} (7,1.3);
\draw[<->] (-6,1.1)-- node [below] {$3\frac{1}4$BS($lS{-}2$)} (7,1.1);

\draw %(1,0) node[below=0mm,
	(1.2,3.2) node[above right,
	text width=5.6cm,
	information text]
	  {
%	  Tower of $\stackrel{\text{est}}{\rightarrow}$
Here, BS$(k)=2^{-k}$ again denotes the size of a brick of level $k$.
	  };
\end{tikzpicture}
\end{center}

The figure shows that $\mu(\tu^k S)=:f(lS{-}k)$ depends only on $lS{-}k$. (Namely $f$ fits the recursion $f(0)=1,f(j{+}1)=\frac {f(j)}2 +2$.)

Hence
\begin{align*}
\mu(\tu S)-\mu(\tu\pa S)	&=f(0)+\dots+f(lS{-}1)-f(0)-\dots-f(\underbrace{l(\pa S){-}1}_{lS-2})\\
						&=f(lS{-}1)\\
						&=\mu(\tu^1(S)).
\end{align*}
It remains to show $\mu(\tu^1(S))\leq 4$. Instead of reflecting on the recursion formula, take four copies of $\tu S$, slightly shifted to the right overlapping each other. This pile has the same shape as the quasitower 
in Figure \ref{fig:quasitower}
above%
: Each of its layers has the size of 4 brick sizes. 

\begin{center}
%%esttowers
\begin{tikzpicture}[information text/.style={fill=gray!10,inner sep=1ex},scale=.5]
\filldraw[fill=gray!60, very thin, draw=gray] %(-6.5,1) -- 
(-6,1) -- (-6,2) -- (-2,2) -- (-2,3) -- (0,3) -- (0,4) -- (1,4) -- (1,3) -- (3,3) -- (3,2) -- (7,2) -- (7,1) --(10,1)--cycle;
\filldraw[fill=gray!40,xshift=1cm, very thin, draw=gray] (-6.5,1) -- (-6,1) -- (-6,2) -- (-2,2) -- (-2,3) -- (0,3) -- 
(0,4) -- (1,4) -- (1,3) -- (3,3) -- (3,2) -- (7,2) -- (7,1)% --(7.5,1)
;
\filldraw[fill=gray!20,xshift=2cm, very thin, draw=gray] (-6.5,1) -- (-6,1) -- (-6,2) -- (-2,2) -- (-2,3) -- (0,3) -- 
(0,4) -- (1,4) -- (1,3) -- (3,3) -- (3,2) -- (7,2) -- (7,1)% --(7.5,1)
;
\filldraw[fill=white,xshift=3cm, very thin, draw=gray] (-6.5,1) -- (-6,1) -- (-6,2) -- (-2,2) -- (-2,3) -- (0,3) -- 
(0,4) -- (1,4) -- (1,3) -- (3,3) -- (3,2) -- (7,2) -- (7,1)% --(7.5,1)
;
\draw[xstep=1,ystep=1,dashed, very thick] (0,3) grid (4,4);
%\draw[xstep=2,ystep=1,dashed, very thick] (-2,2) grid (6,3);
\draw[dashed, very thick] (-2,2) -- (-2,3) -- (0,3) -- (0,2);
\draw[dashed, very thick] (2,2) -- (2,3);
\draw[dashed, very thick] (4,2) -- (4,3) -- (6,3) -- (6,2);
\draw[xshift=2cm,xstep=4,ystep=1,dashed, very thick] (-8,1) grid (8,2);

\draw (2,0) node[below=0mm,text width=8cm,
	information text]
	  {
	  The union of these four $\stackrel{\text{est}}{\rightarrow}$towers has the same shape as a quasitower.
	  };
\end{tikzpicture}
\end{center}

Thus,
\begin{align*}
\mu\tu^1 S\leq \mu(\text{level 1 of the union of the 4 copies of }\tu S)=4.
\end{align*}
%(Claim is proven, theorem is proven.)
\end{proof}
%\emph{Step 5: End of the proof.}
\emph{Summary of the proof.}
Combination of step 3 and step 4 shows $\mu(w_{j+1})\leq \mu(w_j)+4$. This and $\mu(w_0)=0$ lead to $\# (w_N\setminus \T_0)\leq \mu(w_N)\leq 4N$.
\begin{bem}
In preparation for the simplex version of the proof, let
%If 
the sizes of initial bricks vary from 1
%but are at least $d$
with a minimum $d$ and a maximum $D$.
Then $\mu$ and $\ar\est$ can be adapted.
%, 
The size of a level $j$ brick is at least $2^{-j}d$, so 
%we have:
%\eqref{mudef} can be transformed into an inequality with a modified $\mu$:
%For arbitrary Lebesgue measurable $M=M^0{\times}\{0\}\,\cup\, M^1{\times}\{1\}\,\cup\, \dots \subset
%\Omega
%\R
%\times \N$, 
let
%Let
%\index{$\mu$ -- measure estimating $\#$!pile game}
\begin{align*}
\mu(M):=
%\frac1d
\sum_{j=1}^\infty \frac{2^j}d\left|M^j\right|
\end{align*}
for an arbitrary measurable $M=M^0{\times}\{0\}\,\cup\, M^1{\times}\{1\}\,\cup\, \dots \subset
%\Omega
\R
\times \N$. 
Then \eqref{mudef} turns into %the inequality
\begin{align*}
%\#W-\#W^0\leq\sum_{j=1}^\infty \frac{2^j}d\left|\bigcup W^j\right|&=:\mu\left(\bigcup %\iota 
%W\right)
\mu\left(\bigcup %\iota 
W\right)=\sum_{j=1}^\infty \frac{2^j}d\left|\bigcup W^j\right|&\geq \#W-\#W^0
\end{align*}
and \eqref{estdef} can be adjusted by
\begin{align*}
(y,l{+}1)\ar{\text{est}}(x,l) \quad:\LRq x\in y+2^{-l}[-D,D].
\end{align*}
%(This remark does not make much sense here, because the constant does in fact not depend on the size of the initial bricks, but this kind of variation will be needed in the simplex version of the proof).
\end{bem}
%\end{enumerate}
\subsection{Proof of the theorem of Binev--Dahmen--DeVore}
Back to simplices. The above proof is going to be adapted into a proof of the statement in Lemma \ref{forest formulation}.

As all existing proofs for the Binev--Dahmen--DeVore Theorem \cite{BDV,Stevenson,Karkulik} except \cite{Holst} by Holst, Licht and Lyu, this one “geometrises” the setup. In the pile game, the only characteristic property of a brick 
is 
%was 
its size. For simplices, there are several characteristic sizes. It is simple to use volume $\lvert S\rvert$ and the distance 
\index{D@$D_{\Vnew}$}
\begin{align*}
D_{\Vnew} S:=\max_{x\in S}\|x-\Vnew S\|_2.
\end{align*}
\begin{bem}
Since $D_{\Vnew}\leq \diam$, one could also work with $\diam$ instead of $D_{\Vnew}$.
\end{bem}
The finiteness of the number of equivalence classes in $\Simplexe$ has the following consequence, which was already known and used by Binev, Dahmen and DeVore \cite{BDV} and by Stevenson \cite{Stevenson}:
\begin{lem}\label{lem:dD}
There are only finitely many values for $2^{lS}\lvert S\rvert$ and for $2^{\frac{lS}n} D_{\Vnew} S$ in $\Simplexe$, especially a minimum $d$\index{d@$d$} for the first value and a maximum $D$\index{D@$D$} for the second value, such that
\begin{align}
1&\leq \frac{2^{lS}\lvert S\rvert}d\label{areaestimation}\\
\text{and}\quad D_{\Vnew} S&\leq 2^{-\frac{lS}n}D=:D_{lS}.\label{Destimation}
\end{align}
\end{lem}
\begin{proof}
\begin{enumerate}
\item
Bisection bisects volume, 
%so it 
and thus
does not change the value of $2^{lS}\lvert S\rvert$. This shows that there are at most $\#\T_0$ distinct values for $2^{lS}\lvert S\rvert$ in $\Simplexe$, one per tree.
\item
Corollary \ref{similarity classes} on page \pageref{similarity classes}
%stated 
states 
that there are only finitely many similarity classes of admissible simplices (in each tree and thus in $\Simplexe$ as well), so only finitely many values for $\frac{D_{\Vnew}(S)^n}{\lvert S\rvert}$ and thereby for
\begin{align*}
\sqrt[n]{\frac{D_{\Vnew}(S)^n}{\lvert S\rvert}\cdot 2^{lS}\lvert S\rvert}=2^{\frac{lS}n} D_{\Vnew} S.
\end{align*}
\end{enumerate}
\end{proof}
\begin{proof}[Proof of the BDV theorem (Lemma \ref{forest formulation}).]
\index{BDV theorem!SIC!proof}
%Now 
In the following,
%we imitate 
the proof of the toy problem
is imitated. Actually, there are no further ideas needed. The proof is based on the embedding
\begin{SCfigure}[][ht]
%\begin{center}
%%embedding
\begin{tikzpicture}[information text/.style={fill=gray!10,inner sep=1ex}]
%%Einträge der transformationsmatrix
\def\a{1};
\def\b{0};
\def\c{.4};
\def\d{.5};
\def\y{1.6};
\def\yachse{2*\y+.5};
\def\layer{
%\begin{scope}[cm={\a,\b,\c,\d,(0,0)}]
	\draw[->] (0,0) -- (4,0);
	\draw[->]  (0,0) -- (0,3);
	\draw (1,1) -- (3,1) -- (2,2) -- cycle;
%	\draw (.1,.5) node [ right] {$\Omega$};
};

\draw[->] (0,0) node[left]{0} -- (0,2*\y+1) node[left]{$l$};

\begin{scope}[cm={\a,\b,\c,\d,(0,0)}]
	\layer;
	\coordinate (A) at (2,1.5){};
\end{scope}

\begin{scope}[yshift=\y cm]
\draw (0,0) node[left]{1};
\begin{scope}[cm={\a,\b,\c,\d,(0,0)}]
	\layer;
	\draw (2,1)--(2,2);
	\coordinate (B) at (1.6,1.3){};
	\coordinate (C) at (2.4,1.3){};
\end{scope}
\begin{scope}[yshift=\y cm]
\draw (0,0) node[left]{2};
\begin{scope}[cm={\a,\b,\c,\d,(0,0)}]
	\layer;
	\draw (2,1)--(2,2);
	\draw (1.5,1.5) -- (2,1) -- (2.5,1.5);
	\coordinate (D) at (1.5,1.25){};
	\coordinate (E) at (2.5,1.25){};
	\coordinate (F) at (1.75,1.5){};
	\coordinate (G) at (2.2,1.6){};
\end{scope}
\end{scope}
\end{scope}

\draw[->](A) -- (B);
\draw[->] (A)--(C);
\draw[->] (B)--(D);
\draw[->] (B)--(F);
\draw[->] (C)--(E);
\draw[->] (C)--(G);

%\draw (2,0) node[below=3mm,text width=4cm,
%	information text]
%	  {
%	  Embedding of the admissible triangles.
%	  };
\end{tikzpicture}
\caption{Embedding of the admissible triangles.}\label{fig:embedding}
%\end{center}
\end{SCfigure}
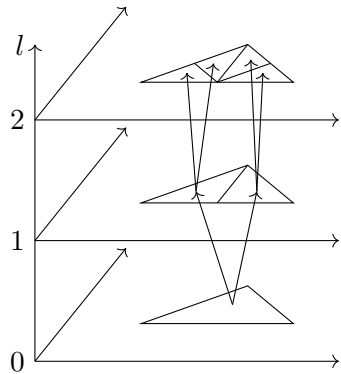
\index{i@$\iota$ -- embedding of $\Simplexe$!}
\begin{align*}
\iota S:=S\times \{l S\}
\end{align*}
of the simplices into 
%$\Omega\times \N$ 
$\R^n\times \N$
as depicted in Figure \ref{fig:embedding}. Thereby, the following analysis does not strictly distinguish between $\iota W=\{\iota S~|~S\in W\}$ and $\bigcup \iota W=\bigcup_{S\in W}\iota S$.

\index{W^k@$W^k$ -- $k$th layer of $W$}
For a forest $W$, let $W^k$ be the set of level $k$ simplices of $W$.

\emph{Step 1: Estimate and extend $\#\bullet-\#\T_0$.} According to \eqref{areaestimation}, for a non-initial $S$ define
\begin{align*}
1&\leq \frac{2^{lS}\lvert S\rvert}d =:\mu(\iota S).
\end{align*}
Union over all $S\in W^k$ %yields
%\begin{align*}
%\# W^k\leq \frac{2^k\left|\bigcup \iota W^k\right|}{d}=:\mu(\bigcup\iota W^k)
%\end{align*}
and over all $k\in \N_{\geq 1}$ %finally
yields
\begin{align}\label{mudefinition}%\label{leqmu}
\# W\setminus \#\T_0\leq \frac1d\sum_{k=1}^\infty{2^k\left|\bigcup W^k\right|}=:\mu\left(\bigcup\iota W\right).
\end{align}
For this reason, for arbitrary Lebesgue measurable $M=M^0{\times}\{0\}\cup M^1{\times}\{1\}\cup \dots \subset
%\Omega
\R^n
\times \N$, let
\index{m@$\mu$ -- measure estimating $\#$}
\begin{align*}
\mu(M):=\frac1d\sum_{k=1}^\infty 2^k\left|M^k\right|.
\end{align*}
%\index{->i@$\ar\iota$}
\emph{Step 2: Replace $\ar1$ by a relation $\ar\iota$ on $
%\Omega
\R^n
\times\N$ and estimate $\ar\iota$ by 
%\index{->est@$\ar\est$ -- estimates $\ar\iota$}
$\ar\est$.}

%%ariota2
\begin{figure}
\centering{
%\begin{center}
\begin{tikzpicture}[information text/.style={fill=gray!10,inner sep=1ex}, scale=2]
\def\a{1};
\def\b{0};
\def\c{.4};
\def\d{.5};
\def\e{0};
\def\y{1.8};
\def\n{1}
\def\layer{
%\begin{scope}[cm={\a,\b,\c,\d,(0,0)}]
	\draw[->] (0,0) -- (5,0);
	\draw[->]  (0,0) -- (0,3);
	\draw (1,1) -- (3,1) -- (2,2) -- cycle;
%	\draw (.1,.5) node [ right] {$\Omega$};
};
\clip (-.5,-.2) rectangle (5.1,3.3);
\draw[->] (0,0) node[left]{$lS$} -- (0,\n*\y+1) node[left]{$l$};

\begin{scope}[cm={\a,\b,\c,\d,(0,\e)}]
	\layer;
	\draw (2,2) -- (2.2,2) node [right]{\small{$\iota\Vnew(S)$}};
	\draw (2,1.4) node {\small{$\iota S$}};
	\coordinate (a) at (1,1){};
	\coordinate (b) at (2,1){};
	\coordinate (c) at (3,1){};
	\coordinate (d) at (1.5,1.25){};
	\coordinate (e) at (2,1.25){};
	\coordinate (f) at (2.5,1.25){};
	\coordinate (g) at (1.5,1.5){};
	\coordinate (h) at (2.5,1.5){};
	\coordinate (i) at (2,1.75){};
	\coordinate (j) at (2,1.5){};
	\coordinate (k) at (2,2){};
	
	\draw (2,2) circle [radius=1.41cm];
	\draw[<->,>=stealth] (2,2) -- node[right] {\small{$D_{lS}$}} (2-1.41*.6,2+1.41*.8);
\end{scope}

\begin{scope}[yshift=\y cm]
\draw (0,0) node[left]{$lS{+}1$};
\begin{scope}[cm={\a,\b,\c,\d,(0,\e)}]
	\fill[white] (0,0) rectangle (5,5);
	\layer;
	\draw (2,2) -- (2.2,2) node [right]{\small{$(\Vnew(S),lS{+}1)$}};
	\draw (2,1)--(2,2);
	\coordinate (A) at (2,2){};
\end{scope}
\end{scope}

%\foreach\x in {(a),(b),(c),(d),(e),(f),(g),(h),(i),(j),(k),(l)}
\begin{scope}[cm={\a,\b,\c,\d,(0,\e)}]
\foreach\x in {(1,1),(2,1),(3,1),(1.5,1.25),%(2,1.25),
(2.5,1.25),(1.5,1.5),(2.5,1.5),
%(2,1.75),
(2,1.5),(2,2)}
	\draw[->,%>=stealth, 
	gray] (A) -- \x;
\end{scope}

%\draw[->](A) -- (B);
%\draw[->] (A)--(C);
%\draw[->] (B)--(D);
%\draw[->] (B)--(F);
%\draw[->] (C)--(E);
%\draw[->] (C)--(G);

%\draw (2.5,0) node[below=3mm,text width=9cm,
%	information text]
%	  {
%	  	  $\stackrel{\iota}{\rightarrow}$ and $\ar{\text{est}}$
%: The point $(\Vnew(S),lS{+}1)$ $\ar{\iota}$-demands all points of $\iota S$ and $\ar{\text{est}}$-demands the circle drawn around.
%	  };
\end{tikzpicture}
\caption{$\stackrel{\iota}{\rightarrow}$ and $\ar{\text{est}}$
: The point $(\Vnew(S),lS{+}1)$ $\ar{\iota}$-demands all points of $\iota S$ and $\ar{\text{est}}$-demands the circle drawn around.}\label{fig:ariota2}
%\end{center}
}
\end{figure}
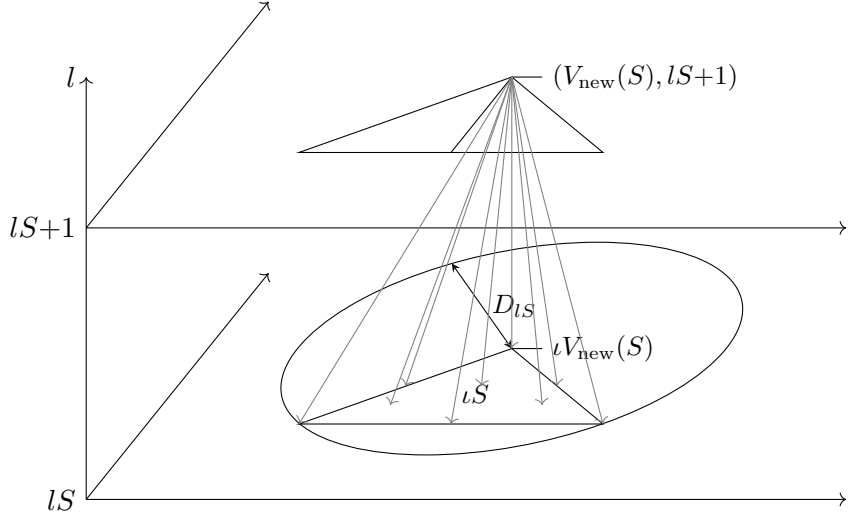

Let \index{B@$B(p,r)$ -- Euclidean ball}$B(p,r)$ be the Euclidean ball of radius $r$ centred at $p$ and \index{V@$V(r)$ -- volume of $B(p,r)$}$V(r)=r^n V(1)$ its volume. Let 
\begin{align*}
(y,l{+}1)\ar\iota (x,l)\quad:\LRq \ex S\in\Simplexe^l,y=\Vnew S, x\in S,
\end{align*}
\index{->i@$\ar\iota$}%
see Figure \ref{fig:ariota2}.
In other words, if $T\ar1 S$, then the point $(\Vnew S,lT)\in\iota T$ $\ar\iota$-demands all points of $\iota S$. Consequently,
\begin{align*}
\iota\left(\ar1\text{-closure of }U\subset \Simplexe\right)\quad\subset\quad \ar\iota\text{-closure of }\iota U.
\end{align*}
Now estimate $\ar\iota$ as follows: $y=\Vnew S, x\in S$ implies $\|x-y\|_2\leq D_{lS}$. So let, %this time 
in this case
for an \emph{arbitrary} $y\in\Omega$,
%Now estimate: If $(x,l+1)\ar\iota (y,l)$, then $x=\Vnew S, y\in S$, which implies $\|x-y\|_2\leq D_{lS}$. So let
\index{->est@$\ar\est$ -- estimates $\ar\iota$}
\begin{align*}
(y,l{+}1)\ar\est (x,l)\quad:\LRq \|x-y\|_2\leq D_{lS}.
\end{align*}
Then $\est$-closures include $\iota$-closures.

Additionally, estimate the embedding of the $\ar0$-equivalence class: By definition of $D_{lT}$, every $S$ with $\Vnew S=\Vnew T$ is included in the ball $B(\Vnew T,D_{lT})$ with the centre $\Vnew T$ and radius $D_{lT}$. %Hence, 

The remarks on Definition \ref{ar} and Lemma \ref{ra-closed is ar01-closed} on pages \pageref{ar} and \pageref{ra-closed is ar01-closed} imply that
\begin{align*}
S\in\Tu^{lS-k} T\quad\LRq \ex \text{ chain } T=T_0\ar0 S_0\ar1 \dots\ar1S_k=S.
\end{align*}
This shows
\begin{align*}
\iota\Tu^{lT}T=
\iota\left\{S~\middle|~T\ar0 S\right\}
%\iota\left\{T~\middle|~\Vnew T=\Vnew S\right\}
\subset B(\Vnew T,D_{lT})\times \{lT\},
\end{align*}
so let \index{tw@$\tu$ -- estimates $\Tu$}
\begin{align*}
\tu T :=\est\text{-closure of }B(\Vnew T,D_{lT})\times \{lT\}\supset \iota\Tu T.
\end{align*}
For a refinement sequence $W_0=\T_0$, $W_{j+1}=W_j\cup \Tu T_j$, let 
\index{w@$w$ -- estimates $W$}
\begin{align}\label{subest}
w_0:=\R^n\times\{0\}\supset \iota W_0, w_{j+1}:=\underbrace{w_j}_{\supset \iota W_j}\cup \underbrace{\tu T_j}_{\supset \iota \Tu T_j}\supset \iota W_{j+1}
\end{align}
be the estimated forests of the refinement sequence.

%Extend $\Omega$ to $\R^n$ before the definition of $\ar\est$. Obviously, the measure of the estimated forests $w_k$ 
%%towers 
%cannot shrink by this (and the number of simplices will be estimated by this measure), so it suffices to show the statement for this extended $\Omega$. 

\emph{Step 3.} \hyperlink{Step 3 from the pile game}{Step 3 from the pile game} above %can be applied 
literally 
applies
to this calculation here, so it holds
\begin{align}\label{setminustrick2}
\mu(w_{j+1}\setminus w_j)\leq\mu(\tu T_j)-\mu(\tu\parent T_j).
\end{align}
%exactly as above.

\emph{Step 4: Shape similarity of the tower layers and $\mu(\tu S)-\mu(\tu\pa S)\leq 2C$.}
Note that $\ar\est$-towers have the very regular shape 
depicted in Figure \ref{fig:simplexesttower}. 
%%esttower
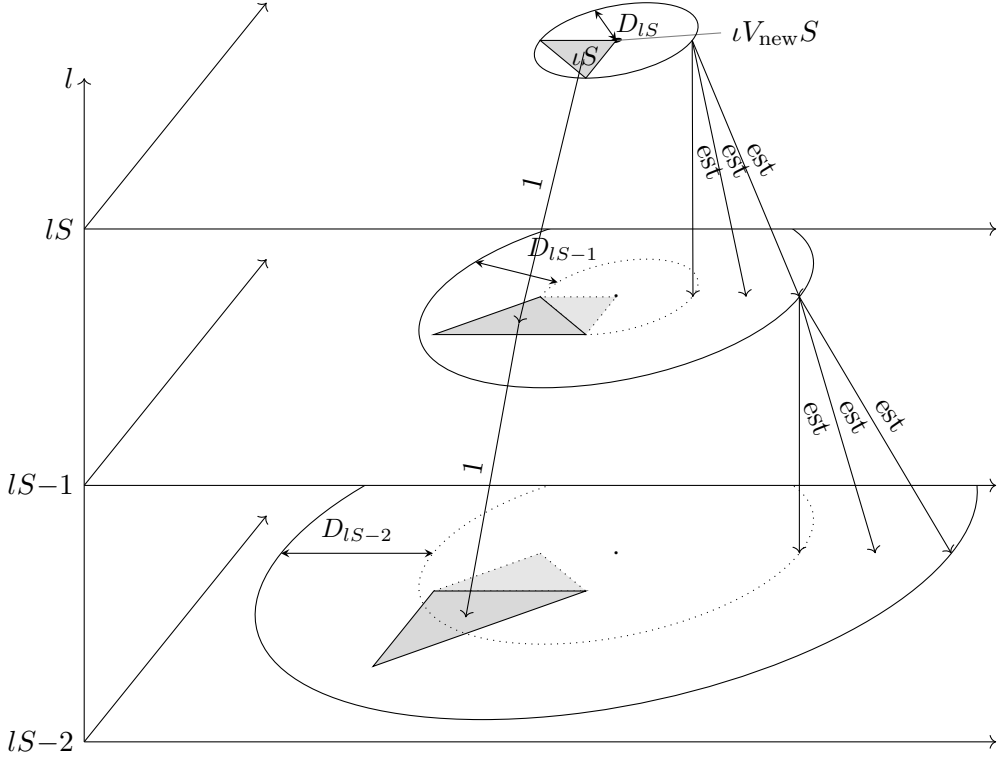
\begin{figure}[ht]
\centering{
\begin{tikzpicture}[information text/.style={fill=gray!10,inner sep=1ex}, scale=2]
\def\a{.5};
\def\b{0};
\def\breite{12};
\def\c{.2};
\def\d{.25};
\def\e{0};
\def\y{1.7};
\def\n{2};
\def\factor{2};
\def\layer{
%\begin{scope}[cm={\a,\b,\c,\d,(0,0)}]
	\fill[white] (0,0) rectangle (12,8);
	\draw[->] (0,0) -- (\breite,0);
	\draw[->]  (0,0) -- (0,6);
%	\draw (1,1) -- (3,1) -- (2,2) -- cycle;
%	\draw (.1,.5) node [ right] {$\Omega$};
};

\clip (-.5,-.2) rectangle (\a*\breite+.1,3.3+\y);
\draw[->] (0,0) node[left]{$lS{-}2$} -- (0,\n*\y+1) node[left]{$l$};

\begin{scope}[cm={\a,\b,\c,\d,(0,\e)}]
	\layer;
%	\draw (2,2) -- (2.2,2) node [right]{\small{$\iota\Vnew(S)$}};
%	\draw[fill=gray!30] (3,5) -- (1,5) --(3,3) -- cycle;
%	\draw[fill=gray!20,dotted] (5,5) -- (4,5) --(5,4) -- cycle; 
	\draw[fill=gray!20,dotted] (4,5) -- (5,4) --(3,4) -- cycle;
	\draw[fill=gray!30] (5,4) -- (3,4) --(3,2) -- cycle;
%	\coordinate (a) at (5,5){};
%	\coordinate (b) at (5-4.41,5){};
	\coordinate (c) at (8.41,5){};
	\coordinate (f) at (3.7,3.3){};
	\coordinate (g) at (1.5,1.5){};
	\coordinate (h) at (2.5,1.5){};
	\coordinate (i) at (2,1.75){};
	\coordinate (j) at (2,1.5){};
	\coordinate (k) at (2,2){};
	
	\draw (5,5) node {$\cdot$};
	\draw (5,5) circle [radius=4.41cm];
	\draw[dotted] (5,5) circle [radius=2.41cm];
%	\draw[dotted] (5,5) circle [radius=1cm];
	\draw[<->,>=stealth] (5-2.41*1,5+2.41*.0) -- node[above] {\small{$D_{lS{-}2}$}} (5-4.41*1,5+4.41*.0);
\end{scope}

\begin{scope}[yshift=\y cm]
\draw (0,0) node[left]{$lS{-}1$};
\begin{scope}[cm={\a,\b,\c,\d,(0,\e)}]
	\layer;
	\coordinate (b) at (7.41,5){};
	\coordinate (e) at (4,4.3){};
	\draw (5,5) node {$\cdot$};
	\draw (5,5) circle [radius=2.41cm];
	\draw[dotted] (5,5) circle [radius=1cm];
	\draw[<->,>=stealth] (5-1*12/13,5+1*5/13) -- node[above right] {\small{$D_{lS{-}1}$}} (5-2.41*12/13,5+2.41*5/13);
	\draw[fill=gray!20,dotted] (5,5) -- (4,5) --(5,4) -- cycle; 
	\draw[fill=gray!30] (4,5) -- (5,4) --(3,4) -- cycle;
\end{scope}

\begin{scope}[yshift=\y cm]
\draw (0,0) node[left]{$lS$};
\begin{scope}[cm={\a,\b,\c,\d,(0,\e)}]
	\fill[white] (0,0) rectangle (10,10);
	\layer;
	\coordinate (A) at (2,2){};
	\fill (5,5) circle [radius=.7mm];
	\draw[thin,gray,text=black] (5,5)--(6.3,5.2) node [right]{$\iota\Vnew S$};
	\draw (5,5) circle [radius=1cm];
	\draw[<->,>=stealth] (5,5) -- node[right] {\small{$D_{lS}$}} (5-1*.6,5+1*.8);
	\coordinate (a) at (6,5){};
	\coordinate (d) at (4.7,4.7){};
	\draw[fill=gray!30] (5,5) -- (4,5) --(5,4) -- cycle; 
	\filldraw (4.75,4.6) node {\small{$\iota S$}};
\end{scope}%%3D
\end{scope}%%erster y shift

%%obere pfeile
\begin{scope}[cm={\a,\b,\c,\d,(0,\e)}]
	\foreach \x in {7.41,6.71,...,6}
		\draw[->] (a) -- node [sloped, %near end,
		above] {est} (\x,5);
\end{scope}%%3D
\end{scope}%%zweiter y shift

%%untere pfeile
\begin{scope}[cm={\a,\b,\c,\d,(0,\e)}]
	\foreach \x in {9.41,8.41,...,7.41}
		\draw[->] (b) -- node [sloped, %near end,
		above] {est} (\x,5);
\end{scope}%%3D

%%1-pfeile
\draw[->] (d) -- node [sloped, above] {1} (e);
\draw[->] (e) -- node [sloped,above%, near start
] {1} (f);

%\foreach\x in {(a),(b),(c),(d),(e),(f),(g),(h),(i),(j),(k),(l)}
%\begin{scope}[cm={\a,\b,\c,\d,(0,\e)}]
%\foreach\x in {(1,1),(2,1),(3,1),(1.5,1.25),%(2,1.25),
%(2.5,1.25),(1.5,1.5),(2.5,1.5),
%%(2,1.75),
%(2,1.5),(2,2)}
%	\draw[->,>=stealth, gray] (A) -- \x;
%\end{scope}

%\draw[->](A) -- (B);
%\draw[->] (A)--(C);
%\draw[->] (B)--(D);
%\draw[->] (B)--(F);
%\draw[->] (C)--(E);
%\draw[->] (C)--(G);

%\draw (3,0) node[below=3mm,text width=12cm,
%	information text]
%	  {
%	  	  $\ar{\text{est}}$-tower of a simplex $S$. The layers are circles with exponentially growing radius. Whenever a simplex' new vertex lies in the next-level layer of the tower (the dotted circle) it must be included in the layer itself (the continuous circle).%of a  $\ar{\text{est}}$-demands the circle drawn around.
%	  };
\end{tikzpicture}
\caption{$\ar{\text{est}}$-tower of a simplex $S$. The layers are circles with approximately exponentially growing radius. Since $\ar{\est}$ estimates $\ar1$, whenever 
%a simplex' new vertex 
the new vertex of a simplex
lies in the next-level layer of the tower (the dotted circle), the simplex must be included in the layer itself (the continuous circle).}\label{fig:simplexesttower}
}
\end{figure}
The figure shows that %$\tu^{lS}(S)=B\left(\Vnew S,D_{lS}\right)$ and
\begin{align*}
\tu^k(S)=B\left(\Vnew S,D_{lS}+\dots+D_k\right).
\end{align*}
Consequently,
\begin{align}\label{Turmvolumen}
\mu\tu^k(S)=\frac{2^k}{d}
%\left|B\left(\Vnew(S),\underbrace{D_{lS}}_{2^{-\frac{lS}{n}}D}+\dots+\underbrace{D_k}_{2^{-\frac{lS}{n}}D}\right)\right|
V\left(\underbrace{D_{lS}}_{2^{-\frac{lS}{n}}D}+\dots+\underbrace{D_k}_{2^{-\frac{k}{n}}D}\right)%%Brüche
%V\left(\underbrace{D_{lS}}_{2^{-{l(S)}/{n}}D}+\dots+\underbrace{D_k}_{2^{-\left(\frac{lS}{n}\right)}D}\right)%%Brüche
&=\frac{D^n}{d}
%\left|B\left(0,2^{-\frac{lS-k}{n}}+\dots+2^{-\frac{0}{n}}\right)\right|\\
V\left(2^{-\frac{lS-k}{n}}+\dots+2^{-\frac{0}{n}}\right)\\%%Brüche
%V\left(2^{-(lS-k)/n}+\dots+2^{-0/n}\right)\\
&=:f(lS-k)\label{f}
\end{align}
depends only on $lS-k$.

Finally, compute
%\begin{align*}
%&\mu\tu S-\mu\tu\pa S\\&=\mu\tu^1 S+\dots+\mu\tu^{lS} S-\mu\tu^1\pa S-\dots-\mu\tu^{l(\pa S)}\pa S\\
%						&=f(lS-1)+\underbrace{f(lS-2)+\dots+f(0)-f(lS-2)-\dots-f(0)}_0\\
%						&=f(lS-1)\\
%						&\stackrel{(\ref{f}=\ref{Turmvolumen})}{\leq}\frac{D^n}{d}
%										%V\left(0,\sum_{j=0}^\infty 2^{-\frac{j}{n}}\right)\right|\\
%										V\left(\sum_{j=0}^\infty 2^{-\frac{j}{n}}\right)\\
%						&=\frac{D^n}{\left(1-2^{-\frac1n}\right)^n d}V(1)\\
%						&=:2C.
%\end{align*}
\begin{align*}
&\mu\tu S-\mu\tu\pa S\\&=\mu\tu^1 S+\dots+\mu\tu^{lS} S-\mu\tu^1\pa S-\dots-\mu\tu^{l(\pa S)}\pa S\\
						&=f(lS{-}1)+\underbrace{f(lS{-}2)+\dots+f(0)-f(lS{-}2)-\dots-f(0)}_0\\
						&=f(lS{-}1)
\end{align*}
This and the equality of $\eqref{f}$ and $\eqref{Turmvolumen}$ lead to 
\begin{align*}
\mu\tu S-\mu\tu\pa S	%&\stackrel{(\ref{f}=\ref{Turmvolumen})}{\leq}
						&\leq\frac{D^n}{d}
										%V\left(0,\sum_{j=0}^\infty 2^{-\frac{j}{n}}\right)\right|\\
										V\left(\sum_{j=0}^\infty 2^{-\frac{j}{n}}\right)%\\&
						=\frac{D^n}{\left(1-2^{-\frac1n}\right)^n d}V(1)\\
						&=:2C.
\end{align*}
%\index{C@$C$ -- BDV constant!SIC!formula}
\index{BDV constant!SIC!formula}
\emph{A short summary of the proof.}
\eqref{setminustrick2} shows
$\mu w_{j+1}-\mu w_j 
%\stackrel{\eqref{setminustrick2}}
\leq \mu \tu T_{j}-\mu \pa\tu T_j \leq 2C$%
. 
$\mu w_0= 0$ (because the sum in \eqref{mudefinition} starts at $k=1$).
The equations \eqref{mudefinition}, \eqref{subest} and the latter
imply 
%the last estimation in 
$\#(W_N\setminus\T_0)
%\stackrel{\eqref{mudefinition}}
\leq \mu(\iota W_N)
%\stackrel{\eqref{subest}}
\leq\mu w_N\leq 2NC$.
\end{proof}

\paragraph{Example.} 
%For a 2-dimensional domain $\Omega$, an 
%isosceles rectangular triangles of leg size 1, with diagonals as refinement edges 
%yields $D=1, d=\frac12,
%B(0,1)
%V(1)=\pi$ and therefore $C=\frac12\frac1{\left(1-\frac{\sqrt{2}}2\right)^2\frac12}\pi<37$.
Some examples for the BDV constants are given in Table \ref{tab:SIC examples}.
\begin{SCtable}[][ht]
\bgroup
%\begin{center}
\def\arraystretch{1.5}
\begin{tabular}{c|c|c|c|c}
$n$ & $D$ & $d$ & $V(1)$ & $C\leq$\\
\hline
2 & 1 & $\frac12$ & $\pi$ & 37\\
\hline
3 & $\sqrt 2$ & $\frac 16$ & $\frac43\pi$ & 4100\\
\hline
4 & $\sqrt 3$ & $\frac 1 {24}$ & $\frac{\pi^2}2$ & $840{,}000$
\end{tabular}
%\end{center}
\caption{BDV constants for an
initial mesh consisting of 
reference simplices of equal hyperlevels and types}
%\index{C@$C$ -- BDV constant!values for SIC}
\index{BDV constant!SIC!values}
\label{tab:SIC examples}
\egroup
\end{SCtable}
%%%%%%%%%%%%%%%%%%%%%%%%%%%%%%%%%%%%%%%%%%%%
\newpage
\section{A generalisation of the initialisation refinement
%for Kossaczký--Maubach Refinement 
proposed by Kossaczký and Stevenson}\label{sec:initial division}
%\subsection{Introduction}
Starting with an arbitrary triangulation without T-arrays, there is no known algorithm how to arrange the vertices of its simplices into T-arrays, such that the strong initial conditions (\hyperlink{SIC}{SIC}) hold true. Kossaczký \cite[Section 4 starting at page 286]{Kossaczky} (for dimension 3) and Stevenson \cite[Appendix A starting on page 239]{Stevenson} (for arbitrary dimension) remedied this shortage by constructing an initial refinement of tagged simplices of an arbitrary triangulation without T-arrays which satisfies SIC. Namely, each simplex is divided into $(n{+}1)!/2$ simplices of equal volume and each vertex of this refinement is assigned to a position in a T-array independently of the simplex, such that SIC is satisfied.

In the 2-dimensional case, this is done by dividing all triangles by their centroids into 3 triangles each, letting all original edges be horizontal parts (i.e.\ refinement edges) of the small triangles and all centroids of original triangles be the vertical vertex of the T-arrays of the small triangles, see Figure \ref{triangledivision}.

%%wird nicht mehr gebraucht
\newsavebox{\tarrays}% Box to store matrix content
\savebox{\tarrays}{$
\begin{pmatrix}
p_2\quad p_3\\
q_1
\end{pmatrix},
\begin{pmatrix}
p_3\quad p_1\\
q_1
\end{pmatrix},
\begin{pmatrix}
p_1\quad p_2\\
q_1
\end{pmatrix},
\begin{pmatrix}
p_3\quad p_4\\
q_2
\end{pmatrix},
\begin{pmatrix}
p_4\quad p_2\\
q_2
\end{pmatrix},
\begin{pmatrix}
p_2\quad p_3\\
q_2
\end{pmatrix}
$}
%%%%%%%%%%%%
\newsavebox{\tarray}
\savebox{\tarray}{$
\begin{pmatrix}
p_i\quad p_j\\
q_k
\end{pmatrix}
$}
\begin{SCfigure}[3][ht]
%\begin{center}
\begin{tikzpicture}[scale=.5]
\draw (0,0) node[left]{$p_1$} -- (6,0) node[right]{$p_2$} -- (9,3) node[right]{$p_4$} -- (3,3) node[left]{$p_3$}-- cycle;
\draw (6,0) -- (3,3);
\draw[dashed] (0,0) -- (3,1) node[below]{$q_1$} -- (3,3) -- (6,2) node[above]{$q_2$} -- (9,3);
\draw[dashed] (3,1) -- (6,0) -- (6,2);

%\draw (11,4) node[below right, text width=7cm]{Initial refinement according to Kossaczký and Stevenson of the triangulation $\{\conv\{p_1,p_2,p_3\},\conv\{p_2,p_3,p_4\}\}$ into tag\-ged simplices with the T-arrays 
%\usebox{\tarray}.};
\end{tikzpicture}
%\end{center}
\caption{Initial refinement according to Kossaczký and Stevenson of the triangulation $\{\triangle p_1p_2p_3,\triangle p_2p_3p_4\}$ into tagged simplices with the T-arrays 
\usebox{\tarray}.}
\label{triangledivision}
\end{SCfigure}

%That initial refinement algorithm is generalised here leading to less than $(n+1)!/2$ times the original number of simplices.
This section generalises their initial refinement algorithm to generate initial refinements with less simplices.
\subsubsection*{Dividing simplices by single points}
\label{sec:divide}
%%Warum wurde der Titel überhaupt geändert?
%\subsection{Preliminaries}
\begin{center}
\begin{tikzpicture}[scale=2]
\draw[thick] (2,0) -- (1,1);
\draw (0,0) --node{$T_2$} (2,0)node[right]{$p_1$}  --node[near start]{$S$} (1,1)node[above]{$p_2$} --node{$T_1$} cycle;
\filldraw (1.5,.5) node[above right]{$q$} circle [radius=.03];
\draw (1,.3) node {$T$};
\draw[dotted] (0,0) -- (1.5,.5);
\end{tikzpicture}
\end{center}
%Let $S=\conv\{p_0,\dots,p_n\}$ be a simplex (of dimension $n\leq 1$), $q$ a point in its relative interior. $S$ can be regularly divided 
A point $q$ in the relative interior of an $m$-simplex $S=\conv\{p_0,\dots,p_m\}$ 
%(of dimension $m\geq 0$) 
regularly \emph{divides} it
into the $m{+}1$ parts $S_j:=\conv(\{q,p_0,\dots,p_m\}\setminus\{p_j\})$ (for $j\in\{0,\dots,m\}$). (If $m=0$, i.e.\ if $S$ is a point, this is indeed not a proper division into several parts.) 
If $S$ is a subsimplex of an $n$-simplex $T$ thereby, $q$ also divides $T$ into $m{+}1$ parts regularly, namely into the convex hulls 
\begin{align*}
&\left\{\conv\left(S_j\cup\left(\Vertices T\setminus \Vertices S\right)\right)~|~j\in\{0,\dots,m\}\right\}\\
&=\left\{\conv\left\{q,T_j\right\}~\middle|~ T_j\text{ is %($n{-}1$)-subsimplex 
a hyperface
of $T$ not containing $q$}\right\}.
\end{align*} 

\subsection*{The generalised algorithm}
\index{initial division}
Let $\T_0$ be a regular triangulation of an $n$-dimensional domain $\Omega$ into $n$-simplices. 
%Let us call $\T_0$ just \emph{the simplices} and the set of their subsimplices just \emph{the subsimplices}.
In the following, the set $\T_0$ will just be called \emph{the simplices} and the set of their subsimplices just \emph{the subsimplices}.
\paragraph{Marking the initial triangulation.}
\index{mark}
%We distribute some points of types $m,\dots,2$ to %the $n$-skeleton of simplices of $\T_0$, but not to the vertices, 
Some points of \emph{types}\index{type of a point} $n,\dots,2$ (this type is not the type of a T-array from above but a new notion and the marking has nothing to do with the marking in the AFEM loop) are distributed to %the $n$-skeleton of simplices of $\T_0$, but not to the vertices, 
$\Omega$ pursuant to the following algorithm: 

\begin{algorithmic}
\For{$m=n,\dots,2$}
\State Let $\S_m\subset\Sub\T_0$\index{S_m@$\S_m$} be the 
$m$-%
subsimplices 
%of dimension $\geq 0$ 
not containing any point of type $>m$. Distribute points of type $m$ to 
%the relative interiors of simplices $S\in
$
\S_m$, such that every $m$-simplex in $\S_m$ contains exactly one. 
%(not necessarily in the relative interior; two $m$-simplices share a point, if it lies in their intersection). 
\EndFor
\end{algorithmic}
%(One could also conclude with $m=1$ instead of $m=2$% by placing a type 1 point on every still free edge% at the end
%, but this will not be necessary as %we will see.
%it would result in a bisection %later
%afterwards.)
The algorithm is outlined more detailed in Algorithm \ref{initial division} below. For an example, look at Figure \ref{divisionexample}.
\begin{bemn}
\begin{enumerate}
\item
A point can be contained in several $m$-subsimplices.
\item
%Prefer to set points onto
The points are preferred on subsimplices of low dimensions (optimally solely 
%onto 
on
vertices) and 
%onto 
on 
the longest edges.
%For 
%do:
%distribute points of type $n$ to the $n$-skeleton of simplices of $\T_0$, but not to the vertices, such that: 
\item
Note that there can be an $m$-subsimplex containing both type $>m$ and type $m$ points. But then it does not belong to $\S_m$, so a type $m$ point
will also lie in a subsimplex belonging to $\S_m$.
%%will not lie 
%are forbidden to be put
%%in 
%into
%the relative interior.
\item
A simple example of such a distribution of points is given by the initialisation of Kossaczký and Stevenson: Let the points of type $m$ be the barycentres of the $m$-subsimplices. The other distributions of points can be obtained by gradual alterations of this one: Choose an $(m{-}1)$-subsimplex $S$. Pull together all the points in the relative interiors of $m$-subsimplices $T\supset S$ and of $S$ itself to one point on $S$. The maximal type of all moved points becomes the type of 
%the new 
this point.
\end{enumerate}
\end{bemn}
\paragraph{Initial divisions.} 
%%Zu knapp
%Now a simplex (an initial one as well as a part of one) is repeatedly divided as described at the beginning by the point of maximal type in it, which has not become a vertex yet. This point of maximal type is unique:
%If this point lies in the relative interior of an $n$-simplex, the simplex will be divided into $n{+}1$ simplices, given by the convex hull of the type
%Now the 
The
level $k$ descendants of an initial simplex are, differently from Maubach Bisection, defined by the points of type $n,\dots, n{-}k{+}1$ recursively as follows: 
Starting with $k=1$,
assume 
that 
a level $k{-}1$ descendant $T$ of the initial simplex 
contains a unique 
%the 
marked point of %maximal type in it, which has not become a vertex of $T$ yet, 
type $n{-}k{+}1$. 
%in it
%is unique. 
Then divide $T$ by this point as described at the beginning. 
%in Subsection \ref{sec:divide}. 
For an example, see Figure \ref{divisionexample}.
Since a simplex can be divided into more than 2 simplices, the trees of initial simplices may be not binary anymore.
The above assumption should be proved 
%now:
in the following lemma:

%%Verworfen, weil Algorithmus und Begründung trennen.
%As induction assumption assume that a simplex $S$ arising through $n{-}m$ recursive divisions of an initial simplex (with $1\leq m\leq n$) is the convex hull of points %$q_{n+1},\dots,q_m$ of type $n{+}1,\dots,m$
%$q_{n},\dots,q_{m+1}$ of type $n,\dots,m{+}1$
% respectively, and an $m$-subsimplex $T_m\in \S_m$. Particularly, the marked point in $S$ which is not a vertex of $S$ of maximal type is a unique point $q_m$ of type $m$ (for $m\leq2$).
\begin{lem}\label{Sm well-defined}
A simplex $T$ arising through $n{-}m$ recursive divisions of an initial simplex (with $1\leq m\leq n$) is the convex hull of points %$q_{n+1},\dots,q_m$ of type $n{+}1,\dots,m$
$q_{n},\dots,q_{m+1}$ of type $n,\dots,m{+}1$
 respectively, and an $m$-subsimplex $S_m\in \S_m$. Particularly, 
 %the non-vertex of maximal type is 
 it contains
 a unique point $q_m$ of type $m$, if $m\geq2$.
%not consisting any type
\end{lem}
\begin{proof}[Proof by induction with decreasing $m$.]
%For $m=n$ let $S_n:=T$ and nothing else has to be shown. 
It is trivial for $m=n$ (with $S_n:=T$, the initial simplex).
Assume the statement for a certain $m$. By construction, the subsimplex $S_m$ contains exactly one point $q_m$ of type $m$. This divides $T$ into the simplices %sharing exactly one $n{-}1$-subsimplex with 
\begin{align*}
\left\{\conv\{q_n,\dots,q_m,S_{m-1}\}~\middle|~
%S_{m-1}\in \Sub{S_m},\dim S_{m-1}=m{-}1,q_m\notin S_{m-1}
S_{m-1}\text{ is 
%an ($m{-}1$)-subsimplex 
hyperface 
of $S_m$ not containing $q_m$}
\right\}.
\end{align*}
%being convex hull of $q_n,\dots,q_m$ and one of the $(m{-}1)$-subsimplices $S_{m-1}$ of $S_m$ not containing $q_m$. 
Since $S_{m-1}\subset S_m\in\S_m$, it does not contain any point of type $> m$ either. So indeed $S_{m-1}\in\S_{m-1}$ as stated. 
%In the marking algorithm, 
During the marking, 
$S_{m-1}$ has got exactly one point of type $m{-}1$ as stated.
\end{proof}

\begin{figure}
\centering{
\begin{tikzpicture}
\def\schritt1{
	%%initial tetrahedron
	\draw[thick] (0,0)node[left]{$a$} -- (4,0)node[right]{$b$} -- (4,2)node[right]{$c$} -- (2,4)node[left]{$d$} -- cycle;
	\draw[thick,dashed] (0,0) -- (4,2);
	\draw[thick] (4,0) -- (2,4);

	%%type 3 point
	\draw[thin] (0,0) -- (2,1.7)node[left]{$e3$};
	\draw[thin] (4,0) -- (2,1.7);
	\draw[thin] (4,2) -- (2,1.7);
	\draw[thin] (2,4) -- (2,1.7);
};

%%Fillings
\fill[gray!20] (0,0) -- (4,0) -- (2,1.7) -- (1,2) -- cycle;
\fill[gray!40] (4,2) -- (2,4) -- (2,1.7) -- (3.7,1.5) -- cycle;

\begin{scope}[xshift=-6 cm]
\fill[gray!20] (0,0) -- (4,0) -- (4,2) -- (2,1.7) -- cycle;
\schritt1
\end{scope}
\schritt1

%%type 2 points
\draw[thin,gray!100] (2,1.7) -- (3.7,1.5)node[below,black]{$f2$};
\draw[thin,gray!100] (4,0) -- (3.7,1.5);
\draw[thin,gray!100] (4,2) -- (3.7,1.5);
\draw[thin,gray!100] (2,4) -- (3.7,1.5);

\draw[thin,gray!100] (2,1.7) -- (1,2)node[left,black]{$g2$};
\draw[thin,gray!100] (4,0) -- (1,2);
\draw[thin,gray!100] (4,2) -- (1,2);

\draw[thin,gray!100] (0,0) -- (2.3,.6)node[left,black]{$h2$};
\draw[thin,gray!100] (4,0) -- (2.3,.6);
\draw[thin,gray!100] (4,2) -- (2.3,.6);
\draw[thin,gray!100] (2,1.7) -- (2.3,.6);

%%Beschriftung
%\draw%[xshift=4cm] 
%(5,4) node [below right,text width=3cm] {
%\begin{align*}
%e3&\in\rel\Int\conv\{a,b,c,d\}\\
%f2&\in\rel\Int\conv\{b,c,d\}\\
%g2&\in\rel\Int\conv\{a,d\}\\
%h2&\in\rel\Int\conv\{a,b,c\}
%\end{align*}
%};
\end{tikzpicture}
}
\caption{
An example for initial division of the tetrahedron 
%$\conv\{a,b,c,d\}$
$\triangle abcd$. The numbers specify the types of the marked points. The left figure shows the tetrahedron after the first division into 4 parts, the right one shows the second division of the 4 parts into 3 parts by $f2$, 2 + 2 parts by $g2$ and 3 parts by $h2$, respectively.
%(The points $e3,f2,g2$ and $h2$ lie in the relative interiors of the faces 
%%spanned by $\{a,b,c,d\},\{b,c,d\},\{a,d\}$ and $\{a,b,c\}$, 
%$\triangle abcd,\triangle bcd,\overline {ad}$ and $\triangle abc$, 
%resp.)
($e3\in\rel\Int\triangle abcd$, $f2\in\rel\Int\triangle bcd$, $g2\in\rel\Int\overline {ad}$ and $h2\in\rel\Int\triangle abc$.)
}
\label{divisionexample}
\end{figure}
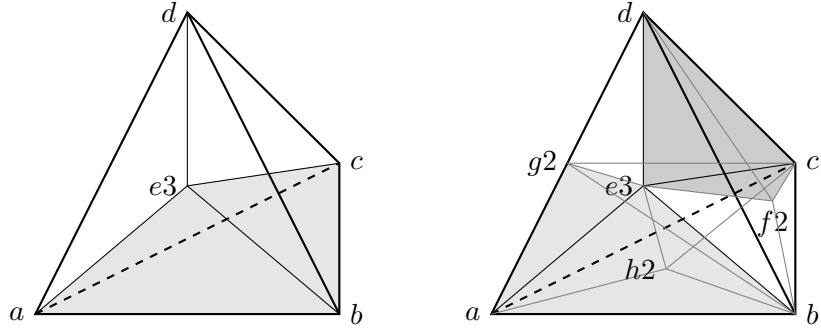

After $n{-}1$ recursive divisions the T-arrays are defined as follows.
%\begin{defn}[
\paragraph{T-arrays for level $n{-}1$ descendants.}%]
\label{T-arrays4divided}
%After $n{-}1$ recursive divisions let $S_1$ be an edge like stated in Lemma \ref{Sm well-defined}. 
Every simplex $\conv\{q_n,\allowbreak\dots,\allowbreak q_2,\allowbreak S_1\}$ (for $S_1$ being a 1-simplex as stated in Lemma \ref{Sm well-defined}) resulting from $n{-}1$ recursive divisions of the initial triangulation is equipped with the T-array
$\begin{pmatrix}
p_0\quad p_1\\
q_2\\
\vdots\\
q_n
\end{pmatrix}$ of type 1, where $p_0$ and $p_1$ are the vertices of $S_1$ (in arbitrary order).
%\end{defn}
%\begin{folg}
%The triangulation arising through $n{-}1$ recursive divisions of the initial triangulation consists of type 1 simplices sufficing the strong initial conditions, if the simplex $\conv\{q_n,\dots,q_2,T_1\}$ is equipped with the T-array
%$\begin{pmatrix}
%p_0\quad p_1\\
%q_2\\
%\vdots\\
%q_n
%\end{pmatrix}$ of type 1, where $p_0$ and $p_1$ are the vertices of $T_1$.
%\end{folg}
\begin{thm}
The triangulation defined by 
%Definition \ref{T-arrays4divided} 
the preceding paragraph
satisfies SIC.
\end{thm}
\begin{proof}
Recall the three properties of SIC from Definition \ref{SIC} on page \pageref{SIC}: \emph{Regularity, all T-arrays must have the same type and for each pair of simplices there must be reference coordinates coinciding on their intersection.}
Indeed, 
%the triangulation is regular and 
all T-arrays have type 1.
Since a point in a subsimplex $S$ defines a regular subdivision of all simplices including $S$, the triangulation is regular.
Reference coordinates for the T-array
$\begin{pmatrix}
p_0\quad p_1\\
q_2\\
\vdots\\
q_n
\end{pmatrix}$
are given by 
$$\vi_1(p_0)=0,\vi_1(p_1)=e_1,\vi_1(q_j)=\frac12(e_1+\dots+e_j) \text{ for $j\geq2$}$$
or alternatively by 
$$\vi_2(p_0)=e_1,\vi_2(p_1)=0,\vi_2(q_j)=\frac12(e_1+\dots+e_j) \text{ for $j\geq2$}.$$
Since the marked points $q_j$ and their types were set globally, their reference coordinates are also global. If two of the level $n{-}1$ simplices intersect in horizontal vertices, an appropriate choice between $\vi_1$ and $\vi_2$ (note that for each pair of simplices the reference coordinates can be chosen individually) ensures coincidence of the reference coordinates also for these vertices and hence for the whole intersection of the two simplices.
\end{proof}
Algorithm \ref{initial division} sums up the result of this chapter.
\begin{algorithm}
\caption{Generalised initial division}\label{initial division}
\paragraph{Distribution of points.}
\begin{algorithmic}
\For{$m:=n, m \geq 2,m${-}{-}}
\State{Let $\S_{m}$ be the $m$-subsimplices of $\T_0$ not containing any point of type $> m$.}
\State{$Remainder:=\bigcup\S_{m}$}
\State{$Points_m:=\leer$}
\While{$Remainder\neq\leer$}
\State{Choose $p\in Remainder$}\\
%\Comment{prefer to choose points $p$ on subsimplices of low dimension}
\Comment{prefer to choose a point $p$ on a subsimplex of low dimension}
\State{$Points_m:=\{p\}\cup Points_m$}
\State{$patch:=\{S\in\S_m~|~p\in S\}$}
\For{$S\in patch$}
\State{$point_S:=p$}
\EndFor
\State{$Remainder:=Remainder\setminus \bigcup%\S_{=m}p
patch$}
\EndWhile
%\State{$\{point_S~|~S\in\S_m\}$
\State{$Points_m$
are the points of type $m$.}
\EndFor
\end{algorithmic}
%%%%%%%%%%%%%%%%%%%%%%%%%
\paragraph{Step by step division of simplices.}
\begin{algorithmic}
\State{$\Help_{n+1}:=\{(\Vertices T,())~|~T\in\T_0\}$}\\
\Comment{New data Structure; a simplex is represented by a pair $(O,N)$, where $O$ is the set of its old vertices and $N$ is the \emph{list} of its new vertices}
\For{$m=n,m\geq 2,m${-}{-}}
\For{$T=(O,(q_{m+1},\dots,q_n))\in\Help_{m+1}$}
\State{$\Hull:=\text{subset of $O$ with }point_{\conv O}\in\rel\Int(\conv\Hull)$}
\Comment{is unique}
\State{$parts_T:=\left\{\big(O{\setminus} \{p_m\}\,,\,(point_{\conv O},q_{m+1},\dots,q_n)\big)~|~p_m\in\Hull\right\}$}
\EndFor
\State{$\Help_m:=\bigcup_{T\in\Help_{m+1}} parts_T$}
\EndFor
\end{algorithmic}
%%%%%%%%%%%%%%%%%%%%%%%%%%%%%%%
\paragraph{T-arrays for the simplices.}
\begin{algorithmic}
\For{$T=(O,(q_2,\dots,q_n))\in\Help_2$}
\State{Provide an order in $O=\{p_0,p_1\}$.}
\State{$T\text{-}array_T:=
\begin{pmatrix}
p_0\quad p_1\\
q_2\\
\vdots\\
q_n
\end{pmatrix}
$}
\EndFor
\State{$\T:=\{T\text{-}array_T~|~T\in\Help_2\}$}
\end{algorithmic}
\end{algorithm}
After this initialisation of T-arrays, further descendants are generated by standard bisection.
\begin{bemn}
\begin{itemize}
\item
If it succeeds to put the marked points solely on vertices of the initial triangulation, there will be no division at all, but just 
%a configuration of the simplices with T-arrays.
an arrangement of the vertices of the simplices into T-arrays of type 1.
\item
%If we start 
Starting with a cube divided into $n!$ reference simplices of type $n$, 
let 
%letting 
each centre of an $m$-dimensional face of the cube be a type $m$ point. Then every initial division step is already 
%Kossaczký--Maubach Bisection
standard bisection. %of these simplices. 
\item
Hence the presented division fills the space between an ordinary initialisation of a triangulation into tagged simplices, the standard bisection and the initialisation division of Kossaczký and Stevenson in a sense.

%Probably even the  of Bänsch and of Arnold, Mukherjee and Pouly are of this type.
\item
Kossaczk\'y chooses different T-arrays for the divided simplices.
%\item
%Instead of subdividing any simplex at once, these divisions could be suspended until they are demanded by the refinement algorithm, as Kossaczký already proposed.
%To this end, the forest of admissible simplices should start with the primordial triangulation as roots and the successively subdivided simplices as their children and grandchildren. Only after the $m{-}1$ initial divisions further descendants are defined by Maubach Bisection.
\end{itemize}
\end{bemn}
%%%%%%%%%%%%%%%%%%%%%%%%%%%%%%%%%%%%%%%%
\newpage
\section{Restricted T-arrays and a generalisation of a theorem of Gallistl, Schedensack and Stevenson}\label{sec:ReTaCo}
\subsection{Introduction: the GSS theorem}
%Dietmar 
Gallistl, 
%Mira 
Schedensack and 
%Rob 
Stevenson proved the following 
\cite[Theorem 3.3]{GSS}:
\begin{thm}[GSS theorem]\label{GSS}\index{GSS theorem}
Suppose SIC. There is some universal constant $C(\T_0)$, such that: 

Let $%S\in
\T$ be an arbitrary %simplex of an 
admissible triangulation and 
$S'\in\refine(\T,T)$ be a descendant of $S\in\T$. Then $l(S')\leq l(S)+C$.
\end{thm}

In uniform refinements $\T$ of $\T_0$, it is possible to bisect every simplex once, getting the (regular) next level uniform refinement $\hat \T$. Therefore, $\refine(\T,T)$, which is coarser than $\hat \T$, bisects every simplex in $\T$ at most \emph{once}.

%In this section, we will argue always with %such estimations: Some $\T'$ is finer than every simplex $S$ in $\T$. Hence, it is finer than $\refine(\T,S)$, which is the coarsest of all refinements $\T'$ of $\T$, strictly finer than $S$.
%this estimation principle: 
This estimation principle is crucial in this section:
Recall that $\refine(\T,T)$ is the coarsest refinement of $\T$ strictly finer than $T$. So 
%if you give 
for 
a 
given
refinement $\hat \T$ of $\T$ strictly finer than $T$, $\refine(\T,T)$ is certainly coarser than $\hat\T$.

In 2D, $C=2$ does the job generally: Bisect every triangle twice as on the 
left of Figure \ref{fig:refine}
%first
%picture
below%
. Then every edge of $\T$ gets bisected but no further edge is bisected and thus hanging nodes do not arise.

This is more complicated in higher dimensions: Consider the 
%tagged simplex 
T-array
$S=\begin{pmatrix}
a\quad b\\c\\d
\end{pmatrix}$. Assume it has a neighbour with refinement edge $\overline{cd}$ which should be refined. So $S$ has to be bisected into descendants not including $\overline{cd}$ as a whole edge. 
%We follow 
Follow
a path in the binary tree:
\begin{align*}
\begin{pmatrix}
a\quad b\\c\\d
\end{pmatrix}
\ra
\begin{pmatrix}
a\\\frac{a+b}2\\c\\d
\end{pmatrix}
\ra
\begin{pmatrix}
\frac{a+b}2&c&d\\ 
&\frac{a+d}2
\end{pmatrix}
\ra
\begin{pmatrix}
c\quad d\\
\frac{a+b}4 {+}\frac d2\\
\frac{a+d}2
\end{pmatrix}
\ra
\begin{pmatrix}
%c\\
d\\
\frac{c+d}2\\
\frac{a+b}4 {+}\frac d2\\
\frac{a+d}2
\end{pmatrix}.
\end{align*}
Note that after this refinement, the face $\triangle abd$ 
%looks like that:
is more than twice uniformly bisected as depicted on the right of Figure \ref{fig:refine} below. 
\begin{figure}[b]
\centering{
%%GSS 2D
\begin{tikzpicture}[information text/.style={fill=gray!10,inner sep=1ex}, scale=.5]
%\draw (0,0) node [left] {$a$}; 
\draw (0,0) %node [left] {$a$}  
-- (4,0) %node [right] {$d$} 
-- (4,4) %node [right] {$b$} 
-- cycle;
\draw 
%(2,0) %node [below] {$\frac{a+d}2$} 
%-- 
(3,1) -- (4,0);
%\draw[thin, gray] (3,1) -- (4.3,1.2) node [right, black] {$\frac{a+b}4+\frac d2$};
\draw %[dashed] 
(2,0) -- (2,2) -- (4,2)  (3,1) -- (2,2);
\draw (4,2) node[right=
%7mm
2mm
,text width=2.8cm,
	information text]
	  {
	  	 After 2 uniform refinements of a triangle, all original edges are bisected.
	  };
\end{tikzpicture}
%%GSS
\begin{tikzpicture}[information text/.style={fill=gray!10,inner sep=1ex}, scale=.5]
%\draw (0,0) node [left] {$a$}; 
\draw (0,0) node [left] {$a$}  -- (4,0) node [right] {$d$} -- (4,4) node [right] {$b$} -- cycle;
\draw (2,0) %node [below] {$\frac{a+d}2$} 
-- (3,1) -- (4,0);
%\draw[thin, gray] (3,1) -- (4.3,1.2) node [right, black] {$\frac{a+b}4+\frac d2$};
\draw %[dashed] 
(2,0) -- (2,2) -- (4,2) -- (3,1) -- (2,2);
\draw (4,2) node[right=
7mm
%2mm
,text width=2.8cm,
	information text]
	  {
	  	 The face $abd$ after refinement.
	  };
\end{tikzpicture}
}
\caption{}\label{fig:refine}
\end{figure}
In this case,
$S$ must be bisected 4 times successively, not only every edge of $S$ is bisected, but also the new edge $(d,\frac{a+b}2)$. That is why the GSS theorem is not trivial. 

The major part of this section unfolds some basic theory for the restriction of T-arrays and for the the initial condition that restricted T-arrays coincide (ReTaCo). ReTaCo is satisfied by every admissible triangulation, if $\T_0$ satisfies SIC (see Lemma \ref{SIC=>ReTaCo}). Proofs are technical and lengthy there. The main result, the generalisation of the GSS theorem, is Theorem \ref{quasi-uniform} in Subsection \ref{sec:quasi-uniform refinement} at the end on page \pageref{quasi-uniform}, together with Corollary \ref{SIC=>ReTaCo} before. 
%Theorem \ref{quasi-uniform} and its proof are mostly independent of the former ReTaCo theory.
The proof of Theorem \ref{quasi-uniform} can be understood without the former ReTaCo theory with $\T_0$ being any admissible refinement of an initial triangulation which satisfies SIC.
%%%%%%%%%%%%%%%%%%%%%%%%%%%%%%%
\subsection{%Restriction of T-arrays and restricted 
Restricted
T-array coincidence (ReTaCo)}
\subsubsection{Restriction of T-arrays}\index{restriction of T-arrays}\index{restricted T-array}\index{restriction}
\begin{defn}[restriction of a T-array to a subsimplex]
A T-array can be \emph{restricted }to subsimplices. For that purpose, erase some entries from the T-array and push the remaining horizontal and vertical part together. 
%To be precise here, we need to extend the notion of a T-array: We need to introduce type -1, where the horizontal part is empty.
For example, the restrictions of
\begin{align*}
\begin{pmatrix}
p_0\quad p_1\quad p_2 \quad p_3\\
p_4\\
p_5
\end{pmatrix}
\end{align*}
to $
%\conv\{p_1,p_3,p_5\}
\triangle p_1p_3p_5
$ and $
%\conv\{p_4,p_5\}
\overline {p_4 p_5}$ are the T-arrays
\begin{align*}
\begin{pmatrix}
p_1\quad p_3\\
p_5
\end{pmatrix}
\quad\text{and}\quad
\begin{pmatrix}
%p_1\\
 p_4\\
 p_5
\end{pmatrix},
\end{align*}
respectively. The restriction of a T-array $T$ to a subsimplex $S$ is denoted by $\rstr(T,S)$.
For brevity, the words \emph{to a subsimplex} are sometimes dropped.
Using the expression $\rstr(T,S)$\index{rstr@$\rstr(T,S)$ -- restriction of T-array $T$ to subsimplex $S$} always implies the statement that $S$ is a subsimplex of $T$.
The set of all restrictions of $T$ is denoted by $\Rstr(T)$\index{Rstr@$\Rstr$ -- set of restrictions}.
\end{defn}
%\small
%Wir wollen diese Einschränkung formal beschreiben. Dafür bezeichnen wir mit $A_1<\dots<A_k$ diejenige Totalordnung auf der Menge $\{A_1,\dots,A_k\}$, für welche $A_i<A_j:\LR i<j$.
%
%Das angeordnete Simplex kann formal beschrieben werden als Bipartition der Menge $\{p_0,\dots,p_n\}$ in $(H,V)$, den horizontalen Teil $H:=\{p_0,\dots,p_k\}$ und den vertikalen Teil $V:=\{p_{k+1},\dots,p_n\}$, zusammen mit einer Menge von zwei Ordnungsrelationen auf $H$, nämlich $\{p_0<\dots<p_k,p_k<\dots<p_0\}$ und einer Ordnungsrelation auf $V$, nämlich $p_{k+1}<\dots<p_n$. Das Simplex einschränken heißt dann, sowohl die Partition als auch die Totalordnungen einzuschränken. Dabei ist die auf eine Teilmenge $T$ eingeschränkte Partition $(H\cap T,V\cap T)$ und die Einschränkung einer Relation $R\subset M\times M$ ist ${R\cap (T{\times} T)}$.) Auf eine formale Beschreibung, wie der Array gedreht wird, wollen wir hier verzichten.
%
%\normalsize
As %Appropriate 
tools to prove the next lemmas, \emph{gappy restrictions} 
%and \emph{tree induction} 
are introduced.
\begin{defn}[gappy T-array, gappy restriction%, push together $\pt$
]
%\index{pt@$\pt$}\index{push together}
A \emph{gappy }T-array has both usual entries (being \emph{vertices} of the simplex) and \emph{gaps}, denoted by $\square$.
If some entries of a T-arrays are entries erased but the gaps are \emph{not} pushed together, it is called \emphindex{gappy restriction}\index{gappy T-array} of the T-array. One can apply Maubach's bisection on gappy T-arrays, extending the arithmetic by the rule that the midpoint between an arbitrary entry and a gap is a gap again (as usual for $\infty$). 
\end{defn}
This way, the tree of a T-array $T$ is mapped bijectively to the tree of a gappy restriction $G$ and even entry by entry. (That is, every node (i.e.\ T-array) $D$ in one tree is mapped to a node $e(D)$ in the other one and furthermore every entry in a T-array $D$ is associated with the entry at the same place in the image $e(D)$.) Here, one and the same gappy T-array can appear several times in the tree as different nodes of the tree.
\begin{lem}\label{restricted children}
Let $S$ be a restriction of a T-array $T$. 
\begin{enumerate}
\item
If $t(T)=0$, then $\rstr(T^T,S)=S^T$.
\item\label{it:Eref not included}
If $\Eref T\not \subset S$, then it holds $\rstr(T_1,\conv S)=S$ for one child $T_1$ of $T$ and $\rstr(T_2,T_2\cap S)=\rstr(S,T_2\cap S)$ for the other child.
\item\label{it:Eref included}
If $\Eref T\subset S$, then the restrictions of the children of $T$ to their intersection with $S$ are the children of $S$.
\end{enumerate}
\end{lem}
\begin{proof}
\begin{enumerate}
\item is obvious.
\end{enumerate}
\emph{2. and 3.} Let $G$ be the gappy restriction of 
%a T-array 
$T$ to $\conv(S)$. It is clear that all entries in the 
%tree 
children
of $G$ which are not gaps equal the associated entries in the 
%tree 
children
of $T$, because 
%all calculations 
the calculation
leading to this entry equal%
s
here and there. Furthermore, it is clear that exactly these vertices (i.e.\ entries) in the 
%tree 
children
of $T$ correspond to non-gaps in the 
%tree 
children
of $G$, which lie in $\conv(S)$, the convex hull of all non-gaps. 
%Third, comparing a T-array with its children, the number of gaps in a child of a T-array can only stay constant or increase by 1. (It increases by 1, if and only if exactly one of the vertices of the refinement edge, say $p_0$, is a gap and the child containing $p_0$ is regarded.) This shows that the T-arrays with equally many gaps as $G$ form a subtree of all its descendants.

The remaining assertion is that pushing together the 
%descendants 
children
of the gappy restriction $G$ 
%of $T$ 
%%with equally many gaps as $G$ 
%leads to all descendants of the restricted T-array $S$.
yields what is stated
for the restrictions of the children of $T$.

%Tree induction is applied 
%%now.
%as follows.
%\\
%1st step: \emph{Every descendant of $G$ is pushed together to a descendant of $S$. %$\rstr(\Descext G)\subset \Descext S$.
%}\\
%By definition, $G$ is pushed together to $S$.\\
%
%By induction assumption, let
Let 
$G=
\begin{pmatrix}
p_0&\dots&p_k\\
&\vdots\\
&p_n
\end{pmatrix}$
be a gappy T-array. 
%which is pushed together to a descendant $S'$ of $S$.

\begin{enumerate}
\item[2.]
If at least one of the vertices of the refinement edge of $G$ is a gap, say $p_0=\square$, then $G$ has the children
\begin{align*}
G_1=\begin{pmatrix}
p_1\quad \dots \quad p_{k}\\
\square\\
p_{k+1}\\
\vdots\\
p_n
\end{pmatrix} \text{ and } G_2=\begin{pmatrix}
\square\quad p_1\quad \dots \quad p_{k-1}\\
\square\\
p_{k+1}\\
\vdots\\
p_n
\end{pmatrix}
\end{align*}
and pushing together $G_1$ still yields %the same T-array as pushing together $G'$
$S$. %If $p_k$ is also a gap, then also $G_2$ yields the same. If not, $G_2$ has an additional gap and is not tracked anymore, nor its descendants are. %In this second case, bisection of the gappy T-array corresponds to 
Pushing together $G_2$ yields
$$
\rstr(T_2,S\cap T_2)%=\pt(G_2)
=\rstr\left(S,\conv\left(\Vertices S{\setminus} \{p_k\}\right)\right)=\rstr(S,S\cap G_2).
$$

\item[3.]
If both $p_0$ and $p_k$ are non-gaps, it makes no difference, whether first to bisect or first to push together the T-array obviously. 
%Transposition commutes with pushing together obviously.
\end{enumerate}
\end{proof}
\begin{defn}[restricted descendant]\index{restricted descendant}
%A \emph{restricted descendant }of a T-array $S$ is 
%%the 
%a restriction of an extended descendant $D$ of $T$% to one of the subsimplices of $D$
%. 
The set of \emph{restricted descendants} of a T-array $T$ is $\Rstr\Descext T$,
the set of restrictions of extended descendants $D$ of $T$.
\end{defn}
The following seemingly more general notion is only tentative. It is equivalent to the restricted descendant %(Lemma \ref{lem:restricted descendants}.\ref{generalised is restricted}).
as the next lemma shows.
\begin{defn}[generalised descendant]
A \emph{generalised descendant }of a T-array $S$ is a T-array, which arises by a finite recursive application of choosing one of the children in the extended binary forest and 
%restricting to a subsimplex.
restriction.
\end{defn}
%An appropriate tool to prove the next lemma is the following \emph{tree induction}:
%
%\begin{lem}[tree induction]\label{tree induction}
%To prove a statement for all nodes of a tree, it suffices to prove it for the root and to infer from that statement for an arbitrary node to the statement for its children.
%\end{lem}
%\noindent (without proof)

\begin{lem}\label{lem:restricted descendants}
\begin{enumerate}
\item\label{generalised is restricted}
The set of restricted descendants of a T-array equals the set of generalised descendants.
\item
Let $S$ be a restriction of a T-array $T$. Then 
\begin{align*}
%\left\{\rstr(D,D\cap S)~\middle|~D\in\Descext T\right\}=\left\{\text{restricted descendants }$Z$ of $S$\right\}
\left\{\Rstr(D,D\cap S)~\middle|~D\in\Rstr\Descext T\right\}=\Rstr\Descext S, 
\end{align*}
i.e.\ the following diagram commutes.
%\begin{center}
%\begin{tikzpicture}[mapsto/.style={Bar->}]
%\def\rechts{4};
%\def\unten{2};
%\def\edgemapsto{edge[mapsto]};
%\node (T) at (0,8) {$T$};
%\node (S) at (\rechts,8) {$S$};
%\node (D) at (0,8-\unten) {$\Rstr\Descext T$};
%\node (Z) at (\rechts,8-\unten) {$\Rstr\Descext S$};
%\node (I) at (.5*\rechts,8-2*\unten) {$\{$intersections $D\cap S~|~D \text{ is descendant of } T\}$};
%\path %[mapsto/.style={Bar->}] 
%(T) edge[mapsto] (D)
%(T) edge[mapsto]
%	node[above]{\scriptsize{restriction}} 
%%	node [below]{\scriptsize{to a subsimplex}} 
%	(S)
%(S) edge[mapsto] (Z)
%(D) edge[->]
%	node[sloped, above]{\scriptsize{restriction}} 
%	(I)
%(Z) edge[->]
%	node[sloped, above]{\scriptsize{restriction}} 
%	(I)
%;
%\end{tikzpicture}
%\end{center}
\begin{align*}
\begin{matrix}
T & \xmapsto%[\text{to a subsimplex}]
{\text{restriction}} & S\\
\Big\downarrow && \Big\downarrow\\
\Rstr\Descext T& \xrightarrow[\text{to }S\cap \bullet]{\text{restriction}}&
\Rstr\Descext S
\end{matrix}
\end{align*}
(The down-pointing arrows are actually $\mapsto$.)
\end{enumerate}
\end{lem}
\begin{proof}[Proof of 1.]
A generalised descendant arises by recursive application of choosing a child and restricting to a subsimplex. A restricted descendant as well, but at first only children can be chosen and only at the end the T-array is restricted. Thus, all what has to be shown is that the choices of children can be brought forward before the restrictions, i.e.\ for a T-array $T$, a restriction $S$ of $T$ and a child $Z$ of $S$, wanted is a descendant $D$ of $T$ which can be restricted to $Z$. 

If the refinement edge of $T$ is not included in $S$, then $T$ has a child with the same restriction $S$ according to Lemma \ref{restricted children}.\ref{it:Eref not included}. This child is chosen recursively until the type becomes 0 or the refinement edge of $T$ is included in 
%the refinement edge of 
$S$. 
If $t(T)=0$, also $t(S)=0$ and also $T^T=:D$ is restricted to $S^T=Z$.
If $\Eref T\subset 
%\Eref 
S$, choosing the child can be brought forward according to Lemma \ref{restricted children}.\ref{it:Eref included}.

\emph{Proof of 2.} %\emph{1st statement: 
For a restricted descendant $D$ of $T$, $S\cap D$ is a subsimplex of $D$,
%.} 
according to Lemma \ref{UntermengeTeilmenge}. 
%Recall from the beginning of the proof of Lemma \ref{restricted children} that in every descendant of the gappy restriction of $T$ the non-gaps correspond to the vertices in $\conv(S)$.
\emph{2nd statement: $\Rstr\Descext S$ are included in the restrictions of $\Rstr\Descext T$ to their intersections with $S$.} $\Rstr\Descext S$ are generalised descendants of $T$ and according to 1., also restricted descendants of $T$.
\emph{3rd statement: %The restrictions of $\Rstr\Descext T$ to their intersections with $S$ are restricted descendants of $S$.} Let be given a 
For a given
restricted 
descendant $D$ of $T$%. It has to be proven that 
, $\rstr(D,D\cap S)$ is a restricted descendant of $S$.%
}
According to 1., it suffices to find a \emph{generalised} descendant $Z$ of $S$ with this property. At first, \emph{descendants (without restriction)} $D$ of $T$ are considered. As induction assumption, assume that $\rstr(\pa D,\allowbreak {\pa (D)\cap S})$ is a generalised descendant $Z^*$ of $S$. Then 
%every case 
the cases 
in Lemma \ref{restricted children} 
%implies 
together imply
that $Z:=\rstr(D,D\cap Z^*)$ is a generalised descendant of $Z^*$ and hence of $S$, while $D\cap Z^*=D\cap\pa D\cap S=D\cap S$, so indeed $Z=\rstr(D,D\cap Z^*)=\rstr(D,D\cap S)$ as needed. Secondly, \emph{restrictions} of descendants $D^*$ of $T$ to subsimplices $U$ of $D^*$ are considered for $D$. This means that $\rstr(D^*,D^*\cap S)$ is further restricted to $U\cap S$, which 
%is 
yields a restricted descendant of $S$, as well as $\rstr(D^*,D^*\cap S)$ is.
%These $D$ are also restricted descendants of $S$, because $\rstr(D^*,D^*\cap S)$ is 
%%just 
%further restricted to $U\cap S$% simply%
%.
\end{proof}

\subsubsection{Identification of some T-arrays}
\index{identical T-array}
Equivalently to Stevenson \cite{Stevenson},
%We are going to identify some T-arrays with each other. 
%Some 
some
T-arrays are identified with each other. 
Firstly, note that a type 0 array has to be transposed if it shall be bisected further. Hence identify it with its transposed. Their descendants are the same anyhow. Write $S\stackrel{T}{\sim} S^T$.

Secondly, 
%for $T=\begin{pmatrix}p_0&\ldots& p_k\\
%&\vdots\\
%&p_m
%\end{pmatrix}$, consider $T_R=\begin{pmatrix}p_k&\ldots& p_0\\
%&\vdots\\
%&p_m
%\end{pmatrix}$, the \emph{reflected} T-array. Obviously the reflected children of $T$ are the children of $T_R$. On that score, let us 
identify $T$ with the reflected T-array $T_R$ (the T-array having the horizontal part of $T$ reflected) writing $T\stackrel{R}{\sim} T_R$.

Together, these identifications identify also
\begin{align*}
\begin{pmatrix}
p_0\\
\vdots\\
p_m
\end{pmatrix}
\stackrel{T}{\sim}
\begin{pmatrix}
p_0&\ldots& p_m
\end{pmatrix}
\stackrel{R}{\sim}
\begin{pmatrix}
p_m&\ldots& p_0
\end{pmatrix}
\stackrel{T}{\sim}
\begin{pmatrix}
p_m\\
\vdots\\
p_0
\end{pmatrix}.
\end{align*}
\begin{bem}
If $S\sim S'$, the nodes of the binary tree of $S$ (i.e.\ the descendants of $S$) are identical to the nodes of the binary tree of $S'$.
%This identification is common: The conventional notation for tagged simplices skips the type 0 and Stevenson \cite{Stevenson} already identified 
\end{bem}
\subsubsection{A weaker initial condition: restricted T-array coincidence.}
\begin{defn}[restricted T-array coincidence (ReTaCo)]\label{def:ReTaCo}
\index{restricted T-array coincidence}\index{ReTaCo -- restricted T-array \\co\-in\-ci\-dence}
The \emph{restricted T-arrays} of 
the initial triangulation 
$\T_0$ \emph{coincide} (\hypertarget{ReTaCo}{\emph{ReTaCo}}), if 
%$\T_0$ is regular and 
for each pair $S,T\in\T_0$ with $S\cap T\neq\emptyset$, the restrictions of the T-arrays of $S$ and $T$ to $S\cap T$ are identical.
\end{defn}
\begin{lem}\label{ReTaCoinheritance}
%\item\label{ReTaCoinheritance}
%Each admissible refinement of $\T_0$ fulfils ReTaCo. 
%ReTaCo for $\T_0$ implies ReTaCo for each admissible triangulation.
Assume ReTaCo for $\T_0$ and let $\{S,T\}\subset \Rstr\Simplexe$ be a regular pair of restrictions of admissible simplices from the extended binary forest to subsimplices of theirs. Then their restrictions to $S\cap T$ coincide.

Especially, every admissible triangulation satisfies ReTaCo.
\end{lem}
\begin{proof}
%[Proof of \ref{ReTaCoinheritance} 
%%by induction on $n$
%]
%The stronger statement is proven that \emph{for any common subsimplex $U$ of any admissible simplices $S$ and $T$ the restrictions of $S$ and $T$ to $U$ are identical (even if $\{S,T\}$ is not regular)} by mathematical induction on $lS+lT$. Base case: $lS=lT=0$, i.e.\ $S,T\in \T_0$, so $S\cap T$ is a common subsimplex of both. Restriction is transitive: The restriction of $S$ to $U$ is the restriction of [the restriction of $S$ to $S\cap T$] to $U$ and by ReTaCo this is identical with the restriction of $T$ to $U$. For the induction step, let $S,T\in\Simplexe$.
%
%\emph{1st case: $\Vnew S\notin \Vertices U$ or $\Vnew T\notin \Vertices U$.} W.l.o.g.\ say $\Vnew S\notin \Vertices U$. Then $\Vertices U\subset \Vertices S\setminus \Vnew S\subset \Vertices\pa S$, so $U$ is also a common subsimplex of $\pa S$ and $T$. By the induction assumption, the restrictions of $\pa S$ and $T$ to $U$ are identical. But the restriction of $S$ to $U$ is identical to the restriction of $\pa S$ to $U$.
%
%\emph{2nd case: }
%\emph{Base case:} If $n=1$, then all T-arrays are identical.
%By induction on $n{:=}\dim \T_0$, the following stronger statement is proven: \emph{Let $\{S,T\}\subset\Rstr\Simplexe$ be a regular pair of restrictions of admissible simplices to subsimplices of theirs. Then their restrictions to $S\cap T$ coincide.}
It is proven by induction on $n{:=}\dim \T_0$.

If $S=T$ is one $n$-simplex, the T-arrays coincide anyway. 
In the base case $n=1$, all T-arrays are identical.
%Let $\{S,T\}$ be a regular pair of two admissible subsimplices. 
Otherwise, let $S_0\supset\dots \supset S_M$ with $S\in\Rstr (S_M)$ and $T_0\supset\dots\supset T_N$ with $T\in\Rstr (T_N)$ be rooted paths %of $S$ and $T$ 
in the forest of admissible simplices $\Simplexe$ and let $j$ be the first index with 
$S_j\not\sim T_j$. 
%$\dim (S_j\cap T_j) < n$. (I.e.\ for all $i<j$, $S_i=T_i$ are $n$-simplices.)

If $j=0$, then $S_0\neq T_0$ and according to the initial condition (ReTaCo), the restrictions of $S_0$ and $T_0$ to $S_0\cap T_0$ are identical. Otherwise, $S_j$ and $T_j$ are the two distinct children of $S_{j-1}$ and their restrictions to $S_j\cap T_j$ coincide. (The restriction of the two children of 
$
\begin{pmatrix}
p_0	&\dots&p_k\\
	&\vdots\\
	&p_n
\end{pmatrix}
$
to their intersection is
$
\begin{pmatrix}
p_1	&\dots&p_{k-1}\\
	&\frac12{(p_0+p_k)}\\
	&p_{k+1}\\
	&\vdots\\
	&p_n
\end{pmatrix}.$)
Whether $j=0$ or not, $S_j$ and $T_j$ are two simplices sufficing ReTaCo with an intersection of lower dimension. 

Then Lemma \ref{lem:restricted descendants} says that 
%the restriction $\rstr(S,S\cap T)=\rstr(S,$ coincides with the restriction of $S':=\rstr(S,S\cap \aff (S_j\cap T_j))$ to $S\cap T$.
%$\rstr(S,S\cap T)$ is a restricted descendant $S'$ of $S_j':=\rstr(S_j,S_j\cap T_j)$ and $\rstr(T,S\cap T)$ is a restricted descendant $T'$ of $T_j':=\rstr(T_j,S_j\cap T_j)$. 
$\rstr(S,S\cap (S_j\cap T_j))$ and $\rstr(T,T\cap(S_j\cap T_j))$ is a restricted descendant 
%$S'$ and $T'$ 
of $S_j':=\rstr(S_j,S_j\cap T_j)$ and of $T_j':=\rstr(T_j,S_j\cap T_j)$, respectively.
$S_j'$ and $T'_j$ are identical T-arrays
%, hence suffice ReTaCo 
and have lower dimension than %$S$ and $T$. 
$n$.
So the induction assumption, applied for the lower-dimensional initial triangulation $\{S'_j\}$ instead of $\T_0$ and $\rstr(S,S\cap(S_j\cap T_j))$ and $\rstr(T,T\cap(S_j\cap T_j))$ instead of $S$ and $T$, deals with the rest of the proof.
\end{proof}
%TODO: Bemerkung
%Hence, all what is stated for $\T_0$ holds true also for every admissible refinement $\T$ (with an adjusted %level function
%statement for the levels).
\begin{folg}[ReTaCo $\Ra$ PC]\label{ReTaCo=>PC}\index{ReTaCo $\Ra$ PC}
If $\T_0$ satisfies ReTaCo, then $\Simplexe$ is pairwise compatible.
\end{folg}
Recall the \hyperlink{SIC}{strong initial conditions} (Definition \ref{SIC}
%from Section \ref{sec:IC}.
on page \pageref{SIC}).
In preparation to prove that SIC imply ReTaCo, it is shown that reference coordinates leave their mark also on subsimplices.
\begin{lem}\label{ReTa from RefCo}
Let $T$ be a reference simplex with reference coordinates $\vi$. The T-array of a restriction $S$ of $T$ to a subsimplex can be reconstructed by the reference coordinates $\vi(\Vertices S)$ of the vertices of that subsimplex up to identity.
\end{lem}
\begin{proof}
Formally, a T-array can be described as a partition of the vertices of a simplex into a non-empty horizontal part $H$ and a vertical part $V$ and a total order $<$ on each. Roughly speaking, a T-array is restricted by restricting both the partition and the order to the subset.
Let denote a total order by writing $p_0<\dots<p_k$. For an $m$-dimensional T-array $T$ of type $k\in\{1,\dots,m{-}1\}$, the equivalence class $\{T,T_R\}$ of $\stackrel{R}{\sim}$ can be described by the partition into the horizontal and the vertical part, the \emph{unordered pair} of the total orders $\{p_0<\dots<p_k,p_k<\dots<p_0\}$ on the horizontal part and the total order $p_{k+1}<\dots<p_m$ on the vertical part; while the equivalence class $\left\{T,T^T,T_R,(T_R)^T\right\}$ of both $\stackrel{R}{\sim}$ and $\stackrel{T}{\sim}$ of a full-type $m$-dimensional T-array $T$ can be described by the unordered pair of orders $\{p_0<\dots<p_m,p_m<\dots<p_0\}$ on the (unpartitioned) set of its vertices. The congruous equivalence class of a restriction of the T-array %to an $m$-subsimplex 
is obtained by restricting the partition and the orders to a subset of the vertices and, if the type of the restricted T-array is 0 or it has no horizontal vertices at all, converting $(\{p_0\},\{p_1<\dots<p_n\})$ or $(\{\},\{p_0<\dots<p_n\})$ into $\{p_0<\dots<p_n,p_n<\dots<p_0\}$.
%Such a pair of orders is 
%characterised by the set of all triples $\{\{p_i,p_k\},p_j\}$
The information of such a pair $\{p_0<\dots<p_n,p_n<\dots<p_0\}$ of orders is obtained, if for any triple of vertices $a,b,c$, it can be decided, which one is in both orders \emph{between} the other two, i.e.\ whether $\{a<b<c,c<b<a\}$ or $\{b<a<c,c<a<b\}$ or $\{a<c<b,b<c<a\}$ are the restrictions of the orders to $\{a,b,c\}$. 
%The first case is denoted here by $abc$ (so $abc\LR cba$).

Now the task is to recover the restricted T-array $S$ from the reference coordinates of the vertices of a subsimplex of a tagged simplex $T$. To recover the \emph{type}, recall Lemma \ref{edgelengths}. It implies that 
%two distinct Chebyshev distances occur between vertices:
%horizontal vertices have Chebyshev distance $2^{-h}$ from 
the Chebyshev distance between two horizontal vertices is twice the Chebyshev distance of the other pairs of vertices. If two or more horizontal vertices and at least one vertical vertex are contained in the subsimplex (i.e.\ if the type of $S$ is neither 0 nor full), horizontal and vertical vertices can be distinguished that way. If all vertices have the same Chebyshev distance from each other, the type is 0 or full. So the type is determined up to the equivalence of type 0 and full type.

The pair of orders on the horizontal part or the equivalence class of a 0/full-type simplex is recovered as follows. The horizontal part is a $2^{-h}$-scaled Kuhn simplex, i.e.\ $p_j-p_{j-1}=\epsilon_je_{\sigma(j)}$ for some signs $\epsilon_j\in\{1,-1\}$ and a permutation $\sigma$% of $\{1,\dots,n\}$
. Hence $\|p_j-p_i\|_1=\left\|\sum_{a=i+1}^j \epsilon_ae_{\sigma(a)}\right\|_1=(j{-}i) 2^{-h}$ for $j>i$. This shows that among three distinct horizontal vertices 
%$p_i,p_j,p_k$ there is always a unique one, for example $p_j$ with $i<j<k$, and hence $\|p_k-p_j\|_1+\|p_j-p_i\|_1=\left(|k-j|+|j-i|\right)2^{-h}=(k-i)2^{-h}=\|p_k-p_i\|_1$ and this one must be between the other two in the (restricted) order.
$p,q,r$ there is always a unique one, for example $q$ with $\|p-q\|_1+\|q-r\|_1=\|p-r\|_1$ and this one must be between the other two in the (restricted) order.

%For a T-array of type $0$ or of full type, the pair of orders is already found with it. 
For a type in $\{1,\dots,m{-}1\}$, the order of the vertical part is wanted yet. As usual, this can be reduced to the task on the horizontal part: Remove all horizontal vertices but one vertex $p_h$ and transpose the T-array into a full-type one. %According to Lemma \ref{lem:restricted descendants}, the resulting T-array 
%has full type and 
%is the restriction of some descendant of $T$. Hence the 
The
pair of 
%the order 
orders
of the vertices of the resulting T-array
%and the converse order 
can be obtained by the procedure described above. One of these two orders is led by the retained horizontal vertex $p_h$. This is the correct order on the vertical part.
\end{proof}
\begin{folg}[SIC $\Ra$ ReTaCo]\label{SIC=>ReTaCo}\index{SIC $\Ra$ ReTaCo}
The strong initial conditions %SIC
for $\T_0$ imply restricted T-array coincidence for every admissible triangulation.
\end{folg}
\begin{proof}
Let $S$, $T$ be two simplices from a triangulation sufficing SIC. %Let $\vi_S$ and $\vi_T$ be reference coordinates for $S$ and $T$ respectively, coinciding on $S\cap T$. 
%Since $\vi_S|_{S\cap T}$ determines
%Since the 
The
restrictions of suitable reference coordinates to $S\cap T$ coincide and determine
the restriction of $S$ and $T$ to $S\cap T$ up to identical T-arrays according to Lemma \ref{ReTa from RefCo}. Hence, the restrictions of $S$ and $T$ to $S\cap T$ coincide up to identification as well. Lemma \ref{ReTaCoinheritance} transfers ReTaCo to each admissible refinement.
\end{proof}
%%%%%%%%%%%%%%%%%%%%%%%%%%%%%%%%%%%%%%%
\subsection{Quasi-uniform refinements for ReTaCo}\label{sec:quasi-uniform refinement}
The following theorem generalises both the GSS theorem
(Theorem \ref{GSS}) and the general refineability of \cite[Definition 2.14]{Gaspoz}.
\begin{thm}[Quasi-uniform refinements for ReTaCo]
\label{quasi-uniform}
\index{quasi-uniform refinement!ReTaCo}
\index{GSS theorem!generalisation}
Suppose ReTaCo for $\T_0$. Then the following statements hold true:
\begin{enumerate}
\item\label{quasi-uniform1}
There is an admissible %regular 
refinement $\T$ of $\T_0$, where all simplices have levels in $\{n,\dots,2n-1\}$ (with $n=\dim \T_0$). 
If type $n$ is excluded, i.e.\ all type $n$ simplices are transposed, all T-arrays of $\T$ have hyperlevel 1.
Especially, $\Simplexe$ is generally refineable.
\item\label{quasi-uniform_n}
There are even quasi-uniform refinements: For all $k\in\N$, there exists a refinement $\T$ of $\T_0$ with
\begin{align*}
l(\T)=\{l(S)~|~S\in \T\}\subset \{kn,\dots,kn+n{-}1\}.
\end{align*}
If the type $n$ is excluded, $h(\T)=\{k\}$.
\end{enumerate}
\end{thm}
\begin{proof}[Proof of \ref{quasi-uniform1}]
%\emph{Proof of \ref{quasi-uniform1}.} 
%Let us investigate, 
Investigate,
in which order the edges are bisected. Recall
%the children of a T-array:
%\begin{align*}
%\begin{pmatrix}
%p_0&\dots& p_{k}\\
%&p_{k+1}\\
%&\vdots\\
%&p_n
%\end{pmatrix}
%\mapsto
%\begin{pmatrix}
%p_1&\dots& p_{k}\\
%&\frac12 (p_0+p_{k})\\
%&p_{k+1}\\
%&\vdots\\
%&p_n
%\end{pmatrix},
%\begin{pmatrix}
%p_0&\dots& p_{k-1}\\
%&\frac12 (p_0+p_{k})\\
%&p_{k+1}\\
%&\vdots\\
%&p_n
%\end{pmatrix}.
%\end{align*}
%Thus, 
that
after several refinements but before transposition, the horizontal part is a centrepiece of the original horizontal part and the vertical part consists of centres of horizontal--horizontal edges, succeeded by the 
%original 
initial vertical part. %Let us denote 
Denote
such a descendant by
\begin{align*}
\begin{pmatrix}
H&\dots& H\\
&HH\\
&\vdots\\
&HH\\
&V\\
&\vdots\\
&V
\end{pmatrix}.
\end{align*}
(Equal letters can mean different points. $H$ stands for an original horizontal vertex, $V$ for an original vertical one and $AB$ for the centre between $A$ and $B$.) 
Consequently, the
%The 
type 
%0 
$n$
descendant after the first transposition will look like
%$\begin{pmatrix}
%H\\
%HH\\
%\vdots\\
%HH\\
%V\\
%\vdots\\
%V
%\end{pmatrix}$
%and is transposed to 
\begin{align*}
\begin{pmatrix}
H &HH &\ldots &HH &V &\ldots&V
\end{pmatrix}.
\end{align*}
%We are going to bisect 
\emph{
%at 
%At
%least 
In those descendants,
every original edge}
should be bisected
exactly once, 
to ensure that $\T$ is \emph{strictly finer than every simplex of $\T_0$}. If there are several $V$s, say 
$%\begin{pmatrix}
(
\ldots %&
HH %&
\ldots %&
$
$V_1 %&
\ldots %&
V_2 %&
\ldots
%\end{pmatrix}
)
$,
%we have to bisect the edge $(HH,V_2)$ before we can bisect $(V_1,V_2)$. For that reason (and for simplicity), let us just bisect \emph{every} edge $(HH,V)$, i.e.\ bisect, until the horizontal part of the descendant looks either like 
the edge $(HH,V_2)$ has to be bisected before $(V_1,V_2)$. For that reason (and for simplicity), just bisect \emph{every} arising edge $(HH,V)$, i.e.\ bisect, until the horizontal part of the descendant looks either like 
$\begin{pmatrix}
H &HH &\ldots &HH
\end{pmatrix}$
or like $(V)$.

Note that until that state, 
%we bisect
bisected are
\begin{itemize}
\item
every edge $(H,H)$ \emph{before} transposition,
\item
then edges $(H,V)$,
\item
then edges $(HH,V)$,
\item
and finally edges $(V,V)$.
\end{itemize}
All of these edges (distributed to many descendants) are bisected and only these.

This would characterise the edges to 
%be bisected% 
bisect%
%without further explanation
, if the classification of vertices into horizontal and vertical ones were \emph{global}, i.e.\ independent of the simplex. But 
it is not, 
%this is not the case,
because %we identify 
$\begin{pmatrix}
p_0\\
\vdots\\
p_m
\end{pmatrix}$ and
$\begin{pmatrix}
p_0&
\ldots&
p_m
\end{pmatrix}$
have been identified. 
%(And this cannot be fixed by excluding type 0, because it must be possible that $\begin{pmatrix}
%a\quad b\\
%c\\
%d
%\end{pmatrix}$ touches $\begin{pmatrix}
%a&c&d\\
%&e
%\end{pmatrix}$.) 
Nevertheless a global characterisation of the \emph{edges which are 
%be bisected
bisected} is possible: %Note that the 
The 
edge $(H_1H_2,V)$ lies in the interior of a 2-subsimplex (a triangle) of some $S\in\T_0$ and the restricted T-array of this triangle must be either $\begin{pmatrix}
H_1\quad H_2\\
V
\end{pmatrix}$ or $\begin{pmatrix}
H_2\quad H_1\\
V
\end{pmatrix}$ (this T-array cannot be transposed), which is a global characterisation. Hence the edges which are bisected are all such edges $(H_1H_2,V)$, all the edges of $\T_0$ and no other ones.
%In sum, an edge is bisected if and only if it is an edge of $\T_0$, or it is $(\frac{a+c}2,b)$ for a triangular subsimplex with restricted T-array identical to $\begin{pmatrix}
%a\quad b\\
%c
%\end{pmatrix}$.

Every simplex is bisected at least $n$ times (because every edge must be bisected and every level $n{-}1$ descendant of $S$ still contains at least $2$ vertices of $S$). The largest number of successive bisections happens to a type $n{-}1$ simplex 
$\begin{pmatrix}
H&
\hdots&H\\
&V
\end{pmatrix}$. It is bisected into level $n{-}1$ and type 0 simplices of the form $\begin{pmatrix}
H\\
HH\\
\vdots\\
HH\\
V
\end{pmatrix},$
which are transposed and bisected further, finally into type 0 arrays $\begin{pmatrix}
V\\
V(HH)\\
\vdots\\
V(HH)\\
VH
\end{pmatrix}$ (among others of larger types). This descendant has level $n{-}1+n=2n-1$.

General refineability: The so defined $\T$ is a refinement of $\T_0$ strictly finer than each simplex $S$ of $\T_0$, so it is finer than $\refine(\T_0,S)$, hence the latter must be finite. Since ReTaCo transfers to every admissible triangulation $\T$, the same holds for $\T$ instead of $\T_0$.

Using only 
%simplices of the normal binary forest of admissible T-arrays 
T-arrays of type ${<}n$
for the just described refinement,
%(i.e.\ type $n$ does not occur). 
%Suppose that $\T_0$ does not contain any type $n$ simplex. If it does, transpose it “back” into a type 0 simplex. (This will not hurt ReTaCo.) Further during refinement, transpose only if needed (i.e.\ only if a type 0 simplex has to be bisected further), such that type $n$ does not occur, because it is bisected at once. 
note carefully that every %bisection sequence (i.e.\ a 
path in the extended binary forest
%) 
%$\underbrace{S_0}_{\in \T_0}\ra\dots\ra \underbrace{S}_{\in \T}$
$\T_0\ni S_0\supset\dots\supset S\in \T$ 
contains exactly one transposition, i.e.\ $h(S)=1$.

\emph{Proof of \ref{quasi-uniform_n}.} 
%Consequently, applying
Applying
the above defined refinement $k$ times successively on $\T_0$ without type $n$ T-arrays leads to a triangulation of hyperlevel $k$ simplices solely. 
%The level increases at least by $kn$.
Since type $n$ is excluded, the hyperlevel must increase, whenever the level increases by $n$. Conversely, the level of a hyperlevel $k$ simplex is at most $kn+n{-}1$. 

On the other hand, %simplices of of level $\geq C\cdot n$, because 
the above defined refinement bisects every simplex %is bisected 
at least $n$ times. 
%(as %we 
%explained above in the proof of the first statement). 
Consequently, $k$ successive applications of this refinement produce simplices of levels $\geq k\cdot n$.
\end{proof}

%%%%%%%%%%%%%%%%%%%%%%%%%%%%%%%%%%%%
\newpage
%\section{Binev--Dahmen--DeVore for $\|\bullet\|_\infty$-isometric coordinate changes}
\section{Binev--Dahmen--DeVore for the AGK initialisation.}
\label{sec:IsoCoChange}
\subsection{Introduction}
There is a 3-dimensional regular triangulation $\T_0$ of simplices which cannot be equipped with T-arrays satisfying SIC (see \cite[Lemma 1.7.14, p.~26]{Schoen}). To provide an alternative to initial division (as it was presented in Section \ref{sec:initial division}), Karkulik, Pavlicek and Praetorius proved the BDV theorem for a 2-dimensional regular initial triangulation with arbitrary T-arrays without any further initial condition \cite{Karkulik,KarkulikE}. Alkämper, Gaspoz and Klöfkorn \cite{Gaspoz} proposed a weaker initial condition which still guarantees general refineability and a simple algorithm 
%that arranges the vertices of each simplex into T-arrays
that assigns tags (i.e.\ T-arrays) to every simplex of an untagged triangulation fulfilling that initial condition. The aim of this section is to prove the BDV theorem for an arbitrary output 
%of a small modification 
of that algorithm, namely Algorithm \ref{AGK algorithm}. 
\paragraph{Specialties of this chapter.}
It should be emphasised that some definitions, namely Definition \ref{restricted T-array!IsoCoChange} of restricted T-arrays, Definition \ref{Wk iota!IsoCoChange} of the layer $W^k$ of a set $W\subset \Simplexe$, the embedding $\iota$ of $W$ into $\Omega\times \N$, Definition \ref{quasi-uniform} of a quasi-uniform refinement and Definition \ref{def:refsimplex} of a reference simplex are varied in this chapter. Definition \ref{def:refsimplex} is compatible with the old definition. The main difference to the previous chapters is that instead of considering reference coordinates, each simplex is embedded into its \emph{Chebyshev lattice} (see 
%Definitions \ref{def:refsimplex} and \ref{lattice of a T-array}
Subsubsection \ref{Chebyshev lattice} below). In this chapter, the hyperlevel $hS$ is part of the structure of a T-array $S$, it is not only the place of the T-array in the forest of admissible simplices. 
Consequently, the hyperlevel of an initial simplex can vary from 0 and restricted T-arrays need a hyperlevel, too. 
%Roughly speaking, the hyperlevel %$h$ 
%takes over the role of the level function $l$.

%This reads as follows:
\begin{algorithm}[h!]
\caption{AGK algorithm}\label{AGK algorithm}
\begin{algorithmic}[1]
\State %Let $\mathcal V=\Vertices \T_0$ be the vertices of the triangulation $\T_0$. 
Start with a regular triangulation (without T-arrays) $\T_0$ of $\Omega$. Divide $\mathcal V 
\T_0
$ into disjoint subsets $\mathcal V_i\subset \mathcal V
\T_0, i\in \{0,1\},\mathcal V_0\cup \mathcal V_1=\mathcal V
\T_0$.
%, such that every simplex has at least one vertex in $\Vertices_0$. Provide an ordering $<_i$ for each of them.
\allowbreak
\State 
Every simplex $S\in \T_0$ with $\Vertices S\cap \Vertices_0\neq \leer$ gets the hyperlevel 0 and the T-array 
$\begin{pmatrix}
p_0 &\dots&p_k\\
&\vdots\\
&p_n
\end{pmatrix}$,
where the horizontal vertices $\{p_0,\dots,p_k\}$ are $\Vertices S\cap\Vertices_0$, the vertical vertices $\{p_{k+1},\dots,p_n\}$ are $\Vertices S\cap\Vertices_1$, $p_0<_0\dots<_0 p_k$ and $p_{k+1}<_1\dots<_1 p_n$.
\State
If $S\cap \mathcal V_0 = \leer$, $S$ gets the hyperlevel $hS:=1$ and the T-array 
$\begin{pmatrix}
p_0&\dots&p_n
\end{pmatrix}$ where $p_0<_1\dots<_1 p_n$. 
%and all vertices are sorted ascendingly from $z_0$ to $z_n$ with respect to $>_1$.
\end{algorithmic}
\end{algorithm}
%\begin{bem}
%The difference to the original AGK algorithm is the condition $S\cap \Vertices_0\neq 0$.
%Instead of this extra condition, Alkämper, Gaspoz and Klöfkorn assign a full T-array in the case $S\cap\Vertices_0=\leer$. The extra condition is not really necessary for BDV, but it fits in the theory here better. The best constants in the BDV theorem are obtained for $\Vertices_0=\Vertices\T_0$ anyway.
%\end{bem}
\begin{thm}[BDV theorem for 
%ReTaHyCo
AGK initialisation]\label{BDV4IsoCoChange}
\index{BDV theorem!IsoCoChange}
%In the forest language, the (equivalent) statement reads 
%Suppose ReTaHyCo for the initial triangulation $\T_0$ of the $n$-dimensional polytope $\Omega$ 
Let the tagged triangulation
%a regular triangulation of the $n$-dimensional polytope $\Omega$ be the input of the AGK algorithm, the tagged triangulation 
$\T_0$ be 
%its
the 
output
of the AGK algorithm
and 
\[
D:=\max
\limits_{\substack{S\in\Simplexe
%\\tS=0
}} 
%\left\{
2^{hS}\diam S,
%~\middle|~S\in\Simplexe,tS=0\right\},
\quad d:=\min_{S\in\Simplexe}
%2^{n\cdot hS-tS}
2^{n\cdot hS+n{-}tS}
\lvert S\rvert\index{d@$d$!IsoCoChange}\index{D@$D$!IsoCoChange}.
\]
Then there is a constant 
%$C(n)$,
$C>0$, 
depending only on 
%%the dimension $n$ of the triangulation
$n$ and linearly on ${D^n}/{d}$, 
%the initial triangulation,
such that for every refinement sequence of triangulations $\T_0,\T_1,\dots,\T_N$, it holds 
that
\begin{align}
\#\T_N- \#\T_0\leq {\underbrace{\frac12\# \left\{S\in \Simplexe\setminus\T_0~\middle|~hS\leq 2+\log_2(n)\right\}}_{\leq (
%4
8
n)^{n}\#\T_0}}+
{C%(n)\frac {D^n}{d}
\cdot
 N}.\label{eq:BDV4IsoCoCh}
%%Falsch
%\index{C(n)@$C(n)$ -- BDV constant!IsoCoChange}\index{BDV constant!IsoCoChange}
%\index{C@$C$ -- BDV constant!IsoCoChange}
%\index{BDV constant!IsoCoChange}
\end{align}
%For an initial triangulation of solely type $n$ simplices, 
If 
%$\Vertices_0=\Vertices\T_0$, 
$\Vertices_0=\leer$ (or, equivalently, if $\Vertices_1=\leer$),
the first summand in \eqref{eq:BDV4IsoCoCh} can be 
%reduced to $\# \left\{S\in \Simplexe~
%%\middle
%\right|~\allowbreak hS\leq 1+\allowbreak\log_2(n)%\right
%\}$.
estimated by
$(4n)^n\#\T_0$.
\end{thm}
\begin{bem}
The well-definedness of $d$ and a number similar to $D$ has been the subject of Lemma \ref{lem:dD} on page \pageref{lem:dD}.
\end{bem}
\paragraph{Schedule of the proof.}
Before setting up the theory in detail in the following subsection, a schedule for the proof may provide an overview. 
It is shown that the output of the AGK algorithm satisfies the initial condition \emph{\hyperlink{ReTaHyCo}{ReTaHyCo} (restricted T-arrays and hyperlevels coincide)}.  The premise that $\T_0$ is an output of the AGK algorithm in the BDV theorem can be replaced by ReTaHyCo. ReTaHyCo is further transformed into \emph{\hyperlink{IsoCoChange}{IsoCoChange} (isometric coordinate changes)}, a geometrical condition introduced in Subsubsection \ref{subsec:IsoCoChange} below. 
%and the proof makes full use of IsoCoChange.
The notions of IsoCoChange are the material the proof is made of.

\begin{defn}[restricted T-array, hyperlevel of restricted T-array]
\index{restricted T-array!IsoCoChange}\index{restriction!IsoCoChange}\index{hyperlevel!restricted T-array}
\label{restricted T-array!IsoCoChange}
Let $S$ be a T-array, $U$ a non-empty subsimplex of $S$. The \emph{restriction $\rstr(S,U)$ of $S$ to $U$} is made by wiping out all the vertices from $S$, which are not vertices of $U$, pushing the remaining vertices together \emph{and transposing $U$, if $U$ does not contain any horizontal vertex of $S$}.
The \emph{hyperlevel $h(\rstr(S,U))$ of that restriction }equals the hyperlevel of $S$, if $U$ contains a horizontal vertex of $S$ and is incremented in the case of that latter transposition.
\end{defn}
\begin{defn}
%[Isometric Coordinate Changes I]
%[restricted T-array coincidence with equal hyperlevels \hypertarget{ReTaHyCo}{(ReTaHyCo)}]
[restricted T-array and hyperlevel coincidence \hypertarget{ReTaHyCo}{(ReTaHyCo)}]\label{def:ReTaHyCo}
\index{ReTaHyCo}
%The (regular) initial triangulation $\T_0$ satisfies
%\emph{isometric coordinate changes} (IsoCoChange)
\emph{Restricted T-array and hyperlevel coincidence with equal hyperlevels (ReTaHyCo)}
%if:
are the following conditions.
\begin{itemize}
\item
%It 
$\T_0$
is conforming.
\item
The hyperlevels of simplices of $\T_0$ are 0 or 1. Hyperlevel 1 simplices have type $n$.
%Any common vertex of an arbitrary pair $S,T\in \T_0$ of initial T-arrays is either horizontal in both T-arrays or vertical in both.
\item
A hyperlevel $hp$ can be assigned to each vertex $p\in\Vertices\T_0$ independently of the simplex such that for every T-array $S$
\begin{itemize}
\item
$p$ is horizontal in $S$ and $hp=hS$ or
\item
$p$ is vertical in $S$ and $hp=hS{+}1$.
\end{itemize}
\item
For each pair $S,T\in \T_0$ of initial 
%simplices
T-arrays (touching in at least one vertex), the restrictions of 
%(the T-arrays) 
$S$ to $S\cap T$ and of $T$ to $S\cap T$ coincide up to reflexion.
\end{itemize}
\end{defn}
%\emph{Why is that initial condition realised by the Algorithm 3.1 in \cite{Gaspoz}?}
%\paragraph{How is that initial condition realised by the Algorithm 3.1 in \cite%[Algorithm 3.1]
%{Gaspoz}?}
ReTaHyCo is assumed for any initial triangulation $\T_0$ in this section, 
except in the Propositions \ref{pps:klebt} and \ref{pps:imPatch}.
\begin{lem}
Every output of the AGK algorithm satisfies ReTaHyCo.
\end{lem}
\begin{proof}
%The property of a vertex to be horizontal or vertical is global as required in ReTaHyCo.
If the vertices in $\Vertices_j$ get hyperlevel $j$ (for $j\in\{0,1\}$), the first three conditions are fulfilled.

%The orderings $<_i$ restricted to the horizontal and the vertical part give the orders of the vertices in the horizontal and the vertical part, respectively.

To restrict a T-array to a subsimplex means to restrict the orders to subsets of the horizontal and the vertical part, and to transpose the T-array, if all of its vertices are vertical. Since the orderings are global, they coincide on the restrictions to the intersection of two T-arrays.
%Additionally, the property of a subsimplex to contain only vertical vertices, too, is global.
% one restriction of a T-array to a subsimplex contains only vertical vertices the same holds for the restriction of any other T-array to the same subsimplex.
If a subsimplex of a simplex contains only vertices of $V_1$, the restriction of any T-array to it is a full-type hyperlevel 1 T-array.
\end{proof}

If all initial vertices belong to $\mathcal V_1$, the initial triangulation consists of type $n$ simplices only.

To anticipate the idea of \hyperlink{IsoCoChange}{IsoCoChange}, it \emph{could} be defined as follows: Reference coordinates $\vi: \R^n\ra \R^n$ are called to \emph{preserve the lattice}, if the horizontal vertices of $\vi(S)$ lie in $2^{-hS}\Z^n$. (It follows that $\vi$ maps the vertical vertices to $2^{-hS-1}\Z^n$.) IsoCoChange means that two arbitrary simplices $S,T\in\T_0$ with common restriction $U$ to their intersection $S\cap T$ and reference coordinates $\vi_S,\vi_T$ for them preserving the lattice 
%two conditions must be fulfilled: 
fulfil two conditions:
\begin{itemize}
\item
The pull backs of the lattice points coincide on their intersection:
\[\vi_S\inv(2^{-hU}\Z^n)\cap\aff(S\cap T)=\vi_T\inv(2^{-hU}\Z^n)\cap\aff(S\cap T).\]
\item
Pull back the metric induced by the Chebyshev norm $\|\bullet\|_\infty$ by 
$\vi_S$ and $\vi_T$, respectively. 
%the reference coordinates 
%and restrict it to $S\cap T$.
Then these metrics coincide 
on $S\cap T$.
%for $S$ and $T$.
\end{itemize}
The 
most 
important 
%milestone 
stepping stone
%crucial result
to prove Theorem \ref{BDV4IsoCoChange}
is Theorem \ref{milestone} and
was inspired by Karkulik, Pavlicek and Praetorius \cite{Karkulik}. They proved the BDV theorem for an arbitrary regular 2D mesh $\T_0$ with an arbitrary initial configuration of refinement edges, 
which is equivalent to the 2D case of IsoCoChange. 
%This is by the way equivalent to the 2D case of IsoCoChange, because on such a mesh T-arrays of type $2$ can be imposed on each triangle sufficing IsoCoChange easily. 
%(On the other hand, refinement edges even sufficing the compatibility condition (i.e.\ SIC) can be imposed in linear time.)
The definitions of the tower layer $W^k$ 
and 
the embedding $\iota$ 
%and the quasi-uniform refinement 
are changed for it.
\begin{defn}[$W^k$, $\iota$]
\index{W^k@$W^k$ -- $k$th layer of $W$!IsoCoChange}\index{i@$\iota$ -- embedding of $\Simplexe$!IsoCoChange}
\label{Wk iota!IsoCoChange}
%For a set (e.g.\ a forest) of admissible simplices $W\subset \Simplexe$, $W^k$ denotes not longer these of \emph{level} $k$ but these of \emph{hyper}level $k$. \emph{and type $\leq n{-}1$}. I.e.\ simplices of hyperlevel $k$ and type $n$ are transposed into those of hyperlevel $k-1$ and type $0$ and then assigned to $W^{k-1}$. 
For a subset 
%(e.g.\ a subforest) of the normal forest (i.e.\ simplices of type $n$ are transposed) 
of the admissible simplices $W\subset \Simplexe$, $W^k$ no longer denotes the simplices of \emph{level} $k$ but instead the simplices of \emph{hyper}level $k$ \emph{and type $\leq n{-}1$} 
(i.e.\ simplices of type $n$ are transposed).
$W^{j,k}$ denotes the set of type $j$ hyperlevel $k$ simplices in $W$.
$\Simplexe^{n,j}$ denotes the set of all admissible type $n$ hyperlevel $j$ simplices.
%Furthermore, the embedding $\iota: S\mapsto S\times \{lS\}$ is altered to the map $\iota: S\mapsto S\times\{hS\}%\times\{tS\}
%\subset \Omega\times\N%\times\{0,\dots,n{-}1\}
%$.
Furthermore, $\iota$ no longer denotes the map $S\mapsto S\times \{lS\}$ but instead $S\mapsto S\times\{hS\}%\times\{tS\}
\subset \Omega\times\N%\times\{0,\dots,n{-}1\}
$.
\end{defn}
%\begin{defn}[quasi-uniform refinements]\label{quasi-uniform}\index{quasi-uniform refinement!IsoCoChange}\index{Q@$\Q$ -- quasi-uniform refinement}
%The quasi-uniform refinement 
%%$\Q^j\T_0$ of $\T_0$ 
%$\Simplexe^{n,j}$
%is the set of type $n$ hyperlevel $j$ simplices.
%\end{defn}
Recall the notion of a \hyperlink{tower}{tower} which is not changed here.

%That milestone reads:
%It is stated as follows.
%%replaced by \milestone1
\begin{thm}\label{milestone}
%%Let $S\in\Simplexe$ with $tS\leq n{-}1$ and $2^{hS}>4n$. Let $S_0$ be the ancestor of $S$ in $\T_0$. Then there is a vertex $p\in\Vertices S_0$ such that each are $T\in\Simplexe$ with $tT\leq n{-}1,2^{hS}>4n$ is included in the patch $\T_0 p$.
%Suppose 
%%IsoCoChange 
%ReTaHyCo
%for $\T_0$. Then
%\begin{enumerate}
%\item\label{it:quasi-uniform}
%\quasiuniform{1}
%\item\label{it:towerstep}\index{h_0@$h_0$}
%\towerstep1
%\end{enumerate}
\milestone1
\end{thm}
Theorem \ref{milestone}.\ref{it:quasi-uniform} was already proven in \cite{Gaspoz} for every output of 
%their
the AGK 
algorithm. Using Theorem \ref{milestone}, the proof of the BDV theorem is just another version of the pile game, where the hyperlevel takes over the role of the level.
Two propositions subdivide the proof of Theorem \ref{milestone}. 
To understand how they interlock does not require understanding of 
%It can be understood how they interlock without understanding 
the 
%following 
notion of a \emph{$\frac1N$-closed set}
which is properly defined in Subsection \ref{sec:e->f}.
It can be seen as a placeholder for the time being. To give a possibility to imagine Proposition \ref{pps:klebt} anyhow, %may be imagined with the following 
assume that $\T_0$ is a 
%single reference simplex 
set of reference simplices
with horizontal vertices in $\Z^n$. Then a $\frac1N$-closed set is a union of subcubes of the cubes $\frac1N(a+[0,1]^n)$ (with $N\in\N_{\geq 1}$ and $a\in\Z^n$).
% intersected with $S$.
%supposing some definitions. %From here on it is important to think of the Chebyshev lattice (of a simplex) as filling the whole space (and in the case of subsimplices its whole affine hull, respectively), even though far simplices do not have anything to do with this lattice.

%\begin{defn}[$\frac1N$-closed]

%\begin{defn}[$\ar\Edges$]
%For edges $e,f\in\Edges\Simplexe$, let $\ar\Edges$ be the relation
%\begin{align*}
%\ar\Edges\quad:\LR\quad \ex S,T\in\Simplexe.S\ar1 T,e=\Eref\pa S,f=\Eref\pa T.
%\end{align*}
%(i.e.\ bisection of $e$ demands bisection of $f$.)
%\end{defn}
\begin{defn}[hyperlevel of an edge]\label{def:hyperlevel of an edge}
\index{hyperlevel of an edge}
The \emph{hyperlevel of an edge }is defined by $h\Eref S:=hS$. Here, simplices of type 0 are excluded. They are considered as not having a refinement edge, but have to be transposed to a simplex of incremented hyperlevel.

%%verschoben
%This implies $2^{-he}=|e|_{\|\bullet\|_S}$ and hence, supposing IsoCoChange, the uniqueness of this definition (i.e.\ the independence of the simplex $S$).
\end{defn}
It is shown below in Lemma \ref{hyperlevel of an edge is unique} that this definition is well-defined.

Recall the Definition \ref{ar} of $\ar\Edges$ and Theorem \ref{Knotenschritt}.\ref{it:Eref} on pages \pageref{ar} and \pageref{it:Eref}.
\begin{pps}[$f\ar\Edges e$ sticks on $2^{-k}$-closed sets]\label{pps:klebt}
%If 
Assume that $\T_0$ satisfies ReTaHyCo and consists only of hyperlevel 0 T-arrays. 
Let $f\ar\Edges e$ for admissible edges $e$ and $f$ 
of hyperlevel $\geq 1$, 
%then
%\begin{enumerate}
%\item\label{item:nonincreasing}
%$hf\leq he$
%\item
%%was hat das label hier gemacht?
%\label{item:klebt}
%Let 
$k\leq he{-}1$ be an integer and $U\subset \Omega$ be $2^{-k}$-closed in $\T_0$. 
%Then
%\begin{align*}
%\Vertices f\cap U\neq \leer\quad\Ra\quad \Vertices e\cap U\neq\leer.
%\end{align*}
%\end{enumerate}
Then the following implication holds true: If one vertex of $f$ lies in $U$, then also one vertex of $e$ does.
\end{pps}
To illustrate the statement: %A sequence $e_1\ar\Edges\dots$ “sticks” on the $2^{-k}$-closed set $U$ like a bar magnet, see Figure \ref{fig:klebt}.
If $e_1 \ar\Edges\cdots\ar\Edges e_N$ and $e_1$ touches $U$, then all $e_i$ also do, see Figure \ref{fig:klebt}. The sequence cannot leave $U$.

\newsavebox{\typdrei}% Box to store matrix content
\savebox{\typdrei}{$\begin{pmatrix} p_0&p_1&p_2&p_3 \end{pmatrix}$}
\newsavebox{\typzwei}% Box to store matrix content
\savebox{\typzwei}{$\begin{pmatrix} p_0&p_1&p_2\\&p_4 \end{pmatrix}$}
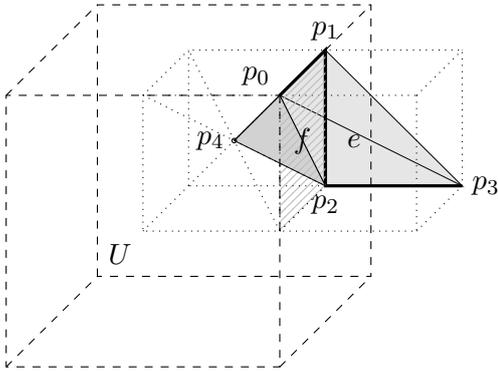
\begin{SCfigure}[10][ht]
\centering{
\begin{tikzpicture}[information text/.style={fill=gray!10,inner sep=1ex}, 
					cube1/.style=dashed, 
					simplex/.style={very thin},
					diag/.style=dotted,
					canonical1/.style={very thick},
					canonical2/.style=double,
					scale=.6
					]
\def\ys{-1};
\def\cube{
\draw[diag] (3,3) rectangle (0,0);
\draw[diag] (1,1) rectangle (4,4);
\draw[diag] (0,0) -- (1,1); 
\draw[diag] (3,3) -- (4,4); 
\draw[diag] (3,0) -- (4,1); 
\draw[diag] (0,3) -- (1,4);
}

\def\largecube{
\draw[cube1] (6,3) rectangle (0,-3);
\draw[cube1] (2,-1) rectangle (8,5);
\draw[cube1] (0,-3) -- (2,-1); 
\draw[cube1] (6,3) -- (8,5); 
\draw[cube1] (6,-3) -- (8,-1); 
\draw[cube1] (0,3) -- (2,5);
}

\begin{scope}[xshift=0cm]
	\coordinate (p0) at (0,3);
	\coordinate (p1) at (1,4);
	\coordinate (p2) at (1,1);
	\coordinate (p3) at (4,1);

	\fill[gray!20] (p0)--(p2)--(p3)--(p1)--cycle;
	\cube;
	\draw[canonical1, %dashed
	]  
	(0,3) -- (1,4)--(1,1) -- (4,1); 
	
	\draw[style=simplex] 
	(1,4) node [above]{$p_1$}-- (4,1)node [right]{$p_3$} --node[left]{$e$} (0,3)node [above left]{$p_0$} -- (1,1)node [below]{$p_2$};
	\coordinate (p3) at (-1,2);
	
	\fill[fill=gray!35] (p0)--(p1)--(p3)--(p2)--cycle;
	\fill[pattern=north east lines,pattern color=gray!50] (p0)--(p1)--(p2)--(0,0)--cycle;

	\draw[style=simplex] 
	(p0) %node [left]{$p_0$}
	--node{$f$} 
	(p2) %node [below right]{$p_2$}
	;
	\draw[simplex] (p0) -- (p3) node [left]{$p_4$};
	\draw[simplex] (p1) %node [above]{$p_1$} 
	-- (p3);
	\draw[simplex] (p2) -- (p3);
	
	\draw (p3) circle [radius=.05];
\end{scope}
\begin{scope}[xshift=-3cm]
	\draw[diag] (0,3) -- (4,1);
	\draw[diag] (3,0) -- (1,4);
	\draw[diag] (1,1) -- (3,3);
	\cube;
\end{scope}
\begin{scope}[xshift=-6cm]
	\largecube;
	\draw (2.5,-.5) node{$U$};
\end{scope}
\end{tikzpicture}
\caption{The refinement edge $f$ of the left simplex, \usebox{\typzwei}, is also the refinement edge of one of the children of the right simplex, \usebox{\typdrei}, so $f$ demands $e$. 
Note that $f$ touches the large cube $U$ and, in accordance with Proposition \ref{pps:klebt}, $e$ also does.}
\label{fig:klebt}
}
\end{SCfigure}

\begin{pps}[$\frac1N$-closure is bounded]\index{1N@$\frac1N$-closure is bounded}\label{pps:imPatch}
Suppose an initial triangulation $\T_0$ of type $n$ hyperlevel $0$ simplices. 
For every 
%even 
integer $N>
%2
n$, 
%the $\frac1N$-closure of an arbitrary point $x\in\Omega$ 
every point $x\in\Omega$ is contained in a $\frac1N$-closed set which
is included in the interior $\Int
\T_0 
%\Simplexe^{n,1}
p$ of a patch for some $p\in\Vertices
\T_0
%\Simplexe^{n,1}
$. %If $\T_0$ consists of type $n$ simplices only, it suffices to assume an even $N>n$.
\end{pps}
While the challenge of Proposition \ref{pps:klebt} is to find this statement and its proof consists of a simple analysis of the bisection rule, this is vice versa in the case of Proposition \ref{pps:imPatch}: 
Its 
%The
statement is not surprising, but additional ideas are required to prove it.

Proposition \ref{pps:klebt} is proved 
%on the following pages 
in Subsection \ref{sec:e->f},
%``$e\ar\Edges f$'' starting on page \pageref{sec:e->f}, 
followed by the proof of Proposition \ref{pps:imPatch} 
and Theorem \ref{milestone}
%starting on page \pageref{sec:imPatch}.
in Subsection \ref{sec:imPatch}
and the 
%proofs of Theorem \ref{milestone} and 
proof of the BDV theorem in Subsection \ref{sec:finish}.

\subsection{Preliminaries}
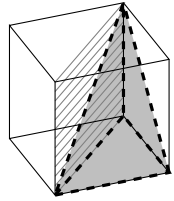
\begin{figure}[h!]
\centering{
\begin{tikzpicture}[scale=.3]
%%Parameter
\tikzmath{\h=5;
			\ax=2;\ay=-2.5;
			\bx=5;\by=1;}
\def\base { -- ++(\ax,\ay) -- ++(\bx,\by) -- ++(-\ax,-\ay) -- ++(-\bx,-\by)};
\def\vertical {-- ++(0,\h)};
%%Zeichnung
%%Würfel
\fill[gray!50] (\ax,\ay)--++(\bx,\by) --++(-\ax,-\ay+\h)--cycle;
%\fill[pattern=horizontal lines,pattern color=gray!120] (\ax,\ay) -- ++(\bx-\ax,\by-\ay) \vertical -- ++(\ax-\bx,\ay-\by) --++ (0,-\h);
\foreach \z in {0,.416,...,\h}
	\draw[gray!100] (\ax,\ay+\z) -- (\bx,\by+\z);
\draw (0,0) \base\vertical\base;
\draw (\ax,\ay) \vertical;
\draw (\bx,\by)\vertical;
\draw (\ax+\bx,\ay+\by)\vertical;
%%Tetraeder
\draw[dashed,very thick] (\ax,\ay) --++(\bx,\by) --++(-\ax,-\ay) \vertical --cycle;
\draw[dashed,very thick] (\ax,\ay) --++(\bx-\ax,\by-\ay);
\draw[dashed,very thick] (\ax+\bx,\ay+\by) --++ (-\ax,-\ay+\h);
\end{tikzpicture}

\caption{(What this subsection is about in one picture)}
}
\end{figure}
At first, some additional notions are introduced. The definition of a reference simplex is revised (see Definition \ref{def:refsimplex})%and the theory including Theorem \ref{cube} is completely replaced
. While the notion of a Chebyshev lattice in Subsubsection \ref{Chebyshev lattice} and Lemma \ref{referencesubsimplex} are essential, most lemmas and proofs in this subsection can be skipped by a reader with a good intuition.

\subsubsection{Parallelotopes}
\begin{defn}[parallelotope, edge vectors, edge basis, vertices, face, subparallelotope, hypersubparallelotope]
\index{parallelotope}\index{edge vectors}\index{edge basis}\index{vertices!parallelotope}\index{face!parallelotope}\index{subparallelotope}\index{hypersubparallelotope}
An $m$-dimensional \emph{parallelotope} in $\R^n$ is a set of the form $P=a+[0,1]v_1+\dots+[0,1]v_m$ for a point $a\in\R^n$ and linearly independent $v_1,\dots,v_m\in\R^n$, the \emph{edge vectors} of $p$, forming an \emph{edge basis} of the parallelotope.

The set of \emph{vertices} of the parallelotope $P$ is $a+\{0,1\}v_1+\dots+\{0,1\}v_m$. A parallelotope is obviously convex and a non-empty \emph{face} of 
%it 
$P$
is given by $a+I_1v_1+\dots+I_mv_m$ with $\{I_1,\dots,I_m\}\subset\{%\leer,
[0,0],[0,1],[1,1]\}$.

Let ${{T_1}\,\dot\cup\, \cdots \,\dot\cup\, {T_k}}\subset \{1,\dots,m\}$ be a partition of a subset of $\{1,\dots,m\}$ into non-empty subsets. Then 
\begin{align}
a+[0,1]\sum_{j\in T_1}v_j+\dots+[0,1]\sum_{j\in T_k} v_j\label{eq:subparallelotope}
\end{align}
is a $k$-dimensional \emph{subparallelotope} of $P$. A \emph{hypersubparallelotope} of $P$ is a subparallelotope of dimension $\dim P-1$. 

Note that there are different representations of a parallelotope, e.g.\
\begin{align}
a+[0,1]v_1+\dots+[0,1]v_m = (a{+}v_1)+[0,1](-v_1)+\dots+[0,1]v_m
\end{align}
and not every subparallelotope of $P$ can be represented like \eqref{eq:subparallelotope} using a single 
%one of them 
representation of $P$
(namely a subparallelotope of $P$ can be written in the form \eqref{eq:subparallelotope}, if and only if it contains the point $a$).
\end{defn}
%\begin{lem}\label{subparallelotope affinely closed}
%A subparallelotope $U$ of a parallelotope $P$ is affinely closed in %its parallelotope 
%$P$, i.e.\ $\aff (U)\cap P=U$.
%\end{lem}
%\begin{proof}
%Let $U$ be
%\begin{align*}
%a+\sum_{i=1}^k [0,1]\sum_{j\in T_i}v_j%\ref{eq:subparallelotope}
%\end{align*}
%for a vertex $a$ of $P$, an edge basis $v_1,\dots, v_m$ of $P$, and a partition $T_1\,\dot\cup\, \cdots \,\dot\cup\, T_k\subset \{1,\dots,m\}$ into non-empty subsets. An arbitrary point 
%%$x\in 
%in $
%\aff U$ has a unique representation
%\begin{align}
%%b&=
%a+\sum_{i=1}^k \beta_i\sum_{j\in T_i}v_j%\\
%%b&
%=a+\sum_{i=1}^k\sum_{j\in T_i} \beta_iv_j.
%\end{align}
%The latter formula is its unique representation in $\aff (P)$ with the origin $a$ and the basis $v_1,\dots,v_m$ and this is a point of $P$ if and only if all the $\beta_i$ lie in $[0,1]$, i.e.\ if the point was already in $U$.
%\end{proof}

Figure \ref{fig:projection} might help to understand the next lemma.
\begin{lem}[Characterisation of hypersubparallelotopes]\label{hypersubparallelotope}
\index{Characterisation of hypersubparallelotopes}
If $H$ is a 
hyper%
subparallelotope of a cube $C$, 
%with $\dim H=\dim C-1$ (i.e.\ a hypersubparallelotope), 
then $H$ is either a face (i.e.\ a subcube) of $C$, or there are edge vectors $e_1{\neq} e_2$ of $C$ such that the projection $\pi_{e_1,e_2}$ of $C$ on these directions (for a rigorous definition of these projections see Definition \ref{def:proj} below) maps $H$ to a diagonal in the square $\pi_{e_1,e_2} C$ and all remaining edge vectors of $C$ are also edge vectors of $H$, thus
\begin{align}
H=
%\underbrace{
C\cap\pi_{e_1,e_2}\inv \pi_{e_1,e_2} H 
.
\end{align}
\end{lem}

{\centering
\begin{SCfigure}[1.35][ht]
%\centering{
\begin{tikzpicture}[scale=.5]
\tikzmath{\abstand=-3pt;
			\lang=6;
			\kurz=3;}
\draw (0,0) -- ++(\lang,0) -- ++(0,\lang) -- ++(-\lang,0) -- cycle;
\draw[thick] (0,0) --node[sloped,fill=gray!20]{$\pi_{e_1,e_2} H$} (\lang,\lang);
\draw(.3*\lang,.7*\lang) node{$\pi_{e_1,e_2} C$};
\end{tikzpicture}
\caption{A hypersubparallelotope which is not a subcube can be projected to a diagonal of a square.}
\label{fig:projection}
%}
\end{SCfigure}
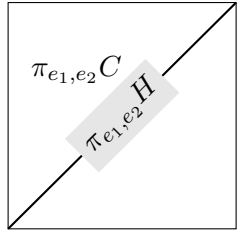

}
\begin{proof}
Let $n:=\dim C$.

Per definition, there are an edge basis $v_1,\dots,v_n$ of $C$ and a partition $T_1\,\dot\cup\,\cdots\,\dot\cup\, T_{n-1}$ of a subset of $\{1,\dots,n\}$ into index sets, such that $\sum_{i\in T_j} v_j$ is an edge vector of $H$ for each $j$. At most one of the index sets $T_j$ contains two indices, the other ones only one.

If all $T_j$ consist of only one index, all edges of $H$ are also edges of $C$, so $H$ is a subcube of $C$.

If otherwise one of the $T_j$ has two elements, then let $e_1$ and $e_2$ be the corresponding edge vectors.
\end{proof}
%%%%%%%%%%%
\subsubsection{Chebyshev lattices}\label{Chebyshev lattice}
The notion of a Chebyshev lattice is simple: Take $\R^m$ with the Chebyshev norm $\|\bullet\|_\infty$ and a subset $\frac1M\Z^m\subset \R^m$ for some $M\in\N_{\geq 1}$.
To obtain a Chebyshev lattice from that, forget all but 
\begin{itemize}
\item
this pair of space and subset, 
\item
the metric induced by the Chebyshev norm and 
\item 
the affine structure on $\R^m$.
\end{itemize}
\small{(One may ask, what can by gained by forgetting structure. This becomes clear when looking for sublattices.)}

\normalsize
Formally, this can be described as follows.
\begin{defn}[Chebyshev lattice, lattice width]
\index{Chebyshev lattice}\index{lattice width}\index{refinement of a Chebyshev lattice}
An \emph{$m$-dimensional $\frac1N$-Chebyshev lattice} (for $m\in\N_{\geq 0},N\in\N_{\geq1}$) 
%in an $m$-dimensional affine space $R$
is a triple $(R,Z,\|\bullet\|)$ of an $m$-dimensional real affine space $R$, 
%pair $(Z,\|\bullet\|)$ of
a norm $\|\bullet\|$ on the vector space corresponding to $R$ such that $(R,\|\bullet\|)$ is isometric to $\left(\R^m,\|\bullet\|_\infty\right)$, and a subset $Z\subset R$ isometric to $\left(\frac1N\Z^m,\|\bullet\|_\infty\right)$. 
If the affine space or the norm are given, Chebyshev lattices are often denoted by the pair $(Z,\|\bullet\|)$, skipping the affine space $R$, or only by $Z$.
The \emph{lattice width} is $\frac1N$% $\lw(Z,\|\bullet\|)$, or briefly $\lw(Z)$
.
\end{defn}
\begin{defn}[$\frac1M Z$, refinement of a Chebyshev lattice]
\index{1M@$\frac1M Z$}\index{Z@$\frac1M Z$}\index{refinement of a Chebyshev lattice}\index{refine(Z,1/M)@$\refine(Z,\frac1M)$}
Let $\frac1M Z$ denote the uniform scaling of a Chebyshev lattice $Z$ at one of 
its
%the lattice 
points (still embedded in the same $R$).

Sometimes, for some multiplier $M$ of $N$, a $\frac 1N$-Chebyshev lattice $Z$ shall be turned into an $\frac1M$-Chebyshev lattice $Z'$. This is given by 
$Z'=\frac NM Z$
%$Z'=\frac1{M\lw Z}Z$ 
and labelled $\refine(Z,\frac1M)$, the \emph{$\frac1M$-refinement of $Z$}. 
\end{defn}
\begin{defn}[isometry respecting the Chebyshev lattice]
\index{isometry respecting the \\Chebyshev lattice}
\index{lattice isometry}
For two Chebyshev lattices 
%$(Z,\|\bullet\|)$ in $R$ and $(Z',\|\bullet\|')$ in $R'$ 
$(R,Z,\|\bullet\|)$ and $(R',Z',\|\bullet\|')$ in $R'$, 
an isometry between $R$ and $R'$ is said to \emph{respect} the Chebyshev lattice and called \emph{lattice isometry}, if it maps $Z$ to $Z'$. 
\end{defn}
\begin{defn}[canonical unit vectors]\index{canonical unit vectors}
Given a lattice isometry 
%from
%%überlange Zeile
which maps
$\left(\frac1N\Z^m,\|\bullet\|_\infty\right)$ 
%$\left(\R^n,\frac1N\Z^m,\|\bullet\|_\infty\right)$ 
to a general $\frac1N$-Chebyshev lattice $(Z,\|\bullet\|)$, the images of the canonical unit vectors $(0,\dots,0,1,0,\dots,0)$ are called \emph{canonical unit vectors}, as well. 
\end{defn}
\begin{defn}[sublattice]
\index{sublattice}
Let $(R,Z,\|\bullet\|)$ be a $\frac1N$-Chebyshev lattice and $U$ an affine subspace of $R$. If $(U,Z\cap U,\|\bullet\|)$ is also a $\frac1N$-Chebyshev lattice (with the same lattice width), then it is called a \emph{sublattice }of $(R,Z,\|\bullet\|)$. For a given Chebyshev lattice $(R,Z,\|\bullet\|)$, $U$ determines the sublattice uniquely. Therefore, it is also written that ``$U$ is a sublattice of $(Z,\|\bullet\|)$''.
\end{defn}
\begin{defn}[cube, subcube in a Chebyshev lattice]\label{def:cube in a Chebyshev lattice}
\index{cube}\index{subcube}
%A 1-(sub)cube in $\left(\Z^n,\|\bullet\|_\infty\right)$ is a set of the form $a+I_1\times\dots\times I_n$ for some $a\in \Z^n$ and $I_1=\dots=I_n=[0,1]$ in the case of a cube and $I_1,\dots,I_n\in\left\{\leer,\{0\},[0,1]\right\}$ in the case of a subcube.
A \emph{1-cube in $(\R^n,\Z^n,\allowbreak {\|\bullet\|_\infty})$} is a set of the form $a+[0,1]^n$ for some $a\in \Z^n$. 
%Its \emph{subcubes }are 
%%$\leer$ and 
%the sets of the form $a+I_1\times\dots\times I_n$ for $I_1,\dots,I_n\in\left\{
%%\leer,
%\{0\},\{1\},[0,1]\right\}$. They are also its faces.
A \emph{subcube }in $(\R^n,\Z^n,\allowbreak\|\bullet\|_\infty)$ is a face of a cube therein.

For $M,N\in\N_{\geq 1}$, a \emph{$\frac 1{M\cdot N}$-(sub)cube in $\left(\R^n,\frac1N\Z^n,\|\bullet\|_\infty\right)$ }is a 1-(sub)cube as above %simply 
multiplied by $\frac 1{M\cdot N}$.

%For a $\frac1{M\cdot N}$-refinement $Z'$ of a A subcube in $Z$ is the union of $\frac{$ subcubes

Let $(R,Z,\|\bullet\|)$ be a Chebyshev lattice with a lattice isometry 
%\begin{align*}
$
\vi: \R^m%&
\ra R%\\
,\allowbreak	\frac1N\Z^m%&
\ra Z.
%\end{align*}
$
Then a \emph{$\frac 1{M\cdot N}$-(sub)cube in $(R,Z,\|\bullet\|)$ }is the image of a $\frac 1{M\cdot N}$-(sub)cube in $\left(\frac1N\Z^m,\|\bullet\|_\infty\right)$ under $\vi$. %simply.
\end{defn}
Note that a subcube in $(Z,\|\bullet\|)$ is a subset of the affine space $R$, not a subset of $Z$. A subcube in $Z$ can contain points of $Z$ only as vertices.
\begin{lem}\label{SPTSG}
Let $C$ be a $\frac 1N$-cube in a $\frac 1N$-Chebyshev lattice  $(Z,\|\bullet\|)$ and $U$ be an $m$-dimensional subparallelotope of $C$. Then 
%$(U,Z\cap \aff U,\|\bullet\|)$ 
$\aff(U)$
is a 
%$\frac 1N$-lattice 
sublattice of $(Z,\|\bullet\|)$
and $U$ is a $\frac 1N$-cube in 
%it. 
that sublattice.
See Figure \ref{fig:sublattice}.
\end{lem}

\begin{figure}[ht]
\centering{
%%sublattices
\begin{tikzpicture}
\begin{scope}[xshift=-6cm, scale = .8]
%%grid
\foreach \x in {0,...,3}
	{\foreach \y in {0,...,3}
		{\fill (\x,\y) circle [radius=.07];}
	}
\draw[very thick] (0,3) --node[fill=gray!20]{\small{1}} (1,2) --node[fill=gray!20]{\small{1}} (2,1) --node[fill=gray!20]{\small{1}} (3,0);

%%Beschriftung
\draw[gray!96,<->,>={Stealth[sep=.06cm]}]  (0,0) --node[black,below]{1}(1,0);
\draw[gray!96,<->,>={Stealth[sep=.06cm]}]  (0,0) --node[black,left]{1}(0,1);
\draw[gray!96,<->,>={Stealth[sep=.06cm]}]  (1,0) --node[black,above right=-2pt]{1}(0,1);

%%line
\foreach \y in {0,...,3}
	{\fill (3+2.7,\y) circle [radius=.07cm];}
\draw[thick] (3+2.7,0) --node[right]{1} ++(0,1) --node[right]{1} ++(0,1) --node[right]{1} (3+2.7,3);

%%Pfeil
\foreach \y in {0,...,3}
	{\draw[->,>={Stealth[scale=1.3,sep=.06cm]}] (3-\y,\y) .. controls (3.7-.5*\y,.1*\y*\y+.9*\y+.5) 
	and (4.9-.5*\y,.1*\y*\y+.9*\y+.5) .. (3+2.7,\y);}
\end{scope}

\begin{scope}[scale=1.5,z={(.15,.2)},very thin]
\tikzmath {\xyend=.1;}
\fill[gray!30] (-\xyend,2+\xyend,-.3) -- (2+\xyend,-\xyend,-.3) -- (2+\xyend,-\xyend,2.3) -- (-\xyend,2+\xyend,2.3) -- cycle;
%xy-grid
\foreach \z in {0,1,2} {
\draw[shift={(0,0,\z)}] (-\xyend,-\xyend,0) grid (2+\xyend,2+\xyend,0);
}
%%Tiefenlinien
\foreach \x in {0,1,2}{
	\foreach \y in {0,...,2}{
		\draw (\x,\y,-.5) -- (\x,\y,2.5);
	}
}
%%Subgrid
\foreach \z in {0,1,2}{
	\draw[dashed] (-\xyend,2+\xyend,\z)--(2+\xyend,-\xyend,\z);
}
%%Beschriftung
\draw[<->,>=stealth,thick]  (2,0,0) --node[black,fill=gray!20,above]{1} (1,1,0);
\draw[<->,>=stealth,thick]  (1,0,0) --node[black,fill=gray!20,below]{1}(2,0,0);
\draw[<->,>=stealth,thick]  (1,0,0) --node[black,fill=gray!20,left]{1}(1,1,0);
\end{scope}
\end{tikzpicture}
\caption{Sublattices}\label{fig:sublattice}
}
\end{figure}
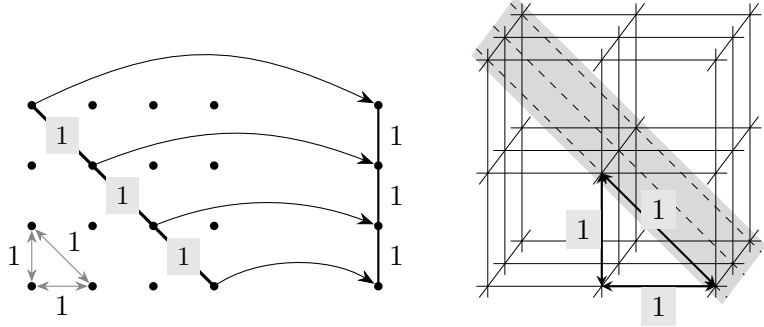

\begin{proof}
It suffices to prove the case $N=1$. 
%Let $v_1,\dots,v_n$ be edge vectors of $C$, $a$ be a point of $C$ such that $C=a+\sum_{i=1}^n v_i$, 
Let $C=a+\sum_{i=1}^n [0,1]v_i$ with $a\in Z$,
$T_1\,\dot\cup\,\cdots\,\dot\cup\, T_m\subset \{1,\dots,n\}$ and $U=a+\sum_{j=1}^m[0,1]\sum_{i\in T_j} v_i$. Then
\begin{align*}
\aff U&=a+
%\Span\left(\sum_{i\in T_1}v_i,\dots,\sum_{i\in T_m}v_i\right), \\
\sum_{j=1}^m\R\sum_{i\in T_j} v_i,\\
Z\cap\aff U&=a+\sum_{j=1}^m\Z\sum_{i\in T_j} v_i
\end{align*} and 
\begin{align*}
\vi:\R^m&\ra \aff U\\
\Z^m&\ra Z\cap\aff U\\
\beta &\mapsto a+\sum_{j=1}^m\beta_j\sum_{i\in T_j} v_i
\end{align*}
is an isometry mapping $[0,1]^m$ to $U$, thus $U$ is a cube in the sublattice $Z\cap \aff U$.
\end{proof}
\begin{lem}\label{lem:intersection of cube and sublattice}
Let $c_Z$ be a $\frac1N$-subcube in a 1-Chebyshev lattice $Z$ and $U$ be a sublattice of $Z$ spanned by a subparallelotope of a 1-cube $C$ in $Z$. Then $c_Z\cap U$ is a subcube in $U$.
\end{lem}
\begin{proof}[Sketch of a proof]
For any subparallelotope $P$ of $C$ there is a sequence $C=P_n\supset\dots\supset P_k=P$ such that $P_j$ is a hyperparallelotope of $P_{j+1}$. For that reason, it suffices to prove the statement for a \emph{hyper}subparallelotope $P$. (The actual statement is obtained by a recursive application to the Chebyshev lattice $\aff P_j$, $c_Z\cap P_j$ and the sublattice spanned by $P_{j-1}$ with $j$ decreasing from $n$ to $k+1$.)

If $P$ is a subcube, the statement is more or less obvious.

Otherwise, the projection from the Characterisation of hypersubparallelotopes yields Figure \ref{fig:intersection of cube and subgrid}, where the projection of $c_Z$ is one of the small gray squares or line segments.
\begin{SCfigure}[][ht]
\begin{tikzpicture}[scale=.4]
\draw[step=2,gray] (0,0) grid (6,6);
\draw (0,0) rectangle (6,6);
\draw (1,5) node{$\pi C$};
\draw[thick] (-.5,-.5) -- node[sloped,fill=gray!30]{$\pi U$} (6.5,6.5);
\end{tikzpicture}
\caption{}\label{fig:intersection of cube and subgrid}
\end{SCfigure}
$\pi c_Z\cap \pi U$ can be a diagonal $d$ or a vertex $v$ of $\pi c_Z$ or empty. Since $U=\pi\inv\pi U$, $c_Z\cap U$ is the full preimage $c_Z\cap \pi\inv(c_Z\cap U)$ of $e$ or $v$ or $\leer$ in $c_Z$ which is a subparallelotope in $c_Z$ and a subcube in $U$.
\end{proof}

\begin{defn}[reference simplex]\label{def:refsimplex}
\index{reference simplex!in a Chebyshev lattice}
A \emph{%
$(k,h)$ 
reference simplex} 
%of hyperlevel $hS$ 
in an $m$-dimensional $2^{-h}$-Chebyshev lattice $(Z,\|\bullet\|)$ is a type $k$ hyperlevel $h$ T-array
$\begin{pmatrix}
p_0&\dots& p_k\\
&\vdots\\
&p_m
\end{pmatrix}$
for which there exists a sequence of $2^{-h}$-subcubes $C_k\subset\dots\subset C_m$ in $(Z,\|\bullet\|)$ such that
\begin{itemize}
\item
each $C_j$ is a $j$-dimensional subcube (i.e.\ a hyperface) of $C_{j+1}$,
\item
for $j=k{+}1,\dots,m$, the vertical vertex $p_j$ is the centre of $C_j$,
%the vertical node $p_{k+1},\dots,p_n$ is the centre of $C_{k+1},\dots,C_n$ respectively.
\item
%and $p_0,\dots,p_{k}$ are vertices of $C_k$.
and $C_k=p_0+\sum_{j=1}^k[0,1](p_j{-}p_{j-1})$.
%$p_0$ is a vertex of $C_k$ and
\end{itemize}
\end{defn}
\begin{lem}\label{lem:uniquelattice}
A T-array $S$ of %a given 
type $k$ %and hyperlevel $hS$ 
in an affine space $R$ uniquely determines a lattice $Z_S$, a norm $\|\bullet\|_S$ and cubes $C_k^S,\dots,C^S_m$ such that $S$ is a reference simplex in $(Z_S,\|\bullet\|_S)$ and $C^S_k,\dots,C^S_m$ fulfil the properties claimed in Definition \ref{def:refsimplex}.
\end{lem}
\begin{proof}
Let $S=:\begin{pmatrix}
p_0&\dots& p_k\\
&\vdots\\
&p_m
\end{pmatrix}$.
At first, the cubes are determined: The cube $C_k$ has to be $p_0+\sum_{j=1}^k[0,1](p_j{-}p_{j-1})$. For $j\in\{k,\dots,m{-}1\}$, the cube $C_{j+1}$ has to be $C_j+[0,2](p_{j+1}{-}
%p_j
\mid C_j)$ to have $C_j$ as a hyperface and %be centred at 
$p_{j+1}$ as center.

Finally, the cube $C_m$ determines the lattice: If $C_m=p_0+\sum_{j=1}^m [0,1]v_j$ for edge vectors $v_j$, then $Z=p_0+\sum_{j=1}^m \Z v_j$ with the norm $\left\|\sum_{j=1}^m \beta_jv_j\right\|:=2^{-hS}\|\beta\|_\infty$ is the only $2^{-hS}$-Chebyshev lattice wherein $C_m$ is a cube.
\end{proof}
\begin{defn}[lattice, cubes of a T-array]\label{lattice of a T-array}
\index{lattice!of a T-array}\index{Chebyshev lattice!of a T-array}\index{cubes of a T-array}\index{Zs@$Z_S$ -- Chebyshev lattice of $S$}
%This lattice
The lattice in Lemma \ref{lem:uniquelattice} 
is called \emph{the lattice $Z_S$ of the T-array $S$}, 
%these cubes 
the cubes there 
are \emph{the cubes of $S$}%$(Z_S,\|\bullet\|_S)$
.
In contrast, \emph{a $\frac1N$-subcube $c$ of a T-array $S$} (always denoted by a lower-case $c$) represents an arbitrary subcube in the Chebyshev lattice of $S$ according to Definition \ref{def:cube in a Chebyshev lattice}.
\end{defn}
\begin{bemn}
\begin{itemize}
\item
The Chebyshev lattice of a (sub-)simplex fills 
%the whole space (and in the case of subsimplices 
its whole affine hull, 
%respectively), 
even though far simplices of the triangulation have nothing to do with this lattice. 
%TODO??? 
Consequently, a subcube of a simplex is not restricted to be a subset of the simplex.
\item
Also the reflexion of a reference simplex is a reference simplex, and has the same cubes $C_k, \dots, C_m$.
\end{itemize}
\end{bemn}
\begin{lem}[Conservation of a Chebyshev lattice in children]\label{lem:conservationchildren}
%\index{Conservation of a Chebyshev lattice in children}
If $S$ is a 
$(k,h)$
reference simplex 
%of type $k\geq 1$ 
in a Chebyshev lattice $(Z,\|\bullet\|)$ for $k\geq 1$ with the cubes $C_k,\dots,C_m$, then each of the children of $S$ is a 
$(k{-}1,h)$
reference simplex with some cube $C_{k-1}$ and the same cubes $C_k,\dots,C_m$.

If $S$ is a 
$(0,h)$
reference simplex 
%of type 0 
in a Chebyshev lattice $(Z,\|\bullet\|)$, then its transposed $S^T$ (i.e.\ its child in the extended binary tree) is a 
$(n,h{+}1)$
reference simplex in $\left(\frac12 Z,\|\bullet\|\right)$.
\end{lem}
A proof of this is not outlined. It is similar to the proof of Theorem \ref{Cubetheorem} about the cube-related geometry of reference simplices.

The next lemma, the main result of this preliminary section here, states that a subsimplex of a reference simplex spans a sublattice in which the subsimplex is a reference simplex itself.

\emph{Preliminary remark on Lemma \ref{referencesubsimplex}.}
Let $S$ be a reference simplex in a Chebyshev lattice $Z$ and $U$ a restriction of $S$. By Definition \ref{restricted T-array!IsoCoChange} of $hU$, 
\begin{align*}
\refine\left(Z,2^{-hU}\right)=
\begin{cases}
Z,&\text{if $U$ contains a horizontal vertex of $S$,}\\
\frac12Z,&\text{otherwise}.
\end{cases}
\end{align*}
\begin{lem}[Restrictions are reference simplices in sublattices.]\label{referencesubsimplex}
\index{restrictions of reference simplices}
Let $S$ be a reference simplex in a $2^{-hS}$-Chebyshev lattice $
%T=
\left(Z,\|\bullet\|\right)$. Let $U$ be the restriction of $S$ to an $m$-dimensional subsimplex 
of $S$. 
%Let $\alpha:=hU-hS$ (i.e.\ $\alpha=0$, if $U$ contains a horizontal vertex of $S$ and otherwise 1). 
Then
\begin{enumerate}
\item
the intersection of
the affine subspace $\aff (U)$ 
with $Z$
is an $m$-dimensional sublattice of 
%$2^{-\alpha}Z$
$\refine\left(Z,2^{-hU}\right)$%
,
\item
the T-array $U$ is a reference simplex in $\left(
%2^{-\alpha}Z
\refine\left(Z,2^{-hU}\right)
\cap\aff U,\|\bullet\|\right)$, 
%of hyperlevel $hU$,
\item\label{subsimplex subparallelotope}
%and 
if $U$ contains some horizontal vertex of $S$, the cube $C^U_m$ of $U$ is a subparallelotope of the cube $C^S_n$ of $S$.
Otherwise, 
%it 
$C^U_m$ 
is a subparallelotope of the cube $C_n^{S'}$ of some descendant $S'$ of $S$ of hyperlevel $hS{+}1$ and full type.
\end{enumerate}
\end{lem}
%\begin{bem}
%If $U$ contains a horizontal vertex of $S$, then $hU=hS$ by definition of hyperlevels for restrictions and $\refine\left(Z,2^{-hU}\right)=Z$. Only if $U$ contains only vertical vertices of $S$, then $hU=hS{+}1$ and $\refine\left(Z,2^{-hU}\right)=\frac12Z$.
%\end{bem}
The proof is prepared by some basic investigations.
\begin{defn}[possible parent, $\sim$ ancestor]
\index{possible parent}\index{possible ancestor}
A T-array 
$S= 	
\begin{pmatrix}
p_0&\dots& p_k\\
&\vdots\\
&p_n
\end{pmatrix}$
of a type $\leq n{-}1$ can be the child either of 
$S_1=
\begin{pmatrix}
2p_{k+1}{-}p_k& p_0&\dots& p_k\\
&&p_{k+2}\\
&&\vdots\\
&&p_n
\end{pmatrix}$ or of 
$S_2=
\begin{pmatrix}
p_0&\dots& p_k&2p_{k+1}{-}p_0\\
&p_{k+2}\\
&\vdots\\
&p_n
\end{pmatrix}$.
\end{defn}
These two T-arrays are the \emph{possible parents} of $S$. The \emph{possible ancestors} of $S$ of the same hyperlevel are defined in the same way.
\begin{lem}\label{possibleancestorequalhyperlevel}
Let $T$
 be a reference simplex% with the corresponding subcubes $C_{tT},\dots,C_n$
. Then a possible ancestor $S$ of $T$ of the same hyperlevel is also a reference simplex.
\end{lem}
\begin{proof}
It suffices to prove the statement for a possible \emph{parent} (instead of \emph{ancestor}) of reference simplex 
$T=
\begin{pmatrix}
p_0&\dots& p_k\\
&\vdots\\
&p_n
\end{pmatrix}$ of type $\leq n{-}1$.

%Note that the notion of a reference simplex is unique up to a lattice isometry, because there is always a lattice isometry mapping $C_n$ to $[0,1]^n$, $C_{n-1}$ to $[0,1]^{n-1}$, etc.\ simultaneously.
%Note that arbitrary two reference simplices are mapped
Note that a lattice isometry maps reference simplices to reference simplices and a reference simplex can be mapped to an arbitrary other one by a lattice isometry, because there exists always a lattice isometry mapping $C_n$ to $[0,1]^n$, $C_{n-1}$ to $[0,1]^{n-1}{\times}\{0\}$, etc.\ simultaneously.

In this way, an arbitrary child $T'$ of an arbitrary reference simplex $S'$ of hyperlevel $hT=hS$ and type $tT'=tT=tS'{-}1=tS{-}1$ is a reference simplex according to Lemma \ref{lem:conservationchildren} and can be mapped to $T$. This
lattice isometry maps $S'$ to \emph{one} of the possible fathers of $T$. Therefore, \emph{this} possible father is also a reference simplex.

Furthermore, there exists a lattice isometry interchanging $p_0$ and $p_k$, $p_1$ and $p_{k-1}$, etc.\ and fixing the vertical vertices $p_{k+1},\dots,p_n$. 
It maps the possible parents 
$\begin{pmatrix}
2p_{k+1}{-}p_k& p_0&\dots& p_k\\
&&p_{k+2}\\
&&\vdots\\
&&p_n
\end{pmatrix}$ 
and
$\begin{pmatrix}
p_0&\dots& p_k&2p_{k+1}{-}p_0\\
&p_{k+2}\\
&\vdots\\
&p_n
\end{pmatrix}$ of $T$ to each other. Hence, 
%they are \emph{both} 
\emph{both} of them are reference simplices.
\end{proof}
\begin{proof}[Proof for Lemma \ref{referencesubsimplex}]
Throughout the proof, let the reference simplex 
$S$ 
in $
%T
(Z,\left\|\bullet\right\|)$ be $%S:=
\begin{pmatrix}
p_0&\dots&p_k\\
&\vdots\\
&p_n
\end{pmatrix}$.

Since several reference simplices are used here, the cubes corresponding to a (supposed) reference simplex $T$ are generally denoted by $C^T_j$.

%1st step: Lemma \ref{referencesubsimplex} for T-arrays of full type.
\emph{1st case: $S$ has full type.}
Let $S=\begin{pmatrix}p_0&\dots&p_n\end{pmatrix}$ and $U=\begin{pmatrix}p_{i_0}&\dots&p_{i_m}\end{pmatrix}$ (with $0\leq i_0<\dots<i_m\leq n$). In this case, to record is only the sublattice and the cube $C^U_m$.
Then, with 
\begin{align*}
%e_1&:=p_0-p_1 \text{ (not $p_1-p_0$)}\\
%&\vdots\\
%e_{i_0}&:=p_{i_0-1}-p_{i_0}\\
% e_{i_0+1}&:=p_{i_0+1}-p_{i_0}\\
%&\vdots\\
%e_n&:=p_n-p_{n-1},
e_j:=\begin{cases}
p_{j-1}-p_j&\text{ for $j\leq i_0$}\\
p_j-p_{j-1}&\text{ otherwise},
\end{cases}
\end{align*} 
$C^S_n=p_0+\sum_{i=1}^n [0,1](p_{i-1}-p_i)=p_{i_0}+\sum_{i=1}^n [0,1]e_i$ and, with
\begin{align*}
T_1&:=\{i_0+1,\dots,i_1\}\\
\vdots\\
T_m&:=\{i_{m-1}+1,\dots,i_m\},
%$T_j:=\{i_{j-1}+1\}
\end{align*}
the following subparallelotope of $C^S_n$
\begin{align}
p_{i_0}+\sum_{j=1}^m[0,1]\sum_{i\in T_j}e_i
=\,p_{i_0}+\sum_{j=1}^m[0,1]\left(p_{i_j}{-}p_{i_{j-1}}\right)\label{edgebasis}
=:\,C^U_m
\end{align}
% of $C^S_n$ 
arises. According to Lemma \ref{SPTSG}, $\aff U=\aff C^U_m$ is a sublattice of $(Z,\|\bullet\|)$, and according to \eqref{edgebasis}, $C^U_m$ is the cube for the full type reference simplex $U$ in that sublattice.
%$C^U_m$ is a cube and  $U$ is a reference simplex in it.

%2nd step: Lemma \ref{referencesubsimplex} for a subsimplex $U=
\emph{2nd case: $U=
\begin{pmatrix}
p_{i_0} &p_{i_0+1}\quad\dots&p_k\\
&p_{k+1}\\
&\vdots\\
&p_n
\end{pmatrix}$ containing a horizontal vertex of $S$ (supposing $i_0\leq k$).
}

$S$ and $U$ have the possible ancestors
\begin{align*}
\begin{pmatrix}
2p_n-p_k&\dots&2p_{k+1}-p_k&p_0&\dots&\dots&\dots&p_k
\end{pmatrix}
&=:\tilde S\text{ and}\\
\begin{pmatrix}
2p_n-p_k&\dots&2p_{k+1}-p_k&\quad p_{i_0}&p_{i_0+1}&\dots&p_k
\end{pmatrix}
&=:\tilde U\text{, respectively}.
\end{align*}
$\tilde U$ is the restriction of $\tilde S$ to a subsimplex.
According to the first 
%step, 
case,
$\tilde U$ is a reference simplex in the sublattice $(Z\cap\aff\tilde U,\|\bullet\|)$. As a descendant of a reference simplex, $U$ is a reference simplex in the same sublattice.
Since $U$ has equal hyperlevel as $\tilde U$, its cube $C_m^U$ is the same as the cube $C_m^{\tilde U}$, i.e.\ a subparallelotope of $C_n^S$.

%3rd step: Lemma \ref{referencesubsimplex} for a subsimplex $U=
\emph{3rd case: $U=
\begin{pmatrix}
p_{i_0} &\dots&p_{i_l}\\
&p_{i_{l+1}}\\
&\vdots\\
&p_{i_m}
\end{pmatrix}$ (with $0\leq i_0<\dots <i_l\leq k<i_{l+1}<i_m\leq n$) containing a horizontal vertex of $S$ (i.e.\ $l\geq 0$).
}

Let $U':=\begin{pmatrix}
p_{i_0} &p_{i_0+1}\quad\dots&p_k\\
&p_{k+1}\\
&\vdots\\
&p_n
\end{pmatrix}$
be the T-array from the second 
%step. 
case. $U'$ and $U$ have the possible ancestors
\begin{align*}
\tilde U'&:=\begin{pmatrix}
p_{i_0}&p_{i_0+1}&\dots&p_k&2p_{k+1}-p_{i_0}&\dots&2p_{n}-p_{i_0}
\end{pmatrix}\text{, and}\\
\tilde U&:=\begin{pmatrix}
p_{i_0}&p_{i_1}&\dots&p_{i_l}& &2p_{i_{l+1}}-p_{i_0}&\dots&2p_{i_m}-p_{i_0}
\end{pmatrix}, \text{ respectively.}
\end{align*}
%which 
$\tilde U$ is the restriction of $\tilde U'$ to a subsimplex. The statement is concluded as in the second %step.
case.

%4th step: Lemma \ref{referencesubsimplex} for a subsimplex $U$ not containing a horizontal vertex of $S$.
\emph{4th case: $U$ does not contain any horizontal vertex of $S$.}
All vertical vertices of $S$ are contained in every descendant $S'$ of $S$ of hyperlevel $hS{+}1$ and type $n$, so $U$ is also a restriction of $S'$ to a subsimplex. $S'$ is a reference simplex in $\left(\frac 12Z,\|\bullet\|\right)$. The statement for $U$ follows by applying the first 
%step 
case
above on $S'$, this lattice and $U$.
\end{proof}
%\begin{kor}\label{halb geht immer}
%A subsimplex of a hyperlevel $k$ simplex spans a subparallelotope of a $2^{-k-1}$-cube.
%\end{kor}
%\begin{proof}
%If the subsimplex does not contain any horizontal vertex, this has already been proven in the 2nd case of Lemma \ref{referencesubsimplex}. Otherwise, the first case of that lemma yields such a subparallelotope of a $2^{-k}$-cube and then Lemma \ref{lem:scale} finds a subparallelotope of a $2^{-k-1}$-cube with the same affine span.
%\end{proof}
\begin{defn}[canonical projection of a Chebyshev lattice]\label{def:proj}
%\index{canonical projection of a \\Chebyshev lattice}
\index{canonical projection of a Cheby\-shev lattice}
%Let $(R_1,Z_1,\|\bullet\|_1)$ and $(R_2,Z_2,\|\bullet\|_2)$ on $R_2$ be $\frac1N$-Chebyshev lattices. A \emph{orthogonal projection} from $(Z_1,\|\bullet\|_1)$ to $(Z_2,\|\bullet\|_2)$ is a surjective affine map $\pi: R_1\ra R_2$ such that %$\pi Z_1=Z_2$
%a $\frac1N$-cube in $Z_1$ is mapped to a $\frac1N$-cube in $Z_2$.

%The \emph{canonical projections of $\left(\frac1N\Z^n,\|\bullet\|_\infty\right)$} in $\R^n$ are maps 
The \emph{canonical projection of $\left(\R^n,\frac1N\Z^n,\|\bullet\|_\infty\right)$ on the canonical unit vectors $e_{i_1},\dots,e_{i_m}$} (with $1\leq i_1<\dots<i_m\leq n$)
 is the map
\begin{alignat*}2
\pi_{i_1,\dots,i_m}: &&\R^n&\ra \R^m\\
					&&(x_1,\dots,x_n)&\mapsto (x_{i_1},\dots,x_{i_m}).
\end{alignat*}
%They map 
It maps 
$\left(\frac1N\Z^n,\|\bullet\|_\infty\right)$ to $\left(\frac1N\Z^m,\|\bullet\|_\infty\right)$.

%For lattice isomorphisms $\vi_1: (\frac1N\Z^n,\|\bullet\|_\infty)\ra (Z_1,\|\bullet\|_1)$ and $\vi_2: (\frac1N\Z^m,\|\bullet\|_\infty)\ra (Z_2,\|\bullet\|_2)$ the maps $\vi_2\circ\pi_{i_1,\dots,i_m}\circ \vi_1$ are the canonical projections from $(Z_1,\|\bullet\|_1)$ to $(Z_2,\|\bullet\|_2)$.

For general Chebyshev lattices, \emph{canonical projections }are obtained by composition of the above defined canonical projection with lattice isomorphisms (in arbitrary order). 
%A \emph{projection on some canonical unit vectors} preserves these while the other canonical unit vectors are annihilated.
\end{defn}
\begin{bem}
A canonical projection maps $\frac1N$-cubes to $\frac1N$-cubes, subcubes of cubes to subcubes of the image cubes
and
subparallelotopes of cubes to subparallelotopes of the image cubes. 
%and reference simplices to reference simplices.
\end{bem}
%\begin{lem}\label{lem:proj}
%A canonical projection maps a reference simplex of a certain 
%%type and 
%hyperlevel to a reference simplex of the same 
%%type and 
%hyperlevel.
%\end{lem}
%\begin{proof}
%It suffices to prove it for one reference simplex in $\left(\Z^n,\|\bullet\|_\infty\right)$ of each type. Let $1\leq i_1<\dots<i_l\leq k<i_{l+1}<\dots<i_m\leq n$. The projection $\pi_{i_1,\dots,i_m}$ maps the reference simplex
%\begin{align*}
%\begin{pmatrix}
%\sum_{j=1}^0 e_j& \dots& \sum_{j=1}^k e_j\\
%&\frac12\sum_{j=1}^{k+1} e_j\\
%&\vdots\\
%&\frac12\sum_{j=1}^n e_j
%\end{pmatrix}
% \mapsto 
%\begin{pmatrix}
%\sum_{j=1}^0 e_j& \dots& \sum_{j=1}^l e_j\\
%&\frac12\sum_{j=1}^{l+1} e_j\\
%&\vdots\\
%&\frac12\sum_{j=1}^{m} e_j
%\end{pmatrix}.
%\end{align*}
%(Some vertices are mapped to the same point and the top vertical vertices may be mapped to $\frac12\sum_{j=1}^l e_j$, which is already included in the convex hull of the first and the last horizontal vertex.)
%\end{proof}
\subsubsection{IsoCoChange}\label{subsec:IsoCoChange}
\begin{defn}[$\|\bullet\|_\infty$-isometric coordinate changes (\hypertarget{IsoCoChange}{IsoCoChange})]\label{def:IsoCoChange}
\index{isometric coordinate changes}
\index{IsoCoChange -- iso\-met\-ric co\-or\-di\-na\-te \\chan\-ges}
%\index{IsoCoChange -- isometric coordinate changes}
%\index{IsoCoChange -- iso\-met\-ric co\-or\-di\-na\-te chan\-ges}
A (regular) triangulation $\T$ satisfies \emph{IsoCoChange}, if for each pair $S,T\in\T$ the following holds true: 
%Let $k$ be the hyperlevel of the restriction of $S$ to $S\cap T$ (i.e.\ $k=hS$, if $S\cap T$ contains a horizontal vertex of $S$ and $k=hS{+}1$, else). %Then for reference coordinates $\vi_S,\vi_T$ it holds
Let $U$ be the restriction of $S$ to $S\cap T$, $V$ be the restriction of $T$ to $S\cap T$ and $\alpha:=\max\{hU,hV\}$.
%Let $\vi_S,\vi_T: \Omega\ra (R,Z,\|\bullet\|)$ be reference coordinates for $S$ and $T$, respectively. Let $(R_S,Z_S,\|\bullet\|_S)$ and $(R_T,Z_T,\|\bullet\|_T)$ be the pulled-back Chebyshev lattices respectively. Then the two sublattices each spanned by $\aff (S\cap T)$ coincide.
Let 
%$(\R^n,Z_S,\|\bullet\|_S)$ 
$(Z_S,\|\bullet\|_S)$ and 
%$(\R^n,Z_T,\|\bullet\|_T)$
$(Z_T,\|\bullet\|_T)$
 be the Chebyshev lattice of $S$ and $T$, respectively. Then the 
 %two 
 sublattice of 
%$(\R^n,Z_S,\|\bullet\|_S)$
$(\refine(Z_S,2^{-\alpha}),\|\bullet\|_S)$
 and 
%the sublattice 
of
%$(\R^n,Z_T,\|\bullet\|_T)$
$(\refine(Z_T,2^{-\alpha}),\|\bullet\|_T)$% respectively
, each spanned by $\aff (S\cap T)$, must coincide.
\end{defn}
\begin{bemn}
\begin{itemize}
\item
Normally, $hU=hV$. There is only one exception: If $tU=0$ and $tV=\dim V$, then $hV=hU{+}1$.
\item
%It is called IsoCoChange for historical reasons: 
The name ``IsoCoChange'' refers to an alternative definition,
which was hinted in the introduction of this chapter.
%The Chebyshev lattices of $S$ and $T$ are something like certain reference coordinates $\vi_S: S\ra \R^n$, where $\R$ is equipped with the Chebyshev norm. In this formulation, IsoCoChange claims that the coordinate change map $\vi_T\circ (\vi_S|_{S\cap T})\inv$ is isometric and maps points of $2^{-\alpha}\Z^n$ into $2^{-\alpha}\Z^n$.
\end{itemize}
\end{bemn}
\begin{lem}\label{lem:IsoCoCh}
\begin{enumerate}
\item\label{twodefs}
If $\T_0$ satisfies 
%the first definition of IsoCoChange, then also the second one.
ReTaHyCo, then it also satisfies IsoCoChange.
\item\label{item:refinementconservation}
If $\T$ satisfies 
%the second definition of 
IsoCoChange, then every admissible refinement of $\T$ also does.
\end{enumerate}
\end{lem}
%Since the proof requires some 
%%other 
%further
%notions and lemmas, it is postponed 
%to 
%%after 
%%(or behind?) 
%the preliminary section.

%\subsubsection{Proof of Lemma \ref{lem:IsoCoCh}}
\begin{proof}[Proof of Lemma \ref{lem:IsoCoCh}.\ref{twodefs}]
According to Lemma \ref{referencesubsimplex}, the restriction of $S$ to $S\cap T$ is a reference simplex in $(Z_S\cap \aff (S\cap T),\|\bullet\|_S)$ and the restriction of $T$ to $S\cap T$ is a reference simplex in $(Z_T\cap \aff (S\cap T),\|\bullet\|_T)$. According to Lemma \ref{lem:uniquelattice}, there is only one Chebyshev lattice of this restriction each. Both restrictions coincide up to reflexion, so the two Chebyshev lattices must coincide.
\\
\emph{Proof of Lemma \ref{lem:IsoCoCh}.\ref{item:refinementconservation}.}
This is clear from Lemma \ref{lem:conservationchildren}: If a simplex $S$ is bisected, its Chebyshev lattice is conserved. If $S$ is transposed, its Chebyshev lattice is scaled by $\frac12$. However, in the definition of IsoCoChange, the Chebyshev lattices are refined to the same lattice width before they are compared. 
%So even non-regular refinements satisfy IsoCoChange.
\end{proof}
\begin{bem}
IsoCoChange is a weaker compatibility condition than 
%SIC \ref{SIC}
\hyperlink{SIC}{SIC} in Definition \ref{SIC} on page \pageref{SIC}%, but a stronger one than ReTACo \ref{ReTACo}
. To demonstrate 
%that it is strictly weaker, 
the difference,
consider an arbitrary 2-dimensional triangulation with hyperlevel 0 and type 2 T-arrays. While SIC claims that every edge is refinement edge of \emph{none} or of \emph{both} adjacent triangles, IsoCoChange is always satisfied, because the sublattices on the edges are always isometric to $\Z$ with the two vertices being two neighbouring lattice points.
%We have the following chain of compatibility conditions:
%SIC $\Ra$ IsoCoChange $\Ra$ ReTACo $\Ra$ trianglewise equivalence of restricted T-arrays $\LR$ general refineability

With SIC, an edge is in reference coordinates an edge of a cube, a face diagonal, or a diagonal of a higher dimensional subcube, but is the same for each simplex. So levels could be assigned to each edge. This is different 
%in 
for 
IsoCoChange: An edge can be an edge of a cube in the lattice of one simplex, but a diagonal in another. That is why the level function $l$ is not very useful in this chapter. The hyperlevel $h$ 
%takes over its role 
proves more useful
(sometimes together with the type $t$). %Consequently, some definitions are varied:
\end{bem}

\begin{lem}[Hyperlevel of an edge is well-defined]\label{hyperlevel of an edge is unique}\index{Hyperlevel of an edge is unique}
IsoCoChange implies that the definition of the hyperlevel of an edge $h\Eref S:=hS$ (Definition \ref{def:hyperlevel of an edge})
%on page \pageref{def:hyperlevel of an edge} 
is unique. In the Chebyshev lattice of an 
%including 
%admissible
initial
simplex 
$S\in\T_0$, the length of a refinement edge $e$ of a descendant of $S$ is $\|e\|_S=2^{-he}$.
\end{lem}
\begin{proof}
IsoCoChange implies that the metrics $\|\bullet\|_S,\|\bullet\|_T$ of the Chebyshev lattices of two simplices $S,T\in\Simplexe$ coincide on their intersection. So this edge length also coincides. The refinement edge of a simplex $S$ connects two distinct vertices of the $2^{-hS}$-cube $C_n$ of $S$, thus $\|e\|_S=2^{-hS}=2^{-he}$.
\end{proof}

\subsection{A $\ar\Edges$-sequence sticks to a $2^{-k}$-closed set.}
%{Proof of Proposition \ref{pps:klebt}}
\label{sec:e->f}
The purpose of this subsection is to prove Proposition \ref{pps:klebt} 
%(see above)
above.
Unlike the output of the AGK algorithm, $\T_0$ consists of hyperlevel 0 simplices in this proposition. This will be minded when the proposition is used in the proof of Theorem \ref{milestone}.
\begin{defn}[$\frac1N$-closed,$\frac1N$-closure]\label{def:frac1N-closed}
\index{1N@$\frac1N$-closed}\index{closed@$\frac1N$-closed}
\index{1N@$\frac1N$-closure}\index{closed@$\frac1N$-closure}
%\begin{small}
%(The definitions of Chebyshev lattices and $\frac1N$-subcubes in them are given in the 
%subsequent
%preliminary section below. For the moment, a $\frac1N$-Chebyshev lattice (with $N\in\N_{\geq 1}$) can be imagined as $\frac1N \Z^n$ in 
%%a reference simplex 
%reference coordinates
%and a subcube in that Chebyshev lattice as a subcube of a cube in $\R^n$ with vertices from 
%%the Chebyshev lattice
%$\frac1N\Z^n$.)
%\end{small}

A subset $U\subset\Omega$ is \emph{$\frac1N$-closed }in $\T_0$, if for each $S\in\T_0$, there is a union $U_S$ of $\frac1N$-subcubes in the Chebyshev lattice $(Z_S,{\|\bullet\|_S})$ such that $U\cap S=U_S\cap S$. Since closed sets are considered only in $\T_0$, the attribute ``in $\T_0$'' is omitted.

The \emph{$\frac1N$-closure }of a set $A\subset\Omega$ is the smallest $\frac1N$-closed set containing $A$.
\end{defn}
Figure \ref{fig:closedsetpicture} provides an example.

\begin{figure}[hbt]
\centering{
\begin{tikzpicture}
\tikzmath{
	coordinate \A,\B,\Ce,\D,\e,\E,\F,\G,\H,\I,\J;
	\A=(0,0);
	\B=(2,2);
	\Ce=(0,2);
	\e=.25*(\A)+.75*(\Ce);
	\E=(.7,1.3);
	\F=(1,3.5);
	\G=.5*(\B)+.5*(\Ce);
	\H=.5*(\B)+.5*(\F);
	\I=.5*(\Ce)+.5*(\F);
	\J=.5*(\F)+.5*(\G);
}
\def\Dreieck{
%node[below]{$A$}
\clip[draw] (\A) -- (\B) -- (\Ce) -- cycle;
\draw[gray!150,semitransparent] (\A) grid (\B);
}
\def\rechtes{
\begin{scope}
\fill[gray!30] (0,1) rectangle (1,2);
\draw 	(\A)node[below]{$a$} 
		(\B)node[right]{$b$}
		(\Ce)node[above]{$c$}
		(\E)node{$B$}
		(\e)--++(.1,0)node[right]{$e$};
\Dreieck
\end{scope}
}
%%rechtes Dreieck
\rechtes

%%Linkes Dreieck
\begin{scope}[cm={.5,.5,-.5,.5,(0,0)}]
\draw (\Ce)node[left]{$d$}
		(\E)node{$A$};
%%drittes Dreieck
\draw[dashed] (\B) -- (\F) -- (\Ce);
\draw[dashed,gray!96] (\H) -- (\G) -- (\I);
\draw (\J) node {$C$};
%%Linkes Dreieck
\Dreieck
\fill[pattern=dots] (1,1) rectangle (2,2);
\end{scope}

%%Beschriftung
\draw (-1,-.5)node[below right]{T-arrays: $A=
\begin{pmatrix}
a&d&c
\end{pmatrix}, B=
\begin{pmatrix}
a&b&c
\end{pmatrix}$};
%\begin{scope}[xshift=4cm]
%\rechtes
%\begin{scope}[cm={.5,.5,-.5,.5,(0,0)}]
%\draw 	(\Ce)node[left]{$d$}
%		(\E)node{$A$};
%\Dreieck
%\fill[semitransparent,gray!60] (1,1) rectangle (2,2);
%\end{scope}
%\end{scope}
\end{tikzpicture}
}
\caption{Examples for a not $\frac12$-closed and a $\frac12$-closed set in $\{A,B\}$: Let the T-arrays have hyperlevel 0, so the drawn grids in them are clips of $\frac12$-Chebyshev lattices. The grey set 
%in the left picture 
is not $\frac12$-closed, because its intersection with the left triangle $A$, the line segment $e$, is ``diagonal'' in the Chebyshev lattice of $A$ and hence not a subcube in 
%the Chebyshev lattice of $A$
it. To make it $\frac1N$-closed, the small 
%grey 
dotted
triangle 
%at $c$ 
(which is the intersection of a $\frac12$-cube with $A$) has to be added. Note that there could be a further triangle $C$ touching $A$ at $\overline{cd}$ such that the dotted triangle is not closed in $C$ and so on.
%, as shown in the right picture.
}
\label{fig:closedsetpicture}
\end{figure}
\begin{bem}
For naturals $N,M\in \N_{\geq 1}$, $\frac 1N$-closed sets are also $\frac 1{MN}$-closed, because a $\frac1N$-subcube is a union of $\frac1{MN}$-subcubes.
\end{bem}

%The $\frac1N$-closure of a set $A\subset\Omega$ is the smallest $\frac1N$-closed set containing $A$.
%\end{defnonum}
As it is familiar from $\ar0,\ar1$ etc., now a relation $\arsquare{\frac1N}$ is defined, such that the $\frac1N$-closed sets are the $\arsquare{\frac1N}$-closed sets. 
\begin{defn}[$\arsquare{\frac1N}$]
\index{->1N@$\arsquare{\frac1N}$}
%%falsch
%Let $(Z,\|\bullet\|)$ be a 
%%$\frac 1M$
%1-Chebyshev lattice on an affine space $R$. %For a multiplier $N$ of $M$, let 
Let 
$\arsquare{\frac1N}$ be the following relation on 
%$R\times R$: 
$\Omega\times\Omega$:
\begin{align*}
q\arsquare{\frac1N} p\quad:\LRq 	&\text{Each $\frac1N$-closed set containing $q$ contains also $p$.}
%\\
%									&\text{\emph{Or, equivalently:} 
%FALSCH							Each $\frac1N$-subcube containing $p$ contains also $q$.}
%There is a simplex $S\in\T_0$ such that each $\frac1N$-subcube of $S$ containing $p$ contains also $q$.}
\end{align*}
\end{defn}
\begin{bemn}
\begin{enumerate}
\item
If $q$ lies in a simplex $S\in\T_0$ and every $\frac1N$-subcube of $S$ containing $q$ contains also $p$, then $q\arsquare{\frac1N} p$.
\item
$\arsquare{\frac 1N}$ is reflexive and transitive.
\item
%For a factor $M$ of $N$, it holds that $q\arsquare{\frac 1N}p\quad\Rq q\arsquare{\frac 1M}p$.
For naturals $M,N\in\N_{\geq 1}$, it holds that $q\arsquare{\frac 1{MN}}p\quad\Rq q\arsquare{\frac 1N}p$.
\end{enumerate}
\end{bemn}
\begin{lem}\label{lem:Mittenspringen}
If $\tilde S=
\begin{pmatrix}
p_0\\
\vdots\\
p_n
\end{pmatrix}
$
is a reference simplex of type 0, then for all indices $0\leq i\leq j\leq n$ it holds that
%\begin{align*}
$
p_j\arsquare{2^{-h\tilde S}} p_i\text{, moreover }\fa k\leq h\tilde S.p_j\arsquare{2^{-k}} p_i.
$%\end{align*}
\end{lem}
\begin{proof}
$p_j$ is the centre and therefore a relatively interior point of the cube $C_j$ of $\tilde S$. %Note that there is a tessellation of $R$ into $2^{-h\tilde S}$-cubes. According to \ref{item:konvexeTeilmenge} of the definition of a face, each subcube of a subcube containing $p_j$ contains 
Hence, each $2^{-h\tilde S}$-subcube in the Chebyshev lattice of 
%$Z$
$\tilde S$
containing $p_j$ includes $C_j$, while $C_j$ includes $C_i$ and hence contains $p_i$.
\end{proof}
Recall Proposition \ref{pps:klebt}: 
\emph{Assume that $\T_0$ satisfies ReTaHyCo and consists of hyperlevel $0$ simplices solely. Let $f\ar\Edges e$ for admissible edges $e$ and $f$ 
of hyperlevel $\geq 1$, 
%then
%\begin{enumerate}
%\item\label{item:nonincreasing}
%$hf\leq he$
%\item
\label{item:klebt}
%Let 
$k\leq he{-}1$ be an integer and $U\subset \Omega$ be $2^{-k}$-closed in $\T_0$. 
%Then
%\begin{align*}
%\Vertices f\cap U\neq \leer\quad\Ra\quad \Vertices e\cap U\neq\leer.
%\end{align*}
If a vertex of $f$ lies in $U$, then a vertex of $e$ also does.}
\begin{proof}[Proof of Proposition \ref{pps:klebt}]
%Let $f\ar\Edges e$. 
According to the definition of $\ar\Edges$ on page \pageref{ar}, there exists an admissible simplex
%Let 
$S\in\Simplexe$ 
%be an admissible simplex 
with $f=\Eref S, \allowbreak e=\Eref\pa S$.
%, according to Lemma \ref{lem:Mittler}. 
Consider the following in the Chebyshev lattice of the root of $S$.

\emph{1st case: $t(S)\in\{1,\dots,n{-}1\}$.} Let $\tilde S=
\begin{pmatrix}
p_0\\
\vdots\\
p_n
\end{pmatrix}
$
be the ancestor of $S$ of hyperlevel $hS{-}1$ and type 0. (%According to the assumptions of the proposition, 
Since 
$he\geq 1$, 
%so 
there exists
such an ancestor% 
%exists
.) 
%so the 
The assumption $k\leq he{-}1$ of the proposition turns into $k\leq h\tilde S$. 
The horizontal part of $S$ is a centrepiece of $\begin{pmatrix}
p_0&\dots&p_n
\end{pmatrix}$.
%$\Vertices e=:\{p_i,p_j\}$
Let $p_i$ and $p_j$ (with $i<j$) 
%are 
be the vertices of $f=\Eref S$, the ends of that centrepiece. Let a $2^{-k}$-closed set $U$ contain one of them. 
%Let $\Vertices e=:\{p_i,p_j\}$ (with $i<j$) intersect a $2^{-k}$-closed set $U$. 
%Lemma \ref{lem:Mittenspringen} ensures that $p_j\arsquare{2^{-k}} p_i$, so $p_i\in U$. 
There are two 
possibilities,
%possible cases, 
what $e=\Eref\pa S$ could be: $\Vertices e=\{p_{i-1},p_j\}$ or $\Vertices e=\{p_i,p_{j+1}\}$. According to Lemma \ref{lem:Mittenspringen}, 
%ensures that 
the vertex in $\Vertices f\cap U$ $\arsquare{2^{-k}}$-demands the 
%$p$ 
vertex of $e$
with the minimal index (i.e.\ $p_{i-1}$ in the first and $p_i$ in the second case)  
%%both 
%$p_i\arsquare{2^{-k}} p_{i-1}$, $p_j\arsquare{2^{-k}} p_{i}$ and $p_j\arsquare{2^{-k}} p_{i-1}$. I.e.\ each vertex of $e$ $\arsquare{2^{-k}}$-demands $p_{i-1}$ in the first case and %both vertices $\arsquare{2^{-k}}$-demands 
%$p_{i}$ in the second case 
to lie in $U$
%if it does itself
as well. %Hence, $\Vertices e\cap U\neq\leer$.

\emph{2nd case: $t(S)\in\{0,n\}$.}
Let $\pa (S)=:
\begin{pmatrix}
p_0\quad p_1\\
\vdots\\
p_n
\end{pmatrix}$. Then $\Vertices e=\Vertices\Eref\pa S=\{p_0,p_1\}$ and $\Vertices f=\{p_0,p_n\}$ or $\Vertices f=\{p_1,p_n\}$. The vertex $p_n$ is the centre of the $2^{-he}$-cube $C_n$ for $\pa (S)$, which has $p_0$ and $p_1$ as vertices, therefore $p_n$ $\arsquare{2^{-he}}$-demands $p_0$ and $p_1$. %As its centre, $p_n$ is not contained in any subcube of $C_n$ 
Consequently, for a vertex $q$ of $f$ supposed to lie in $%\Vertices f\cap 
U$ (namely $p_0,p_1$ or $p_n$), one of the vertices $p$ in $\{p_0,p_1\}=\Vertices e$ fulfils $q\arsquare{2^{-he}} p$. 
%Additionally
%, with $he\leq hf{+}1$
%it can be implied that
%$p\arsquare{2^{-hf}} q\quad\Rq 
%p\arsquare{2^{-(he-1)}} q\quad \Rq 
Furthermore, this implies $q\arsquare{2^{-k}} p$, which shows that also $p\in U$%, therefore $\Vertices e\cap U\neq \leer$
.
\end{proof}
\subsection{The $\frac1N$-closure and tower layers are included in patches}\label{sec:imPatch}
This subsection is devoted to prove Proposition \ref{pps:imPatch} and Theorem \ref{milestone}.
\subsubsection{The $\frac1N$-closure is included in a patch}
Suppose a fixed regular initial triangulation $\T_0$ of 
type $n$ 
hyperlevel 0 simplices satisfying IsoCoChange throughout this subsubsection unlike the output of the AGK algorithm.

%%Here were Proposition \ref{pps:imPatch}
\begin{defn}[$\CoN$]
\index{C1@$\CoN$ -- set of the $\frac1N$-subcubes}
Let $\Cubes_{\frac1N}S$ be the set of non-empty $\frac1N$-subcubes in the Chebyshev lattice of a hyperlevel 0 simplex $S\in\Simplexe$.%, let $\Cubes_{\frac1N}\T_0$ be the set of these $\frac1N$-subcubes in the lattices of all simplices

\begin{bem}
Note that the elements of $\Cubes_{\frac1N}S$ are not necessarily subsets of $S$.
\end{bem}

%For such a subcube $C$, let $S_C$ be the simplex such that $C\in\Cubes_{\frac1N}S$. (Distinguish the subcubes of two Chebyshev lattices of two distinct simplices from each other, even if the Chebyshev lattices coincide, such that $S_C$ is unique.)

\end{defn}
%\begin{defn}[The hyperface ${HF}(S,p)$ opposite to $p$]
%For a given simplex $S$ and a vertex $p$ of it, let $HF(S,p):=\conv\left(\Vertices(S)\setminus \{p\}\right)$ be the hyperface opposite of $p$ in $S$.
%\end{defn}
\begin{defn}[hyperplane ${HP}(S,p)$ spanned by the hyperface opposite to $p$]
\index{HP@${HP}(S,p)$ -- hyperplane opposite to $p$}
For a given simplex $S$ and a vertex $p$ of it, let $HP(S,p):=\aff\left(\Vertices(S)\setminus \{p\}\right)$ be the hyperplane spanned by the hyperface of $S$ opposite to $p$.
\end{defn}
Recall the definition of the interior of a patch:
$
\Int \T p:=\Omega\setminus \bigcup\left(\T\setminus\T p\right).
$
%\begin{bem}
A subset of 
this
%it 
is given by
%\begin{align*}
%\bigcup\left\{S\setminus H_p~\middle|~S\in\T p, H_p \text{ is the hyperface of $S$ opposite $p$}\right\}.
%\end{align*}
$\bigcup_{S\in\T p}S\setminus HP(S,p)$ obviously. 
%with $HF(S,p)$ being the hyperface of $S$ opposite to $p$.
%\end{bem}

\emph{Idea of the proof of Proposition \ref{pps:imPatch}.}
For each vertex $p\in\Vertices\T_0$, a 
set 
%union $U_p$
of $\frac1N$-subcubes $\Cubes_p\subset \CoN(\T_0 p)$ of the Chebyshev lattices of the simplices in $\T_0 p$ is going to be assigned
to $p$%
, such that %$\bigcup \left\{C\cap S_C$ (i.e.\ 
the union $U_p$ of these subcubes, %$\Cubes_p$ is consisting of) 
each intersected with its simplex,
%\begin{itemize}
%\item
is $\frac1N$-closed,
%\item
%$\bigcup \Cubes_p\subset\Int\T_0 p$ (it is contained in the interior of a patch.
contained in the interior of the patch of $p$
%$\bigcup\left\{\bigcup \Cubes_p~\middle|~p\in\Vertices\T_0\right\}=\Omega$ 
%\item
and these unions $U_p$ 
%of those subcubes 
for all of the vertices in $\Vertices\T_0$ altogether cover $\Omega$.
%\end{itemize}

The cubes of $\Cubes_p$ are chosen by the following rule.	
\begin{defn}[relation $r_S$]
\index{r@$r_S$}
%Let $r$ be the following relation on $\Vertices \T_0\times \Cubes_{\frac1N}\T_0$:
For 
%an initial simplex $S\in\T_0$, 
%a tagged simplex $S$,
a T-array $S$ of hyperlevel 0, %or 1,
let $r_S$ be the following relation on $\Vertices S\times \Cubes_{\frac1N}S$:
\begin{align*}
r_S&:=\left\{(p,c)~\middle|~
c\cap HP(S,p)=\leer
%c\subset \left(C_n^S\setminus HP(S,p)\right)
\right\},
\\
\text{so}\quad p\,r_Sc&:\Leftrightarrow~ c\text{ 
%lies in the subcube $C_n^S$ and 
does not intersect the hyperplane $HP(S,p)$.}
\end{align*}
%For $HF_p$ being the Hyperface opposite to $p$.
\end{defn}
\begin{defn}[$\Cubes_p,U_p$]\index{U@$U_p$}\index{C@$\Cubes_p$}

For an initial vertex $p\in\Vertices\T_0$, the set $\Cubes_p\subset\Cubes_{\frac1N}(\T_0p)$ of $\frac 1N$-subcubes %in the simplices of the patch $\T_0 p$ of $p$, 
and the set $U_p\subset\Omega$ are
\begin{align*}
\Cubes_p:=&\,\bigcup_{S\in\T_0p}\left\{c\in\Cubes_{\frac1N}S ~\middle|~ p\,r_Sc\right\},\\
%\left\{c~\middle|~\ex S\in\T_0 p.c\in\Cubes_{\frac1N}S \text{ with } p\,r_Sc\right\}\\
U_p:=&\,\bigcup\left\{c\cap S~\middle|~
%c\in \Cubes_p,S\in\T_0 p \text{ with } c\in\Cubes_{\frac1N}S\right\}\\
S\in\T_0 p,c\in\Cubes_{\frac1N}S \text{ with } c\in \Cubes_p\right\}\\
=&\,\bigcup\left\{c\cap S~\middle|~S\in\T_0 p,c\in\Cubes_{\frac1N}S \text{ with } p\,r_Sc\right\}.
\end{align*}
\end{defn}
%$\Cubes_p$ will not be used subsequently.
Only $U_p$ will be used in the following.
%\begin{bem}
%It does not matter whether to add or to omit the condition that $c$ lies in $C_n^S$ because afterwards it is intersected with $S$. $S$ is a subset of $C_n^S$ but the intersection of any $\frac1N$-subcube $c$ with $C_n^S$ is a $\frac1N$-subcube included in $C_n^S$.
%\end{bem}
\begin{lem}\label{vertexrelationcube}
The relation $r_S$ and the set $U_p$ have the following properties:
\begin{enumerate}
\item\label{item:patch}
$U_p\subset\Int\T_0 p$,
\item\label{item:every}
%Let $S$ be a 
%%tagged hyperlevel 0 simplex.
%hyperlevel 0 T-array.
%\begin{enumerate}
%\item
%If $N>2n$% is even
%, %then for any hyperlevel $0$ simplex $S$
%%$\fa S\in\T_0~
%$
%\fa c\in\CoN S\;\allowbreak\ex p\in\Vertices S\,.\,p\,r_S c$.
%\item
%If $t(S)=n$, even $N>n$ suffices.
%\end{enumerate}
%%
%For 
%a 
%type $n$
%hyperlevel 0 T-array $S$ and an integer $
$\fa S\in\T_0~\fa
N>n%$, $
~\fa c\in\CoN S~\allowbreak\ex p\in\Vertices S\,.\,p\,r_S c$.
\item\label{item:closed}
%$N\in 2\N_{\geq 1}$
%Assume that $\T_0$ satisfies IsoCoChange, $h(\T_0)=0$ and $N$ is even. Then
%For even $N$, 
$U_p$ is $\frac1N$-closed in $\T_0$.
\end{enumerate}
\end{lem}
\begin{proof}[Proof of Lemma \ref{vertexrelationcube}.\ref{item:patch}]
%$U_p$ consists only of subsets of simplices meeting at $p$, but is disjoint from any hyperface opposite to $p$. 
%%The remark on Definition \ref{def:int patch} of the interior of a patch 
%This
%implies $U_p\subset\Int\T_0 p$.
By definition of $U_p$,
$%\begin{align*}
U_p\subset\bigcup_{S\in\T_0 p} S\setminus HP(S,p)\subset \Int \T_0 p.
$%\end{align*}
\end{proof}
\begin{proof}[Proof of Lemma \ref{vertexrelationcube}.\ref{item:every}]
In the representative 1-Chebyshev lattice $\left(%\R^n,
\Z^n,\|\bullet\|\right)$, a non-empty $\frac1N$-subcube is given by 
\begin{align}
c=\prod_{i=1}^n \underbrace{[a_i,b_i]}_{=:I_i}
\text{ with }
a_i,b_i\in\frac 1N\Z, 
|I_i|=b_i-a_i\in\left\{0,\frac1N\right\}\label{intervallength}.
\end{align}

The vertices of a%
n $(n,0)$ reference simplex 
$S=
%\begin{pmatrix}
%p_0&\dots&p_k\\
%&\vdots&\\
%&p_n
%\end{pmatrix}
\begin{pmatrix}
p_0&\dots&p_n
\end{pmatrix}
$
%of hyperlevel 0 and type $k$ 
in $\Z^n$
are given for example by
\begin{align*}
p_j=\sum_{i=1}^j e_i.
%\begin{cases}
%&\sum_{i=1}^j e_i \text{ for $j\leq k$}\\
%\frac12&\sum_{i=1}^j e_j\text{ for $j\geq k{+}1$}.
%\end{cases}
\end{align*}
%An $(n,0)$ reference simplex 
%in $\Z$
%is given for example by
%\begin{align*}
%S=
%\begin{pmatrix}
%\sum_{i=1}^0 e_i&\dots&\sum_{i=1}^n e_i
%\end{pmatrix}.
%\end{align*}
%Since all 
All
reference simplices in arbitrary 1-Chebyshev lattices are lattice-isometric, 
so
regarding only this one does not 
%lessen 
affect
generality.
Using $x_0:=1$ and $x_{n+1}:=0$
as auxiliary coordinates, 
the hyperplanes $HP_j
:=HP(S,p_j)
$ 
%spanned by the hyperfaces $HF(S,p_j)$
 opposite to $p_j$ are given by
\begin{align}
%HP_0&=\left\{x\in\R^n~\middle|~x_1
%%+x_{k+1}
%=1\right\}\text{ (with $x_{n+1}:=0$)} \label{HP0}\\
%HP_1&=\left\{x\in\R^n~\middle|~x_1=x_2\right\}\\
%&~\vdots\\
%HP_{n-1}&=\left\{x\in\R^n~\middle|~x_{n-1}=x_n\right\}\\
%HP_{n}&=\left\{x\in\R^n~\middle|~x_n=0\right\}
HP_j=\left\{x=(x_1,\dots,x_n)\in\R^n~\middle|~x_j=x_{j+1}\right\} 
\end{align}
for $j\in\{0,\dots,n\}$
%with $x_0:=1,x_{n+1}:=0$.
(which can 
%easily 
be verified by checking that $HP_j$ contains all vertices except $p_j$).
The simplex $S$ itself is the intersection of closed half spaces bounded by these hyperplanes $HP_j$ and also containing $p_j$. 
Therefore, it is given by
%on page 5
\begin{align*}
S=\left\{x
%=(x_1,\dots,x_n)
\in\R^n~\middle|~
%x_1\geq\dots\geq x_n\geq 0 \text{ and } x_1+x_{k+1}\leq 1
1\geq x_1\geq\dots\geq x_{n}\geq 0
\right\}. 
%\text{ (with $x_0:=1$, $x_{n+1}:=0$)}.
\end{align*}

If the subcube $c=\prod_{i=1}^n {[a_i,b_i]}=\prod_{i=1}^n I_i$ from above 
%does 
did
not satisfy any of the relations $p_jr_S c$, it 
would
intersect all these hyperplanes%
. 
%and the following %sequence of 
%inequalities can be derived:
%The following calculation is illustrated in Figure \ref{fig:cubeinsimplex}.
%
%\centering{
%\begin{SCfigure}[10][ht]
%\begin{tikzpicture}[scale=3]
%\draw {(0,0)node[left]{0} -- ++(1,0)node[right]{1}};
%\draw (0,-.03) -- ++(0,.06);
%\draw (1,-.03) -- ++(0,.06);
%
%\tikzmath{\abstand=.3;\abstandeins=.2;}
%\foreach \n in {1,2,3}
%	{
%	\tikzmath{\x=3-\n;
%				\xstart=.3*\x;
%				\ystart=\abstand*\x+\abstand;}
%	%waagerechte Striche
%	\draw (.3*\x, \ystart)node[left]{$a_\n$} --node[above]{$I_\n$} ++(0.33,0)node[right]{$b_\n$};
%	%Intervallstriche
%	\draw (.3*\x,\ystart-.03) -- ++(0,.06);
%	\draw (.33+.3*\x,\ystart-.03) -- ++(0,.06);
%	%Hilfslinie
%	\draw[dotted] (\xstart,\ystart) -- ++(0,-\abstand);
%	}
%\end{tikzpicture}
%\caption{Graphical derivation for Inequality \eqref{ineq:3}: If the cube $I_1\times I_2\times I_3$ touches all hyperfaces $\{x_3=0\},\{x_{2}=x_3\},\{x_1=x_2\}$, then $a_3\leq 0$ and the intervals must overlap as depicted, hence $b_1\leq |I_1|+|I_2|+|I_3|$.}
%\label{fig:cubeinsimplex}
%\end{SCfigure}
%}

%For the indices $k\in \{%1
%0,\dots,n
%%{-}1
%\}$,
The intersection $c\cap HP_j$ is not empty if and only if $I_j\cap I_{j+1}\neq \leer$ (with $I_{0}:=[1,1],I_{n+1}=[0,0]$), 
%the cube $c$ intersects the hyperplane $HP_n$, if $I_k$ intersects $I_{k+1}$,
which implies $a_j\leq b_{j+1}$.
%; ${c\cap HP_n}$ is not empty if and only if $0\in I_n$, which implies $a_n\leq 0$. 
The conjunction of all 
%this yields
these inequalities would lead to
%\begin{align}
%b_k&=\underbrace{(b_k-a_k)}_{|I_k|}+\underbrace{(a_k-b_{k+1})}_{\leq 0}+\underbrace{(b_{k+1}-a_{k+1})}_{|I_{k+1}|}+\dots+\underbrace{(b_n-a_n)}_{|I_n|}+\underbrace{(a_n-0)}_{\leq 0}\\
%&\leq |I_k|+\dots+ |I_n|.
%\end{align}
\begin{align}
%b_k&=\underbrace{b_k-a_k}_{|I_k|}+\underbrace{a_k-b_{k+1}}_{\leq 0}+\underbrace{b_{k+1}-a_{k+1}}_{|I_{k+1}|}+\dots+\underbrace{b_n-a_n}_{|I_n|}+\underbrace{a_n-0}_{\leq 0}\\
%&\leq |I_k|+\dots+ |I_n|.\label{ineq:3}
1=a_0-b_{n+1}=&\underbrace{a_0-b_1}_{\leq 0}+\underbrace{b_{1}-a_{1}}_{|I_{1}|}+\dots+\underbrace{b_n-a_n}_{|I_n|}+\underbrace{a_n-%\overbrace{
b_{n+1}%}^{0}
}_{\leq 0}\\
&\leq |I_1|+\dots+ |I_n|\leq \frac nN,\label{ineq:3}
\end{align}
%\begin{alignat}3
%c\cap HP_n\neq\leer &\Ra a_n\leq 0, &b_n&=a_n+|I_n|\leq |I_n|\label{ineq:1}\\
%c\cap HP_{n-1}\neq\leer &\Ra a_{n-1}\leq b_{n}\stackrel{\ref{ineq:1}}{\leq} |I_n|,&b_{n-1}&=a_{n-1}+|I_{n-1}|\leq |I_{n-1}|+|I_n|.\\
%\vdots&&&\vdots\\
%%c\cap HP_k\neq \leer
%\ldots&&b_k&\leq|I_k|+\dots+|I_n|\label{ineq:3}.\\
%\vdots&&&\vdots\\
%c\cap HP_1\neq\leer&\Ra a_1\leq b_2\stackrel{\ref{ineq:3}}{\leq} |I_2|+\dots+|I_n|, &b_1&\leq |I_1|+\dots+|I_n|.\label{ineq:4}
%\end{alignat}
%Additionally, $c\cap HP_0\neq\leer$ implies
%\begin{align}
%1\stackrel{\eqref{HP0}}{\leq} b_1+b_{k+1}\stackrel{\ref{ineq:3}
%%\ref{ineq:4}
%}{\leq}2\left(|I_1|+\dots+|I_n|\right)\stackrel{\ref{intervallength}}{\leq} \frac{2n}{N}
%\label{eq:HP0}.
%\end{align}
%
%This is false for all $N>
%%2
%n$, so all $\frac1N$-cubes are related by $r_S$ to at least one vertex.
contradicting $N>n$.
%
%In case of a type $n$ simplex, inequality \eqref{eq:HP0} reads
%\begin{align*}
%1\leq b_1\leq |I_1|+\dots +|I_n|\leq \frac nN.
%\end{align*}
%%so 
%This implies that it is sufficient to assume $N> n$.
\end{proof}
Proving Lemma \ref{vertexrelationcube}.\ref{item:closed} requires the most extensive investigation of reference simplices and their relation %to the unit cube
between subsimplices, subparallelotopes and sublattices among all theory here. Some preparation seems convenient.

%, namely that $U_p$ is $\frac1N$-closed, some preparations are required:

%This statement 
Let $p\in\Vertices\T_0$. It must be shown that 
%for an arbitrary $S\in\T_0$, 
the intersection $S\cap U_p$ is a union of $\frac 1N$-subcubes in the Chebyshev lattice of $S$
for every $S\in\T_0 p$. 
Have a closer look at the set of interest
\begin{align}
%do not change U_p and S!
U_p\cap S=\bigcup\left\{c_T\cap T\cap S~\right|~\underbrace{T\in \T_0 p,c_T \text{ with } 
p\,r_Tc_T
%c_T\cap HP(T,p)=\leer
}_{(A)}\}. \label{replace other}
\end{align}
The aim is to show that in this union, the intersections $S\cap c_T$ of $S$ with subcubes of other simplices $T\neq S$ are already included in subcubes of $S$ itself, such that in the %expression 
condition
$(A)$ in equation \eqref{replace other} one can restrict to $T=S$ without %lessening 
reducing
the 
%set. 
union. It follows that $S\cap U_p=S\cap \bigcup \{c~|~p\,r_Sc\}$ is a union of $\frac1N$-subcubes of $S$, i.e.\ a $\frac1N$-closed set.

For this aim, the relation $r_S$ is studied on subcubes of \emph{subsimplices}.
%\begin{defn}[stretched subcube]\index{stretched subcube}
\begin{defn}[stretch of a subcube in a sublattice]\index{stretch}
Let $(R,Z,\|\bullet\|)$ be an $n$-dimensional Chebyshev lattice with an $m$-dimensional sublattice $U$ 
with the same lattice width. Let $c_U$ be a $\frac1N$-subcube in $U$. Then the 
%\emph{subcube $c_Z$ stretched by $c_U$ in $Z$}
\emph{stretch 
%$c_Z$ 
of $c_U$ in $Z$}
is the %$\frac1N$-cube
%\begin{align}
%\pi_1c_U\times\dots
%\end{align}•
smallest subcube in $Z$ including $c_U$. For $(R,Z,\|\bullet\|)=(\R^n,\Z^n,\|\bullet\|_\infty)$,
%in $\R^n$,
it is
given by
\begin{align*}
\pi_1c_U\times\dots\times \pi_n c_U
\end{align*}
(for $\pi_j: (x_1,\dots,x_n)\mapsto x_j$ being the $j$-th coordinate).

%This notion is used also with respect to spanning subsets of $U$ and $R$: 
%For a simplex $S$, the \emph{stretch in $S$} 
The \emph{stretch in a simplex $S$}
means the 
%subcube stretched 
stretch
in the Chebyshev lattice of $S$ etc.
\end{defn}
Figure \ref{fig:stretch} provides an example for the stretches of subcubes.
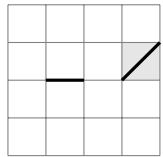
\begin{SCfigure}[4.1][ht]
\begin{tikzpicture}[scale=.5]
\fill[gray!20,very thin] (3,2) rectangle (4,3);
\draw[gray] (0,0) grid (4,4);
\draw[very thick] (3,2) -- (4,3);
\draw[very thick] (1,2) -- (2,2);
\end{tikzpicture}
\caption{The stretch of the diagonal line segment is the grey square, while the stretch of the horizontal line segment is the line segment itself, because it is already a subcube in the lattice.}
\label{fig:stretch}
\end{SCfigure}

%The following lemma will be lead back to Lemma \ref{core}.
\begin{lem}\label{relation on subsimplices}
Let $c_U$ be a $\frac1N$-subcube in the Chebyshev lattice of a subsimplex $U$ of a% hyperlevel 0 
n $(n,0)$
reference simplex $S$. 
%(again with an even $N$). 
Let $c_S$ be the 
%subcube stretched by 
stretch of 
$c_U$ in $S$, $p$ be a vertex of $U$ and $r_S$ and $r_U$ be the above defined relations between vertices and subcubes in $S$ and $U$ respectively.
Then
\begin{align*}
p\,r_S c_S\quad\Longleftrightarrow\quad p\,r_Uc_U.
\end{align*}
\end{lem}
%%Here was the beginning of Lemma \ref{relation on subcubes} and the Verweis on Lemma \ref{core}.

%Consequently, the lemma follows from the following, slightly more general %(namely because of the ``$\geq$'') 
%lemma.
%The differences are that $H$ can equal $S$ and that $U$ may be a subset of $H$.
%%, the core of this small theory here:
%\end{proof}
%\begin{lem}\label{core}
%%%$U\not\subset H$ nicht nötig
%%
%%Let $C_U$ and $C_H$ be $\frac1{2}$-cubes
%%parallel to the axis
%%in a %$\frac1$-
%%Chebyshev lattice $(R,Z,\|\bullet\|)$,
%Let $c_U$ be a $\frac1N$-subcube in the Chebyshev lattice of a subsimplex $U$ of a hyperlevel 0 reference simplex $S$. 
%%(again with an even $N$). 
%Let
%%$U$ and 
%$H$ be 
%a subsimplex
%%subsimplices 
%of %a reference simplex 
%$S$ 
%%%in $(R,Z,\|\bullet\|)$ 
%%in $(Z,\|\bullet\|)$ 
%with $\dim H\geq \dim S{-}1$ 
%%Furthermore, for an even number $N$, let $c_U$ be a $\frac1N$-cube in the sublattice $U\cap \refine\left( Z,2^{-hU}\right)$ 
%and $c_S$ be the 
%%subcube stretched by 
%%stretch of
%%$c_U$ in $Z$. 
%%and $c_S$ be the 
%%subcube stretched by 
%stretch of
%$c_U$ in $S$.
%Then %the following implication holds true:
%\begin{align*}
%c_S\cap \aff H\neq\leer\quad\LRq c_U\cap \aff H\neq\leer.
%\end{align*}
%\end{lem}

\begin{proof}
The definition of $r_S$ and $r_U$ turns the statement (with $H$ being the hyperface opposite to $p$) into this: For a hyperface $H$ of $S$ not including $U$ (this extra condition is equivalent to $p\in\Vertices U$), it holds that
\begin{align*}
c_S\cap\aff H\neq \leer\quad\LRq c_U\cap \aff H\neq \leer.
\end{align*}
%Note that $c_S,c_U$ are included in $C$ and $H$ is affinely closed in $C$ (see Lemma \ref{subparallelotopes affinely closed}), so $\aff H$ can be replaced by $H$.
Because of $c_U\subset c_S$ the implication `$\Leftarrow$' is trivial. It is left to prove ‘$\Rightarrow$’.
Since $S$ has only horizontal vertices,
Lemma \ref{referencesubsimplex}.\ref{subsimplex subparallelotope} yields a subparallelotope $P_H$ of a cube $C
=C^S_n
$, such that $\aff H=\aff P_H$. 
%where $C$ is either $C_n^S$, namely if $H$ contains a horizontal vertex of $S$; or the cube $C_n^{S'}$ of a descendant $S'$ of $S$ of hyperlevel $hS{+}1$, if $H$ contains only vertical vertices of $S$. % and type $n$
 According to the characterisation of hypersubparallelotopes, Lemma \ref{hypersubparallelotope}, two cases can occur.

\emph{1st case: $P_H$ is a hyperface of $C$.} 
W.l.o.g.\ say $C=[0,1]^n$, $\aff H=\{x\in \R^n~|~\allowbreak x_1=0\}$. Then 
\begin{align*}
c_S\cap\aff H\neq \leer \quad\LRq 0\in\pi_1 c_S 
%\stackrel{\text{def.\ \emph{stretch%ed
%}}}{\Longleftrightarrow} 
~\xLeftrightarrow[]{\text{def.\ \emph{stretch%ed
}}}~ 
0\in \pi_1 c_U\quad\LRq c_U\cap\aff H\neq \leer.
\end{align*}
\emph{2nd case: $P_H$ is not a 
%subcube 
hyperface 
of $C$.}

The projection $\pi$ from the characterisation of hypersubparallelotopes, Lem\-ma \ref{hypersubparallelotope}, yields 
%the following picture:
Figure \ref{fig:2squares}.
\begin{SCfigure}[.8][h!]
\centering{
\begin{tikzpicture}[scale=%.6
.3]
\tikzmath{\abstand=-3pt;
			\lang=6;
			\kurz=3;}
\begin{scope}			
%\clip (-.5,-.5) rectangle (\lang+2.2,\lang+1.2);
%\draw[gray!30,step=1.5] (-1.5,-1.5) grid (\lang+3,\lang+2);
\end{scope}
\draw (0,0) node[left]{$X$} -- ++(\lang,0)node[right]{$Y=\pi (\Vertices S\setminus \Vertices H)$} -- ++(0,\lang) node[right]{$Z$}-- ++(-\lang,0) node[left]{$W$}-- cycle;
\draw (-.5,-.5) --node[sloped,
%near start,
fill=gray!20]{$\pi \aff H$} 
%(\lang+1.5,\lang+1.5)
(\lang+.5,\lang+.5);
\draw(.3*\lang,.7*\lang) node{$\pi C$};
%\begin{scope}[shift={(4.7,3.7)}]
%\draw (0,0)node[below left=\abstand]{$x$} -- ++(\kurz,0)node[below right=\abstand]{$y$} --node[right]{$\pi c_S\;(=\pi c_U?)$} ++(0,\kurz)node[above right=\abstand]{$z$} -- ++(-\kurz,0)node[above left=\abstand]{$w$} -- cycle;
%\draw[dotted] (0,0)--node[sloped, near end]{$\pi c_U$?} ++(\kurz,\kurz) ++(-\kurz,0) --node[sloped,near end]{$\pi c_U$?} ++(\kurz,-\kurz);
%%\draw (.75*\kurz,.5*\kurz) ;
%\end{scope}
\end{tikzpicture}
\caption{
%Possible position of 
$\pi C
$ and $
\pi \aff H$. %$\pi c_S$ and $\pi c_U$. $\pi c_S$ is the square $\conv\{x,\allowbreak y,\allowbreak z,\allowbreak w\}$. The three possible positions of $\pi c_U$ are discussed below. \footnotesize{(Actually, $c_S$ is a $\frac1N$-subcube in the Chebyshev lattice of $S$, so $\pi c_S$ would have to ``fit'' into $\pi C$ as the small grey squares do. Anyhow, the reader is invited to follow the subsequent discussion imagining $\pi c_S$ in a more ``general'' position, as it is depicted.)}
}
\label{fig:2squares}
}
\end{SCfigure}
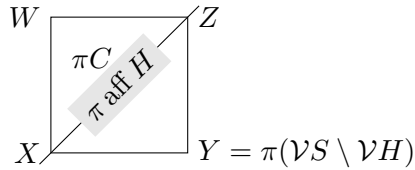

Lemma \ref{hypersubparallelotope} particularly says that $P_H=C\cap\pi\inv\pi P_H$ and thus $\aff(H) = \pi\inv\pi \aff H$. 
Suppose there is an $y\in\pi c_U\cap \pi(\aff H)$ with a preimage $x\in c_U$ under $\pi$. Then $x\in\pi\inv\pi \aff H=\aff H$ and thus $x\in c_U\cap\aff H$.
Hence, $\pi (\aff H)\cap\pi c_U\neq \leer$ already implies that $\aff (H)\cap c_U\neq \leer$ and it suffices to prove $\pi c_S\cap\pi\aff H\neq \leer ~\LR~ \pi c_U\cap\pi\aff H\neq \leer$.
%the former.

Since $S$ has full type, all vertices of $S$ are vertices of $C$ and projected to vertices of $\pi C$. $S$ has exactly one vertex which is not in $H$ and thus not projected to $\pi P_H$, but, say to $Y$ w.l.o.g. Since $U\not\subset H$, $U$ must contain this vertex. 
%Since $U$ has full type as well as 
Consequently, $\pi U$ can be $\overline {XY}$, $\overline {YZ}$ or $\triangle XYZ$. In all theses three cases, a subcube in the Chebyshev lattice of $U$ is already a subcube in the Chebyshev lattice of $\pi S$. 
%so the stretch $\pi c_S$ of $\pi c_U$ in $S$ equals $\pi c_U$ 
%
%For $c_U$ stretching $c_S$, 
%If the stretch of $c_U$ is $c_S$,
%necessarily 
%also 
Since $\pi c_S$ must 
also
be 
the stretch of $\pi c_U$ in $\pi S$, %must be 
%a stretching subparallelotope of 
%$\pi c_S$ 
$\pi c_S=\pi c_U$
and 
the proof is done.
\end{proof}

\begin{bem}
If the condition $\dim H
%\geq 
=\dim S{-}1$ is omitted in the proof, the statement does not hold true anymore: For instance (see Figure \ref{fig:dimH}), let
$S:=\begin{pmatrix}
a&b&c&d
\end{pmatrix}$
be a T-array of type 3, $U:=c_U:=
%\conv\{a,c\}
\overline{ac}$ and $H:=
%\conv\{b,d\}
\overline{bd}$. Then $c_S$ is the face $a+[0,1](b{-}a)+[0,1](c{-}b)$ of the cube and meets $H$ in $b$, while $c_U$ is disjoint from $\aff(H)$.
\end{bem}

\begin{SCfigure}[][ht]
%\centering{
\begin{tikzpicture}[scale=.7,
					information text/.style={fill=gray!10,inner sep=1ex}, 
					cube1/.style=dashed, 
					simplex/.style={very thick,dashed},
					diag/.style=dotted,
					canonical1/.style={very thick},
					canonical2/.style=double,
					scale=.7
					]
\tikzmath{
	coordinate \h,\x,\y,\A,\B,\Ce,\D,\E,\F,\G,\H,\mitte,\punkt;
	\h=(-.5,4);
	\x=(4.5,0);
	\y=(2,3);
	\A=(0,0);
	\B=(\A) + (\x);
	\Ce=(\A) + (\y);
	\D=(\B) + (\y);
	\E=(\A)+(\h);
	\F=(\B)+(\h);
	\G=(\Ce)+(\h);
	\H=(\D)+(\h);
	\mitte=.5*(\A) + .5*(\H);
	\punkt=.3*(\A) + .7*(\B) + .15*(\y)
;
}
%%Würfel
\begin{scope}[dotted]
\fill[gray!20] (\A) -- (\B) -- (\D) -- (\Ce) -- cycle;
\draw (\A) -- (\B) -- (\D) -- (\Ce) -- (\G) -- (\H) -- (\F) -- (\E) -- cycle;
\draw (\A) --++(\y);
\draw (\E) -- (\G);
\draw (\B) -- ++(\h);
\draw(\D) --++(\h);
\end{scope}
%%Tetraeder
\draw (\A)node[left]{$a$} -- (\B)node[right]{$b$} -- (\D)node[right]{$c$} -- (\H)node[above]{$d$} -- cycle;
\draw[very thick] (\A) --node[sloped,fill=white,pos=.4]{$U{=}c_U$} (\D);
\draw[very thick] (\B) --node[fill=gray!30,pos=.6]{$H$} (\H);
\draw (\punkt) node{$c_S$};
\fill (\B) circle [radius=.12];
\end{tikzpicture}
\caption{$\dim H
%\geq 
=\dim S{-}1$ is necessary in Lemma 
%\ref{core}
\ref{relation on subsimplices}.}
\label{fig:dimH}
%}
\end{SCfigure}
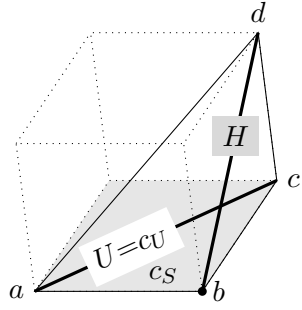

%Recall the following was wanted for simplifying 
The following lemma simplifies
\eqref{replace other} on page \pageref{replace other}.
\begin{lem}
For a simplex $S\in\T_0$, a vertex $p\in\Vertices S$, 
%an even 
$N\in\N_{\geq 1}$ and a $\frac1N$-cube $c_T$ in the Chebyshev lattice of another simplex $T\in\T_0p$ with $p\,r_Tc_T$ intersecting $S$, there is a $\frac1N$-cube $c_S$ in 
%the Chebyshev lattice of $S$ 
$Z_S$
with $p\,r_Sc_S$ such that $c_T\cap T\cap S\subset c_S\cap S$.
\end{lem}
\begin{proof}
%Since $N$ is even, $c_T$ is a subcube also in the refined Chebyshev lattice $\refine\left(Z_T,2^{-hU}\right)$ (because $2^{-hU}$ is 1 or $\frac12$). 
Let $U:=S\cap T, c_U:=c_T\cap\aff U$, which is a subcube in the sublattice $Z_U$ of 
%$\refine\left(Z_T,2^{-hU}\right)$. 
$Z_T$ according to Lemma \ref{lem:intersection of cube and sublattice}.
Furthermore, let $\tilde c_T$ and $c_S$ be the %cubes stretched by 
stretches of
$c_U$ in $T$ and $S$, respectively. 
Since $\tilde c_T\subset c_T$, $p\,r_Tc_T$ implies $p\,r_T\tilde c_T$. Then Lemma \ref{relation on subsimplices} deduces that $p\,r_Uc_U$ and also $p\,r_Sc_S$.
Finally, $c_T\cap T\cap S =c_T\cap\aff U\cap U= c_U\cap U\subset c_S\cap S$, as deserved.

\end{proof}
\begin{proof}[Proof of Lemma \ref{vertexrelationcube}.\ref{item:closed}]
For each $S\in\T_0$, the preceding lemma allows 
%to manipulate \eqref{replace other} as follows
to simplify \eqref{replace other}.
\begin{align*}
U_p\cap S=\bigcup\left\{c_T\cap T\cap S~\middle|~{T\in \T_0 p,p\,r_Tc_T}\right\}=\bigcup\left\{S\cap c_S~\middle|~p\,r_Sc_S\right\}. 
\end{align*}
Therefore, $U_p$ is $\frac1N$-closed.
\end{proof}
Recall Proposition \ref{pps:imPatch}, which is 
%going to be 
%proven 
concluded
%in this subsection:
now:
\emph{Suppose an initial triangulation $\T_0$ of type $n$ hyperlevel $0$ simplices. 
For every 
%even 
integer $N>
%2
n$, 
%the $\frac1N$-closure of an arbitrary point $x\in\Omega$ 
every point $x\in\Omega$ is contained in a $\frac1N$-closed set which is included in the interior $\Int
\T_0 
%\Simplexe^{n,1}
p$ of a patch for some $p\in\Vertices
\T_0
%\Simplexe^{n,1}
$. %If $\T_0$ consists of type $n$ simplices only, it suffices to assume an even $N>n$.
}
\begin{proof}
%[Proof for Proposition \ref{pps:imPatch}]
\label{pr:imPatch}
According to Lemma \ref{vertexrelationcube}, the sets $U_p$ cover all subcubes, thus cover $\Omega$, are contained in the interior of a patch and $\frac 1N$-closed each, i.e.\ each point is contained in a $\frac 1N$-closed set included in the interior of a patch%, as it has had to be proven.
. This concludes the proof.
\end{proof}
%\begin{proof}[Proof of Proposition \ref{pps:klebt}.\ref{item:nonincreasing}]
%Lemma \ref{lem:Mittler} directly implies Proposition \ref{pps:klebt}.1, because $hf:=h\pa S\leq hS=:he$.
%\end{proof}
\subsubsection{The stepping stone: Tower layers are included in patches}
Recall the stepping stone Theorem \ref{milestone}: 
\emph{\milestone2}
%\emph{
%Suppose 
%%IsoCoChange 
%ReTaHyCo
%for $\T_0$. Then
%\begin{enumerate}
%\item
%for every $j\in\N\geq 1$, the \emph{quasi-uniform refinement} 
%%$\Q^j \T_0$ 
%$\Simplexe^{n,j}$
%of $\T_0$, i.e.\ a refinement that consists of the simplices that have hyperlevel $j$ 
%%and type $n$.
%the quasi-uniform refinements 
%%$\Q^j\T_0%
%$\Simplexe^{n,j}%
%=\allowbreak\{S\in\Simplexe~|~hS=j,\allowbreak tS=n\}$ are regular.
%\item
%\towerstep2
%\end{enumerate}
%}

%Recall Theorem \ref{milestone}.\ref{it:quasi-uniform}: \emph{If $T_0$ satisfies IsoCoChange, 
%\quasiuniform2}
%
%for every $j\in\N$ there is an admissible quasi-uniform refinement 
%%$\Q^j \T_0$ 
%$\Simplexe^{n,j}$
%of hyperlevel $j$ 
%and type $n$
%simplices.}
\begin{proof}[Proof of 
%Theorem \ref{milestone}.
\ref{it:quasi-uniform}.]
To get 
%$\Q^{j+1}\T_0$,
$\Simplexe^{n,j+1}$,
%Simply 
simply
bisect all edges of hyperlevel $\leq j$ (as long as there are some). %According to the proposition, they cannot demand higher level edges. 
By the Definition \ref{ar} of $\ar\Edges$ and the well-definedness of $he$, edges never demand higher (hyper)level edges. Bisection of hyperlevel $\leq j$ edges results in hyperlevel $\leq j$ simplices (%in the normal binary forest of admissible simplices, where the type $n$ was excluded; 
because a simplex of type 0 does not have a refinement edge, but only its transposed which already has the next hyperlevel). On the other hand, the resulting triangulation does not contain any hyperlevel $\leq j$ edge anymore, so all simplices have hyperlevel $j$ and type 0. Transposing all T-arrays yields simplices of hyperlevel $j{+}1$ and type $n$.
%%Moreover, all these simplices are of type $n$.
\end{proof}
%\begin{lem}
%For each $T\in \Tu(S)$, there is a sequence $\Eref \pa S\ar\Edges=e_0\ar\Edges\cdots\ar\Edges e_m=\Eref \pa T$ of edges.
%\end{lem}
%\begin{proof}
%According to the remark on the second characterisation of admissible forests (Theorem \ref{Knotenschritt}), there is a chain 
%$S=S_0\ar0 S_1\ar1\cdots\ar1 S_m$ with
%$\Vnew S=p_1\ar\Vertices\cdots\ar\Vertices p_m=\Vnew T$. 
%Since $\mid$ maps the refinement edges injectively to the new vertices
%Every new vertex is the mid of the refinement edge of its parent, so 
%this yields the chain $\Eref \pa S=e_1\ar\Eref\cdots\ar\Eref e_m=\Eref \pa T$.
%\end{proof}

%Recall Theorem \ref{milestone}.\ref{it:towerstep}: 
%
%\emph{Suppose IsoCoChange for $\T_0$. Then \towerstep2}

%\emph{Suppose IsoCoChange for $\T_0$. Then
%for all $S\in\Simplexe$ with $hS\geq h_0:=3+\lfloor\log_2(n)\rfloor$ and for all $j\in \{h_0,\dots,hS\}$ there is a vertex $p_j$ of a hyperlevel $(j{-}h_0)$ ancestor of $S$, such that the tower layer $\Tu^j S
%%%already in changes
%%:=\bigcup\left\{T\in\Tu S~\middle |~hT=j\right\}
%$ is included in the patch 
%%$\left(\Q^{j-h_0}\T_0\right)(p_j)$.
%$\Simplexe^{n,j-h_0}p_j$.
%For an initial triangulation of solely type $n$ simplices, the same is true even for $h_0:=2+\lfloor\log_2(n)\rfloor$.}
\begin{proof}[Proof of \ref{it:towerstep}.]
%\emph{(The proof of the %stronger 
%version of the theorem in case of an initial triangulation of type $n$ simplices is obtained by choosing the variant in round brackets.)}
According to Lemma \ref{lem:IsoCoCh}.\ref{twodefs}, $\T_0$ also satisfies IsoCoChange. 

\emph{1st step: Assume $\T_0$ would consist of type $n$ hyperlevel 0 simplices solely. Then the statement is proven for all simplices $S$ with $hS\geq h_0$ and $j=h_0$.}
At first, note that $hS\geq 2+\lfloor\log_2 n\rfloor$ 
%$(hS \geq 2+\lfloor\log_2(n)\rfloor$) 
is equivalent to $2^{hS}\geq 2^{2+\lfloor\log_2 n \rfloor}>2n$ 
%($\dots >2n$) 
and $T\in\Tu^{h_0} (S)$ means $hT=h_0$ and $tT\leq n{-}1$. According to the characterisation of admissible forests (Theorem \ref{Knotenschritt} on page \pageref{Knotenschritt}) and the remark on it, there is a chain 
 $\Eref \pa S=e_1\ar\Edges\cdots\ar\Edges e_m=\Eref \pa T$.
%Resolving the definition of tower, the following needs to be shown:
%\begin{thm}\label{thm:patch}
%\emph{
%Let $S\in\Simplexe$ with $2^{hS}>4n$ ($\dots >2n$) and $tS\leq n{-}1$. Let $S_0$ be the root of $S$ (i.e.\ the ancestor in $\T_0$). Then there is a vertex $p\in\Vertices S_0$ such that each $T$ with $2^{hT}>4n$ ($\dots >2n$) and $tT\leq n{-}1$ and $S\ra T$ is included in the patch $\T_0 p$.
%}

%\end{thm}
%\begin{bem}[Convention and Remark]
%The conditions $tS,tT\leq n{-}1$ just ensure that $\pa S,\pa T$ have the same hyperlevel as $S$ and $T$, respectively; %$h\pa S=hS$, 
%defining here that the parent of a type $n{-}1$ T-array is always the parent in the \emph{extended} binary tree, i.e.\ a type $n$ (column) array of the same hyperlevel, not the transposed one of type 0 and decremented hyperlevel. 
%
%Note that this condition is satisfied for all simplices $S$ (in the binary tree) of a hyperlevel with $2^{hT}>4n{+}1$ by transposing these of type $n$.
%\end{bem}

Let $q$ be the ``oldest" vertex of $S$, i.e.\ this vertex of the refinement edge, which is also a vertex of $\Eref\pa S$; let 
%$k:=hS{-}1$ 
$k:=
%\min\{hS,hT\}
h_0
{-}1$ 
and $c$ be a $2^{-k}$-cube in the root $S_0$ of $S$ containing $q$.

Since $2^k= 
%2^{hS-1}
2^{h_0-1}
=2^{1+\lfloor\log_2 n\rfloor}
>n$, 
%($\dots >n$), 
Proposition \ref{pps:imPatch} yields a $p\in\Vertices S_0$ such that the $2^{-k}$-closure of $q$ is included in $\Int(\T_0 p)$. It contains $q\in\Vertices\Eref\pa S=\Vertices e_1$, hence, according to Proposition \ref{pps:klebt}, a vertex of $\Vertices e_m=\Vertices\Eref\pa T$ as well, because 
$k
%\leq h\pa T-1
\leq h T-1
=h\Eref\pa T-1$. (Note that $tT\leq n{-}1$ ensures that $T$ has a parent of equal hyperlevel with type in $\{1,\dots,n\}$, thus $hT=h\pa T=h\Eref\pa T$.)

Consequently, $\pa(T)$ contains a point lying in $\Int(\T_0 p)$, i.e.\ only in simplices of $\T_0 p$, but not in others. The same must hold true for the root $T_0\in\T_0$ of $T$, 
%which thus also must belong to $\T_0 p$.
thus also this must belong to $\T_0 p$.
%\end{proof}

%\begin{folg}[Corollary to Theorem \ref{thm:patch}]
%Let $S\in\Simplexe$, with $2^{hS}>2^{j+2}n$ and $tS\leq n{-}1$. Let $S_j$ be the ancestor of $S$ in $quni^{(j)} \T_0$. Then there is a vertex $p_j\in\Vertices S_j$ such that each $T$ with $2^{hT}>2^{j+2}n$ and $tT\leq n{-}1$ and $S\ra T$ is included in the patch $\left(quni^{(j)}\T_0\right)(p_j)$.
%\end{folg}
%\begin{proof}
\emph{2nd step: $j\in\{h_0{+}1,\dots,hS\}$.}
Just start with 
%$\Q^{j}\T_0$
$\Simplexe^{n,j}$ instead of $\T_0$, shift all the hyperlevels by $h_0-j$ and scale the reference coordinates by $2^{j-h_0}$. Then it is the same statement as in 
%Theorem \ref{thm:patch}
the first step.
\end{proof}
\subsection{Another pile game}\label{sec:finish}

Recall the forest formulation of Theorem \ref{BDV4IsoCoChange} (see the Conversion Table \ref{tab:conversion table} on page \pageref{tab:conversion table}):
\emph{Suppose ReTaHyCo for $\T_0$ and 
$
D:=\max
%\limits
_{\substack{S\in\Simplexe
%\\tS=0
}} 
%\left\{
2^{hS}\diam S
$,
%~\middle|~S\in\Simplexe,tS=0\right\},
$
d:=\min_{S\in\Simplexe}
%2^{n\cdot hS-tS}
2^{n\cdot hS+n{-}tS}
\lvert S\rvert$
and 
%$h_0:=3+\lfloor \log_2(n)\rfloor$ or 
$h_0:=2+\lfloor \log_2(n)\rfloor$. 
%in the case of an initial triangulation of solely type $n$ simplices.
Then there is a constant 
%$C(n) > 0$, depending only on the dimension $n$ of the triangulation, 
$C>0$, depending only on $n$ and linearly on $D^n/d$,
such that for every \hyperlink{refinement sequence of admissible forests}{refinement sequence of admissible forests} $\T_0=W_0, W_1,\dots, W_N$, it holds that
\begin{align*}
%\#\left(W_N\setminus W_0\right)\leq {\frac12\underbrace{\# \left\{S\in \Simplexe~\middle|~hS\leq 2+\log_2(n)\right\}}_{\leq 2(
\#\left(W_N\setminus \T_0\right)\leq {\underbrace{\#\left\{S\in \Simplexe\setminus\T_0~\middle|~hS\leq h_0\right\}}_{\leq 2(
%4
8
n)^{n}\#\T_0}}+
{2
%C(n)\frac {D^n}{d}
C N}.
\end{align*}
}

\begin{proof}\index{BDV theorem!IsoCoChange!proof}
For the tower layers of hyperlevel $h_0$ or below, the theory here does not make any statement. So the worst case is assumed that all the simplices of hyperlevel ${\leq}h_0$ are in the forest. This is the first summand in the stated inequality.

%The number of simplices in the first summand 
This number 
is estimated by
\begin{align}
&\#\left\{S\in \Simplexe\setminus\T_0~\middle|~hS \leq h_0\right\}\\
&\leq\#\T_0\big(\overbrace{1
}^{\#\Simplexe^{n,0}}+
\overbrace{2}^{\#\Simplexe^{n-1,0}}+\dots+\overbrace{2^n}^{\#\Simplexe^{0,0}}+\dots+\overbrace{2^{
%\left(
h_0
%{-}1\right)
\cdot 
n+n}}^{\#\Simplexe^{0,h_0
%-1
}}
\overbrace{-1}^{\#\T_0}\big)\\
%&= 2\cdot \left(2^{n\cdot h_0}-1\right)\#\T_0\\
&\leq \left(2\cdot 2^{n\cdot 
%h_0
(h_0+1)
}-2\right)\#\T_0\label{eq:simplices below h_0}\\
&\stackrel{2^{h_0}\leq 4n}{\leq} 2 (8n)^{n}\#\T_0.
\end{align}
In case of 
%a initial triangulation of solely type $n$ simplices, $2^{h_0}$ is even bounded by $4n$, leading to $$2(4n)^{n}\#\T_0$$ for this number.
$\Vertices_0=\leer$, the hyperlevel is bounded from below by 1, which leads to
\begin{align*}
\#\left\{S\in \Simplexe~\middle|~hS \leq h_0\right\}\leq 2(4n)^n\#\T_0.
\end{align*}

The remaining part of the proof again follows similar steps as the proof for the pile game.

Recollect the function $\iota$, which %embeds 
maps
a subset $W\subset \Simplexe$ to a subset of $\R^n\times\N%\times\{0,\dots,n{-}1\}
$ by 
\[\iota W:=\bigcup_{S\in W}\{S\}{\times}\{hS\}%\times\{tS\}
.\]

\emph{1st step: Estimate the counting measure $\#$ by a measure $\mu$ on $\R\times \N$.}
The number $d$ is the minimal volume of a simplex of hyperlevel 0 and type 
$n$, 
%$0$, 
so the number of simplices of hyperlevel 0 and type 
$n$ 
%0
in a set $W^{n,0}\subset\Simplexe$ is at most $\frac{|\iota W^{n,0}|}d$. Analogically, $\#W^{k,j}%\subset\Simplexe^j
$ is at most $\frac{|\iota W^{k,j}|}d 
2^{nj+n-k}$ and the total number of simplices in $W^{h_0+1}\cup W^{h_0+2}\cup \cdots$ with $\iota W=\bigcup_{j=0}^\infty w^j{\times}\{j\}$
%$w^0\times\{0\}\cup w^1\times\{1\}\cup\cdots$ 
is bounded by 
\begin{align}
\sum_{j=h_0+1}^\infty \frac{|w^j|}{d}2^{nj} \big(
\overbrace{2^1}^{\text{type }n-1}+\dots+\overbrace{2^n}^{\text{type }0}
%2^{-(n-1)}+\dots+2^0
\big)
&=\frac{
2^{n+1}-2
%2-2^{1-n}
}{d}\sum_{j=h_0+1}^\infty 2^{nj}|w^j|\\
&=:\mu%\left(
\bigcup_{j=0}^\infty w^j{\times}\{j\}%\right)
.
\label{mudef IsoCoChange}
\end{align}\index{m@$\mu$ -- measure estimating $\#$!IsoCoChange}
Therefore, \eqref{mudef IsoCoChange} defines a measure on $\R^n\times\N$
%a measure $\mu$ on $\R^n\times\N%\{0,\dots,n{-}1\}
%$ %these sets 
%%estimating the number of admissible simplices %in these sets 
%%in the subset $w^j$
such that $\mu \iota W\geq \# W^{h_0+1}+\#W^{h_0+2}+\dots$.

%is given by 
%\begin{align*}
%\mu\bigcup_{j=0}^\infty w^j\times\{j\}:=\sum_{j=h_0}^\infty \frac{|w^j|}{d}2^{nj} \left(2^1+\dots+2^n\right)=2^{nj+1}\left(2^n{-}1\right)\frac{|\bullet|}{d},
%\end{align*}
%$$\mu w^j:=\frac{|\bullet|}{d}2^{nj}\left(2^1+\dots+2^n\right)=2^{nj+1}\left(2^n{-}1\right)\frac{|\bullet|}{d},\index{m@$\mu$ -- measure estimating $\#$!IsoCoChange}$$
%but this time only the hyperlevels $\leq h_0$ are measured. 
%$\mu w^j:=\frac{|\bullet|}{d}2^{nj}\left(2^{1-n}+\dots+2^0\right)=2^{nj+1}\left(1{-}2^{-n}\right)\frac{|\bullet|}{d}$. 
%

%Let $p_S$ be an arbitrary point in $S$. %The Corollary to 
%Theorem \ref{milestone} shows that 
%$
%\Tu^{h_0+j}S\subset \bigcup
%%(\Q^{j}\T_0)(p_j),
%\Simplexe^{n,j},
%$
%which again is a subset of $$B(p_j,2^{-j}D)\subset B(p_S,\allowbreak 2^{-j}D+\|p_S{-}p_j\|_2)\subset B(p_S,2\cdot2^{-j}D).$$ 

\emph{2nd step: Estimate the towers $\Tu$ and the forests $W_k$ by supersets $\tu$ and $w_k$, respectively.}
For each $S\in\Simplexe$, choose an arbitrary point $p_S\in S$. Theorem \ref{milestone} 
%says
gives $p^k_S$ with
%\begin{align*}
$
\bigcup\Tu^k S\subset 
%\left(\Q^{k-h_0}\T_0\right)
\bigcup\Simplexe^{n,k-h_0}
p_S^k,
$
%\end{align*}
which implies
\begin{align*}
S\subset \bigcup\Tu^k S 	&\subset \bigcup 
%\left(\Q^{k-h_0}\T_0\right)
\Simplexe^{n,k-h_0}
p_S^k\\
							&\subset B\left(p_S^k,2^{h_0-k}D\right),
\end{align*}
hence $\left\|p_S-p_S^k\right\|\leq 2^{h_0-k} D$ and 
\begin{align}
\bigcup\Tu^k S 	&\subset B\left(p_S^k,2^{h_0-k} D\right)\nonumber\\
				&\subset B\left(p_S,2^{h_0-k} D+\|p_{S}-p_S^k\|\right)\nonumber\\
				&\subset B\left(p_S,2^{1+h_0-k} D\right)\label{ball around pS}.
		%=:\tu^k(S).
\end{align}
Therefore, the following estimated tower $\tu(S)$ includes $\iota\Tu(S)$.

\begin{defn}\index{tw@$\tu$ -- estimates $\Tu$!IsoCoChange}
The estimated tower of $S$ is defined by
\begin{align}
\tu^k S:=
\begin{cases}
\leer, \text{ for } k>hS,\\
B\left(p^k_S,2^{h_0-k}D\right), \text{ for } 
%k\geq hS-n+1,\\
%hS{-}n{+}1\leq k\leq hS ,\\
hS{-}n< k\leq hS ,\\
B\left(p_S,2^{1+h_0-k}D\right), \text{ for } h_0
%{+}1\leq 
<k\leq hS{-}n,\\
\R^n, \text{ for }k\leq h_0.
\end{cases}
\end{align}
\end{defn}
\begin{bem}
The index $k=hS-n$ of the change from $B(p_S^k,2^{h_0-k}D)$ to $B(p_{S},\allowbreak 2^{h_0-k+1}D)$ could be chosen arbitrarily. This choice turns out to be convenient in \eqref{eq:BDV constant IsoCoChange} below.
\end{bem}
For a given refinement sequence of forests $\T_0=W_0,\dots,W_N$ with corresponding simplices $S_j\in \children(\leaves W_j)$ such that 
%%Achtung! anders als oben!
$W_{k}=W_{k-1}\cup \Tu S_k=\T_0\cup\bigcup_{j=1}^{k} \Tu S_j$, let $w_{k}:=\bigcup_{j=1}^k \tu S_j$, which is a superset of $\iota W_k$.\index{w@$w$ -- estimates $W$!IsoCoChange}
\emph{3rd step: Find a superset of $w_N\setminus w_{N-1}$ in terms of estimated towers and measure it.}
Let $\tilde S_N$ be an ancestor of $S_N$ of hyperlevel $hS_N{-}1$. By definition of a refinement sequence, the simplex $\tilde S_N$ must be included in $W_{N-1}$ as well as $\pa (S_N)$. Let $M<N$ be an index such that $\tilde S_N\in \Tu S_M$. Since simplices cannot demand simplices of higher hyperlevel, $h(S_M)\geq h\tilde S_N=hS_N-1$. Figure \ref{fig:estimation of tower difference} might give an overview about the following splitting of the estimation.
\begin{align}
w_{N}\setminus w_{N-1}%\nonumber\\
						&\subset \tu S_N\setminus \tu S_M\nonumber\\
						&\subset {\bigcup_{k=hS_N-n}^{hS_N}\tu^{k} S_N 
						\times\{k\}\cup 
							\bigcup_{k=h_0+1}^{hS_N-n-1}
								\left(\tu^k S_N\setminus \tu^k S_M\right)
								\times\{k\}}\label{tu of the layers}
\end{align}
%\begin{align}
%&\mu w_{N}-\mu w_{N-1}\nonumber\\	
%						&\leq \mu \left(\tu S_N\setminus \tu S_M\right)\nonumber\\
%						&\leq\mu\tu^{hS_N} S_N + 
%							\sum_{k=h_0}^{hS_N-1}
%								\mu\left(\tu^k S_N\setminus \tu^k S_M\right)\label{mu of the layers}
%\end{align}

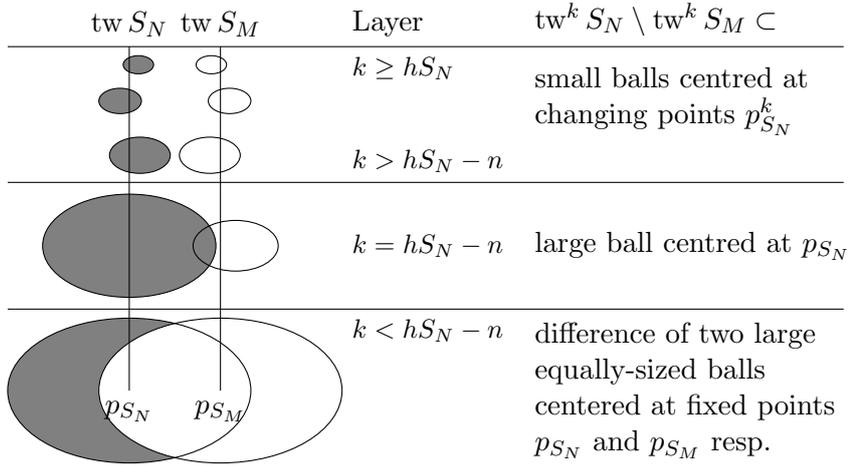
\begin{figure}[hbt]
\centering{
\begin{tikzpicture}[scale=.2,yscale=.6,
					information text/.style={%fill=gray!10,
					inner sep=1ex}]
\def\horizontale{++(55,0)}
\def\yshift{16}
\def\erste{.5*\yshift+1}
\def\yshifteins{10}
\def\yshiftzwei{6}
\def\yshiftdrei{4}
\def\zweite{.5*\yshifteins+2}
\def\gesamt{\yshift+\yshifteins+\yshiftzwei+\yshiftdrei}
\def\senkrechte{\draw (0,0) -- (0,\gesamt+2) node[above]}
\def\textbreite{4.5cm}
%%linker Turm
%%nullte Ebene
\fill[gray] (0,0) circle [radius=8];
%%horizontale
\draw (-8,\erste) -- \horizontale %node[below left]{\small $k<hS_N-n$}left
;
%%erste Ebene
\begin{scope}[yshift=\yshift cm]
\filldraw[fill=gray] (0,0) circle [radius=5.7];
\draw (-8,\zweite) -- \horizontale;
\filldraw[fill=gray] 	++(.7,\yshifteins) circle [radius=2]
					++(-1.3,\yshiftzwei) circle [radius=1.4]
					++(1.2,\yshiftdrei) circle [radius=1];
\end{scope}
\draw (-8,\gesamt+2) -- \horizontale %node[below left]{\small $k\geq hS_N$}
;
%%rechter Turm
\begin{scope}[xshift=6cm]
%%unterste Ebene
\filldraw[fill=white] (0,0) circle [radius=8];
%%Beschriftung
\draw (8,\erste) node[below right]{\small $k<hS_N-n$};	
\draw		(20,0)node[right, text width=\textbreite]{difference of two large equally-sized balls centered at fixed points $p_{S_N}$ and $p_{S_M}$ resp.};
\begin{scope}[yshift=\yshift cm]
\draw 	(8,0) node[right]{\small $k=hS_N-n$} 
		(8,\zweite)	node[above right]{\small $k>hS_N-n$};
\draw (1,0) circle [radius=2.8]
	 ++(-1.7,\yshifteins) circle [radius=2]
	 ++(1.3,\yshiftzwei) circle [radius=1.4]
	 ++(-1.2,\yshiftdrei) circle [radius=1];
%%Beschriftung
\draw (20,0) node[right
%, text width=\textbreite
]{large ball centred at $p_{S_N}$}
	 ++(0,\yshifteins+\yshiftzwei) node[right, text width=\textbreite]{small balls centred at changing points $p^k_{S_N}$}
%%Tabellenkopf
	 ++(0,\yshiftdrei+5) node[right]{$\tu^k S_N\setminus \tu^k S_M\subset$};
\end{scope}
\draw (8,\gesamt+2) node[below right]{\small $k\geq hS_N$}
					node[above right]{Layer};
\draw (0,0) node[below] {$p_{S_M}$};
\senkrechte{$\tu{S_M}$};
\end{scope}
%%linker Turm, unterste Ebene, Kreis nachzeichnen
\draw (0,0) node[below] {$p_{S_N}$};
\senkrechte{$\tu{S_N}$};
\draw (0,0) circle [radius=8];
\end{tikzpicture}
\caption{Estimation of the differences between the forests. Differently from the figure, the radius doubles with decremented layer actually, except of the step from $k=hS_N{-}n{+}1$ to $k=hS_N{-}n$, where it quadruples. The left circles represent the new tower $\tu S_N$, the right circles the old tower $\tu S_M$. The set $w_N\setminus w_{N-1}$, which is added in the $N$-th round to the estimated forest, is estimated by the grey set, i.e.\ only the \emph{lower} layers of $\tu S_M$, the large balls centred at $p_{S_M}$, are subtracted in the estimate, while the upper ones, the small non-concentric circles, are ignored.}
\label{fig:estimation of tower difference}
}
\end{figure}

The measure of this is estimated piecewise.
From here, 
$
V_m
%(r)
=%r^m
\frac{\pi^{m/2}}{\Gamma({m/2+1})}
$\index{V_n@$V_n$ -- volume of the $n$-dimensional unit ball} denotes the volume of the $m$-dimensional 
unit
ball. 
%and $V_m:=V_m(1)$.
The total measure %$\mu\tu^{hS_N} S_N$ 
of the top layers 
%$\tu^{hS_N} S_N\cup \dots\cup \tu^{hS_N-n+1} S_N$ 
$\mu\left(\bigcup_{k=hS_N-n+1}^{hS_N}\tu^{k} S_N
\times\{k\}
\right)$
is 
\begin{align}
\frac {
2^{n+1}{-}2
%2-2^{1-n}
}d \sum_{k=hS_N-n+1}^{hS_N}
2^{n k} \left|B\left(p^k_{S_N},2^{h_0-k}D\right)\right|
=\frac{D^n}d \left(
2^{n}{-}1
%1-2^{-n}
\right)2^{nh_0+1}n V_n,\label{top layers}
\end{align}
the measure of the layer $\tu^{hS_N-n} S_N$ is
\begin{align}
\frac {2^{n+1}{-}2}d %\sum_{k=hS_N-n+1}^{hS_N}
2^{n (hS_N-n)} \left|B\left(p^k_{S_N},2^{h_0+1-hS_N+n}D\right)\right|
=\frac{D^n}d \left(2^{n}{-}1\right)
%2^{n(h_0+1)} 
2^{nh_0+1}2^n
%V(1)
V_n.\label{intermediate layer}
\end{align}
Since $hS_M\geq hS_N{-}1$,
the 
%sum equals
%total 
measure of the layers $\bigcup_{k=h_0+1}^{hS_N-n-1}\left(\tu^k S_N\setminus \tu^k S_M\right)\allowbreak\times\allowbreak \{k\}$ is (with $e_n$ being the $n$-th canonical unit vector)
\begin{align}
						&\frac{2^{n+1}{-}2}{d}
%						\left(
%							2^{n\cdot hS_N}\left|B\left(p_S,2^{1+h_0-hS_N}\right)\right|+
							\sum_{k=h_0+1}^{hS_N-n-1}2^{nk}
								\left|B\left(p_{S_N},2^{1+h_0-k}D\right)\setminus 
									B\left(p_{S_M},2^{1+h_0-k}D\right)\right|
%						\right)
\label{eq:lower layers untransformed}\\
						&=\frac{D^n}d \left(2^{n}{-}1\right)2^{n
						%\left(h_0+1\right)
						h_0+1}2^n
%						\left(
%							\left|B(0,1)\right|+
							\sum_{k=h_0+1}^{hS_N-n-1}
								\left|B(0,1)\setminus 
										B\left(2^{k-h_0-1}\frac{\left\|p_{S_N}-p_{S_M}\right\|}De_n,
												1\right)\right|.\label{eq:lower layers}
%						\right)
\end{align}
(From \eqref{eq:lower layers untransformed} to \eqref{eq:lower layers} the set difference between the two balls is rotated and scaled by the inverse of the radius.) 
%The differences of the balls can be calculated as 

How can the differences of the balls be estimated?
The sequence of inclusions 
%together 
(with \eqref{ball around pS} from above)
\begin{align*}
S_N\subset\tilde S_N
%\subset\bigcup\Tu^{hS-1} S_M
\subset\Tu^{hS_N-1} S_M
\subset B\left(p_{S_M},2^{1+h_0-(hS_N-1)}D\right)
\end{align*}
implies that $\|p_{S_N}-p_{S_M}\| \leq 2^{2+h_0-hS_N}D$ and the volume of the set difference $B(0,1)\setminus B(ae_n,1)$ is monotonically increasing in $a$, thus the summand in \eqref{eq:lower layers} can be estimated from above by 
$
\left|B(0,1)\setminus B\left(2^{1+k-hS_N}e_n,1\right)\right|.
$
The substitution $j=hS_N{-}k
-1
$ turns the sum into
\begin{align*}
\sum_{j=n
%+1
}^{hS_N-h_0
-2
}\left|B(0,1)\setminus B\left(2^{
%1
-j}e_n,1\right)\right|.
\end{align*}
%Since the balls $B(0,1)$ and $B(2,1)$ are disjoint, adding the summand for $j=0$ is equivalent to adding the measure of the top layer which was calculated in \eqref{top layer}, so \eqref{mu of the layers} is estimated from above by
%\begin{align*}
%\frac{2^{n+1}{-}1}d2^{n\left(h_0+1\right)}\sum_{j=0}^{\infty}\left|B(0,1)\setminus B\left(2^{1-j}e_n,1\right)\right|.
%\end{align*}
%
In the following calculation, $\{x_n\geq c\}$ denotes the half-space $\{(x_1,\dots,x_n)~|~\allowbreak x_n\geq c\}$ and the like.
To calculate the volume of the set difference $B(0,1)\setminus B(2ce_n,1)$ of two balls, note that (see Figure \ref{fig:difference of two balls})
\begin{align*}
B(0,1)\cap B(2ce_n,1)
\allowbreak=\allowbreak\left(B(0,1)\cap \left\{x_n\geq c\right\}\right)
\allowbreak\dot\cup\allowbreak \left(B(2ce_n,1)\cap \left\{x_n < c\right\}\right).
\end{align*}
%%Differenz zweier Kugeln
\begin{SCfigure}[2]
\begin{tikzpicture}[scale=1]
%%halber Abstand
\def\c{.65};
%%Füllungen
%%unterer Kreis
\fill[gray!30] (0,0) circle [radius=1cm];
%%oberer Kreis
\fill[white] (0,2*\c) circle [radius =1cm];
%gepunktete Menge
\begin{scope}
\clip (-1,-\c) rectangle (1,\c);
\fill[pattern=dots] (0,0) circle [radius=1cm];
\end{scope}
%%unterer Kreis
\draw (0,0) circle [radius=1cm] ++(-1,0) -- ++(-.1,0) node[left]{$B(0,1)$};
%%oberer Kreis
\draw (0,2*\c) circle [radius =1cm] ++(-1,0) -- ++(-.1,0) node[left]{$B(2ce_n,1)$};
%%Geraden
\draw (-1.2,\c) -- (2.7,\c) node[below left]{$\{x_n< c\}$} node[above left]{$\{x_n> c\}$};
\draw (-1.2,-\c) -- (2.7,-\c) node[below left]{$\{x_n< -c\}$};
\end{tikzpicture}
\caption{The volume of the grey set ${B(0,1)\setminus B(2ce_n,0)}$ equals the volume of the dotted set ${B(0,1)\cap \{-c\leq x_n\leq c\}}$, because the white dotted set ${B(2ce_n,1)\cap\{x_n<c\}}$ is congruent to the grey undotted set ${B(0,1)\cap \{x_n<-c\}}$.}
\label{fig:difference of two balls}
\end{SCfigure}
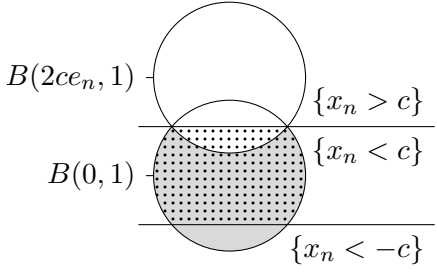
The set $B(2ce_n,1)\cap \left\{x_n < c\right\}$ is congruent to 
$B(0,1)\cap\{x_n< -c\}$
and thus%, with $V_{n-1}(a)$ being the volume of a $(n{-}1)$-dimensional ball of radius $a$, 
\begin{align*}
\left|B(0,1)\setminus B(2c,1)\right|
&=\left|B(0,1)\setminus \left(B(0,1)\cap B(2c,1)\right)\right|\\
&=\left|B(0,1)\cap \left\{-c\leq x_n \leq c\right\}\right|\\
&=\int_{-c}^{c} 
\Big|
\underbrace{\left\{(x_1,\dots,x_{n-1})~\middle|~x_1^2+\dots+x_{n-1}^2=1-x_n^2\right\}}_
{
B^{n-1}\left(0,\sqrt{1-x_n^2}\right)
%V_{n-1}\left(\sqrt{1-x_n^2}\right)
}
\Big|
dx_n\\
&=V_{n-1}\int_{-c}^{c} \underbrace{\left(1{-}x^2\right)^{\frac{n-1}{2}}}_{\leq 1} dx\leq 2cV_{n-1}.
\end{align*}
With $2c=2^{
%1
-j}$, the above sum is bounded by
\begin{align}
\sum_{j=n
%+1
}^{\infty}\left|B(0,1)\setminus B\left(2^{
%1
-j}e_n,1\right)\right|
%= V_{n-1}(1)\sum_{j=0}^\infty \int_{-2^{-j}}^{2^{-j}} \left(1-x^2\right)^{\frac{n-1}{2}} dx.\label{sum of ball differences}
\leq 
%2
V_{n-1}
%(1)
\sum_{j=n
%+1
}^\infty 2^{-j}=2^{1-n}V_{n-1}%(1)
.\label{sum of ball differences}
\end{align}
%For odd $n$, the antiderivative of the integrand is a polynomial, hence the sum turns into a sum of geometric series. To estimate it for small even $n=2,4$, compute the integral for $j=0,1$ numerically and estimate 
%\begin{align}
%\sum_{j=2}^\infty\int_{-2^{-j}}^{2^{-j}} \underbrace{\left(1-x^2\right)^{\frac{n-1}{2}}}_{\leq 1} dx
%\leq \sum_{j=2}^\infty {2^{-j}}-\left(-2^{-j}\right)=1.
%\end{align}
%%obsolet
%This leads to the the bounds for \eqref{sum of ball differences} in Table \ref{tab:sum of ball differences}.
%\begin{table}
%%\begin{tabular}{}
%%\end{tabular}
%\caption{Bounds for the sum of ball differences in \eqref{sum of ball differences}}\label{tab:sum of ball differences}
%\end{table}
Finally, summing up the estimates for \eqref{top layers}, \eqref{intermediate layer} and \eqref{eq:lower layers} using \eqref{sum of ball differences} (note the cancellation of the factor $2^n$ in \eqref{eq:lower layers} using \eqref{sum of ball differences})
leads to the estimate
\begin{align}
\mu\left(w_N\setminus w_{N-1}\right)\leq
\frac{D^n}{d}\left(2^{n}{-}1\right) 2^{nh_0+1}\left(\left(n{+}2^n\right)V_n+2V_{n-1}\right)=:2C.\label{eq:BDV constant IsoCoChange}
\end{align}
\index{BDV constant!IsoCoChange!formula}

%This implies
Estimating $2^{h_0}\leq 4n$ and $(n+2^n)V_n+2V_{n-1}< 2^{n+1}$ yields the simpler but coarser estimate
%%which yields the constant
%and the constant reads
%\begin{align*}
%%C(n)
%C	&=\left(2^{n}-1\right) \left(2^{h_0+1}+
%		2
%		%4
%		\right)^n
%		\frac {D^n}d 
%		V(1)\\
%	&\stackrel{2^{h_0}\leq 8n}{\leq} \left(32n+
%%8
%4
%\right)^n
%\frac {D^n}d 
%V(1)
%\end{align*}
%%\index{C@$C$ -- BDV constant!formula for IsoCoChange}
%\index{BDV constant!IsoCoChange!formula}
%and 
%$$
%%C(n)
%C
%\leq\left(16n+4\right)^n
%\frac {D^n}d 
%V(1)
%$$ 
%
%\begin{align*}
%C\leq \frac {D^n}d2^{5n+1}n^n V_n
%\end{align*}
%using $2^{h_0}\leq 8n$, and 
\begin{align*}
C< \frac {D^n}d2^{4n+1}n^n V_n
\end{align*}
%using $2^{h_0}\leq 4n$ 
%in the case of solely type $n$ initial simplices
%in the BDV theorem.
for the BDV constant.

Having found a bound for $\mu w_{N}-\mu w_{N-1}$, the proof is concluded exactly as the proof under SIC: 
\begin{align*}
\#\T_N-\#\T_0 	&=\frac 12(\#W_N-\#W_0)\\
				&\leq \frac12 \left(\#\left\{S\in\Simplexe\setminus\T_0~|~hS\leq h_0\right\}+\mu w_N\right)\\
				&=\frac12\left(\#\left\{S\in\Simplexe\setminus\T_0~|~hS\leq h_0\right\}+\sum_{K=1}^N\mu(w_K\setminus w_{K-1})\right)\\
%				&\leq \left(2^{nh_0}-1\right)\#\T_0
				&\stackrel{\eqref{eq:simplices below h_0}\eqref{eq:BDV constant IsoCoChange}}{\leq} \left(2^{nh_0}-1\right)\#\T_0
%				+\frac{D^n}{d}\left(
%				2^{n}{-}1
%%				1{-}2^{-n}
%				\right) 2^{nh_0+1}\left(\left(n{+}2^n\right)V_n+2V_{n-1}\right)N
+CN
\end{align*}
\end{proof}
Table \ref{tab:Beispielzahlen IsoCoChange} shows exemplary values for the BDV constant. %in the case of initial simplices of solely type $n$ for an initial triangulation of Kuhn simplices. (For other shapes of initial simplices, $\frac {D^n}{d}$ is probably larger.)
%$n=2$ yields $C=(32\cdot 2+4)^2\frac{2}{\frac12}\pi=68\cdot 4\cdot \pi=$
The values are much larger than for SIC.
\begin{table}[ht]%\label{tab:Beispielzahlen IsoCoChange}
\bgroup
\def\arraystretch{1.5}
\begin{tabular}{c|c|c|c|c|c|c}
$n$	&	$h_0=2{+}\lfloor\log_2 n\rfloor$& $D$ & $\frac1d$ & $V%(1)
															_n$ &$\#\T_N
%- \#\T_0
\leq$&$\frac{C_{\text{\tiny{IsoCoCh}}}}{C_{\text{SIC}}}\approx$\\
\hline
%	&								&2^(2*3)-1,	(2^2-1)(2^(3+1)+2)^2(2/(1/2))*pi
2&3	&$\sqrt 2$&%$\frac12$&
				2&$\pi$&$64\#\T_0+18{,}000N$& 240\\
\hline
%	&						&2^(3*3)-1, (2^3-1)(2^(3+1)+2)^3*sqrt(3)^3/(1/6)*4/3pi
3&3	& $\sqrt 3$ & %$\frac16$ 
				6& $\frac43\pi$ & $512\#\T_0+5.9{\cdot}10^6N$ & 800\\
\hline
4&4	&2& %$\frac1{24}$ 
			24& $\frac{\pi^2}2$ &  $65{,}536\#\T_0+4.1{\cdot}10^{10}N$ & 30{,}000
%\\\hline
\end{tabular}
\egroup
\caption{Examplary values for %$C(n)$ 
the BDV theorem 
%in the case of initial simplices of solely type $n$ and hyperlevel 1 
for an initial triangulation of $\frac12$-scaled Kuhn simplices
of type $n$ and hyperlevel $1$.
% of hyperlevel 1. 
(For other shapes of initial simplices, $\frac {D^n}{d}$ is probably larger.)}
\label{tab:Beispielzahlen IsoCoChange}
%\index{C(n)@$C(n)$ -- BDV constant!IsoCoChange, values}
%\index{C@$C$ -- BDV constant!IsoCoChange, values}
\index{BDV constant!IsoCoChange!values}
\end{table}
%%%%%%%%%%%%%%%%%%%%%%%%%%%%%%%%%%%%%%
\newpage
\section{Conclusion and Outlook}\label{Conclusion and Outlook}
The notions of the demanding relations $\ra,\ar0,\ar1$ etc.\ and the tower are a fundamental concept. Whoever is going to propose a refinement algorithm in the future should analyse not only its general refineability, as one usually did, but in addition also its towers. 

The question, if %and how 
it is possible to impose \hyperlink{SIC}{SIC} to an arbitrary triangulation has to be answered yet. 

The step in the original proof of Binev, Dahmen and DeVore \cite{BDV} explaining why each simplex gets at least $c$ dollars was overestimating. 
%For the 2D reference triangle $\conv\{(0,0),\allowbreak(1,0),\allowbreak(1,1)\}$ as initial triangulation, the resulting BDV constant was  greater than 5000. 
For 2D, the resulting BDV constant was greater than 5000, even with a single isosceles right-angled triangle as initial triangulation.
The new proof yields a constant smaller than 37 for this initial triangulation which is not far from the lower bound 24 given by an example hinted at the title page.

Such examples for lower bounds for the constant should also be collected for higher dimensions.

For regular initial grids (i.e.\ all initial simplices are reference simplices in the grid $\Z^n$), the estimated towers can be estimated with other shapes than Euclidean balls to estimate them sharper. Moreover, a metric could be used which is based on the reference coordinates of each simplex instead of the underlying coordinates where $\Omega$ is usually embedded in. This carries the potential for even sharper bounds for the constant and is necessary to apply the theory to %curved spaces $\Omega$ where the volume of a ball does not depend on its radius exclusively. 
combinatorial triangulations without an embedding into $\R^n$. This ground has already been broken by Holst, Licht and Lyu in \cite{Holst}. Their techniques should be combined with these in this thesis.
%On the other hand, it
Note that an uneven underlying metric
complicates the last step of the proof here, because then the measures of the 
%upper 
tower layers differ from each other.

The generalised initial division algorithm of Section 6 is only a sketch. An actual algorithm how to mark the initial triangulation is highly desirable.

%Section 7 gives rise to the question for the weakest possible initial conditions which imply general refineability, some kind of quasi-uniform refinements or a BDV theorem. 
From a theoretical point of view, the question arises what are the weakest possible initial conditions which imply general refineability, some kind of quasi-uniform refinements or a BDV theorem, especially if the BDV theorem even holds true for ReTaCo in Section 7. ReTaCo implies an analogue of IsoCoChange, namely that the coordinate changes are \emph{similarities}. A similar approach to prove (with subcuboids of size $1\times \dots\times 1\times\allowbreak
 \frac12\times\dots\times\frac12\times \allowbreak
 0\times\dots\times 0$ instead of subcubes) seems possible, but Lemma \ref{relation on subsimplices} fails at the end.

Sections 6 and 8 proclaim a competition between initialisation refinements leading to meshes satisfying SIC and T-array initialisations satisfying 
%ReTaHyCo. 
IsoCoChange.
The question, if there could be a hybrid of both, suggests itself. 
A serious disadvantage of
%in the derivation of the constant for 
%ReTaHyCo 
IsoCoChange
is %the assumption fact
that the hyperlevel could grow by 1 in each refinement step, while SIC let only the level grow by 1 each round. The global order of the vertices in the AGK algorithm which causes $\ar\Edges$ to be a partial order has not been used yet. 
%With a skilful initialisation of the T-arrays this could be improved. 
Additionally, a skilful choice of the order of the vertices can improve the BDV constant.

A third approach would be to triangulate the domain (i.e.\ the polytope) $\Omega$ in a way that is especially suitable for 
adaptive mesh refinement by Kossaczký--Maubach bisection.
%the bisection method. 
The embedding of simplices into cubes (Theorem \ref{Cubetheorem}) suggests to partition the domain into cubes firstly which are partitioned into reference simplices after that.

Is it possible to generalise the concept of a T-array to include all the types of simplices proposed by Bänsch \cite{Baensch} and Arnold, Mukherjee and Pouly \cite{Arnold}? In the light of the proved BDV theorem for the always realisable 
%\hyperlink{ReTaHyCo}{ReTaHyCo} 
\hyperlink{IsoCoChange}{IsoCoChange},
these questions seem unimportant at the moment, but one should bear in mind that the estimate of the BDV constant for 
%ReTaHyCo 
IsoCoChange
here is still too large for practical considerations.

One may ask, if it is appropriate to bound only the worst case for the BDV ratio or if there could be an investigation of a reasonable average case. This would require to simulate the marking as a random process.

Is it possible to prove BDV for \normalindex{Longest Edge Bisection} (LEB) in a similar way?

It shall not be discussed here why LEB differs from standard bisection basically, but what about a modified LEB?% like the following?

A modified LEB (mLEB) is a rule deciding how to bisect a simplex taking based on the geometry of it. In contrast to standard LEB, mLEB should bisect a reference simplex exactly as standard bisection would do, even if the vertices are slightly deviated.
It could be defined as follows:
%\begin{defn}[modified LEB, first suggestion]
\paragraph{1st suggestion.}
Modified LEB bisects the longest edge of a simplex according to the $p$-norm $\|\bullet\|_p$ where $p$ is large enough that the refinement edge of a reference simplex is longer than every other edge, i.e.\ $2^p>n$.
%\end{defn}
%\begin{defn}[modified LEB, second suggestion]
\paragraph{2nd suggestion.}
An edge $f=\overline{CD}$ of a simplex $S\subset\R^n$ $\eps$-\emph{demands} another %(``riper'') 
edge $e=\overline{AB}\in\Edges S$, if
\begin{align*}
\frac{\|A-B\|_\infty}{\|C-D\|_\infty}\geq (1{-}\eps)\cdot 2
%$\|C-D\|_\infty\leq \frac{1+\eps}2 \|A-B\|_\infty$
\quad\text{or}\quad
\frac{\max_{V\in\Vertices e}\|\mid f-V\|_\infty}{\max_{V\in\Vertices f}\|\mid e-V\|_\infty}
%$\max\{\|\frac{C+D}2-A\|_\infty,\|\frac{C+D}2-B\|_\infty\}
\geq (1{-}\eps)\cdot 2
%\max\{\|\frac{A+B}2-C\|_\infty,\|\frac{A+B}2-D\|_\infty\}
.
\end{align*}
(Maybe the first condition could even be skipped.)
\emph{Modified LEB with para\-meter $\eps$} allows bisection of every edge which does not $\eps$-demand other edges.
%\end{defn}

%Modified LEB bisects a reference simplex of arbitrary type exactly as standard bisection would do, even if the vertices are slightly deviated. Unconventionally, the usage of the Chebyshev norm causes that the bisection is not even rotationally invariant, but this could be restored. 
Unconventionally, mLEB is not rotation-invariant.
%The definition immediately raises the questions, if on the one hand, for a too large $\eps$ there exists a simplex where every edge demands another (which is excluded for normal LEB) and if on the other hand, a sufficiently large $\eps$ ensures shape regularity.
Many questions immediately arise: If the rule ensures shape regularity, if it is generally refineable etc.
%%%%%%%%%%%%%%%%%%%%%%%%%%%%%%%%%%%%%%
\newpage
\renewcommand*{\thesection}{(\Alph{section})}
\setcounter{part}{0}
\setcounter{section}{0}
\section{ Appendix. Faces of convex sets}
%\addcontentsline{alles_außer_bilder.toc}{section}{\numberline {(A)}Appendix. Faces of convex sets.}{74}{section.1}

%To characterise the admissible forests, some notions has to be defined at first.

%\begin{defn}[face of a convex set]
%A \emph{face} $F$ of a convex set $C$ is a subset \emph{affinely closed in $C$}, i.e.\ the intersection of the affine hull %$\aff(S):=\{\sum_{i=1}^M a_is_i ~|~s_i\in S,\sum a_i=1\}$ 
%\[
%\aff(F):=\left\{\sum_{i=1}^M a_is_i ~\middle|~s_i\in F,\sum a_i=1\right\}
%\] 
%with $C$ is $F$, with the additional property that $C\setminus F$ is convex.
%\end{defn}
\begin{lem}
%\begin{lemmanonum}{A1}
The relative interior of a non-empty and finite-dimensional set is non-empty.
\end{lem}
%\end{lemmanonum}
\begin{proof}
A convex set $C$ with $\dim(C)=n$ contains $n{+}1$ affine independent points $p_0,\dots,p_n$, which span an $n$-simplex. The relative interior of $C$ is at least as large as the relative interior of this simplex, given by
\begin{align*}
\left\{\sum_{i=0}^n a_ip_i~\middle|~\sum_{i=0}^n a_i=1,\fa i.a_i>0\right\}=\bigcap_{j=0}^n \underbrace{\left\{\sum_{i=0}^n a_ip_i ~\middle|~\sum_{i=0}^n a_i=1, a_j>0\right\}}_{\text{open half-space}}
\end{align*}
(since the latter expression is 
%an open half-space,
a finite intersection of open half spaces, 
it is in fact open) and contains the point $\sum \frac1{n+1}p_i$.
\end{proof}
%\paragraph{Exercise.}
%What about the relative interior of a closed convex set in a Banach space?
\begin{thm}\label{Seite}
%\begin{thmnonum}{A2}\label{Seite}
Let $F$ be a subset of a convex set $C$. Consider the following properties:
%\begin{enumerate}
%\item
\begin{enumerate}
\item\label{item:konvex}
$F$ is convex.
\item
\begin{enumerate}
\item\label{item:Strecke}
$\fa x,y\in C. x\in F\text{ or }F\cap \relint\overline{xy}=\leer.$
%For arbitrary points $x,y\in 
%jede konvexe Teilmenge $T\subset
%C$, it holds: If the interior of the (possibly degenerated) line segment between these points intersects $F$, then $x$ lies in $F$. 
\item\label{item:konvexeTeilmenge}
For %beliebige Punkte $x,y\in 
each convex subset $T\subset
 C$, it holds: %Schneidet das Innere der (möglicherweise degenerierten) Strecke zwischen den Punkten $S$, so liegt $x$ in $S$. 
\begin{align*}
T\subset F\text{ or } {F\cap\relint T}=\leer.
\end{align*}
\end{enumerate}
\end{enumerate}
%\item
%\begin{enumerate}
%\item\label{2a}
%$F$ is \emph{affinely closed within $C$}, i.e.\ the intersection of the affine hull $\aff(F):=\left\{\sum_{i=1}^M a_if_i ~\middle|~f_i\in F,\sum a_i=1\right\}$ with $C$ is $F$.
%\item
%$C\setminus F$ is convex.
%\end{enumerate}
%\end{enumerate}
%Dann erfüllt $S$ 1a und (\ref{1bi} oder \ref{item:konvexeTeilmenge}) genau dann, wenn es \ref{2a} und 2b erfüllt.
Then \ref{item:konvex} %and \ref{item:Strecke} hold true if and only if \ref{item:konvex} and \ref{item:konvexeTeilmenge} hold true.
implies the equivalence of \ref{item:Strecke} and \ref{item:konvexeTeilmenge}.
%Then the following conjunctions are equivalent: 
%\begin{itemize}
%\item
%1a and \ref{item:Strecke}
%\item
%1a and \ref{item:konvexeTeilmenge}
%\item
%\ref{2a} and 2b.
%\end{itemize}
%\end{thmnonum}
\end{thm}
\begin{proof}
\ref{item:konvexeTeilmenge} implies \ref{item:Strecke} obviously. %Next, from 1a and \ref{item:Strecke} the property \ref{item:konvexeTeilmenge} is concluded: 
Conversely, let
%Für $T=\leer$ ist nichts zu zeigen. Sonst sei 
%Let 
$y\in F\cap\relint T$ for a convex $T\subset C$. 
%We we want to show for 
%\emph{Claim: For
%an arbitrary $x\in T$, 
%%that 
%$x\in F$.} 
Take an arbitrary $x\in T$.
Since $y\in\relint T$, there exists a $z\in T$ with $y\in\relint\overline{xz}$. Then \ref{item:Strecke} implies that also $x\in F$. also
%(Note that the argumentation is even true for $x=y$ an thus for a set $T$ consisting of a single point.)
%Now suppose 1a, \ref{item:Strecke} and \ref{item:konvexeTeilmenge} for some $F$ and show:
%
%\emph{2a.} 
%%erster Beweis: 
%If $F=\leer$, there is nothing to prove. 
%$F\subset C\cap\aff F$ holds true anyway. To show the inverse inclusion,
%let $t=\sum_{i=0}^M a_if_i\in C$ with $\sum_{i=0}^M a_i=1$ and $f_i\in F$. %O.\,b.\,d.\,A. ist a_i\neq 0
%Now apply \ref{item:konvexeTeilmenge} on $T=\conv\{f_0,\dots,f_{M},t\}$, see that the intersection of 
%$%\relint\conv\{f_0,\dots,f_M\}\subset
%\relint T$ with $F$ includes $\relint\conv\{f_0,\dots,f_M\}$, is thus non-empty 
%and conclude %$f_M\in T\subset F$.
%$t\in T\subset F$.
%
%
%\emph{2b.} If $x,y\in C{\setminus} F$, then \ref{item:Strecke} means that $\overline{xy}\in C{\setminus} F$. This is convexity of $C{\setminus} F$.
%
%Conversely, suppose 2a and 2b for $F$.
%
%\emph{1a.} According to 2a, $F=C\cap \aff F$ is, as an intersection of two convex sets, convex.
%
%\emph{\ref{item:Strecke}.} %Für $T=\leer$ ist nichts zu zeigen. 
%Let $y\in F\cap%\relint T
%\relint\conv(x,z)$. %x\in T$ beliebig. 
%It has to be proven that $x\in F$. %Da $y\in\relint T$ gibt es ein $z\in T$, sodass $y=ax+(1-a)z$ für ein $a\in(0,1)$. 
%$x$ and $z$ cannot lie in $C{\setminus} F$ both, because $C{\setminus} F$ is convex. But if two of the points $x,y,z$ lie in $F$, so, because of the affine closedness, also the third one.
\end{proof}
\begin{defn}[face of a convex set]\label{def:face}\index{face}
%\begin{defnonumber}[face of a convex set]
%\begin{defnonum}{A3}[face of a convex set]
If a subset $F$ of a convex set $C$ satisfies the equivalent conditions of Theorem \ref{Seite}, it is a \emph{face} of $C$.
%\end{defnonum}
\end{defn}
%\paragraph{Exercise.}
%\begin{bem}
%%Show that a 
%A 
%further equivalent definition is given by:
%\begin{enumerate}
%\item%\label{2a}
%$F$ is \emph{affinely closed within $C$}, i.e.\ the intersection of the affine hull $\aff(F) :=\left\{\sum_{i=1}^M a_if_i ~\middle|~f_i\in F,\sum a_i=1\right\}$ with $C$ is $F$.
%\item
%$C\setminus F$ is convex.
%\end{enumerate}
%\end{bem}
\begin{thm}\label{SeitenUntersimplexe}\index{faces are subsimplices}
%\begin{thmnonum}{A4}\label{SeitenUntersimplexe}
The faces of a simplex are its subsimplices.
\end{thm}
%\end{thmnonum}
\begin{proof}
\emph{First statement: A subsimplex of a simplex is a face.} 
%%Die leere Menge und die Ecken eines Simplex $S$ sind offenbar sowohl Seiten als auch Untersimplexe. 
%%Sei $F$ eine nichtleere Seite eines Simplex $S$. 
A subsimplex $\subsetmarking S$ is $\conv (\subsetmarking \Vertices)$ for some $\subsetmarking \Vertices\subset\Vertices(S)=:\{p_0,\dots,p_m\}$.
\emph{\ref{item:konvex}.} Per definition, a subsimplex is convex.
\emph{\ref{item:Strecke}.} Since $\Vertices(S)$ is affinely independent, the representation of a point as affine combination of $p_0,\dots,p_m$ is unique. %, i.e.\ $\sum a_ip_i=\sum b_ip_i\Ra \fa i. a_i=b_i$. 
A point of $S$, $x=\sum_{i=0}^m a_i p_i$ %with $a_i\geq0$ and $\sum_{i=0}^n a_i=1$ 
lies in $\subsetmarking S$, if and only if $\fa p_i\notin \subsetmarking \Vertices$, $a_i=0$. Let $y=\sum_{i=0}^m b_i p_i$ be 
also
a 
%second 
point of $S$. The relative interior of their convex hull is
\begin{align*}
\relint\overline{xy}=\left\{\sum_{i=0}^m (\lambda a_i+(1{-}\lambda)b_i)p_i~\middle|~\lambda\in (0,1)\right\}.
\end{align*}
%Either 
%a subset of $\subsetmarking T$, namely if 
%$\fa p_i\notin \subsetmarking \Vertices. a_i=b_i=0$, i.e.\ 
The point $x%,y
$
%\in 
lies in $
\subsetmarking S$, or, if $\ex p_i\notin \subsetmarking \Vertices. a_i>0$% or $b_i>0$
, it follows that $\lambda a_i+(1{-}\lambda)b_i>0$ for each $\lambda\in(0,1)$ (because $\sum a_ip_i,\sum b_ip_i\in \conv\Vertices$ are \emph{convex} combinations, so generally $a_i, b_i\geq0$), i.e.\ $\subsetmarking S\cap\relint
%\conv\{x,y\}
\overline{xy}
=\leer$.

\emph{Second statement: A face of a simplex is a subsimplex.}
For some vertices $p_0,\dots,p_k$ of a simplex $S$ and \emph{positive} real coefficients $a_0,\dots,a_k$ summing up to 1, let $x=\sum_{j=0}^k a_jp_j$ be a point of a (non-empty) face $F$ of $S$. 
%We claim that $F$ includes all these vertices. 
%
%If $S\setminus F$ is also convex, it cannot contain every $p_j$, otherwise it would have to contain $x$ as well. Therefore, let w.l.o.g.\ $p_k\in F$. But now either $k=0$ and our claim holds true, or also the intersection point of the line through $p_k$ and $x$ with the subsimplex $\conv\{p_0,\dots,p_{k-1}\}$ is contained in $F$. This point is the convex combination $\sum_{j=0}^{k-1} b_jp_j$ for certain positive $b_j$ and %we get 
%our claim 
%follows 
%by induction on $k$.
$x$ lies in the intersection of $F$ and $\relint
%\conv\{p_0,\dots,p_k\}
\triangle p_0\dots p_k
$. According to \ref{item:konvexeTeilmenge}, $F$ must include the subsimplex 
%$\conv\{p_0,\dots,p_k\}$. 
$\triangle p_0\dots p_k$.
%Induktion über $k$. Für $k=0$ ist offenbar $p_k=x\in F$. Induktionsschritt: Soll $S\setminus F$ konvex sein, kann es nicht alle $p_j$ enthalten, denn dann müsste es auch $x$ enthalten. Sei deshalb o.\,B.\,d.\,A.\, $p_k\in F$. Wegen der affinen Abgeschlossenheit ist auch der Schnittpunkt der Gerade durch $p_k$ und $x$ mit 
%%$\conv\{p_0,\dots,p_{k-1}\}$ 
%$\triangle p_0\dots p_{k-1}$
%in $F$ enthalten. Dieser Punkt ist die Konvexkombination $\sum_{j=0}^{k-1} b_jp_j$ mit positiven $b_j$ und unsere Induktionsvoraussetzung besagt, dass $p_0,\dots,p_{k-1}\in F$.
%%We will 
%Let us
%say 
It is written
that %the above 
$x$ \emph{demands} the vertices $p_0,\dots,p_k$. Remark that $x\in
%\triangle{
\conv\{p_0\dots p_k
%}
\}$.
%Sei nun $N$ die Menge aller Ecken von $S$, die von Punkten von $F$ verlangt werden.
%\begin{align*}
%N=\{p\in\Vertices S~|~\ex x\in S, x\in F\text{ verlangt } p\}.
%\end{align*}
Hence
\begin{align*}
F=\{x\in F\}\subset\conv \underbrace{\{p\in \Vertices(S)~|~\ex x\in F, x \text{ demands }p\}}_{=:\subsetmarking\Vertices}\subset \conv F=F,
\end{align*}
using the convexity of $F$ in the last equation% (\ref{item:konvex})
. %that $F=\aff F\cap C$ as intersection of an affine subspace with a convex set is convex itself, i.e.\ $\conv F=F$. 
Thus, $F$ is the subsimplex $\conv(\subsetmarking\Vertices)$.
\end{proof}
%\begin{lemmanonum}{A5}\label{UntermengeTeilmenge}
\begin{lem}\label{UntermengeTeilmenge}
Let $C$ be a convex set with a face $F$. If $T$ is a convex subset of $C$, then $F\cap T$ is a face of $T$.
\end{lem}
\begin{proof}
$F\cap T$ is convex. If for a convex subset $U$ of $T$, the intersection $\relint U\cap (F\cap T)$ is non-empty, $U$ must be a subset of $F$ according to Theorem \ref{Seite}.\ref{item:konvexeTeilmenge} and thus a subset of $F\cap T$, so \ref{item:konvexeTeilmenge} holds true for $F\cap T$.
%$F\cap T$ is affinely closed in $T$: 
%\begin{align*}
%\aff(F\cap T)\cap T&\subset \aff F\cap C=F\qquad | \cap T\\
%%\aff(F\cap T)\cap T&\subset T\\
%\aff(F\cap T)\cap T&\subset F \cap T.
%\end{align*}
%Additionally, $T\setminus (F\cap T)=T\cap (C\setminus F)$ is an intersection of convex sets, hence convex.
\end{proof}
Instead of the latter Lemma, the following corollary is almost exclusively used.
\begin{folg}\label{NachkommeUntersimplex}
%\begin{cornonum}{A6}\label{NachkommeUntersimplex}
%If $S\cap T$ is a subsimplex of $T$ and $\hat T$ a descendant of $T$, so also $S\cap T'$ a subsimplex of $T'$.
Suppose $T$ and $S$ are simplices, $T'$ is a descendant of $T$ and $T\cap S$ is as subsimplex of $T$. Then $S\cap T'$ is a subsimplex of $T'$.
\end{folg}
%\end{cornonum}

%\begin{lem}\label{eine}\index{union of faces is a face}
%%\begin{lemmanonum}{A7}\label{eine}
%%If a union of faces is convex, then it is a face itself.
%If a union of faces of a convex set $C$ is convex, then it is a face of $C$ as well.
%\end{lem}
%%\end{lemmanonum}
%\begin{proof}
%Let $V$ be a union %konvexer Untermengen 
%of faces $F_i$. To show \ref{item:Strecke} for $V$, let $x,y\in C$, such that $\relint\overline{xy}\cap F_i\neq\leer$ for some $i$. Since $F_i$ is a 
%face, \ref{item:Strecke} implies that $x\in F_i\subset V$. If $V$ is convex into the bargain, it is a face.
%%Die zweite Aussage folgt aus der ersten, da der Schnitt zweier konvexer Mengen stets konvex ist.
%\end{proof}
%%\begin{cornonum}{A8}\label{eine2}
%\begin{folg}\label{eine2}
%Let $C$ and $K$ be convex sets, and $C=\bigcup\{C_i~|~i\in I\}$ for convex $C_i$. 
%%Then we have the following implication: 
%Then the following implication holds true: 
%\begin{align*}
%\fa i.C_i\cap K \text{ is a face of $K$. }\quad\Rq C\cap K \text{ is a face of $K$.}
%\end{align*}
%\end{folg}
%\begin{proof}
%Apply lemma \ref{eine} with $F_i=C_i\cap K$.
%\end{proof}
%%%%%%%%%%%%%%%%%%%%%%%%%%%%%%%%%%%%%%%
\newpage
%\phantomsection
%\addcontentsline{toc}{section}{Acknowledgements}
\section*{Acknowledgements}
First of all, I want to thank the good Lord for the continuous grace to save, to send good ideas down from heaven, to show me mistakes, to liberate from distraction and to wake me up every morning. I thank my supervisors Ma Rui and Carsten Carstensen for their patience to listen to my cluttered sketches in the early phase of the work. I want to thank Martin Alkämper, Fernando Gaspoz and Robert Klöfkorn, who preserved me from claiming, the weakest compatibility condition would already imply \hyperlink{SIC}{SIC}. I want to thank many people for providing inspiring environments: Ana%
, Natalia% 
, Danuta and Maciej% 
, Irina %
and Giorgi% 
, further Julie %
and the Flixbus and the Polskibus corporations. I want to thank Sebastian and Hermann for listening and reading, my brother Albrecht for the advice to start the introduction with the aim from the large scale perspective, from Karin and Collin for grammatical support, my flatmates Jonatan and Steven for some LaTeX and layout advice and for feeding me, and my parents for financial support.
\newpage
\phantomsection
\addcontentsline{toc}{section}{References}
\bibliography{literatur}

\begin{thebibliography}{10}

\bibitem{Gaspoz}
Martin Alk{\"a}mper, Fernando Gaspoz, and Robert Kl{\"o}fkorn.
\newblock A weak compatibility condition for newest vertex bisection in any
  dimension.
\newblock {\em SIAM J. Sci. Comput.}, 40(6):a3853--a3872, 2018.
\newblock \href {https://doi.org/10.1137/17M1156137}
  {\path{doi:10.1137/17M1156137}}.

\bibitem{Arnold}
Douglas~N. Arnold, Arup Mukherjee, and Luc Pouly.
\newblock Locally adapted tetrahedral meshes using bisection.
\newblock {\em SIAM J. Sci. Comput.}, 22(2):431--448, 2000.
\newblock \href {https://doi.org/10.1137/S1064827597323373}
  {\path{doi:10.1137/S1064827597323373}}.

\bibitem{Atalay}
F.~Betul Atalay and David~M. Mount.
\newblock Bounds on the cost of compatible refinement of simplex decomposition
  trees in arbitrary dimensions.
\newblock {\em Comput. Geom.}, 79:14--29, 2019.
\newblock \href {https://doi.org/10.1016/j.comgeo.2019.01.004}
  {\path{doi:10.1016/j.comgeo.2019.01.004}}.

\bibitem{Baensch}
Eberhard B{\"a}nsch.
\newblock Local mesh refinement in 2 and 3 dimensions.
\newblock {\em IMPACT Comput. Sci. Eng.}, 3(3):181--191, 1991.
\newblock \href {https://doi.org/10.1016/0899-8248(91)90006-G}
  {\path{doi:10.1016/0899-8248(91)90006-G}}.

\bibitem{biedl}
Therese~C. Biedl, Prosenjit Bose, Erik~D. Demaine, and Anna Lubiw.
\newblock Efficient algorithms for {Petersen}'s matching theorem.
\newblock {\em J. Algorithms}, 38(1):110--134, 2001.
\newblock \href {https://doi.org/10.1006/jagm.2000.1132}
  {\path{doi:10.1006/jagm.2000.1132}}.

\bibitem{BDV}
Peter Binev, Wolfgang Dahmen, and Ronald DeVore.
\newblock Adaptive finite element methods with convergence rates.
\newblock {\em Numer. Math.}, 97(2):219--268, 2004.
\newblock \href {https://doi.org/10.1007/s00211-003-0492-7}
  {\path{doi:10.1007/s00211-003-0492-7}}.

\bibitem{axioms}
C.~Carstensen, M.~Feischl, M.~Page, and D.~Praetorius.
\newblock Axioms of adaptivity.
\newblock {\em Comput. Math. Appl.}, 67(6):1195--1253, 2014.
\newblock \href {https://doi.org/10.1016/j.camwa.2013.12.003}
  {\path{doi:10.1016/j.camwa.2013.12.003}}.

\bibitem{Cascon}
J.~Manuel Cascon, Christian Kreuzer, Ricardo~H. Nochetto, and Kunibert~G.
  Siebert.
\newblock Quasi-optimal convergence rate for an adaptive finite element method.
\newblock {\em SIAM J. Numer. Anal.}, 46(5):2524--2550, 2008.
\newblock \href {https://doi.org/10.1137/07069047X}
  {\path{doi:10.1137/07069047X}}.

\bibitem{GSS}
Dietmar Gallistl, Mira Schedensack, and Rob~P. Stevenson.
\newblock A remark on newest vertex bisection in any space dimension.
\newblock {\em Comput. Methods Appl. Math.}, 14(3):317--320, 2014.
\newblock \href {https://doi.org/10.1515/cmam-2014-0013}
  {\path{doi:10.1515/cmam-2014-0013}}.

\bibitem{Holst}
Michael Holst, Martin Licht, and Zhao Lyu.
\newblock Newest vertex bisection over general triangular meshes.
\newblock Preprint, 2010.
\newblock URL: \url{"https://www.math.ucsd.edu/~mlicht/pdf/preprint.nvb.pdf"}.

\bibitem{Karkulik}
Michael Karkulik, David Pavlicek, and Dirk Praetorius.
\newblock On 2d newest vertex bisection: optimality of mesh-closure and {{\(H
  ^{1}\)}}-stability of {{\(L _{2}\)}}-projection.
\newblock {\em Constr. Approx.}, 38(2):213--234, 2013.
\newblock \href {https://doi.org/10.1007/s00365-013-9192-4}
  {\path{doi:10.1007/s00365-013-9192-4}}.

\bibitem{KarkulikE}
Michael Karkulik, David Pavlicek, and Dirk Praetorius.
\newblock Erratum to: ``{On} 2d newest vertex bisection: optimality of
  mesh-closure and {{\(H^1\)}}-stability of {{\(L_2\)}}-projection''.
\newblock {\em Constr. Approx.}, 42(3):349--352, 2015.
\newblock \href {https://doi.org/10.1007/s00365-015-9309-z}
  {\path{doi:10.1007/s00365-015-9309-z}}.

\bibitem{Kossaczky}
Igor Kossaczk{\'y}.
\newblock A recursive approach to local mesh refinement in two and three
  dimensions.
\newblock {\em J. Comput. Appl. Math.}, 55(3):275--288, 1994.
\newblock \href {https://doi.org/10.1016/0377-0427(94)90034-5}
  {\path{doi:10.1016/0377-0427(94)90034-5}}.

\bibitem{Maubach}
Joseph~M. Maubach.
\newblock Local bisection refinement for {{\(n\)}}-simplicial grids generated
  by reflection.
\newblock {\em SIAM J. Sci. Comput.}, 16(1):210--227, 1995.
\newblock \href {https://doi.org/10.1137/0916014} {\path{doi:10.1137/0916014}}.

\bibitem{Mitchell}
William~F. Mitchell.
\newblock A comparison of adaptive refinement techniques for elliptic problems.
\newblock {\em ACM Trans. Math. Softw.}, 15(4):326--347, 1989.
\newblock URL: \url{www.acm.org/pubs/contents/journals/toms/1989-15/}, \href
  {https://doi.org/10.1145/76909.76912} {\path{doi:10.1145/76909.76912}}.

\bibitem{Schoen}
Patrick Sch{\"o}n.
\newblock {\em Scalable adaptive bisection algorithms on decomposed simplicial
  partitions for efficient discretizations of nonlinear partial differential
  equations}.
\newblock PhD thesis, Univ. Freiburg, Fakult{\"a}t f{\"u}r Mathematik und
  Physik, 2017.
\newblock URL: \url{d-nb.info/1159877963/34}, \href
  {https://doi.org/10.6094/UNIFR/15576} {\path{doi:10.6094/UNIFR/15576}}.

\bibitem{Stevenson}
Rob Stevenson.
\newblock The completion of locally refined simplicial partitions created by
  bisection.
\newblock {\em Math. Comput.}, 77(261):227--241, 2008.
\newblock \href {https://doi.org/10.1090/S0025-5718-07-01959-X}
  {\path{doi:10.1090/S0025-5718-07-01959-X}}.

\bibitem{Traxler}
C.~T. Traxler.
\newblock An algorithm for adaptive mesh refinement in {{\(n\)}} dimensions.
\newblock {\em Computing}, 59(2):115--137, 1997.
\newblock \href {https://doi.org/10.1007/BF02684475}
  {\path{doi:10.1007/BF02684475}}.

\bibitem{Whitney}
Hassler Whitney.
\newblock {\em Geometric integration theory}, volume~21 of {\em Princeton Math.
  Ser.}
\newblock Princeton University Press, Princeton, NJ, 1957.

\end{thebibliography}
%%%%%%%%%%%%%%%%%%%%%
\printonly{
\newpage
\section*{Selbstständigkeitserklärung}
Ich erkläre, dass ich die vorliegende Arbeit selbstständig verfasst und
noch nicht für andere Prüfungen eingereicht habe. Sämtliche Quellen,
einschließlich Internetquellen, die unverändert oder abgewandelt
wiedergegeben werden, insbesondere Quellen für Texte, Grafiken,
Tabellen und Bilder, sind als solche kenntlich gemacht. Mir ist bekannt,
dass bei Verstößen gegen diese Grundsätze ein Verfahren wegen
Täuschungsversuchs bzw.\ Täuschung eingeleitet wird.

\vspace{1.5cm}
\noindent\makebox[\linewidth]{\rule{\textwidth}{0.4pt}}
%Leipzig, \today\hspace{5 cm} Lukas Gehring
Ort, Datum\hspace*{\fill} Lukas Gehring
}
%%%%%%%%%%%%%%
\newpage
\phantomsection
\addcontentsline{toc}{section}{Index}

\newgeometry{margin=2.5cm}
\small
\begin{multicols}{3}
%\section*{Index}
\printsubindex*
\end{multicols}
\end{document}